\documentclass[11pt]{preprint}
\usepackage[full]{textcomp}
\usepackage[osf]{newtxtext} 
\usepackage[cal=boondoxo]{mathalfa}
\usepackage{colortbl}

\usepackage{comment}

\usepackage{amssymb}
\usepackage{mathtools}
\usepackage{hyperref}
\usepackage{breakurl}
\usepackage{mhenvs}
\usepackage{mhequ} 
\usepackage{mhsymb}
\usepackage{booktabs}
\usepackage{tikz}
\usepackage{tcolorbox}
\usepackage{mathrsfs}
\usepackage[utf8]{inputenc}
\usepackage{longtable}
\usepackage{wrapfig}
\usepackage{subcaption}
\usepackage{mathrsfs}
\usepackage{epsfig}
\usepackage{microtype}
\usepackage{comment}
\usepackage{wasysym}
\usepackage{centernot}
\usepackage{enumitem}
\usepackage{bm}
\usepackage{stackrel}
\usepackage{graphicx}

\makeatletter
\newcommand{\globalcolor}[1]{%
  \color{#1}\global\let\default@color\current@color
}
\makeatother

\usetikzlibrary{calc}
\usetikzlibrary{decorations}
\usetikzlibrary{positioning}
\usetikzlibrary{shapes}
\usetikzlibrary{external}

\definecolor{blush}{rgb}{0.87, 0.36, 0.51}
	\definecolor{brightcerulean}{rgb}{0.11, 0.67, 0.84}
	\definecolor{greenryb}{rgb}{0.4, 0.69, 0.2}

\newif\ifdark
\darkfalse

\ifdark
\definecolor{darkred}{rgb}{0.9,0.2,0.2}
\definecolor{darkblue}{rgb}{0.7,0.3,1}
\definecolor{darkgreen}{rgb}{0.1,0.9,0.1}
\definecolor{franck}{rgb}{0,0.8,1}
\definecolor{pagebackground}{rgb}{.15,.21,.18}
\definecolor{pageforeground}{rgb}{.84,.84,.85}
\pagecolor{pagebackground}
\AtBeginDocument{\globalcolor{pageforeground}}
\definecolor{symbols}{rgb}{0,0.7,1}
\colorlet{connection}{red!80!black}
\colorlet{boxcolor}{blue!50}

\else

\definecolor{darkred}{rgb}{0.7,0.1,0.1}
\definecolor{darkblue}{rgb}{0.4,0.1,0.8}
\definecolor{darkgreen}{rgb}{0.1,0.7,0.1}
\definecolor{franck}{rgb}{0,0,1}
\definecolor{pagebackground}{rgb}{1,1,1}
\definecolor{pageforeground}{rgb}{0,0,0}
\colorlet{symbols}{blue!90!black}
\colorlet{connection}{red!30!black}
\colorlet{boxcolor}{blue!50!black}

\fi

\def\slash{\leavevmode\unskip\kern0.18em/\penalty\exhyphenpenalty\kern0.18em}
\def\dash{\leavevmode\unskip\kern0.18em--\penalty\exhyphenpenalty\kern0.18em}

\DeclareMathAlphabet{\mathbbm}{U}{bbm}{m}{n}

\DeclareFontFamily{U}{BOONDOX-calo}{\skewchar\font=45 }
\DeclareFontShape{U}{BOONDOX-calo}{m}{n}{
  <-> s*[1.05] BOONDOX-r-calo}{}
\DeclareFontShape{U}{BOONDOX-calo}{b}{n}{
  <-> s*[1.05] BOONDOX-b-calo}{}
\DeclareMathAlphabet{\mcb}{U}{BOONDOX-calo}{m}{n}
\SetMathAlphabet{\mcb}{bold}{U}{BOONDOX-calo}{b}{n}

\setlist{noitemsep,topsep=4pt,leftmargin=1.5em}

\DeclareMathAlphabet{\mathbbm}{U}{bbm}{m}{n}

\DeclareMathAlphabet{\mcb}{U}{BOONDOX-calo}{m}{n}
\SetMathAlphabet{\mcb}{bold}{U}{BOONDOX-calo}{b}{n}
\DeclareFontFamily{U}{mathx}{\hyphenchar\font45}
\DeclareFontShape{U}{mathx}{m}{n}{
      <5> <6> <7> <8> <9> <10>
      <10.95> <12> <14.4> <17.28> <20.74> <24.88>
      mathx10
      }{}
\DeclareSymbolFont{mathx}{U}{mathx}{m}{n}
\DeclareMathSymbol{\bigtimes}{1}{mathx}{"91}

\def\emptyset{{\centernot\ocircle}}

\providecommand{\figures}{false}
{ \ifthenelse{\equal{\figures}{false}} {#1}{\[ {\rm Figure \ missing !} \]} }{}
\def\id{\mathrm{id}}

\newcommand{\cT}{{\mathcal T}}

\def\CH{\mathcal{H}}
\def\CP{\mathcal{P}}

\def\CA{\mathcal{A}}

\def\CC{\mathcal{C}}
\def\CQ{\mathcal{Q}}

\def\CM{\mathcal{M}}
\def\CT{\mathcal{T}}

\tikzstyle{tinydots}=[dash pattern=on \pgflinewidth off \pgflinewidth]
\tikzstyle{superdense}=[dash pattern=on 4pt off 1pt]

\newcommand{ \black }{\color{black} }
\newcommand{\beq}{\begin{equation}}
\newcommand{\eeq}{\end{equation}}
\newcommand{\nab}{\langle \nabla \rangle_{\frac{1}{\varepsilon}}}

\usepackage{empheq}

\newcommand*\widefbox[1]{\fbox{\hspace{2em}#1\hspace{2em}}}

\newcommand{\T}{\mathbf{T}}

\def\Labp{\mathfrak{p}}
\def\Labe{\mathfrak{e}}

\def\Labn{\mathfrak{n}}
\def\Labo{\mathfrak{f}}

\def\Labhom{\mathfrak{t}}

\def\Lab{\mathfrak{L}}
\def\Adm{\mathfrak{A}}

\def\Deltap{\Delta^{\!+}}

\def\${|\!|\!|}

\newcounter{theorem}

\newtheorem{ex}[theorem]{Example}

\newenvironment{DIFnomarkup}{}{} 

\newtheorem{assumption}{Assumption}
 
\theorembodyfont{\rmfamily}

\newfont{\indic}{bbmss12}

\def\Nabla_#1{\nabla_{\!#1}}

\makeatletter
\pgfdeclareshape{crosscircle}
{
  \inheritsavedanchors[from=circle] 
  \inheritanchorborder[from=circle]
  \inheritanchor[from=circle]{north}
  \inheritanchor[from=circle]{north west}
  \inheritanchor[from=circle]{north east}
  \inheritanchor[from=circle]{center}
  \inheritanchor[from=circle]{west}
  \inheritanchor[from=circle]{east}
  \inheritanchor[from=circle]{mid}
  \inheritanchor[from=circle]{mid west}
  \inheritanchor[from=circle]{mid east}
  \inheritanchor[from=circle]{base}
  \inheritanchor[from=circle]{base west}
  \inheritanchor[from=circle]{base east}
  \inheritanchor[from=circle]{south}
  \inheritanchor[from=circle]{south west}
  \inheritanchor[from=circle]{south east}
  \inheritbackgroundpath[from=circle]
  \foregroundpath{
    \centerpoint%
    \pgf@xc=\pgf@x%
    \pgf@yc=\pgf@y%
    \pgfutil@tempdima=\radius%
    \pgfmathsetlength{\pgf@xb}{\pgfkeysvalueof{/pgf/outer xsep}}%
    \pgfmathsetlength{\pgf@yb}{\pgfkeysvalueof{/pgf/outer ysep}}%
    \ifdim\pgf@xb<\pgf@yb%
      \advance\pgfutil@tempdima by-\pgf@yb%
    \else%
      \advance\pgfutil@tempdima by-\pgf@xb%
    \fi%
    \pgfpathmoveto{\pgfpointadd{\pgfqpoint{\pgf@xc}{\pgf@yc}}{\pgfqpoint{-0.707107\pgfutil@tempdima}{-0.707107\pgfutil@tempdima}}}
    \pgfpathlineto{\pgfpointadd{\pgfqpoint{\pgf@xc}{\pgf@yc}}{\pgfqpoint{0.707107\pgfutil@tempdima}{0.707107\pgfutil@tempdima}}}
    \pgfpathmoveto{\pgfpointadd{\pgfqpoint{\pgf@xc}{\pgf@yc}}{\pgfqpoint{-0.707107\pgfutil@tempdima}{0.707107\pgfutil@tempdima}}}
    \pgfpathlineto{\pgfpointadd{\pgfqpoint{\pgf@xc}{\pgf@yc}}{\pgfqpoint{0.707107\pgfutil@tempdima}{-0.707107\pgfutil@tempdima}}}
  }
}
\makeatother

\def\symbol#1{\textcolor{symbols}{#1}}

\def\decorate#1#2{
        \ifnum#2>0
    		\foreach \count in {1,...,#2}{
	       	let
				\p1 = (sourcenode.center),
                \p2 = (sourcenode.east),
				\n1 = {\x2-\x1},
				\n2 = {1mm},
				\n3 = {(1.3+0.6*(\count-1))*\n1},
				\n4 = {0.7*\n1}
			in 
        		node[rectangle,fill=symbols,rotate=30,inner sep=0pt,minimum width=0.2*\n2,minimum height=\n2] at ($(sourcenode.center) + (\n3,\n4)$) {}
				}
		\fi
        \ifnum#1>0
    		\foreach \count in {1,...,#1}{
	       	let
				\p1 = (sourcenode.center),
                \p2 = (sourcenode.east),
				\n1 = {\x2-\x1},
				\n2 = {1mm},
				\n3 = {(1.3+0.6*(\count-1))*\n1},
				\n4 = {0.7*\n1}
			in 
        		node[rectangle,fill=symbols,rotate=-30,inner sep=0pt,minimum width=0.2*\n2,minimum height=\n2] at ($(sourcenode.center) + (-\n3,\n4)$) {}
				}
		\fi
}

\tikzset{
    dectriangle/.style 2 args={
        triangle,
        alias=sourcenode,
        append after command={\decorate{#1}{#2}}
    },
    dectriangle/.default={0}{0},
}

\tikzset{
	cross/.style={path picture={ 
  		\draw[symbols]
			(path picture bounding box.south east) -- (path picture bounding box.north west) (path picture bounding box.south west) -- (path picture bounding box.north east);
		}},
root/.style={circle,fill=green!50!black,inner sep=0pt, minimum size=1.2mm},
        dot/.style={circle,fill=pageforeground,inner sep=0pt, minimum size=1mm},
        dotred/.style={circle,fill=pageforeground!50!pagebackground,inner sep=0pt, minimum size=2mm},
        var/.style={circle,fill=pageforeground!10!pagebackground,draw=pageforeground,inner sep=0pt, minimum size=3mm},
         var1/.style={circle,fill=pageforeground!10!pagebackground,draw=pageforeground,inner sep=0pt, minimum size=4mm},
        kernel/.style={semithick,shorten >=2pt,shorten <=2pt},
        kernels/.style={snake=zigzag,shorten >=2pt,shorten <=2pt,segment amplitude=1pt,segment length=4pt,line before snake=2pt,line after snake=5pt,},
        rho/.style={densely dashed,semithick,shorten >=2pt,shorten <=2pt},
           testfcn/.style={dotted,semithick,shorten >=2pt,shorten <=2pt},
        renorm/.style={shape=circle,fill=pagebackground,inner sep=1pt},
        labl/.style={shape=rectangle,fill=pagebackground,inner sep=1pt},
        xic/.style={very thin,circle,draw=symbols,fill=symbols,inner sep=0pt,minimum size=1.2mm},
        g/.style={very thin,rectangle,draw=symbols,fill=symbols!10!pagebackground,inner sep=0pt,minimum width=2.5mm,minimum height=1.2mm},
        xi/.style={very thin,circle,draw=symbols,fill=symbols!10!pagebackground,inner sep=0pt,minimum size=1.2mm},
	xies/.style={very thin,rectangle,fill=green!50!black!25,draw=symbols,inner sep=0pt,minimum size=1.1mm},
	xiesf/.style={very thin,rectangle,fill=green!50!black,draw=symbols,inner sep=0pt,minimum size=1.1mm},
        xix/.style={very thin,crosscircle,fill=symbols!10!pagebackground,draw=symbols,inner sep=0pt,minimum size=1.2mm},
        X/.style={very thin,cross,rectangle,fill=pagebackground,draw=symbols,inner sep=0pt,minimum size=1.2mm},
	xib/.style={thin,circle,fill=symbols!10!pagebackground,draw=symbols,inner sep=0pt,minimum size=1.6mm},
	xie/.style={thin,circle,fill=green!50!black,draw=symbols,inner sep=0pt,minimum size=1.6mm},
	xid/.style={thin,circle,fill=symbols,draw=symbols,inner sep=0pt,minimum size=1.6mm},
	xibx/.style={thin,crosscircle,fill=symbols!10!pagebackground,draw=symbols,inner sep=0pt,minimum size=1.6mm},
	kernels2/.style={very thick,draw=connection,segment length=12pt},
	keps/.style={thin,draw=symbols,->},
	kepspr/.style={thick,draw=connection,->},
	krho/.style={thin,draw=symbols,superdense,->},
	krhopr/.style={thick,draw=connection,superdense},
	triangle/.style = { regular polygon, regular polygon sides=3},
	not/.style={thin,circle,draw=connection,fill=connection,inner sep=0pt,minimum size=0.5mm},
	diff/.style = {very thin,draw=symbols,triangle,fill=red!50!black,inner sep=0pt,minimum size=1.6mm},
	diff1/.style = {very thin,dectriangle={1}{0},fill=red!50!black,draw=symbols,inner sep=0pt,minimum size=1.6mm},
	diff2/.style = {very thin,dectriangle={1}{1},fill=red!50!black,draw=symbols,inner sep=0pt,minimum size=1.6mm},
		diffmini/.style = {very thin,rectangle,fill=black,draw=black,inner sep=0pt,minimum size=0.75mm},
	 kernelsmod/.style={very thick,draw=connection,segment length=12pt},
	 rec/.style = {very thin,rectangle,fill=black,draw=black,inner sep=0pt,minimum size=2mm},
	cerc/.style={very thin,circle,draw=black,fill=symbols,inner sep=0pt,minimum size=2mm},
	stars/.style={very thin,star,star points=6,star point ratio=0.5, draw=black,fill=red,inner sep=0pt,minimum size=0.7mm},
	>=stealth,
        }
        \tikzset{
root/.style={circle,fill=black!50,inner sep=0pt, minimum size=3mm},
        circ/.style={circle,fill=white,draw=black,very thin,inner sep=.5pt, minimum size=1.2mm},
        round1/.style={fill=white,outer sep = 0,inner sep=2pt,rounded corners=1mm,draw,text=black,thin,minimum size=1.2mm},
          circ1/.style={circle,fill=red!10,draw=red,very thin,inner sep=.5pt, minimum size=1.2mm},
        rect/.style={fill=white,outer sep = 0,inner sep=2pt,rectangle,draw,text=black,thin,minimum size=1.2mm},
        rect1/.style={fill=white,outer sep = 0,inner sep=2pt,rectangle,draw,text=black,thin,minimum size=1.2mm},
        round2/.style={fill=red!10,outer sep = 0,inner sep=2pt,rounded corners=1mm,draw,text=black,thin,minimum size=1.2mm},
       round3/.style={fill=blue!10,outer sep = 0,inner sep=2pt,rounded corners=1mm,draw,text=black,thin,minimum size=1.2mm}, 
        rect2/.style={fill=black!10,outer sep = 0,inner sep=2pt,rectangle,draw,text=black,thin,minimum size=1.2mm},
        dot/.style={circle,fill=black,inner sep=0pt, minimum size=1.2mm},
        dotred/.style={circle,fill=black!50,inner sep=0pt, minimum size=2mm},
        var/.style={circle,fill=black!10,draw=black,inner sep=0pt, minimum size=3mm},
        kernel/.style={semithick,shorten >=2pt,shorten <=2pt},
         diag/.style={thin,shorten >=4pt,shorten <=4pt},
        kernel1/.style={thick},
        kernels/.style={snake=zigzag,shorten >=2pt,shorten <=2pt,segment amplitude=1pt,segment length=4pt,line before snake=2pt,line after snake=5pt,},
		kernels1/.style={snake=zigzag,segment amplitude=0.5pt,segment length=2pt},
		rho1/.style={densely dotted,semithick},
        rho/.style={densely dashed,semithick,shorten >=2pt,shorten <=2pt},
           testfcn/.style={dotted,semithick,shorten >=2pt,shorten <=2pt},
           visible/.style={draw, circle, fill, inner sep=0.25ex},
        renorm/.style={shape=circle,fill=white,inner sep=1pt},
        labl/.style={shape=rectangle,fill=white,inner sep=1pt},
        xic/.style={very thin,circle,fill=symbols,draw=black,inner sep=0pt,minimum size=1.2mm},
        xi/.style={very thin,circle,fill=blue!10,draw=black,inner sep=0pt,minimum size=1.2mm},
	xib/.style={very thin,circle,fill=blue!10,draw=black,inner sep=0pt,minimum size=1.6mm},
	xie/.style={very thin,circle,fill=green!50!black,draw=black,inner sep=0pt,minimum size=1mm},
	xid/.style={very thin,circle,fill=symbols,draw=black,inner sep=0pt,minimum size=1.6mm},
	edgetype/.style={very thin,circle,draw=black,inner sep=0pt,minimum size=5mm},
	nodetype/.style={very thick,circle,draw=black,inner sep=0pt,minimum size=5mm},
	kernels2/.style={very thick,draw=connection,segment length=12pt},
clean/.style={thin,circle,fill=black,inner sep=0pt,minimum size=1mm},	not/.style={thin,circle,fill=symbols,draw=connection,fill=connection,inner sep=0pt,minimum size=0.8mm},
	>=stealth,
        }

\makeatletter
\def\DeclareSymbol#1#2#3{%
	\expandafter\gdef\csname MH@symb@#1\endcsname{\tikzsetnextfilename{symbol#1}%
	\tikz[baseline=#2,scale=0.15,draw=symbols,line join=round]{#3}}%
	\expandafter\gdef\csname MH@symb@#1s\endcsname{\scalebox{0.75}{\tikzsetnextfilename{symbol#1}%
	\tikz[baseline=#2,scale=0.15,draw=symbols,line join=round]{#3}}}%
	\expandafter\gdef\csname MH@symb@#1ss\endcsname{\scalebox{0.65}{\tikzsetnextfilename{symbol#1}%
	\tikz[baseline=#2,scale=0.15,draw=symbols,line join=round]{#3}}}%
	}
\def\<#1>{\ifthenelse{\boolean{mmode}}{\mathchoice{\csname MH@symb@#1\endcsname}{\csname MH@symb@#1\endcsname}{\csname MH@symb@#1s\endcsname}{\csname MH@symb@#1ss\endcsname}}{\csname MH@symb@#1\endcsname}}
\makeatother

\DeclareSymbol{Xi22}{0.5}{\draw (0,0) node[xi] {} -- (-1,1) node[not] {} -- (0,2) node[xi] {};} 
\DeclareSymbol{Xi2}{-2}{\draw (-1,-0.25) node[xi] {} -- (0,1) node[xi] {};} 
\DeclareSymbol{Xi2b}{-2}{\draw (-1,-0.25) node[xic] {} -- (0,1) node[xic] {};} 
\DeclareSymbol{Xi2g}{-2}{\draw (-1,-0.25) node[xies] {} -- (0,1) node[xi] {};} 
\DeclareSymbol{Xi2g2}{-2}{\draw (-1,-0.25) node[xi] {} -- (0,1) node[xies] {};} 
\DeclareSymbol{cXi2}{-2}{\draw (0,-0.25) node[xi] {} -- (-1,1) node[xic] {};}
\DeclareSymbol{Xi3}{0}{\draw (0,0) node[xi] {} -- (-1,1) node[xi] {} -- (0,2) node[xi] {};}
\DeclareSymbol{XiIIXi}{0}{\draw (0,0) node[xi] {} -- (-1,1); \draw[kernels2] (-1,1) node[not] {} -- (0,2) node[xi] {};}

\DeclareSymbol{Xi4}{2}{\draw (0,0) node[xi] {} -- (-1,1) node[xi] {} -- (0,2) node[xi] {} -- (-1,3) node[xi] {};}
\DeclareSymbol{Xi4_1}{2}{\draw (0,0) node[xic] {} -- (-1,1) node[xic] {} -- (0,2) node[xi] {} -- (-1,3) node[xi] {};}
\DeclareSymbol{Xi4_2}{2}{\draw (0,0) node[xic] {} -- (-1,1) node[xi] {} -- (0,2) node[xi] {} -- (-1,3) node[xic] {};}
\DeclareSymbol{Xi2X}{-2}{\draw (0,-0.25) node[xi] {} -- (-1,1) node[xix] {};}
\DeclareSymbol{XXi2}{-2}{\draw (0,-0.25) node[xix] {} -- (-1,1) node[xi] {};}
\DeclareSymbol{IIXi}{0}{\draw (0,-0.25) node[not] {} -- (-1,1) node[xi] {} -- (0,2) node[xi] {};}
\DeclareSymbol{IXi^2}{-1}{\draw (-1,1) node[xi] {} -- (0,0) node[not] {} -- (1,1) node[xi] {};}
\DeclareSymbol{IIXi^2}{-4}{\draw (0,-1.5) node[not] {} -- (0,0);
\draw[kernels2] (-1,1) node[xi] {} -- (0,0) node[not] {} -- (1,1) node[xi] {};}
\DeclareSymbol{XiX}{-2.8}{\node[xibx] {};}
\DeclareSymbol{tauX}{-2.8}{ \node[X] {};}
\DeclareSymbol{Xi}{-2.8}{\node[xib] {};}

\DeclareSymbol{IXiX}{-1}{\draw (0,-0.25) node[not] {} -- (-1,1) node[xix] {};}
\DeclareSymbol{IXi3}{2}{\draw (0,-0.25) node[not] {} -- (-1,1) node[xi] {} -- (0,2) node[xi] {} -- (-1,3) node[xi] {};}
\DeclareSymbol{IXi}{-2}{\draw (0,-0.25) node[not] {} -- (-1,1) node[xi] {};}
\DeclareSymbol{XiI}{-2}{\draw (0,-0.25) node[xi] {} -- (-1,1) node[not] {};}

\DeclareSymbol{Xi4b}{0}{\draw(0,1.5) node[xi] {} -- (0,0); \draw (-1,1) node[xi] {} -- (0,0) node[xi] {} -- (1,1) node[xi] {};}
\DeclareSymbol{Xi4b'}{0}{\draw(0,1.5) node[xi] {} -- (0,-0.2); \draw (-1,1) node[xi] {} -- (0,-0.2) node[not] {} -- (1,1) node[xi] {};}
\DeclareSymbol{Xi4c}{0}{\draw (0,1) -- (0.8,2.2) node[xi] {};\draw (0,-0.25) node[xi] {} -- (0,1) node[xi] {} -- (-0.8,2.2) node[xi] {};}
\DeclareSymbol{Xi4d}{-4.5}{\draw (0,-1.5) node[not] {} -- (0,0); \draw (-1,1) node[xi] {} -- (0,0) node[xi] {} -- (1,1) node[xi] {};}
\DeclareSymbol{Xi4e}{0}{\draw (0,2) node[xi] {} -- (-1,1) node[xi] {} -- (0,0) node[xi] {} -- (1,1) node[xi] {};}
\DeclareSymbol{Xi4e'}{0}{\draw (0,2) node[xi] {} -- (-1,1) node[xi] {} -- (0,-0.2) node[not] {} -- (1,1) node[xi] {};}

\DeclareSymbol{Xitwo}
{0}{\draw[kernels2] (0,0) node[not] {} -- (-1,1) node[not] {}
-- (-2,2) node[not]{} -- (-3,3) node[xi]  {};
\draw[kernels2] (0,0) -- (1,1) node[xi] {};
\draw[kernels2] (-1,1) -- (0,2) node[xi] {};
\draw[kernels2] (-2,2) -- (-1,3) node[xi] {};}

\DeclareSymbol{IXitwo}
{0}{\draw (-.7,1.2) node[xi] {} -- (0,-0.2) -- (.7,1.2) node[xi] {};}
\DeclareSymbol{I1Xitwo}
{0}{\draw[kernels2] (0,0) node[not] {} -- (-1,1) node[xi] {};
\draw[kernels2] (0,0) -- (1,1) node[xi] {};}

\DeclareSymbol{I1Xitwobis}
{0}{\draw[kernels2] (0,0) node[not] {} -- (-1,1) node[xies] {};
\draw[kernels2] (0,0) -- (1,1) node[xies] {};}

\DeclareSymbol{I1Xitwog}
{0}{\draw[kernels2] (0,0) node[not] {} -- (-1,1) node[xies] {};
\draw[kernels2] (0,0) -- (1,1) node[xi] {};}

\DeclareSymbol{cI1Xitwo}
{0}{\draw[kernels2] (0,0) node[not] {} -- (-1,1) node[xic] {};
\draw[kernels2] (0,0) -- (1,1) node[xi] {};}

\DeclareSymbol{I1IXi3}{0}{\draw (0,0) node[xi] {} -- (-1,1) ; 
\draw[kernels2] (-1,1) node[not] {} -- (0,2) node[xi] {};
\draw[kernels2] (-1,1) node[not] {} -- (-2,2) node[xi] {};}

\DeclareSymbol{I1Xi3c}{-1}{\draw[kernels2](0,1.5) node[xi] {} -- (0,0) node[not] {}; \draw (-1,1) node[xi] {} -- (0,0) ; \draw[kernels2] (0,0) -- (1,1) node[xi] {};}

\DeclareSymbol{I1Xi3cbis}{-1}{\draw[kernels2](0,1.5) node[xies] {} -- (0,0) node[not] {}; \draw (-1,1) node[xies] {} -- (0,0) ; \draw[kernels2] (0,0) -- (1,1) node[xies] {};}

\DeclareSymbol{I1IXi3b}{0}{\draw[kernels2] (0,0) node[not] {} -- (-1,1) ; \draw[kernels2] (0,0)   -- (1,1) node[xi] {} ;
\draw (-1,1) node[xi] {} -- (0,2) node[xi] {};
}

\DeclareSymbol{I1IXi3c}{0}{\draw[kernels2] (0,0) node[not] {} -- (-1,1) ; \draw[kernels2] (0,0)   -- (1,1) node[xi] {} ;
\draw[kernels2] (-1,1) node[not] {} -- (0,2) node[xi] {};
\draw[kernels2] (-1,1) node[not] {} -- (-2,2) node[xi] {};}

\DeclareSymbol{I1IXi3cbis}{0}{\draw[kernels2] (0,0) node[not] {} -- (-1,1) ; \draw[kernels2] (0,0)   -- (1,1) node[xies] {} ;
\draw[kernels2] (-1,1) node[not] {} -- (0,2) node[xies] {};
\draw[kernels2] (-1,1) node[not] {} -- (-2,2) node[xies] {};}

\DeclareSymbol{I1Xi}{0}{\draw[kernels2] (0,0) node[not] {} -- (-1,1)  node[xi] {} ;}

\DeclareSymbol{I1Xi4a}{2}{\draw[kernels2] (0,0) node[not] {} -- (-1,1) ; \draw[kernels2] (0,0) node[not] {} -- (1,1) node[xi] {} ;
\draw (-1,1) node[xi] {} -- (0,2) node[xi] {} -- (-1,3) node[xi] {};}

\DeclareSymbol{cI1Xi4a}{2}{\draw[kernels2] (0,0) node[not] {} -- (-1,1) ; \draw[kernels2] (0,0) node[not] {} -- (1,1) node[xic] {} ;
\draw (-1,1) node[xic] {} -- (0,2) node[xi] {} -- (-1,3) node[xi] {};}

\DeclareSymbol{I1Xi4b}{2}{\draw (0,0) node[xi] {} -- (-1,1) node[xi] {} -- (0,2) ; \draw[kernels2] (0,2) node[not] {} -- (-1,3) node[xi] {};\draw[kernels2] (0,2)  -- (1,3) node[xi] {};
}

\DeclareSymbol{cI1Xi4b}{2}{\draw (0,0) node[xic] {} -- (-1,1) node[xic] {} -- (0,2) ; \draw[kernels2] (0,2) node[not] {} -- (-1,3) node[xi] {};\draw[kernels2] (0,2)  -- (1,3) node[xi] {};
}

\DeclareSymbol{I1Xi4c}{2}{\draw (0,0) node[xi] {} -- (-1,1) node[not] {}; \draw[kernels2] (-1,1) -- (0,2) ; 
\draw[kernels2] (-1,1) -- (-2,2) node[xi] {} ;
\draw (0,2) node[xi] {} -- (-1,3) node[xi] {};}

\DeclareSymbol{cI1Xi4c}{2}{\draw (0,0) node[xic] {} -- (-1,1) node[not] {}; \draw[kernels2] (-1,1) -- (0,2) ; 
\draw[kernels2] (-1,1) -- (-2,2) node[xic] {} ;
\draw (0,2) node[xi] {} -- (-1,3) node[xi] {};}

\DeclareSymbol{I1Xi4ab}{2}{\draw[kernels2] (0,0) node[not] {} -- (-1,1) ; \draw[kernels2] (0,0) node[not] {} -- (1,1) node[xi] {};\draw (-1,1) node[xi] {} -- (0,2) ; \draw[kernels2] (0,2) node[not] {} -- (-1,3) node[xi] {};\draw[kernels2] (0,2)  -- (1,3) node[xi] {}; }

\DeclareSymbol{cI1Xi4ab}{2}{\draw[kernels2] (0,0) node[not] {} -- (-1,1) ; \draw[kernels2] (0,0) node[not] {} -- (1,1) node[xic] {};\draw (-1,1) node[xic] {} -- (0,2) ; \draw[kernels2] (0,2) node[not] {} -- (-1,3) node[xi] {};\draw[kernels2] (0,2)  -- (1,3) node[xi] {}; }

\DeclareSymbol{I1Xi4bc}{2}{\draw (0,0) node[xi] {} -- (-1,1) node[not] {}; \draw[kernels2] (-1,1) -- (0,2) ; 
\draw[kernels2] (-1,1) -- (-2,2) node[xi] {} ; \draw[kernels2] (0,2) node[not] {} -- (-1,3) node[xi] {};\draw[kernels2] (0,2)  -- (1,3) node[xi] {};
}

\DeclareSymbol{cI1Xi4bc}{2}{\draw (0,0) node[xic] {} -- (-1,1) node[not] {}; \draw[kernels2] (-1,1) -- (0,2) ; 
\draw[kernels2] (-1,1) -- (-2,2) node[xic] {} ; \draw[kernels2] (0,2) node[not] {} -- (-1,3) node[xi] {};\draw[kernels2] (0,2)  -- (1,3) node[xi] {};
}

\DeclareSymbol{I1Xi4abcc1}{2}{\draw[kernels2] (0,0) node[not] {} -- (-1,1) node[not] {}
-- (-2,2) node[not]{} -- (-3,3) node[xic]  {};
\draw[kernels2] (0,0) -- (1,1) node[xic] {};
\draw[kernels2] (-1,1) -- (0,2) node[xi] {};
\draw[kernels2] (-2,2) -- (-1,3) node[xi] {};
}

\DeclareSymbol{I1Xi4abcc1b}{2}{\draw[kernels2] (0,0) node[not] {} -- (-1,1) node[not] {}
-- (-2,2) node[not]{} -- (-3,3) node[xi]  {};
\draw[kernels2] (0,0) -- (1,1) node[xic] {};
\draw[kernels2] (-1,1) -- (0,2) node[xic] {};
\draw[kernels2] (-2,2) -- (-1,3) node[xi] {};
}

\DeclareSymbol{I1Xi4abcc2}{2}{\draw[kernels2] (0,0) node[not] {} -- (-1,1) node[not] {}
-- (-2,2) node[not]{} -- (-3,3) node[xic]  {};
\draw[kernels2] (0,0) -- (1,1) node[xi] {};
\draw[kernels2] (-1,1) -- (0,2) node[xi] {};
\draw[kernels2] (-2,2) -- (-1,3) node[xic] {};
}

\DeclareSymbol{I1Xi4ac}{2}{\draw[kernels2] (0,0) node[not] {} -- (-1,1) ; \draw[kernels2] (0,0) node[not] {} -- (1,1) node[xi] {}; 
\draw[kernels2] (-1,1) node[not] {} -- (0,2) ; 
\draw[kernels2] (-1,1) -- (-2,2) node[xi] {} ;
\draw (0,2) node[xi] {} -- (-1,3) node[xi] {};}

\DeclareSymbol{cI1Xi4ac}{2}{\draw[kernels2] (0,0) node[not] {} -- (-1,1) ; \draw[kernels2] (0,0) node[not] {} -- (1,1) node[xic] {}; 
\draw[kernels2] (-1,1) node[not] {} -- (0,2) ; 
\draw[kernels2] (-1,1) -- (-2,2) node[xic] {} ;
\draw (0,2) node[xi] {} -- (-1,3) node[xi] {};}

\DeclareSymbol{I1Xi4acc1}{2}{\draw[kernels2] (0,0) node[not] {} -- (-1,1) ; \draw[kernels2] (0,0) node[not] {} -- (1,1) node[xic] {}; 
\draw[kernels2] (-1,1) node[not] {} -- (0,2) ; 
\draw[kernels2] (-1,1) -- (-2,2) node[xi] {} ;
\draw (0,2) node[xic] {} -- (-1,3) node[xi] {};}

\DeclareSymbol{I1Xi4acc2}{2}{\draw[kernels2] (0,0) node[not] {} -- (-1,1) ; \draw[kernels2] (0,0) node[not] {} -- (1,1) node[xic] {}; 
\draw[kernels2] (-1,1) node[not] {} -- (0,2) ; 
\draw[kernels2] (-1,1) -- (-2,2) node[xi] {} ;
\draw (0,2) node[xi] {} -- (-1,3) node[xic] {};}

\DeclareSymbol{2I1Xi4}{2}{\draw[kernels2] (0,0) node[not] {} -- (-1,1) node[not] {};
\draw[kernels2] (0,0) -- (1,1) node[not] {};
\draw[kernels2] (-1,1) -- (-1.5,2.5) node[xi] {};
\draw[kernels2] (-1,1) -- (-0.5,2.5) node[xi] {};
\draw[kernels2] (1,1) -- (0.5,2.5) node[xi] {};
\draw[kernels2] (1,1) -- (1.5,2.5) node[xi] {};
}

\DeclareSymbol{2I1Xi4dis}{2}{\draw[kernels2] (0,0) node[not] {} -- (-1,1) node[not] {};
\draw[kernels2] (0,0) -- (1,1) node[not] {};
\draw[kernels2] (-1,1) -- (-1.5,2.5) node[xies] {};
\draw[kernels2] (-1,1) -- (-0.5,2.5) node[xies] {};
\draw[kernels2] (1,1) -- (0.5,2.5) node[xies] {};
\draw[kernels2] (1,1) -- (1.5,2.5) node[xies] {};
}

\DeclareSymbol{2I1Xi4c1}{2}{\draw[kernels2] (0,0) node[not] {} -- (-1,1) node[not] {};
\draw[kernels2] (0,0) -- (1,1) node[not] {};
\draw[kernels2] (-1,1) -- (-1.5,2.5) node[xic] {};
\draw[kernels2] (-1,1) -- (-0.5,2.5) node[xi] {};
\draw[kernels2] (1,1) -- (0.5,2.5) node[xic] {};
\draw[kernels2] (1,1) -- (1.5,2.5) node[xi] {};
}

\DeclareSymbol{2I1Xi4c2}{2}{\draw[kernels2] (0,0) node[not] {} -- (-1,1) node[not] {};
\draw[kernels2] (0,0) -- (1,1) node[not] {};
\draw[kernels2] (-1,1) -- (-1.5,2.5) node[xic] {};
\draw[kernels2] (-1,1) -- (-0.5,2.5) node[xic] {};
\draw[kernels2] (1,1) -- (0.5,2.5) node[xi] {};
\draw[kernels2] (1,1) -- (1.5,2.5) node[xi] {};
}

\DeclareSymbol{2I1Xi4b}{2}{\draw[kernels2] (0,0) node[not] {} -- (-1,1) ;
\draw[kernels2] (0,0) -- (1,1);
\draw (-1,1) node[xi] {} -- (-1,2.5) node[xi] {};
\draw (1,1)  node[xi] {} -- (1,2.5) node[xi] {};
}

\DeclareSymbol{2I1Xi4bb}{2}{\draw[kernels2] (0,0) node[not] {} -- (-1,1) ;
\draw[kernels2] (0,0) -- (1,1);
\draw (-1,1) node[xi] {} -- (-1,2.5) node[xiesf] {};
\draw (1,1)  node[xi] {} -- (1,2.5) node[xic] {};
}

\DeclareSymbol{2I1Xi4c}{2}{\draw[kernels2] (0,0) node[not] {} -- (-1,1);
\draw[kernels2] (0,0) -- (1,1) node[not] {};
\draw (-1,1)  node[xi] {} -- (-1,2.5) node[xi] {};
\draw[kernels2] (1,1) -- (0.4,2.5) node[xi] {};
\draw[kernels2] (1,1) -- (1.6,2.5) node[xi] {};
}

\DeclareSymbol{2I1Xi4cc1}{2}{\draw[kernels2] (0,0) node[not] {} -- (-1,1);
\draw[kernels2] (0,0) -- (1,1) node[not] {};
\draw (-1,1)  node[xic] {} -- (-1,2.5) node[xi] {};
\draw[kernels2] (1,1) -- (0.4,2.5) node[xic] {};
\draw[kernels2] (1,1) -- (1.6,2.5) node[xi] {};
}

\DeclareSymbol{2I1Xi4cc2}{2}{\draw[kernels2] (0,0) node[not] {} -- (-1,1);
\draw[kernels2] (0,0) -- (1,1) node[not] {};
\draw (-1,1)  node[xic] {} -- (-1,2.5) node[xic] {};
\draw[kernels2] (1,1) -- (0.4,2.5) node[xi] {};
\draw[kernels2] (1,1) -- (1.6,2.5) node[xi] {};
}

\DeclareSymbol{Xi4ba}{0}{\draw(-0.5,1.5) node[xi] {} -- (0,0); \draw (-1.5,1) node[xi] {} -- (0,0) node[not] {}; \draw[kernels2] (0,0) -- (1.5,1) node[xi] {};
\draw[kernels2] (0,0) -- (0.5,1.5) node[xi] {} ;}

\DeclareSymbol{Xi4badis}{0}{\draw(-0.5,1.5) node[xies] {} -- (0,0); \draw (-1.5,1) node[xies] {} -- (0,0) node[not] {}; \draw[kernels2] (0,0) -- (1.5,1) node[xies] {};
\draw[kernels2] (0,0) -- (0.5,1.5) node[xies] {} ;}

\DeclareSymbol{Xi4ba1}{0}{\draw(-0.5,1.5) node[xi] {} -- (0,0); \draw (-1.5,1) node[xi] {} -- (0,0) node[not] {}; \draw[kernels2] (0,0) -- (1.5,1) node[xic] {};
\draw[kernels2] (0,0) -- (0.5,1.5) node[xic] {} ;}

\DeclareSymbol{Xi4ba1b}{0}{\draw(-0.5,1.5) node[xic] {} -- (0,0); \draw (-1.5,1) node[xic] {} -- (0,0) node[not] {}; \draw[kernels2] (0,0) -- (1.5,1) node[xi] {};
\draw[kernels2] (0,0) -- (0.5,1.5) node[xi] {} ;}

\DeclareSymbol{Xi4ba1bdiff}{0}{\draw(-0.5,1.5) node[xic] {} -- (0,0); \draw (-1.5,1) node[xic] {} -- (0,0) node[not] {}; \draw (0,0) -- (1.5,1) node[xi] {};
\draw (0,0) -- (0.5,1.5) node[xi] {};
\draw(0,0) node[diff] {};}

\DeclareSymbol{Xi4ba1bb}{0}{\draw(-0.5,1.5) node[xic] {} -- (0,0); \draw (-1.5,1) node[xiesf] {} -- (0,0) node[not] {}; \draw[kernels2] (0,0) -- (1.5,1) node[xi] {};
\draw[kernels2] (0,0) -- (0.5,1.5) node[xi] {} ;}

\DeclareSymbol{Xi4ba2}{0}{\draw(-0.5,1.5) node[xi] {} -- (0,0); \draw (-1.5,1) node[xic] {} -- (0,0) node[not] {}; \draw[kernels2] (0,0) -- (1.5,1) node[xi] {};
\draw[kernels2] (0,0) -- (0.5,1.5) node[xic] {} ;}

\DeclareSymbol{Xi4ba2b}{0}{\draw(-0.5,1.5) node[xi] {} -- (0,0); \draw (-1.5,1) node[xic] {} -- (0,0) node[not] {}; \draw[kernels2] (0,0) -- (1.5,1) node[xi] {};
\draw[kernels2] (0,0) -- (0.5,1.5) node[xiesf] {} ;}

\DeclareSymbol{Xi4ca}{0}{\draw (0,1) -- (-1,2.2) node[xi] {};\draw (0,-0.25) node[xi] {} -- (0,1) ; \draw[kernels2] (0,1) node[not] {} -- (1,2.2) node[xi] {};
\draw[kernels2] (0,1) {} -- (0,2.7) node[xi] {};
}

\DeclareSymbol{Xi4cb}{0}{\draw (-1,1) -- (-2,2) node[xi] {};\draw[kernels2] (0,0)  -- (-1,1) node[xi] {} ; \draw[kernels2] (0,0) node[not] {} -- (1,1) node[xi] {} ; 
\draw (-1,1) node[xi] {} -- (0,2) node[xi] {};}

\DeclareSymbol{Xi4cbb}{0}{\draw (-1,1) -- (-2,2) node[xiesf] {};\draw[kernels2] (0,0)  -- (-1,1) node[xi] {} ; \draw[kernels2] (0,0) node[not] {} -- (1,1) node[xi] {} ; 
\draw (-1,1) node[xi] {} -- (0,2) node[xic] {};}

\DeclareSymbol{Xi4cbc1}{0}{\draw (-1,1) -- (-2,2) node[xic] {};\draw[kernels2] (0,0)  -- (-1,1) node[xic] {} ; \draw[kernels2] (0,0) node[not] {} -- (1,1) node[xi] {} ; 
\draw (-1,1) node[xic] {} -- (0,2) node[xi] {};}

\DeclareSymbol{Xi4cbc2}{0}{\draw (-1,1) -- (-2,2) node[xi] {};\draw[kernels2] (0,0)  -- (-1,1) node[xi] {} ; \draw[kernels2] (0,0) node[not] {} -- (1,1) node[xic] {} ; 
\draw (-1,1) node[xic] {} -- (0,2) node[xi] {};}

\DeclareSymbol{Xi4cab}{0}{\draw (-1,1) -- (-2,2) node[xi] {};\draw[kernels2] (0,0)  -- (-1,1); \draw[kernels2] (0,0) node[not] {} -- (1,1) node[xi] {} ; 
\draw[kernels2] (-1,1)  {} -- (0,2) node[xi] {};
\draw[kernels2] (-1,1) node[not] {} -- (-1,2.5) node[xi] {};
}

\DeclareSymbol{Xi4cabdis}{0}{\draw (-1,1) -- (-2,2) node[xies] {};\draw[kernels2] (0,0)  -- (-1,1); \draw[kernels2] (0,0) node[not] {} -- (1,1) node[xies] {} ; 
\draw[kernels2] (-1,1)  {} -- (0,2) node[xies] {};
\draw[kernels2] (-1,1) node[not] {} -- (-1,2.5) node[xies] {};
}

\DeclareSymbol{Xi4cabc1}{0}{\draw (-1,1) -- (-2,2) node[xi] {};\draw[kernels2] (0,0)  -- (-1,1); \draw[kernels2] (0,0) node[not] {} -- (1,1) node[xic] {} ; 
\draw[kernels2] (-1,1)  {} -- (0,2) node[xic] {};
\draw[kernels2] (-1,1) node[not] {} -- (-1,2.5) node[xi] {};
}

\DeclareSymbol{Xi4cabc2}{0}{\draw (-1,1) -- (-2,2) node[xic] {};\draw[kernels2] (0,0)  -- (-1,1); \draw[kernels2] (0,0) node[not] {} -- (1,1) node[xic] {} ; 
\draw[kernels2] (-1,1)  {} -- (0,2) node[xi] {};
\draw[kernels2] (-1,1) node[not] {} -- (-1,2.5) node[xi] {};
}

\DeclareSymbol{Xi4ea}{1.5}{\draw (-1,2.5) node[xi] {} -- (-1,1) node[xi] {} -- (0,0); 
 \draw[kernels2] (0,0)  -- (1,1) node[xi] {};
\draw[kernels2] (0,0) node[not] {} -- (0,1.5) node[xi] {}; }

\DeclareSymbol{Xi4eac1}{1.5}{\draw (-1,2.5) node[xic] {} -- (-1,1) node[xi] {} -- (0,0); 
 \draw[kernels2] (0,0)  -- (1,1) node[xic] {};
\draw[kernels2] (0,0) node[not] {} -- (0,1.5) node[xi] {}; }

\DeclareSymbol{Xi4eac1b}{1.5}{\draw (-1,2.5) node[xic] {} -- (-1,1) node[xi] {} -- (0,0); 
 \draw[kernels2] (0,0)  -- (1,1) node[xiesf] {};
\draw[kernels2] (0,0) node[not] {} -- (0,1.5) node[xi] {}; }

\DeclareSymbol{Xi4eac2}{1.5}{\draw (-1,2.5) node[xic] {} -- (-1,1) node[xic] {} -- (0,0); 
 \draw[kernels2] (0,0)  -- (1,1) node[xi] {};
\draw[kernels2] (0,0) node[not] {} -- (0,1.5) node[xi] {}; }

\DeclareSymbol{Xi4eact1}{1.5}{\draw (-1,2.5) node[xic] {} -- (-1,1) node[xi] {} -- (0,0); 
 \draw (0,0)  -- (1,1) node[xic] {};
\draw[rho] (0,0) node[not] {} -- (0,1.5) node[xi] {}; }

\DeclareSymbol{Xi4eact2}{1.5}{\draw[rho] (-1,2.5) node[xic] {} -- (-1,1) node[xi] {} -- (0,0); 
 \draw (0,0)  -- (1,1) node[xic] {};
\draw (0,0) node[not] {} -- (0,1.5) node[xi] {}; }

\DeclareSymbol{Xi4eabis}{1.5}{\draw (-1,2.5) node[xi] {} -- (-1,1) ; \draw[kernels2] (-1,1) node[xi] {} -- (0,0); 
 \draw (0,0)  -- (1,1) node[xi] {};
\draw[kernels2] (0,0) node[not] {} -- (0,1.5) node[xi] {}; }

\DeclareSymbol{Xi4eabisc1}{1.5}{\draw (-1,2.5) node[xic] {} -- (-1,1) ; \draw[kernels2] (-1,1) node[xi] {} -- (0,0); 
 \draw (0,0)  -- (1,1) node[xi] {};
\draw[kernels2] (0,0) node[not] {} -- (0,1.5) node[xic] {}; }

\DeclareSymbol{Xi4eabisc1b}{1.5}{\draw (-1,2.5) node[xic] {} -- (-1,1) ; \draw[kernels2] (-1,1) node[xi] {} -- (0,0); 
 \draw (0,0)  -- (1,1) node[xi] {};
\draw[kernels2] (0,0) node[not] {} -- (0,1.5) node[xiesf] {}; }

\DeclareSymbol{Xi4eabisc1bis}{1.5}{\draw (-1,2.5) node[xi] {} -- (-1,1) ; \draw[kernels2] (-1,1) node[xi] {} -- (0,0); 
 \draw (0,0)  -- (1,1) node[xi] {};
\draw[kernels2] (0,0) node[not] {} -- (0,1.5) node[xi] {};
\draw (-2,1) node[] {\tiny{$i$}};
\draw (-2,2.5) node[] {\tiny{$\ell$}};
\draw (2,1) node[] {\tiny{$k$}};
\draw (0,2.5) node[] {\tiny{$j$}};
 }

\DeclareSymbol{Xi4eabisc1tris}{1.5}{\draw (-1,2.5) node[xi] {} -- (-1,1) ; \draw[kernels2] (-1,1) node[xi] {} -- (0,0); 
 \draw (0,0)  -- (1,1) node[xi] {};
\draw[kernels2] (0,0) node[not] {} -- (0,1.5) node[xi] {};
\draw (-2,1) node[] {\tiny{i}};
\draw (-2,2.5) node[] {\tiny{j}};
\draw (2,1) node[] {\tiny{j}};
\draw (0,2.5) node[] {\tiny{i}};
 }

\DeclareSymbol{Xi4eabisc1quater}{1.5}{\draw (-1,2.5) node[xic] {} -- (-1,1) ; \draw[kernels2] (-1,1) node[xi] {} -- (0,0); 
 \draw (0,0)  -- (1,1) node[xic] {};
\draw[kernels2] (0,0) node[not] {} -- (0,1.5) node[xi] {};
 }

\DeclareSymbol{Xi4eabisc2}{1.5}{\draw (-1,2.5) node[xic] {} -- (-1,1) ; \draw[kernels2] (-1,1) node[xi] {} -- (0,0); 
 \draw (0,0)  -- (1,1) node[xic] {};
\draw[kernels2] (0,0) node[not] {} -- (0,1.5) node[xi] {}; }

\DeclareSymbol{Xi4eabisc2l}{1.5}{\draw (-1,2.5) node[xiesf] {} -- (-1,1) ; \draw[kernels2] (-1,1) node[xi] {} -- (0,0); 
 \draw (0,0)  -- (1,1) node[xic] {};
\draw[kernels2] (0,0) node[not] {} -- (0,1.5) node[xi] {}; }

\DeclareSymbol{Xi4eabisc2r}{1.5}{\draw (-1,2.5) node[xic] {} -- (-1,1) ; \draw[kernels2] (-1,1) node[xi] {} -- (0,0); 
 \draw (0,0)  -- (1,1) node[xiesf] {};
\draw[kernels2] (0,0) node[not] {} -- (0,1.5) node[xi] {}; }

\DeclareSymbol{Xi4eabisc3}{1.5}{\draw (-1,2.5) node[xic] {} -- (-1,1) ; \draw[kernels2] (-1,1) node[xic] {} -- (0,0); 
 \draw (0,0)  -- (1,1) node[xi] {};
\draw[kernels2] (0,0) node[not] {} -- (0,1.5) node[xi] {}; }

\DeclareSymbol{Xi4eb}{0}{
\draw[kernels2] (0,2) node[xi] {} -- (-1,1) ; \draw[kernels2] (-2,2)  node[xi] {} -- (-1,1) ; \draw (-1,1)  node[not] {} -- (0,0); 
 \draw (0,0) node[xi] {}  -- (1,1) node[xi] {};
}

\DeclareSymbol{Xi4eab}{1.5}{\draw[kernels2] (-1,2.5) node[xi] {} -- (-1,1) ; \draw[kernels2] (-2,2)  node[xi] {} -- (-1,1) ; \draw (-1,1)  node[not] {} -- (0,0); 
 \draw[kernels2] (0,0)  -- (1,1) node[xi] {};
\draw[kernels2] (0,0) node[not] {} -- (0,1.5) node[xi] {}; 
}

\DeclareSymbol{Xi4eabdis}{1.5}{\draw[kernels2] (-1,2.5) node[xies] {} -- (-1,1) ; \draw[kernels2] (-2,2)  node[xies] {} -- (-1,1) ; \draw (-1,1)  node[not] {} -- (0,0); 
 \draw[kernels2] (0,0)  -- (1,1) node[xies] {};
\draw[kernels2] (0,0) node[not] {} -- (0,1.5) node[xies] {}; 
}

\DeclareSymbol{Xi4eabc1}{1.5}{\draw[kernels2] (-1,2.5) node[xic] {} -- (-1,1) ; \draw[kernels2] (-2,2)  node[xi] {} -- (-1,1) ; \draw (-1,1)  node[not] {} -- (0,0); 
 \draw[kernels2] (0,0)  -- (1,1) node[xic] {};
\draw[kernels2] (0,0) node[not] {} -- (0,1.5) node[xi] {}; 
}

\DeclareSymbol{Xi4eabc2}{1.5}{\draw[kernels2] (-1,2.5) node[xi] {} -- (-1,1) ; \draw[kernels2] (-2,2)  node[xi] {} -- (-1,1) ; \draw (-1,1)  node[not] {} -- (0,0); 
 \draw[kernels2] (0,0)  -- (1,1) node[xic] {};
\draw[kernels2] (0,0) node[not] {} -- (0,1.5) node[xic] {}; 
}

\DeclareSymbol{Xi4eabbis}{1.5}{\draw[kernels2] (-1,2.5) node[xi] {} -- (-1,1) ; \draw[kernels2] (-2,2)  node[xi] {} -- (-1,1) ; \draw[kernels2] (-1,1)  node[not] {} -- (0,0); 
 \draw (0,0)  -- (1,1) node[xi] {};
\draw[kernels2] (0,0) node[not] {} -- (0,1.5) node[xi] {}; 
}

\DeclareSymbol{Xi4eabbisc1}{1.5}{\draw[kernels2] (-1,2.5) node[xic] {} -- (-1,1) ; \draw[kernels2] (-2,2)  node[xi] {} -- (-1,1) ; \draw[kernels2] (-1,1)  node[not] {} -- (0,0); 
 \draw (0,0)  -- (1,1) node[xic] {};
\draw[kernels2] (0,0) node[not] {} -- (0,1.5) node[xi] {}; 
}

\DeclareSymbol{Xi4eabbisc1perm}{1.5}{\draw[kernels2] (-1,2.5) node[xi] {} -- (-1,1) ; \draw[kernels2] (-2,2)  node[xic] {} -- (-1,1) ; \draw[kernels2] (-1,1)  node[not] {} -- (0,0); 
 \draw (0,0)  -- (1,1) node[xic] {};
\draw[kernels2] (0,0) node[not] {} -- (0,1.5) node[xi] {}; 
}

\DeclareSymbol{Xi4eabbisc2}{1.5}{\draw[kernels2] (-1,2.5) node[xi] {} -- (-1,1) ; \draw[kernels2] (-2,2)  node[xi] {} -- (-1,1) ; \draw[kernels2] (-1,1)  node[not] {} -- (0,0); 
 \draw (0,0)  -- (1,1) node[xic] {};
\draw[kernels2] (0,0) node[not] {} -- (0,1.5) node[xic] {}; 
}

\DeclareSymbol{Xi2cbis}{0}{\draw[kernels2] (0,1) -- (0.8,2.2) node[xi] {};\draw[kernels2] (0,-0.25) node[not] {} -- (0,1); \draw[kernels2] (0,1) node[not] {} -- (-0.8,2.2) node[xi] {};}

\DeclareSymbol{Xi2cbis1}{0}{\draw (0,1) -- (-0.8,2.2) node[xi] {};\draw[kernels2] (0,-0.25) node[not] {} -- (0,1) node[xi] {}; }

\DeclareSymbol{Xi2Xbis}{-2}{\draw[kernels2] (0,-0.25)  -- (-1,1) ; \draw (-1,1) node[xix] {};
\draw[kernels2] (0,-0.25) node[not] {} -- (1,1) node[xi] {};}

\DeclareSymbol{XXi2bis}{-2}{\draw[kernels2] (0,-0.25) -- (-1,1) node[xi] {};
\draw[kernels2] (0,-0.25) node[X] {} -- (1,1) node[xi] {};}

\DeclareSymbol{I1XiIXi}{0}{\draw[kernels2] (0,-0.25) -- (1,1) node[xi] {};
\draw (0,-0.25) node[not] {} -- (-1,1) node[xi] {};}

\DeclareSymbol{I1XiIXib}{0}{\draw  (0,-0.25) node[xi] {} -- (0,1) node[not] {};
\draw[kernels2] (0,1) -- (0,2.25) ; \draw (0,2.25) node[xi]{}; }

\DeclareSymbol{I1XiIXic}{0}{
\draw[kernels2] (0,0) -- (1,1) node[xi] {} ; 
\draw[kernels2] (0,0) node[not] {}  -- (-1,1) node[not] {} -- (0,2) node[xi] {};
}

\DeclareSymbol{thin}{1.4}{\draw[pagebackground] (-0.3,0) -- (0.3,0); \draw  (0,0) -- (0,2);}
\DeclareSymbol{thin2}{1.4}{\draw[pagebackground] (-0.3,0) -- (0.3,0); \draw[tinydots]  (0,0) -- (0,2);}

\DeclareSymbol{thick}{1.4}{\draw[pagebackground] (-0.3,0) -- (0.3,0); \draw[kernels2]  (0,0) -- (0,2);}

\DeclareSymbol{thick2}{1.4}{\draw[pagebackground] (-0.3,0) -- (0.3,0); \draw[kernels2,tinydots]  (0,0) -- (0,2);}

\DeclareSymbol{Xi4ind}{2}{\draw (0,0) node[xi,label={[label distance=-0.2em]right: \scriptsize  $ i $}]  { } -- (-1,1) node[xi,label={[label distance=-0.2em]left: \scriptsize  $ j $}] {} -- (0,2) node[xi,label={[label distance=-0.2em]right: \scriptsize  $ k $}] {} -- (-1,3) node[xi,label={[label distance=-0.2em]left: \scriptsize  $ \ell $}] {};}

\DeclareSymbol{Xi4c1}{2}{\draw (0,0) node[xic] {} -- (-1,1) node[xi] {} -- (0,2) node[xic] {} -- (-1,3) node[xi] {};} 
\DeclareSymbol{IXi2ex}{0}{\draw (0,-0.25) node[xie] {} -- (-1,1) node[xi] {} ; \draw (0,-0.25)-- (1,1) node[xi] {};}
\DeclareSymbol{IXi2ex1}{0}{\draw (0,-0.25) node[xie] {} -- (-1,1) node[xi] {} -- (0,2) node[xi] {};}

\DeclareSymbol{Xi4b1}{0}{\draw(0,1.5) node[xic] {} -- (0,0); \draw (-1,1) node[xic] {} -- (0,0) node[xi] {} -- (1,1) node[xi] {};}

\DeclareSymbol{Xi4ec1}{0}{\draw (0,2) node[xi] {} -- (-1,1) node[xic] {} -- (0,0) node[xic] {} -- (1,1) node[xi] {};}
\DeclareSymbol{Xi4ec2}{0}{\draw (0,2) node[xic] {} -- (-1,1) node[xi] {} -- (0,0) node[xic] {} -- (1,1) node[xi] {};}
\DeclareSymbol{Xi4ec3}{0}{\draw (0,2) node[xic] {} -- (-1,1) node[xic] {} -- (0,0) node[xi] {} -- (1,1) node[xi] {};}

\DeclareSymbol{I1Xi4ac1}{2}{\draw[kernels2] (0,0) node[not] {} -- (-1,1) ; \draw[kernels2] (0,0) node[not] {} -- (1,1) node[xic] {} ;
\draw (-1,1) node[xi] {} -- (0,2) node[xic] {} -- (-1,3) node[xi] {};}

\DeclareSymbol{I1Xi4ac2}{2}{\draw[kernels2] (0,0) node[not] {} -- (-1,1) ; \draw[kernels2] (0,0) node[not] {} -- (1,1) node[xic] {} ;
\draw (-1,1) node[xi] {} -- (0,2) node[xi] {} -- (-1,3) node[xic] {};}

\DeclareSymbol{I1Xi4bp}{2}{\draw (0,0) node[not] {} -- (-1,1) node[xi] {} -- (0,2) ; \draw[kernels2] (0,2) node[not] {} -- (-1,3) node[xi] {};\draw[kernels2] (0,2)  -- (1,3) node[xi] {};
}

\DeclareSymbol{I1Xi4bc1}{2}{\draw (0,0) node[xic] {} -- (-1,1) node[xi] {} -- (0,2) ; \draw[kernels2] (0,2) node[not] {} -- (-1,3) node[xi] {};\draw[kernels2] (0,2)  -- (1,3) node[xic] {};
}

\DeclareSymbol{I1Xi4bc2}{2}{\draw (0,0) node[xic] {} -- (-1,1) node[xi] {} -- (0,2) ; \draw[kernels2] (0,2) node[not] {} -- (-1,3) node[xic] {};\draw[kernels2] (0,2)  -- (1,3) node[xi] {};
}

\DeclareSymbol{I1Xi4cp}{2}{\draw (0,0) node[not] {} -- (-1,1) node[not] {}; \draw[kernels2] (-1,1) -- (0,2) ; 
\draw[kernels2] (-1,1) -- (-2,2) node[xi] {} ;
\draw (0,2) node[xi] {} -- (-1,3) node[xi] {};}

\DeclareSymbol{I1Xi4cc1}{2}{\draw (0,0) node[xic] {} -- (-1,1) node[not] {}; \draw[kernels2] (-1,1) -- (0,2) ; 
\draw[kernels2] (-1,1) -- (-2,2) node[xi] {} ;
\draw (0,2) node[xic] {} -- (-1,3) node[xi] {};}

\DeclareSymbol{I1Xi4cc2}{2}{\draw (0,0) node[xic] {} -- (-1,1) node[not] {}; \draw[kernels2] (-1,1) -- (0,2) ; 
\draw[kernels2] (-1,1) -- (-2,2) node[xi] {} ;
\draw (0,2) node[xi] {} -- (-1,3) node[xic] {};}

\DeclareSymbol{I1Xi4abc1}{2}{\draw[kernels2] (0,0) node[not] {} -- (-1,1) ; \draw[kernels2] (0,0) node[not] {} -- (1,1) node[xic] {};\draw (-1,1) node[xi] {} -- (0,2) ; \draw[kernels2] (0,2) node[not] {} -- (-1,3) node[xic] {};\draw[kernels2] (0,2)  -- (1,3) node[xi] {}; }

\DeclareSymbol{I1Xi4abc2}{2}{\draw[kernels2] (0,0) node[not] {} -- (-1,1) ; \draw[kernels2] (0,0) node[not] {} -- (1,1) node[xic] {};\draw (-1,1) node[xi] {} -- (0,2) ; \draw[kernels2] (0,2) node[not] {} -- (-1,3) node[xi] {};\draw[kernels2] (0,2)  -- (1,3) node[xic] {}; }

\DeclareSymbol{R1}{0}{\draw (-1,1) node[xi] {} -- (0,0) node[not] {};
\draw[kernels2] (0,1.5) node[xic] {} -- (0,0) -- (1,1) node[xic] {};}
\DeclareSymbol{R2}{0}{\draw (-1,1) node[xic] {} -- (0,0) node[not] {};
\draw[kernels2] (0,1.5)  {} -- (0,0) -- (1,1)  {};
\draw (0,1.5) node[xi] {};
\draw (1,1) node[xic] {};
}
\DeclareSymbol{R3}{1}{\draw[kernels2] (-1,1.5)  {} -- (0,0) node[not] {} -- (1,1.5);
\draw (-1,1.5) node[xi] {};
\draw[kernels2] (0,3) {} -- (1,1.5) -- (2,3)  {};
\draw  (0,3) node[xic] {} ;
\draw (2,3) node[xic] {};}
\DeclareSymbol{R4}{1}{\draw[kernels2] (-1,1.5) node[xic] {} -- (0,0) node[not] {} -- (1,1.5);
\draw[kernels2] (0,3) {} -- (1,1.5) -- (2,3) node[xic] {};
\draw (0,3) node[xi] {};}

\DeclareSymbol{I1Xi4bcp}{2}{\draw (0,0) node[not] {} -- (-1,1) node[not] {}; \draw[kernels2] (-1,1) -- (0,2) ; 
\draw[kernels2] (-1,1) -- (-2,2) node[xi] {} ; \draw[kernels2] (0,2) node[not] {} -- (-1,3) node[xi] {};\draw[kernels2] (0,2)  -- (1,3) node[xi] {};
}

\DeclareSymbol{I1Xi4bcc1}{2}{\draw (0,0) node[xic] {} -- (-1,1) node[not] {}; \draw[kernels2] (-1,1) -- (0,2) ; 
\draw[kernels2] (-1,1) -- (-2,2) node[xi] {} ; \draw[kernels2] (0,2) node[not] {} -- (-1,3) node[xi] {};\draw[kernels2] (0,2)  -- (1,3) node[xic] {};
}

\DeclareSymbol{I1Xi4bcc2}{2}{\draw (0,0) node[xic] {} -- (-1,1) node[not] {}; \draw[kernels2] (-1,1) -- (0,2) ; 
\draw[kernels2] (-1,1) -- (-2,2) node[xi] {} ; \draw[kernels2] (0,2) node[not] {} -- (-1,3) node[xic] {};\draw[kernels2] (0,2)  -- (1,3) node[xi] {};
} 

\DeclareSymbol{2I1Xi4bc1}{2}{\draw[kernels2] (0,0) node[not] {} -- (-1,1) ;
\draw[kernels2] (0,0) -- (1,1);
\draw (-1,1) node[xic] {} -- (-1,2.5) node[xi] {};
\draw (1,1)  node[xic] {} -- (1,2.5) node[xi] {};
}

\DeclareSymbol{2I1Xi4bc2}{2}{\draw[kernels2] (0,0) node[not] {} -- (-1,1) ;
\draw[kernels2] (0,0) -- (1,1);
\draw (-1,1) node[xi] {} -- (-1,2.5) node[xic] {};
\draw (1,1)  node[xic] {} -- (1,2.5) node[xi] {};
}

\DeclareSymbol{diff2I1Xi4bc2}{2}{\draw (0,0) node[diff] {} -- (-1,1) ;
\draw (0,0) -- (1,1);
\draw (-1,1) node[xi] {} -- (-1,2.5) node[xic] {};
\draw (1,1)  node[xic] {} -- (1,2.5) node[xi] {};
}

\DeclareSymbol{2I1Xi4bc3}{2}{\draw[kernels2] (0,0) node[not] {} -- (-1,1) ;
\draw[kernels2] (0,0) -- (1,1);
\draw (-1,1) node[xic] {} -- (-1,2.5) node[xic] {};
\draw (1,1)  node[xi] {} -- (1,2.5) node[xi] {};
}

\DeclareSymbol{Xi41}{0}{\draw (0,1) -- (0.8,2.2) node[xic] {};\draw (0,-0.25) node[xi] {} -- (0,1) node[xi] {} -- (-0.8,2.2) node[xic] {};} 

\DeclareSymbol{Xi42}{0}{\draw (0,1) -- (0.8,2.2) node[xi] {};\draw (0,-0.25) node[xic] {} -- (0,1) node[xi] {} -- (-0.8,2.2) node[xic] {};}

\DeclareSymbol{Xi4ca1}{0}{\draw (0,1) -- (-1,2.2) node[xic] {};\draw (0,-0.25) node[xi] {} -- (0,1) ; \draw[kernels2] (0,1) node[not] {} -- (1,2.2) node[xic] {};
\draw[kernels2] (0,1) {} -- (0,2.7) node[xi] {};
}

\DeclareSymbol{Xi4ca2}{0}{\draw (0,1) -- (-1,2.2) node[xi] {};\draw (0,-0.25) node[xi] {} -- (0,1) ; \draw[kernels2] (0,1) node[not] {} -- (1,2.2) node[xic] {};
\draw[kernels2] (0,1) {} -- (0,2.7) node[xic] {};
}

\DeclareSymbol{Xi4cap}{0}{\draw (0,1) -- (-1,2.2) node[xi] {};\draw (0,-0.25) node[not] {} -- (0,1) ; \draw[kernels2] (0,1) node[not] {} -- (1,2.2) node[xi] {};
\draw[kernels2] (0,1) {} -- (0,2.7) node[xi] {};
}

\DeclareSymbol{Xi3a}{0}{
 \draw (-1,1)  node[xi] {} -- (0,0); 
 \draw (0,0) node[xi] {}  -- (1,1) node[xi] {};
}

\DeclareSymbol{Xi4ebc1}{0}{
\draw[kernels2] (0,2) node[xi] {} -- (-1,1) ; \draw[kernels2] (-2,2)  node[xic] {} -- (-1,1) ; \draw (-1,1)  node[not] {} -- (0,0); 
 \draw (0,0) node[xic] {}  -- (1,1) node[xi] {};
}

\DeclareSymbol{Xi4ebc2}{0}{
\draw[kernels2] (0,2) node[xi] {} -- (-1,1) ; \draw[kernels2] (-2,2)  node[xi] {} -- (-1,1) ; \draw (-1,1)  node[not] {} -- (0,0); 
 \draw (0,0) node[xic] {}  -- (1,1) node[xic] {};
}

\DeclareSymbol{Xi2cbispex}{0}{\draw[kernels2] (0,1) -- (0.8,2.2) node[xi] {};\draw (0,-0.25) node[xie] {} -- (0,1); \draw[kernels2] (0,1) node[not] {} -- (-0.8,2.2) node[xi] {};}

\DeclareSymbol{Xi2cbis1p}{0}{\draw (0,1) -- (-0.8,2.2) node[xi] {};\draw (0,-0.25) node[not] {} -- (0,1) node[xi] {}; }

\DeclareSymbol{Xi2Xp}{-2}{\draw (0,-0.25) node[not] {} -- (-1,1) node[xix] {};} 

\DeclareSymbol{I1XiIXib}{0}{\draw  (0,-0.25) node[xi] {} -- (0,1) node[not] {};
\draw[kernels2] (0,1) -- (0,2.25) ; \draw (0,2.25) node[xi]{}; }

\DeclareSymbol{IXi2b}{0}{\draw  (0,-0.25) node[xi] {} -- (0,1) node[not] {};
\draw (0,1) -- (0,2.25) ; \draw (0,2.25) node[xi]{}; }

\DeclareSymbol{IXi2bex}{0}{\draw  (0,-0.25) node[xi] {} -- (0,1) node[xie] {};
\draw (0,1) -- (0,2.25) ; \draw (0,2.25) node[xi]{}; }

 \def\1{\mathbf{\symbol{1}}}

\def\one{\mathbf{1}}
\def\eps{\varepsilon}

\DeclareSymbol{diff}{0}{
\draw (0,0.5) node[diff] {};
}

\DeclareSymbol{diff1}{0}{
\draw (0,0.5) node[diff1] {};
}

\DeclareSymbol{diff2}{0}{
\draw (0,0.5) node[diff2] {};
}

\DeclareSymbol{geo}{0}{
\draw (0,0) node[diff] {};
\draw (0.3,0) node[diff] {};
}

\DeclareSymbol{generic}{0}{
\draw (0,0.6) node[xi] {};
}

\DeclareSymbol{g}{0}{
\draw (0,0.6) node[g] {};
}

\DeclareSymbol{Ito}{0}{
\draw (0,0.6) node[xies] {};
}

\DeclareSymbol{Itob}{0}{
\draw (0,0.6) node[xiesf] {};
}

\DeclareSymbol{greycirc}{0}{
\draw (0,0.3) node[xi] {};
}

\DeclareSymbol{not}{0}{
\draw (0,0.6) node[not] {};
\draw[tinydots] (0,0.6) circle (0.8);
}

\DeclareSymbol{genericb}{0}{
\draw (0,0.6) node[xic] {};
}

\DeclareSymbol{bluecirc}{0}{
\draw (0,0.3) node[xic] {};
}

\DeclareSymbol{genericxix}{0}{
\draw (0,0.6) node[xix] {};
}

\DeclareSymbol{genericX}{0}{
\draw (0,0.6) node[X] {};
}

\DeclareSymbol{diffIto}{1}{
\draw  (0,2.5) -- (0,0) ;
\draw (0,-0.1) node[diff] {};
\draw (0,2.5) node[xies] {};
}
\DeclareSymbol{Itodiff}{2}{
\draw(0,2.9) -- (0,-0.2);
\draw (0,2.9) node[diff] {};
\draw (0,-0.1) node[xies] {};
}

\DeclareSymbol{diffgeneric}{1}{
\draw  (0,2.5) -- (0,0) ;
\draw (0,-0.1) node[diff] {};
\draw (0,2.5) node[xi] {};
}

\DeclareSymbol{genericdiff}{2}{
\draw(0,2.9) -- (0,-0.2);
\draw (0,2.9) node[diff] {};
\draw (0,-0.1) node[xi] {};
}

\DeclareSymbol{diffdot}{2}{
\draw  (0,3) -- (0,-0.1) ;
\draw (0,3) node[not] {};
\draw (0,-0.1) node[diff] {};
}

\DeclareSymbol{diffdotmini}{0}{
\draw  (0,0) -- (0,1.2) ;
\draw (0,1.2) node[not] {};
\draw (0,0) node[diffmini] {};
}

\DeclareSymbol{dotdiff}{2}{
\draw[kernelsmod]  (0,3) -- (0,-0.1) ;
\draw (0,3) node[diff] {};
\draw (0,-0.1) node[not] {};
}

\DeclareSymbol{dotdiff1}{2}{
\draw[kernelsmod]  (0,3) -- (0,-0.1) ;
\draw (0,3) node[diff1] {};
\draw (0,-0.1) node[not] {};
}

\DeclareSymbol{dotdiff1mini}{0}{
\draw[kernelsmod]  (0,1.2) -- (0,0) ;
\draw (0,1.2) node[diffmini] {};
\draw (0,0) node[not] {};
}

\DeclareSymbol{dotdiff2}{2}{
\draw (0,3) -- (0,-0.1) ;
\draw (0,3) node[diff] {};
\draw (0,-0.1) node[not] {};
}

\DeclareSymbol{dotdiff2mini}{0}{
\draw (0,1.2) -- (0,0) ;
\draw (0,1.2) node[diffmini] {};
\draw (0,0) node[not] {};
}

\DeclareSymbol{dotdiffstraight}{0}{
\draw  (0,3) -- (0,-0.1) ;
\draw (0,3) node[diff] {};
\draw (0,-0.1) node[not] {};
}

\DeclareSymbol{arbre1}{0}{
\draw  (0,0) -- (1.5,1.5) ;
\draw (1.5,1.5) node[not] {};
\draw (0,0) node[not] {};
}

\DeclareSymbol{arbre2}{0}{
\draw  (0,0) -- (1.5,1.5) ;
\draw[kernelsmod] (0,0) -- (-1.5,1.5);
\draw (1.5,1.5) node[not] {};
\draw (0,0) node[not] {};
\draw (-1.5,1.5) node[xi] {};
}

\DeclareSymbol{arbre3}{0}{
\draw  (0,0) -- (1.5,1.5) ;
\draw[kernelsmod] (1.5,1.5) -- (0,3);
\draw (0,0) node[not] {};
\draw (1.5,1.5) node[not] {};
\draw (0,3) node[xi] {};
}

\DeclareSymbol{treeeval}{0}{
\draw (0,0) -- (1,1);
\draw (0,0) node[xi] {};
\draw (1.25,1.25) node[xi] {};
\draw (-0.6,0.6) node[]{\tiny{$i$}};
\draw (0.65,1.85) node[]{\tiny{$j$}};
}

\DeclareSymbol{testeval}{0}{
\draw (0,0) -- (1,1);
\draw (0,0) -- (-1,1);
\draw (0,0) node[xi] {};
\draw (1.25,1.25) node[xi] {};
\draw (-1.25,1.25) node[xi] {};
\draw (-0.6,-0.6) node[]{\tiny{$i$}};
\draw (0.65,1.85) node[]{\tiny{$j$}};
\draw (-1.95,1.85) node[]{\tiny{$k$}};
}

\DeclareSymbol{treeeval2}{0}{
\draw[kernelsmod] (-0.25,-1) -- (1,0.5) ;
\draw[kernelsmod] (1,0.5) -- (-0.25,2);
\draw (1,0.5) node[diff2] {};
\draw (-0.25,-1) node[not] {};
\draw (-0.25,2) node[xi] {};
\draw (-0.6,1.2) node[]{\tiny{1}};
}

\DeclareSymbol{arbreact}{1}{
\draw (0,0) node[not] {};
\draw[kernelsmod] (0,0) -- (1,1);
\draw[kernelsmod] (0,0) -- (-1,1);
\draw (-1,1) node[xic] {};
\draw  (0,2) -- (1,1) ;
\draw (0,2) node[xic] {};
\draw (1,1) node[xi] {};
}

\DeclareSymbol{arbreact1}{0}{
\draw (0,-1.5) -- (0,0);
\draw[kernelsmod] (0,0) -- (1,1);
\draw[kernelsmod] (0,0) -- (-1,1);
\draw  (0,2) -- (1,1) ;
\draw (0,-1.5) node[diff] {};
\draw (0,0) node[not] {};
\draw (-1,1) node[xic] {};
\draw (0,2) node[xic] {};
\draw (1,1) node[xi] {};
}

\DeclareSymbol{arbreact2}{0}{
\draw (0,-0.75) -- (-1,0.5); 
\draw (0,-0.75) -- (1,0.5);
\draw (0,1.5) -- (1,0.5);
\draw (0,1.5) node[xic] {};
\draw (1,0.5) node[xi] {};
\draw (-1,0.5) node[xic] {};
\draw (0,-0.75) node[diff] {};
}

\DeclareSymbol{arbreact3}{0}{
\draw[kernelsmod] (0,-0.75) -- (-1,0.5); 
\draw[kernelsmod] (0,-0.75) -- (1,0.5);
\draw (0,1.75) -- (1,0.5);
\draw (2,1.75) -- (1,0.5);
\draw (0,1.75) node[xic] {};
\draw (1,0.5) node[diff] {};
\draw (-1,0.5) node[xic] {};
\draw (2,1.75) node[xi] {};
\draw (0,-0.75) node[not] {};
}

\DeclareSymbol{pre_im_I1Xitwo}{0}{
\draw[kernels2] (0,-0.3) node[not] {} -- (-0.6,0.7) ;
\draw[kernels2] (0,-0.3) -- (0.6,0.7);
\draw (0,0.9) node[g] {};
}

\DeclareSymbol{pre_im_cI1Xi4ab}{2}{
\draw[kernels2] (0,-1) node[not] {} -- (-0.6,0) ;
\draw[kernels2] (0,-1) -- (0.6,0);
\draw (0,0.2) node[g] {};
\draw (0,0.6) -- (0,1.5);
\draw[kernels2] (0,1.5) node[not] {} -- (-0.6,2.5) ;
\draw[kernels2] (0,1.5) -- (0.6,2.5);
\draw (0,2.7) node[g] {};
}

\DeclareSymbol{pre_im_I1Xi4acc2}{0}{
\draw[kernels2] (-1,-0.5) node[not] {} -- (-1.6,0.5) ;
\draw[kernels2] (-1,-0.5) -- (-0.4,0.5);
\draw[kernels2] (-1,-0.5) -- (0.2,-1.5) node[not] {} ;
\draw (-1,1.1) -- (-1,2);
\draw[kernels2] (0.2,-1.5) -- (0.2,2);
\draw (-1,0.7) node[g] {};
\draw (-0.3,2.2) node[g] {};
}

\DeclareSymbol{pre_im_I1Xi4abcc2}{2}{
\draw[kernels2] (0,-1) node[not] {} -- (-1,0) node[not] {};
\draw[kernels2] (-1,1.2) node[not] {} -- (-1,0);
\draw[kernels2] (-1,1.2) -- (-1.5,2.5);
\draw[kernels2] (-1,1.2) -- (-0.5,2.5);
\draw[kernels2] (-1,0) -- (0.7,2.5);
\draw[kernels2] (0,-1) -- (1.5,2.5);
\draw (-1,2.7) node[g] {};
\draw (1,2.7) node[g] {};
}

\DeclareSymbol{pre_im_2I1Xi4c1}{2}{
\draw[kernels2] (0,-0.5) node[not] {} -- (-1,0.5) node[not] {};
\draw[kernels2] (0,-0.5) -- (1,0.5) node[not] {};
\draw[kernels2] (-1,0.5) node[not] {}-- (-1.7,2);
\draw[kernels2]  (-1,2) -- (1,0.5);
\draw[kernels2] (-1,0.5) -- (1,2);
\draw[kernels2] (1,0.5) -- (1.7,2);
\draw (-1.2,2.2) node[g] {};
\draw (1.2,2.2) node[g] {};
}

\DeclareSymbol{pre_im_Xi4eabisc2}{2}{
\draw[kernels2] (1.2,-0.5) node[not] {} -- (-0.7,0.8) ;
\draw[kernels2] (1.2,-0.5) -- (0.4,0.8);
\draw (0,1.4)  -- (0,2.2);
\draw (1.2,2.2) -- (1.2,-0.6);
\draw (0,1) node[g] {};
\draw (0.6,2.4) node[g] {};
}

\DeclareSymbol{pre_im_Xi4eabisc22}{2}{
\draw (1.2,-0.5) node[not] {} -- (-0.7,0.8) ;
\draw[kernels2] (1.2,-0.5) -- (0.4,0.8);
\draw (0,1.4)  -- (0,2.2);
\draw[kernels2] (1.2,2.2) -- (1.2,-0.6);
\draw (0,1) node[g] {};
\draw (0.6,2.4) node[g] {};
}

\DeclareSymbol{pre_im_Xi4eabisc222}{2}{
\draw[kernels2] (0.4,-0.5) node[not] {} -- (-0.6,1) ;
\draw[kernels2] (1.2,0) -- (0.3,1);
\draw (0,1.1)  -- (0,2.5);
\draw[kernels2] (1.2,2.5) -- (1.2,0) node[not] {} -- (0.4,-0.6);
\draw (0,1.2) node[g] {};
\draw (0.6,2.5) node[g] {};
}

\DeclareSymbol{pre_im_Xi4eabc2}{2}{
\draw (0,-0.5) node[not] {} -- (-1,0.5) node[not] {};
\draw[kernels2] (-1,0.5) -- (-1.5,2);
\draw[kernels2] (0,-0.5)  -- (0.7,2);
\draw[kernels2] (-1,0.5) -- (-0.5,2);
\draw[kernels2] (0,-0.5) -- (1.5,2);
\draw (-1,2.2) node[g] {};
\draw (1,2.2) node[g] {};
}

\DeclareSymbol{pre_im_Xi4eabbisc2}{2}{
\draw[kernels2] (0,-0.5) node[not] {} -- (-1,0.5) node[not] {};
\draw[kernels2] (-1,0.5) -- (-1.5,2);
\draw[kernels2] (0,-0.5)  -- (0.7,2);
\draw[kernels2] (-1,0.5) -- (-0.5,2);
\draw (0,-0.5) -- (1.5,2);
\draw (-1,2.2) node[g] {};
\draw (1,2.2) node[g] {};
}

\DeclareSymbol{pre_im_I1Xi4abcc1}{2}{
\draw[kernels2] (0,-1) node[not] {} -- (-1,0) node[not] {};
\draw[kernels2] (0,1.1) node[not] {} -- (-1,0);
\draw[kernels2] (-1,0) -- (-1.5,2.5);
\draw[kernels2] (0,1.1) node[not] {} -- (-0.5,2.5);
\draw[kernels2] (0,1.1) -- (0.5,2.5);
\draw[kernels2] (0,-1) -- (1.5,2.5);
\draw (-1,2.7) node[g] {};
\draw (1,2.7) node[g] {};
}

\DeclareSymbol{pre_im_Xi4eabc1}{2}{
\draw (0,-0.5) node[not] {} -- (-1,0.5) node[not] {};
\draw[kernels2] (-1,0.5) -- (-1.5,2);
\draw[kernels2] (0,-0.5)  -- (-0.5,2);
\draw[kernels2] (-1,0.5) -- (0.8,2);
\draw[kernels2] (0,-0.5) -- (1.5,2);
\draw (-1,2.2) node[g] {};
\draw (1,2.2) node[g] {};
}

\DeclareSymbol{pre_im_Xi4ba1b}{2}{
\draw[kernels2] (0,0) node[not] {}  -- (1.8,1.5);
\draw[kernels2] (0,0) -- (0.8,1.5);
\draw (0,-0.1) -- (-1.8,1.5);
\draw (0,-0.1) -- (-0.8,1.5);
\draw (-1,1.7) node[g] {};
\draw (1,1.7) node[g] {};
}

\DeclareSymbol{pre_im_Xi4ba2}{2}{
\draw (0,-0.1) node[not] {}  -- (1.8,1.5);
\draw[kernels2] (0,0) -- (0.8,1.5);
\draw (0,-0.1) -- (-1.8,1.5);
\draw[kernels2] (0,0) -- (-0.8,1.5);
\draw (-1,1.7) node[g] {};
\draw (1,1.7) node[g] {};
}

\DeclareSymbol{pre_im_Xi4cabc2}{2}{
\draw[kernels2] (0,-0.5) node[not] {} -- (-1,0.5) node[not] {};
\draw[kernels2] (-1,0.5) -- (-1.5,2);
\draw (-1,0.5)  -- (0.7,2);
\draw[kernels2] (-1,0.5) -- (-0.5,2);
\draw[kernels2] (0,-0.5) -- (1.5,2);
\draw (-1,2.2) node[g] {};
\draw (1,2.2) node[g] {};
}

\DeclareSymbol{pre_im_Xi4cabc1}{2}{
\draw[kernels2] (0,-0.5) node[not] {} -- (-1,0.5) node[not] {};
\draw (-1,0.5) -- (-1.5,2);
\draw[kernels2] (-1,0.5)  -- (0.7,2);
\draw[kernels2] (-1,0.5) -- (-0.5,2);
\draw[kernels2] (0,-0.5) -- (1.5,2);
\draw (-1,2.2) node[g] {};
\draw (1,2.2) node[g] {};
}

\DeclareSymbol{pre_im_Xi4eabbisc1}{2}{
\draw[kernels2] (0,-0.5) node[not] {} -- (-1,0.5) node[not] {};
\draw[kernels2] (-1,0.5) -- (-1.5,2);
\draw[kernels2] (0,-0.5)  -- (-0.5,2);
\draw[kernels2] (-1,0.5) -- (0.8,2);
\draw (0,-0.5) -- (1.5,2);
\draw (-1,2.2) node[g] {};
\draw (1,2.2) node[g] {};
}

\DeclareSymbol{pre_im_1}{0}{
\draw[kernels2] (0,-0.5) node[not] {} -- (-0.6,0.5) ;
\draw[kernels2] (0,-0.5) -- (0.6,0.5);
\draw (0,1.1)  -- (-0.55,2);
\draw (0,1.1)  -- (0.55,2);
\draw (0,0.7) node[g] {};
\draw (0,2.2) node[g] {};
}

\DeclareSymbol{disconnect}{0}{
\draw[kernels2] (0,-0.5) node[not] {} -- (-0.6,0.5) ;
\draw[kernels2] (0,-0.5) -- (0.6,0.5);
\draw (-0.55,1.1)  -- (-0.55,2.3);
\draw (0.55,2.3) -- (0.55,1.5) -- (1.2,1.5) -- (1.2,3.5) -- (0.55,3.5) -- (0.55,2.7);
\draw (0,0.7) node[g] {};
\draw (0,2.5) node[g] {};
}

\DeclareSymbol{pre_im_2}{2}{\draw[kernels2] (0,0) node[not] {} -- (-1,1) node[not] {};
\draw[kernels2] (0,0) -- (1,1) node[not] {};
\draw[kernels2] (-1,1) -- (-1.5,2.5);
\draw[kernels2] (-1,1) -- (-0.5,2.5);
\draw[kernels2] (1,1) -- (0.5,2.5);
\draw[kernels2] (1,1) -- (1.5,2.5);
\draw (-1,2.7) node[g] {};
\draw (1,2.7) node[g] {};
}

\DeclareSymbol{CX_rec}{0}{
\draw [black] (-0.3,1) to (-0.3,-0.3);
\draw [black] (0.3,1) to (0.3,-0.3);
\draw [black] (-0.3,1) to (-0.3,2.3);
\draw [black] (0.3,1) to (0.3,2.3);
\draw (0,1) node[rec] {};
}

\DeclareSymbol{CX_cerc}{0}{
\draw [black] (0,1) to (0,-0.3);
\draw (0,1) node[cerc] {};
}

\DeclareMathAlphabet{\mathpzc}{OT1}{pzc}{m}{it}

\let\eps\varepsilon

\def\eqref#1{(\ref{#1})}

\makeatletter 
\newcommand*{\bigcdot}{}
\DeclareRobustCommand*{\bigcdot}{%
  \mathbin{\mathpalette\bigcdot@{}}%
}
\newcommand*{\bigcdot@scalefactor}{.5}
\newcommand*{\bigcdot@widthfactor}{1.15}
\newcommand*{\bigcdot@}[2]{%
  \sbox0{$#1\vcenter{}$}
  \sbox2{$#1\cdot\m@th$}%
  \hbox to \bigcdot@widthfactor\wd2{%
    \hfil
    \raise\ht0\hbox{%
      \scalebox{\bigcdot@scalefactor}{%
        \lower\ht0\hbox{$#1\bullet\m@th$}%
      }%
    }%
    \hfil
  }%
}
\makeatother

\tcbset
{colframe=boxcolor,colback=symbols!7!pagebackground,coltext=pageforeground,
fonttitle=\bfseries,nobeforeafter,center title,size=fbox,boxsep=1.5pt,
top=0mm,bottom=0mm,boxsep=0mm,tcbox raise base}

\def\two{{\<generic>\kern0.05em\<genericb>}}
\def\twoI{{\<Ito>\kern0.05em\<Itob>}}

\def\mail#1{\burlalt{#1}{mailto:#1}}

\usepackage{thmtools}

\declaretheorem[style=definition]{example}

\begin{document}
\renewcommand\thmcontinues[1]{Continued}

\title{Resonance based schemes for dispersive equations via decorated trees}
\author{Yvain Bruned$^1$, Katharina Schratz$^2$}
\institute{University of Edinburgh \and
 LJLL (UMR 7598), Sorbonne Université\\
Email:\ \begin{minipage}[t]{\linewidth}
\mail{Yvain.Bruned@ed.ac.uk}, \mail{schratz@ljll.math.upmc.fr}.
\end{minipage}}

\maketitle

\begin{abstract}
We introduce a numerical framework for dispersive equations embedding their underlying resonance structure  into the discretisation.   This will allow us to resolve  the nonlinear  oscillations of the PDE and to approximate with high order accuracy a large class of equations under lower regularity assumptions than classical techniques require. The   key idea to control the nonlinear frequency interactions in the system up to arbitrary high order thereby lies in  a tailored decorated tree formalism. Our algebraic structures are  close to the ones developed for singular SPDEs with Regularity Structures. We adapt them to the context of dispersive PDEs by using a novel class of decorations {which encode the dominant frequencies}.  The structure proposed in this paper is new and gives a variant of the Butcher-Connes-Kreimer Hopf algebra on decorated trees. 
We observe a similar Birkhoff type factorisation as in SPDEs and perturbative quantum field theory. This factorisation allows us to single out oscillations  and to optimise the local error by mapping it to the particular regularity of the solution. This use of the Birkhoff factorisation seems new in comparison to the literature.  The field of singular SPDEs took advantage of numerical methods and renormalisation in perturbative quantum field theory by extending their structures via the adjunction of decorations and Taylor expansions. Now, through this work, Numerical Analysis is taking advantage of these extended structures and provides a new perspective on them.  
\end{abstract}
\setcounter{tocdepth}{2}
\tableofcontents

\section{Introduction}
We consider nonlinear dispersive equations
\begin{equation}\label{dis}
\begin{aligned}
& i \partial_t u(t,x) +   \mathcal{L}\left(\nabla, \tfrac{1}{\varepsilon}\right) u(t,x) =\vert \nabla\vert^\alpha p\left(u(t,x), \overline u(t,x)\right)\\
& u(0,x) = v(x),
\end{aligned}
\end{equation}
where we assume a polynomial nonlinearity $p$ and that the structure of \eqref{dis} implies at least local wellposedness of the problem on a finite time interval $]0,T]$, $T<\infty$ in an appropriate functional space. Here, $u$ is the complex-valued solution that we want to approximate. Concrete examples are discussed in Section \ref{sec:examples}, including the cubic nonlinear  Schr\"odinger (NLS) equation
\begin{equation}\label{nlsIntro}
i \partial_t u + \mathcal{L}\left(\nabla\right)  u = \vert u\vert^2 u, \quad \mathcal{L}\left(\nabla\right) = \Delta,
\end{equation}
the Korteweg--de Vries (KdV) equation 
\begin{equation}\label{kdvIntro}
 \partial_t u +\mathcal{L}\left(\nabla\right) u = \frac12 \partial_x u^2, \quad \mathcal{L}\left(\nabla\right) = i\partial_x^3,
\end{equation}\black
as well as highly oscillatory Klein--Gordon type  systems
\begin{align}\label{kgrIntro}
 i \partial_t u = -\mathcal{L}\left(\nabla, \tfrac{1}{\varepsilon}\right) u + \frac{1}{\varepsilon^2}\mathcal{L}\left(\nabla, \tfrac{1}{\varepsilon}\right)^{-1} \textstyle p(u,\overline u), \quad \mathcal{L}\left(\nabla, \tfrac{1}{\varepsilon}\right)  = \frac{1}{\varepsilon}\sqrt{\frac{1}{\varepsilon^2}-\Delta}.  
\end{align}

In the last decades,  Strichartz  and Bourgain  space estimates allowed  to establish well-posedness results for dispersive equations in low regularity spaces \cite{Burq-Gerard-Tzvetkov,Bour93a,Keel-Tao,Strichartz,Tao06}.  Numerical theory for dispersive PDEs, on the other hand, is in general still restricted to smooth solutions. This is due to the fact that
most classical approximation techniques were originally developed for linear problems  and thus, in general, neglect  nonlinear frequency interactions  in a system. In the dispersive setting \eqref{dis} the interaction of the    differential operator $\mathcal{L}$ with the nonlinearity $p$, however, triggers  oscillations both in space and in time and, unlike for parabolic problems, no smoothing can be expected.  At  low regularity and high oscillations, these nonlinear frequency interactions  play  an essential role: Note that while the influence of $i\mathcal{L}$ can be small, the influence of the interaction of $+i\mathcal{L}$ and $-i\mathcal{L}$ can be huge, and vice versa.   Classical \emph{linearised} frequency approximations, used, e.g., in splitting methods or exponential integrators, see Table \ref{tab1} below, are therefore restricted to smooth solutions.   The latter is not only a technical formality: The severe order reduction in case of non-smooth solutions is also observed numerically, see, e.g., \cite{JL00,OS18} and Figure \ref{fig:osc},   and only  very little is known on how to overcome this issue.  For an  extensive overview on  numerical methods for Hamiltonian systems, geometric numerical analysis, structure preserving algorithms, and  highly oscillatory problems we refer to  the books Butcher \cite{Butcher}, Engquist et al. \cite{EFHI09}, Faou \cite{Faou12}, E. Hairer et al. \cite{HW,H2Tri}, Holden et al. \cite{HLRS10},  Leimkuhler  $\&$ Reich  \cite{LR04}, McLachlan  $\&$  Quispel \cite{McLacQ02}, Sanz-Serna $\&$ Calvo \cite{SanBook} and the references therein. 

In this work, we establish a new framework of resonance based approximations for dispersive equations which will allow us to approximate with high order accuracy  a   large class of  equations under (much) lower regularity assumptions than classical techniques require.  The key in the construction of the new methods lies in analysing the underlying oscillatory structure of the system~\eqref{dis}. We look at the  corresponding mild solution given by Duhamel's formula
\begin{equation}\label{duh}
u(t) = e^{ it  \mathcal{L}\left(\nabla, \frac{1}{\varepsilon}\right)} v  - ie^{ it  \mathcal{L}\left(\nabla, \frac{1}{\varepsilon}\right)}\vert \nabla\vert^\alpha  \int_0^t e^{ -i\xi  \mathcal{L}\left(\nabla, \frac{1}{\varepsilon}\right)}   p\left(u(\xi), \overline u(\xi)\right) d\xi 
\end{equation}
and its iterations
\begin{equs}\label{It1}
u(t) = e^{ it  \mathcal{L}\left(\nabla, \frac{1}{\varepsilon}\right)} v - ie^{ it  \mathcal{L}\left(\nabla, \frac{1}{\varepsilon}\right)}  \vert \nabla\vert^\alpha\mathcal{I}_1( t, \mathcal{L},v,p)  +\vert \nabla\vert^{2\alpha} \int_0^t \int_0^\xi \ldots d\xi_1 d \xi.
\end{equs}
The principal oscillatory integral $\mathcal{I}_1( t, \mathcal{L},v,p)$ thereby takes the form
\[
\mathcal{I}_1( t, \mathcal{L},v,p) = \int_0^t  \mathcal{Osc}(\xi, \mathcal{L}, v,p) d\xi
\]
with the central oscillations 
\begin{align}\label{osc}
\mathcal{Osc}(\xi, \mathcal{L}, v,p)  = e^{ -i\xi  \mathcal{L}\left(\nabla, \frac{1}{\varepsilon}\right)}
 p\left(e^{ i \xi  \mathcal{L}\left(\nabla, \frac{1}{\varepsilon}\right)} v , 
 e^{ - i \xi  \mathcal{L}\left(\nabla, \frac{1}{\varepsilon}\right)} \overline v \right) 
\end{align}
driven by the nonlinear frequency interactions between the differential operator $\mathcal{L}$ and the nonlinearity $p$. In order to obtain a suitable approximation at low regularity,  it is central to resolve these oscillations -- characterised by the underlying structure of resonances -- numerically.  Classical linearised frequency approximations, however,  neglect the nonlinear interactions in~\eqref{osc}. This linearisation is illustrated in Table \ref{tab1} below for splitting and exponential integrator methods (\cite{H2Tri,HochOst10}). \begin{table}[h!]
\begin{subequations}
\begin{empheq}[box=\widefbox]{align*}
& \text{\em Numerical scheme} &\quad & \text{\em Approximation of nonlinear oscillations}\\
&\text{splitting method }&\quad & \mathcal{Osc}(\xi, \mathcal{L}, v,p)  \approx
 p\left(v , 
  \overline v \right)\\
&\text{exponential method }&\quad & \mathcal{Osc}(\xi, \mathcal{L}, v,p)     \approx 
e^{ -i\xi  \mathcal{L}\left(\nabla, \frac{1}{\varepsilon}\right)} p(v,\overline v)
\end{empheq}
\end{subequations}
\caption{Classical  {frequency approximations} of the   principal oscillations \eqref{osc}.}\label{tab1}
\end{table}

\noindent The aim of this paper is to introduce a  framework which allows us to embed the underlying {nonlinear} oscillations \eqref{osc} and their higher-order counterparts   into the numerical discretisation. The main idea for tackling this problem is to introduce a  decorated tree formalism that optimises the structure of the local error by mapping it to the particular regularity of the solution. 


While  first-order resonance-based discretisations have been presented for particular examples, e.g., the Nonlinear Schrödinger (NLS), Korteweg--de Vries (KdV), Boussinesq, Dirac  and Klein--Gordon equation,  see \cite{HS16,BFS17,BS19,OS18,OstS19,SWZ20}, no general framework could be established so far.   Each and every equation had to be targeted  carefully one at a time  based on  a {sophisticated resonance analysis}. This is due to the fact that the structure of the underlying oscillations \eqref{osc} strongly depends on the form of the leading operator $\mathcal{L}$, the nonlinearity $p$ and in particular their nonlinear interactions. 

In addition, to the lack of a general framework, only very little is known about the higher-order counterpart of resonance based discretisations. Indeed, some attempts  have been made for second-order schemes (see, e.g.,  \cite{HS16} for KdV and \cite{KOS19} for NLS) but they are not optimal. This is due to the fact that  the leading differential operator $\mathcal{L} $ triggers  a full spectrum of frequencies  $k_j \in \Z^{d}$. Up to now it was an unresolved issue on how to control their nonlinear interactions up to higher order, in particular, in higher spatial dimensions where stability poses a key problem.  Even in case of a simple NLS equation it  was an open question so far whether stable low regularity approximations of order higher than one can be achieved in  spatial dimensions $d\geq 2$. In particular, previous works  suggest a severe order reduction (\cite{KOS19}).

To overcome this we introduce a new tailored decorated tree formalism.  Thereby the decorated trees encode the Fourier coefficients in the iteration of  Duhamel's formula, where the  node decoration  encodes the frequencies which is in the spirit close to \cite{Christ,oh1,Gub11}. The main difficulty then lies in controlling the nonlinear frequency interactions within these iterated integrals up to the desired order with the constraint of a given a priori regularity of the solution. The latter is achieved by embedding the underlying  oscillations, and their higher order iterations, via  well-chosen Taylor series expansions into our formalism: The dominant interactions will be embedded exactly whereas only the lower order parts are approximated within the discretisation.


We base our algebraic structures on the ones developed for SPDEs with Regularity Structure \cite{reg} which is a generalisation of Rough Paths \cite{Lyo91,Lyo98,Gub04,Gub10}. Part of the formalism is inspired by \cite{BHZ} and the recentering map used for giving a local description of the solution of {singular SPDEs}. We adapt it to the context of dispersive PDEs by using a new class of decorated trees {encoding the underlying dominant frequencies}.

The framework of decorated trees and the underlying Hopf algebras have allowed the resolution of a large class of singular SPDEs \cite{reg,BHZ,ajay,BCCH} which include 
a natural random dynamic on the space of loops in
a Riemannian manifold in \cite{BGHZ}, see \cite{EMS} for a very brief survey on these developments. With this general framework, one can study properties of singular SPDEs solutions in full subcritical-regimes \cite{CHS,Berglund,support,CMW}. The formalism of decorated trees together with the description of the renormalised equation in this context (see \cite{BCCH}) was directly inspired from numerical analysis of ODEs, more precisely, from the characterisation of Runge-Kutta methods via B-series.
Indeed, B-series are numerical
(multi-)step methods for ODEs represented by a tree expansion, see, e.g.,   \cite{Butcher72,Berland,MR2657947,H2Tri,IQT,MR2803804}. We also refer to  \cite{LieSeries} for a review of  B-series on Lie groups and homogeneous manifolds as well as to   \cite{WordSeries} providing an alternative structure via word series. The field of singular SPDEs took advantage of the B-series formalism and extended their structures via the adjunction of decorations and Taylor expansions. Now, through this work, numerical analysis is taking advantage of these extended structures and enlarges their scope.

 This work proposes a new application of the Butcher-Connes-Kreimer Hopf algebra \cite{Butcher72,CK} to dispersive PDEs. It gives a new light on structures that have been used in various fields such as numerical analysis, renormalisation in quantum field theory, singular SPDEs and 
also dynamical systems for classifying singularities via resurgent functions introduced by Jean Ecalle (see \cite{Ecalle1,Ecalle1,FM}).
This is another testimony of the universality of this structure and adds a new object to this landscape. Our construction is motivated by two main features:
Taylor expansions that are at the foundation of the numerical scheme (added at the level of the algebra as for singular SPDEs) and the frequency interaction (encoded in a tree structure for dispersive PDEs). The combination of the two  together with the Butcher-Connes-Kreimer Hopf algebra allows us to design a novel class of schemes at low regularity. We observe a similar Birkhoff type factorisation as in SPDEs and perturbative quantum field theory. This factorisation allows us to single out oscillations and to perform the local error analysis.

Our main result is the new general resonance based scheme presented in Definition \ref{scheme}  with its error structure given in Theorem \ref{thm:genloc}. Our general framework is illustrated on concrete examples in Section \ref{sec:examples} and simulations show the efficacy of the scheme. The algebraic structure in Section \ref{General framework} has its own interest where the main objective is to understand the frequency interactions. The Birkhoff factorisation given in Section~\ref{sec::Brikhoff} is designed for this purpose and is a good help for  proving  Theorem  \ref{thm:genloc}. This factorisation seems new in comparison to the literature.

{\bf Assumptions.} We impose periodic boundary condition that is $x \in \T^d$. However, our theory can be extended to the full space $\R^d$.  We assume that the   differential operator $\mathcal{L}$ is real and consider two types of structures of the system \eqref{dis} which will allow us to handle dispersive equations at low regularity (such as NLS and KdV) and highly oscillatory Klein--Gordon type systems; see also \eqref{nlsIntro}-\eqref{kgrIntro}.
\begin{itemize}
\item  The differential operators $\mathcal{L}\left(\nabla, \frac{1}{\varepsilon}\right)  = \mathcal{L}\left(\nabla \right) $ and $\vert \nabla\vert^\alpha$ cast in Fourier  space into the form 
\begin{equs}\label{Lldef}
 \mathcal{L}\left(\nabla \right)(k) = k^\sigma + \sum_{\gamma : |\gamma| < \sigma} a_{\gamma} \prod_{j} k_j^{\gamma_j} ,\qquad \vert \nabla\vert^\alpha(k) = \prod_{ \gamma : |\gamma|  {\leq \alpha}} k_j^{\gamma_j}
\end{equs}
for some $ \alpha \in \R $, $ \gamma \in \Z^d $ and $ |\gamma| = \sum_i \gamma_i $,
where  for $k = (k_1,\ldots,k_d)\in \Z^d$ and $m = (m_1, \ldots, m_d)\in \Z^d$ we set \begin{equs}
k^\sigma  = k_1^\sigma + \ldots + k_d^\sigma, \qquad k \cdot m = k_1 m_1 + \ldots + k_d m_d.
\end{equs}
\item We also consider the setting of  a given  high frequency $\frac{1}{\vert \varepsilon \vert} \gg   1$. In this case we assume that the  operators $\mathcal{L}\left(\nabla, \frac{1}{\varepsilon}\right)  $ and $\vert \nabla\vert^\alpha$ take the form
 \begin{equs}\label{Leps} 
\mathcal{L}\left(\nabla, \frac{1}{\varepsilon}\right) =  \frac{1}{\varepsilon^{\sigma}} + \mathcal{B}\left(\nabla,  \frac{1}{\varepsilon}\right), \qquad \vert \nabla\vert^\alpha = \mathcal{C}\left(\nabla,  \frac{1}{\varepsilon}\right)
\end{equs} 
for some differential operators $\mathcal{B}\left(\nabla,  \frac{1}{\varepsilon}\right)$ and $\mathcal{C}\left(\nabla,  \frac{1}{\varepsilon}\right)$ which can be bounded uniformly in $ \vert \varepsilon\vert$ and are relatively bounded by differential operators of degree $\sigma$ and degree $\alpha< \sigma$, respectively.    This allows us to include for instance highly oscillatory Klein--Gordon type equations \eqref{kgrIntro} (see, also Section \ref{sec:kgr}).
\end{itemize}

\begin{figure}[h!]
\begin{subfigure}[c]{0.49\textwidth}
\includegraphics[width=1.\textwidth]{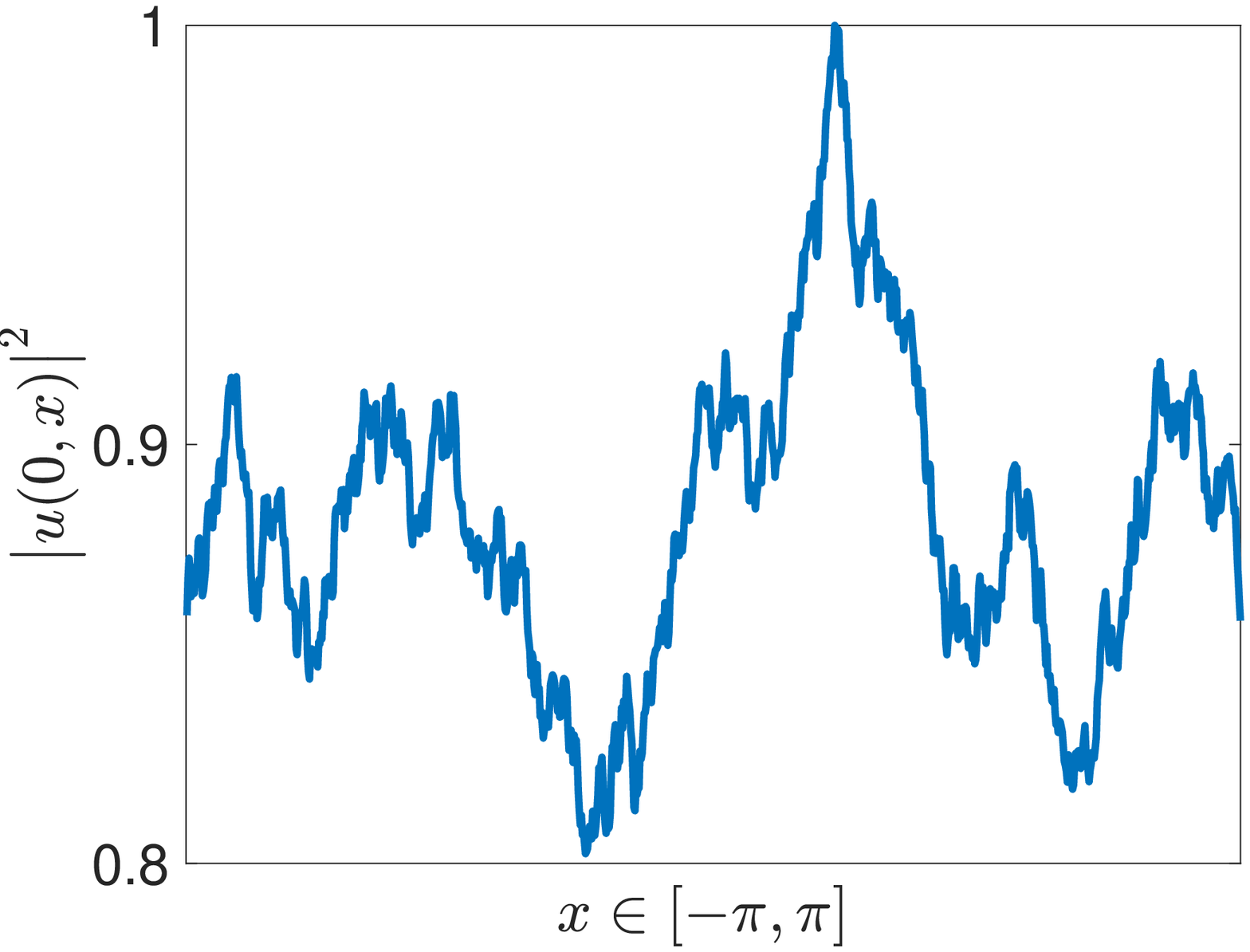}
\subcaption{$H^1$ data}
\end{subfigure}
\begin{subfigure}[c]{0.49\textwidth}
\includegraphics[width=1.\textwidth]{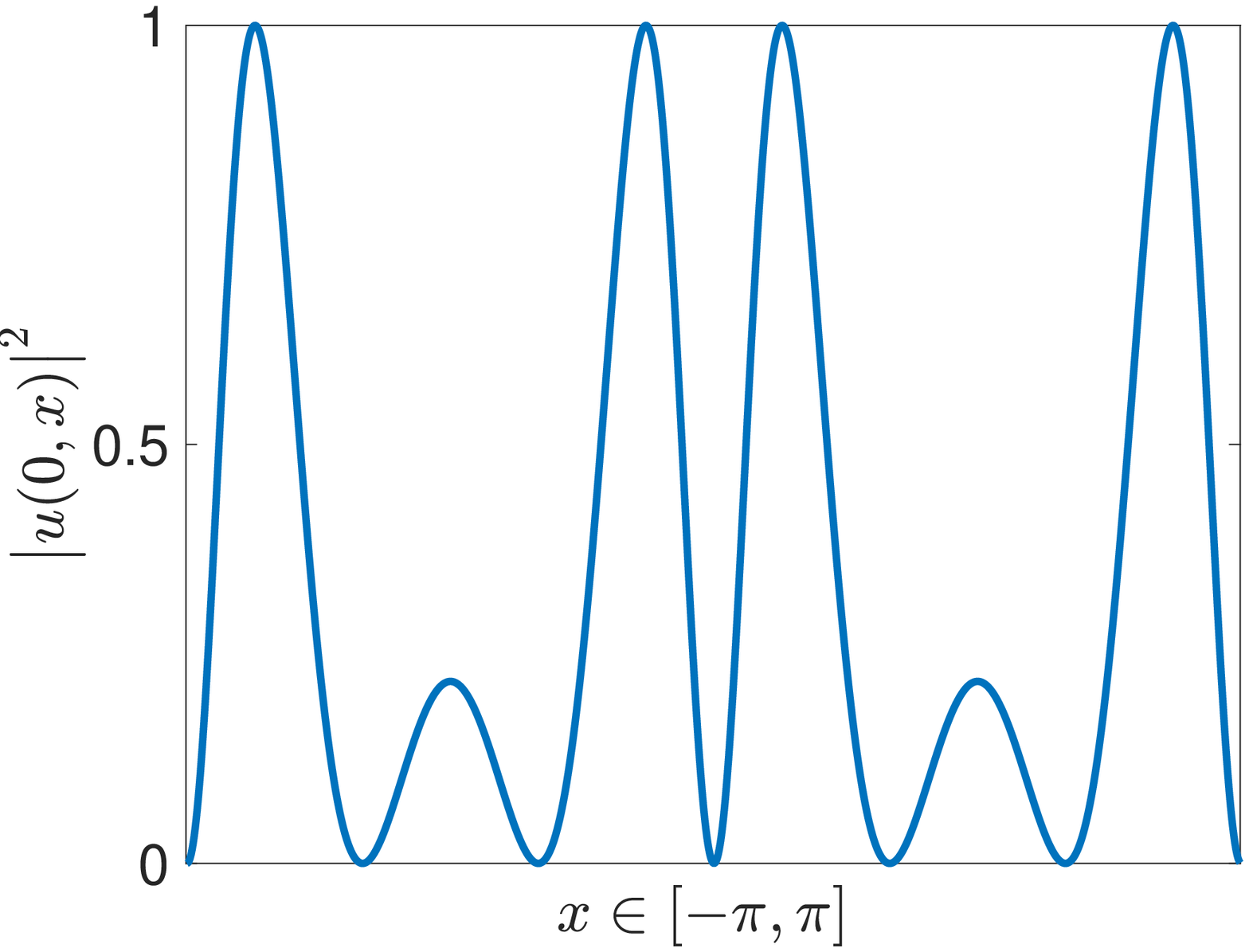}
\subcaption{$\mathcal{C}^\infty$ data}
\end{subfigure}\caption{Initial values for  Figure \ref{fig:osc}: $u_0 \in H^1$ (left) and $u_0 \in \mathcal{C}^\infty$ (right).}\label{fig:ini}
\end{figure}

\begin{figure}[h!]
\begin{subfigure}[c]{0.49\textwidth}
\includegraphics[width=1.\textwidth]{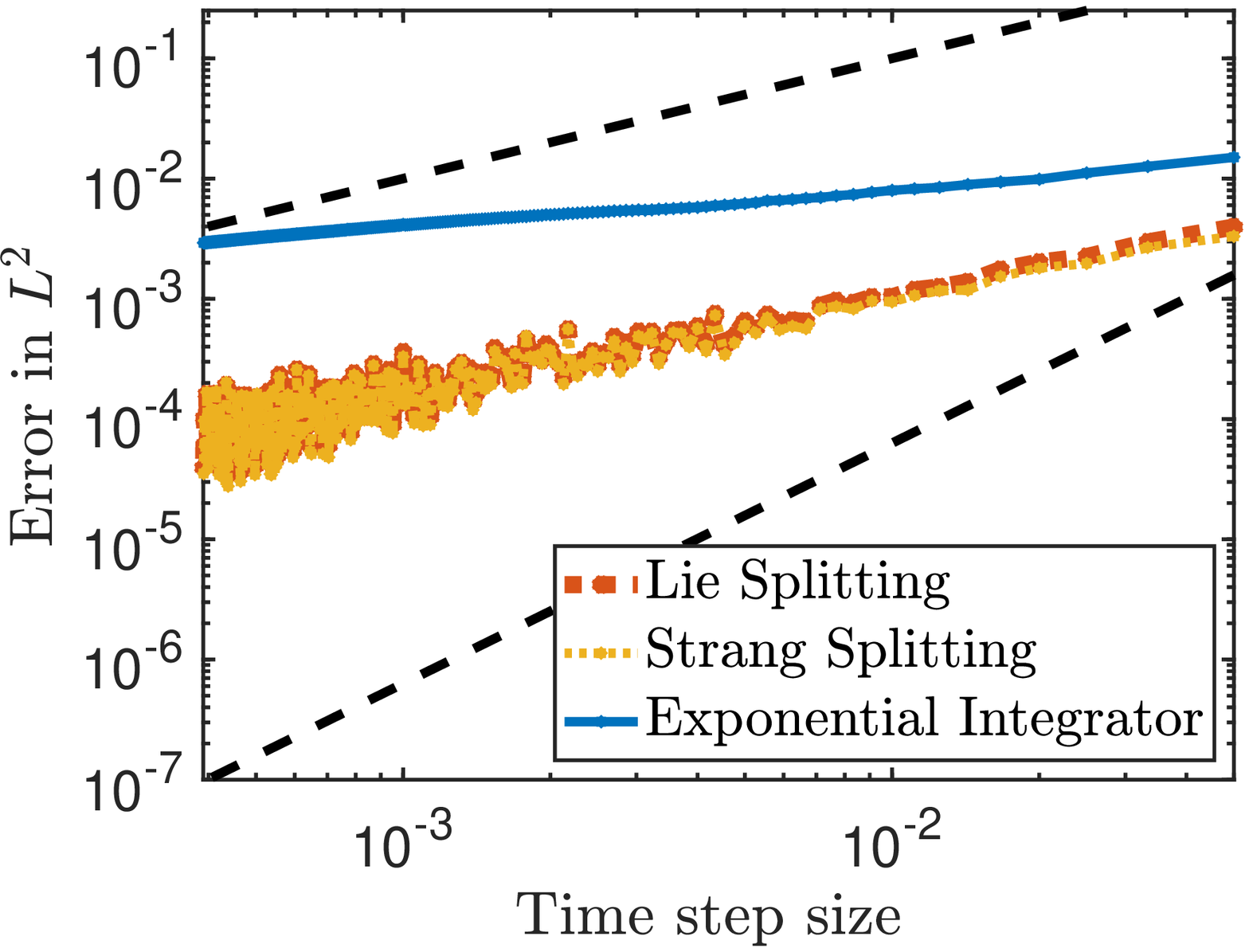}
\subcaption{$H^1$ data}
\end{subfigure}
\begin{subfigure}[c]{0.49\textwidth}
\includegraphics[width=1.\textwidth]{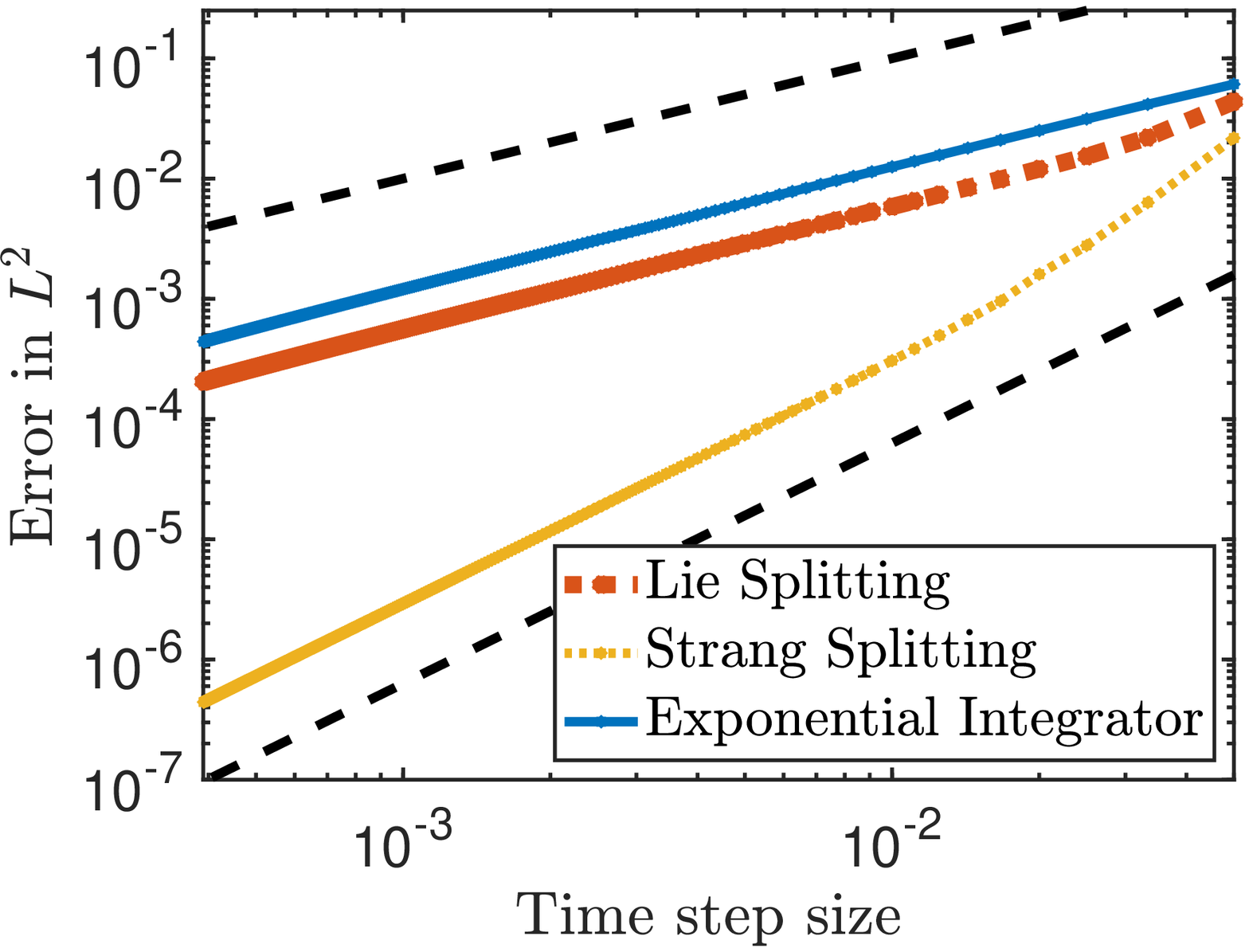}
\subcaption{$\mathcal{C}^\infty$ data}
\end{subfigure}
\caption{Order reduction of classical schemes based on linearised frequency approximations (cf. Table \ref{tab1}) in case of low regularity data (error versus step size for  the cubic Schrödinger equation).  For smooth solutions classical methods reach their full order of convergence (right picture). In contrast, for less smooth solutions they suffer from severe order reduction (left picture). The initial values in $H^1$ and $\mathcal{C}^\infty$ are plotted in Figure \ref{fig:ini}. The slope of the reference solutions (dashed lines) is one and two, respectively.}\label{fig:osc}
\end{figure}
In the next section we introduce the resonance based techniques to solve the dispersive PDE \eqref{dis} and illustrate our approach  on the example of cubic nonlinear Schr\"odinger equation \eqref{nlsIntro}, see Example \ref{ex:introExNLS}. 
\subsection{Resonances as a computational tool}\label{sec:res}
Instead of employing classical linearised frequency approximations (cf. Table~\ref{tab1})  we want to embed  the underlying nonlinear oscillations 
\begin{align}\label{oscii}
\mathcal{Osc}(\xi, \mathcal{L}, v,p)  = e^{ -i\xi  \mathcal{L}\left(\nabla, \frac{1}{\varepsilon}\right)}
p\left(e^{ i \xi  \mathcal{L}\left(\nabla, \frac{1}{\varepsilon}\right)} v , 
 e^{ - i \xi  \mathcal{L}\left(\nabla, \frac{1}{\varepsilon}\right)} \overline v \right)
\end{align}
  (and their higher-order counterparts) into the numerical discretisation.    In case of the one dimensional cubic Schr\"odinger equation \eqref{nlsIntro} the central oscillations~\eqref{oscii}  for instance  take in Fourier the form (see Example  \ref{ex:introExNLS} for details)
  $$
  \mathcal{Osc}(\xi, \Delta, v,\text{cub})  =\sum_{\substack{k_1,k_2,k_3 \in \Z\\-k_1+k_2+k_3 = k} }
 e^{i k  x } \overline{\hat{v}}_{k_1} \hat{v}_{k_2} \hat{v}_{k_3} \int_0^\tau e^{i s \mathscr{F}(k) } ds
  $$
  with the underlying resonance structure
\begin{align}\label{resiNL}
\mathscr{F}(k) = 2 k_1^2 - 2 k_1 (k_2+k_3) + 2 k_2  k_3.
\end{align}
 Ideally we would like to resolve all  nonlinear frequency interactions \eqref{resiNL}  exactly in our scheme.     \black However, these consequent into a generalised convolution (of Coifman--Meyer type \cite{Coif}) which can not be converted as a product into the physical space. Thus  the iteration would need to be carried out fully in Fourier space which does not  yield a scheme which can be  {practically} implemented  in higher spatial dimensions, see also Remark \ref{rem:FFT} below.  The latter in general also holds true in the abstract setting  \eqref{oscii}.\black 
 
In order to obtain an  {efficient and practical} resonance based discretisation we extract the dominant and lower-order parts from the resonance structure \eqref{oscii}. More precisely, we filter out the dominant parts $ \mathcal{L}_{\text{dom}}$  and treat them exactly while only approximating the lower order terms in the spirit of
\begin{equation}\label{oscNew}
\mathcal{Osc}(\xi, \mathcal{L}, v,p) 
 = 
\left[
e^{i \xi  \mathcal{L}_{\text{dom}} \left(\nabla, \frac{1}{\varepsilon}\right)} p_{\text{dom}}\left(v,\overline v\right)
\right] p_{\text{low}}(v,\overline v) + \mathcal{O}\Big(\xi\mathcal{L}_\text{low}\left(\nabla\right)v\Big).
\end{equation}
Here, $\mathcal{L}_{\text{dom}}$ denotes a suitable dominant part of the high frequency interactions and 
\begin{equation}\label{Llow}
\mathcal{L}_\text{low} = \mathcal{L} - \mathcal{L}_{\text{dom}}
\end{equation}
the corresponding non-oscillatory parts (details will be  given in Definition \ref{dom_freq}).  The crucial issue is to determine  $\mathcal{L}_{\text{dom}}$, $p_{\text{dom}}$ and $\mathcal{L}_\text{low}, p_{\text{low}}$ in~\eqref{oscNew} with an interplay between keeping the underlying structure of PDE and allowing a practical implementation at a reasonable cost. We refer to Example \ref{ex:introExNLS} for the concrete characterisation in case of   cubic NLS, where $\mathcal{L}_\text{low} = \nabla$ and  $\mathcal{L}_{\text{dom}} = \Delta$.

Thanks to the resonance based ansatz \eqref{oscNew} the principal oscillatory integral 
\[
\mathcal{I}_1( t, \mathcal{L},v,p)  = \int_0^t \mathcal{Osc}(\xi, \mathcal{L}, v,p)  d\xi
\]
in the expansion of the exact solution \eqref{It1}
\begin{equation}\label{duli}
\begin{aligned}
u(t)   = e^{ it  \mathcal{L}\left(\nabla, \frac{1}{\varepsilon}\right)} v - ie^{ it  \mathcal{L}\left(\nabla, \frac{1}{\varepsilon}\right)}   \vert \nabla\vert^\alpha \mathcal{I}_1( t, \mathcal{L},v,p) 
  + \mathcal{O}\left( t^2\vert \nabla\vert^{2 \alpha}{ q_1( v)} \right)
  \end{aligned}
\end{equation}
(for some polynomial $q_1$) then takes the form
\begin{equation}\label{pri1}
\begin{aligned}
\mathcal{I}_1( t, \mathcal{L},v,p)  
& = \int_0^t \left[
e^{i \xi  \mathcal{L}_{\text{dom}}} p_{\text{dom}}\left(v,\overline v\right)
\right] p_{\text{low}}(v,\overline v)  +  \mathcal{O}\Big(\xi\mathcal{L}_\text{low}\left(\nabla\right){ q_2(v)}\Big)d \xi\\
 & =  t
 p_{\text{low}}(v,\overline v)   \varphi_1\left(i t \mathcal{L}_{\text{dom}} \right) p_{\text{dom}}\left(v,\overline v\right)
  + \mathcal{O}\Big(t^2\mathcal{L}_\text{low}\left(\nabla\right){ q_2(v)}\Big)
\end{aligned}
\end{equation}
(for some polynomial $q_2$) where for shortness we write $\mathcal{L} = \mathcal{L}\left(\nabla, \frac{1}{\varepsilon}\right)$ and define $\varphi_1(\gamma) = \gamma^{-1}\left(e^\gamma -1\right)$ for $\gamma \in \C$. Plugging \eqref{pri1} into \eqref{duli} yields for a small time step $\tau$ that
\begin{multline}\label{ue1}
u(\tau)   = e^{ i\tau  \mathcal{L}} v - \tau ie^{ i\tau  \mathcal{L}}  \vert \nabla\vert^\alpha
 \Big[p_{\text{low}}(v,\overline v)   \varphi_1\left(i \tau \mathcal{L}_{\text{dom}}  \left(\nabla, \frac{1}{\varepsilon}\right)\right) p_{\text{dom}}\left(v,\overline v\right)
\Big] \\
  +   \mathcal{O}\left( \tau^2\vert \nabla\vert^{2 \alpha} q_
  1(v) \right) +  \mathcal{O}\Big(\tau^2\vert \nabla\vert^{ \alpha}\mathcal{L}_\text{low} \left(\nabla\right)q_2({ v})\Big) 
\end{multline}
for some polynomials $q_1, q_2$. The expansion of the exact solution  \eqref{ue1} builds the foundation of the first-order resonance based discretiatzion
\begin{align}\label{GenScheme1}
u^{n+1} = 
 e^{ i\tau  \mathcal{L}} u^n - \tau ie^{ i\tau  \mathcal{L}} 
{ \vert \nabla\vert^\alpha}  \Big[p_{\text{low}}(u^n,\overline u^n)   \varphi_1\left(i \tau\mathcal{L}_{\text{dom}} \left(\nabla, \frac{1}{\varepsilon}\right) \right) p_{\text{dom}}\left(u^n,\overline u^n\right)
\Big] .
\end{align}
Compared to classical linear frequency approximations (cf. Table \ref{tab1})  the main gain of the more involved resonance based approach \eqref{GenScheme1} is the following: All dominant parts $\mathcal{L}_\text{dom}$ are captured exactly in the discretisation, while only the lower order/non-oscillatory parts $\mathcal{L}_\text{low}$ are approximated. Henceforth, within the resonance based approach \eqref{GenScheme1} the local error only depends on the lower order, non-oscillatory operator $\mathcal{L}_\text{low}$, while the local error of classical methods involves the full operator $\mathcal{L}$ and, in particular, its dominant part $\mathcal{L}_\text{dom}$.  
Thus, the resonance based approach \eqref{GenScheme1} allows us to approximate a  more general class of solutions
\begin{multline}\label{DDD}
u \in \underbrace{ \mathcal{D}\left(\vert \nabla\vert^\alpha \mathcal{L}_{\text{low}}\left(\nabla, \frac{1}{\varepsilon}\right)\right) }_{\text{resonance domain}} \cap \mathcal{D}\left(\vert \nabla\vert^{2\alpha}\right)\\ \supset { \mathcal{D}\left(\vert \nabla\vert^\alpha \mathcal{L}\left(\nabla, \frac{1}{\varepsilon}\right)\right) }\cap \mathcal{D}\left(\vert \nabla\vert^{2\alpha}\right)
=\underbrace{ \mathcal{D}\left(\vert \nabla\vert^\alpha \mathcal{L}_{\text{dom}}\left(\nabla, \frac{1}{\varepsilon}\right)\right) }_{\text{classical domain}}\cap \mathcal{D}\left(\vert \nabla\vert^{2\alpha}\right).
\end{multline}
 
\noindent{\bf Higher order resonance based methods.} 
Classical approximation techniques, such as splitting or exponential integrator methods, can easily be extended to higher order, see, e.g., \cite{H2Tri,HochOst10,Ta12}. The step from a first- to  higher-order approximation lies in subsequently employing a higher order Taylor series expansion to the exact solution
\begin{align*}
u(t)  = u(0) + t\partial_t u(0) + \ldots + \frac{t^{r}}{r!} \partial_t^{r} u(0) + \mathcal{O} \left( t^{r+1} \partial_{t}^{r+1} u\right).
\end{align*}
Within this expansion, the higher order iterations of the oscillations \eqref{osc} in the exact solution are, however, not resolved, but subsequently linearised. Therefore,  classical high order methods are restricted to smooth solutions as their local approximation error in general involves high order derivatives \begin{equation}\label{highT}
\mathcal{O} \left( t^{r+1} \partial_{t}^{r+1} u\right)=\mathcal{O} \left( t^{r+1} \mathcal{L}^{r+1}\left(\nabla, \tfrac{1}{\varepsilon}\right) u\right).
\end{equation}  This phenomenon is also illustrated in Figure \ref{fig:osc} where we numerically  observe  the order reduction of the Strang splitting method (of classical order two) down to the order of the Lie splitting method (of classical order one) in case of rough solutions. In particular we observe that classical high order methods do not pay off at low regularity as their  error behaviour reduces to the one of lower order methods.

At first glance our resonance based approach can  also be straightforwardly extended to higher order. Instead of considering only the first order iteration~\eqref{It1} the natural idea is to iterate Duhamel's formula \eqref{duh} up to the desired order $r$, i.e, for initial value $u(0)=v$,
\begin{equation}
\begin{aligned}\label{Itpi}
 u(t)  &= e^{i t \mathcal{L}} v -i  e^{i t \mathcal{L}}\nabla^{\alpha} \int_0^te^{ -i \xi_1 \mathcal{L}}  p\left( e^{ i \xi_1 \mathcal{L}} v,e^{- i \xi_1 \mathcal{L}} \overline v\right) d\xi_1\\&
-e^{i t \mathcal{L}}\nabla^{\alpha}  \int_0^te^{ -i \xi_1 \mathcal{L}}  \Big[D_1 p \left( e^{ i \xi_1 \mathcal{L}} v,e^{- i \xi_1 \mathcal{L}} \overline v\right)\\&\qquad \cdot
e^{ i \xi_1\mathcal{L}}\nabla^{\alpha}  \int_0^{\xi_1} e^{ -i \xi_2 \mathcal{L}} p\left( e^{ i \xi_1 \mathcal{L}} v,e^{- i \xi_1 \mathcal{L}} \overline v\right) d\xi_2
\Big]d\xi_1 \\
&
+e^{i t \mathcal{L}}\nabla^{\alpha}  \int_0^te^{ -i \xi_1 \mathcal{L}}  \Big[D_2p \left( e^{ i \xi_1 \mathcal{L}} v,e^{- i \xi_1 \mathcal{L}} \overline v\right)\\&\qquad \cdot
e^{ -i \xi_1\mathcal{L}}\nabla^{\alpha}  \int_0^{\xi_1} e^{ i \xi_2 \mathcal{L}} \overline{p\left( e^{ i \xi_1 \mathcal{L}} v,e^{- i \xi_1 \mathcal{L}} \overline v\right) }d\xi_2
\Big]d\xi_1 \\
  & + \ldots+\nabla^{\alpha} \int_0^t\nabla^{\alpha} \int_0^\xi \ldots \nabla^{\alpha}\int_0^{\xi_r} d\xi_{r} \ldots d\xi_1 d \xi
\end{aligned}
\end{equation}
 where $ D_1 $ (resp. $ D_2 $) corresponds to the derivative
in the first (resp. second) component of $ p $.
The key idea  will then be the following: Instead of linearising the frequency interactions by a simple Taylor series expansions of  the oscillatory terms $e^{\pm i \xi_\ell \mathcal{L}}$ (as classical methods would do),  we want to embed the dominant frequency interactions  of \eqref{Itpi} exactly into our numerical discretisation.
By neglecting the last term involving the iterated integral of order $r$  we will then  introduce the  desired local error  $\mathcal{O}\Big(\nabla^{(r+1)\alpha} t^{r+1}q(u)\Big)$ for some polynomial $q$.

Compared to the first order approximation \eqref{oscNew} this is, however,  much more involved as  {high order  iterations} of  the nonlinear frequency interactions need to be controlled. The control of these iterated  oscillations is not only a delicate problem on the discrete (numerical) level, concerning accuracy, stability, etc., but already on the continuous level: We have to encode the structure (which strongly depends on the underlying structure of the PDE, i.e., the form of operator $\mathcal{L}$ and the shape of nonlinearity $p$) and at the same time keep track of the regularity assumptions.  In order to achieve this in the general setting \eqref{dis} we will introduce the decorated tree formalism in Section \ref{sec:rho}. Beforehand let us first illustrate the main ideas on the example of the cubic periodic Schr\"odinger equation.
\begin{example}[cubic periodic Schr\"odinger equation]\label{ex:introExNLS}
We consider the one dimensional cubic Schr\"odinger equation
\begin{align}\label{nlsIntro2}
i \partial_t u + \partial_x^2 u = \vert u\vert^2 u    
\end{align}
equipped with periodic boundary conditions, that is $x \in \T$. The latter  casts into the general form \eqref{dis} with
\begin{equation}\label{nlsDo}
\begin{aligned}
\mathcal{L}\left(\nabla, \tfrac{1}{\varepsilon}\right) & = \partial_x^2, \quad \alpha = 0 \quad \text{and}\quad  p(u,\overline u) =u^2 \overline u  .
 \end{aligned}
\end{equation} 
In the case of cubic NLS, the central oscillatory integral (at first order)  takes the form (cf. \eqref{osc})
\begin{equation}\label{I1d}
\mathcal{I}_1(\tau, \partial_x^2,v) = \int_0^\tau  e^{-i s   \partial_x^2}\left[ \left( e^{- i s   \partial_x^2} \overline v \right) \left ( e^{ i s  \partial_x^2} v \right)^2\right] d s.
 \end{equation}
Assuming that $v\in L^2$  the Fourier transform 
$
v(x)  = \sum_{k \in \Z}\hat{v}_k  e^{i k x}
$ 
allows us to express the action of the free Schrödinger group as a Fourier multiplier, i.e., 
\[
e^{\pm i t  \partial_x^2}v(x) = \sum_{k \in \Z} e^{\mp i t k^2} \hat{v}_k  e^{i k x}.
\]
With this at hand we can express the oscillatory integral \eqref{I1d} as follows
\begin{equation}\label{Ia}
\begin{aligned}
\mathcal{I}_1(\tau, \partial_x^2,v) 
 & = \sum_{\substack{k_1,k_2,k_3 \in \Z\\-k_1+k_2+k_3 = k} }
 e^{i k  x } \overline{\hat{v}}_{k_1} \hat{v}_{k_2} \hat{v}_{k_3} \int_0^\tau e^{i s \mathscr{F}(k) } ds
\end{aligned}
\end{equation}
with the underlying resonance structure
\begin{align}\label{resNLS}{
\mathscr{F}(k) = 2 k_1^2 - 2 k_1 (k_2+k_3) + 2 k_2  k_3}.
\end{align}
In the spirit of \eqref{oscNew} we need to extract the dominant and lower-order parts from the resonance structure \eqref{resNLS}. The choice is based on the following  observation. Note that $2k_1^2$ corresponds to a second-order derivative, i.e., with the inverse Fourier transform $\mathcal{F}^{-1}$, we have
\[
\mathcal{F}^{-1}\left(2k_1^2 \overline{\hat v}_{k_1} \hat v_{k_2} \hat v_{k_3}\right)
= \left(- 2\partial_x^2 \overline v\right) v^2
\]
while the terms $k_\ell \cdot k_m$ with $\ell \neq m$ correspond only to first-order derivatives, i.e.,
\[
\mathcal{F}^{-1}\left(k_1 \overline{\hat v}_{k_1} k_2 \hat v_{k_2} \hat v_{k_3}\right)
= -\vert\partial_x v\vert^2 v, \quad \mathcal{F}^{-1}\left( \overline{\hat v}_{k_1} k_2 \hat v_{k_2} k_3 \hat v_{k_3}\right)
= - (\partial_x v)^2\overline v.
\]
This motivates the choice
 \[
 \mathscr{F}(k) = \mathcal{L}_{\text{dom}}(k_1)  +   \mathcal{L}_{\text{low}}(k_1,k_2,k_3) 
 \]
with
\begin{equation}\label{domNLS0}
{
\mathcal{L}_{\text{dom}}(k_1) = 2k_1^2 \quad \text{and}\quad  \mathcal{L}_{\text{low}}(k_1,k_2,k_3) = - 2 k_1 (k_2+k_3) + 2 k_2 k_3}.
\end{equation}
In terms of \eqref{GenScheme1} we thus have
\begin{equation}
\begin{aligned}
\label{domNLS}
& \mathcal{L}_{\text{dom}} = - 2\partial_x^2, 
\quad \quad p_{\text{dom}}(v,\overline v) = \overline v  \quad \text{and}\quad 
p_{\text{low}}(v,\overline v) = v^2
\end{aligned}
\end{equation}
and the first-order NLS resonance based discretisation \eqref{GenScheme1} takes the form
\begin{align}\label{schemeNLSintro}
u^{n+1} = 
 e^{ i\tau  \partial_x^2} u^n - \tau ie^{ i\tau  \partial_x^2} 
 \Big[(u^n)^2 \varphi_1\left(-2 i \tau  \partial_x^2 \right)  \overline u^n
\Big] .
\end{align}
 Thanks to \eqref{ue1}  we readily see by \eqref{domNLS0} that the NLS scheme \eqref{schemeNLSintro} introduces  the approximation error
\begin{equation}\label{natscal}
\mathcal{O}\left(\tau^2 \mathcal{L}_{\text{low}}q(u)\right)= \mathcal{O}\left(\tau^2 \partial_xq(u)\right)
\end{equation}
for some polynomial $q$ in $u$. 
Compared to the error structure of classical discretisation techniques, which involve the full and thus dominant operator $\mathcal{L}_{\text{dom}} = \partial_x^2$, we thus gain one derivative with the resonance based scheme \eqref{schemeNLSintro}. This favorable error at low regularity is underlined in Figure \ref{fig.nlsintro}.

\begin{figure}[h!]\centering
\includegraphics[width=0.65\textwidth]{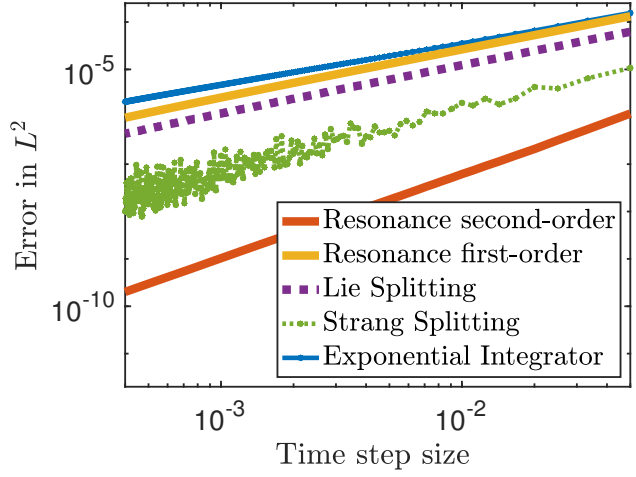}
\caption{Error versus step size (double logarithmic plot). Comparison of classical and resonance based schemes for the cubic Schr\"odinger equation \eqref{nlsIntro2} with $H^2$ initial data.}\label{fig.nlsintro}
\end{figure}

\end{example}

In Example \ref{ex:introExNLS} we illustrated the idea of the resonance based discretisation on the cubic periodic Schr\"odinger equation in one spatial dimension. In order to control frequency interactions in the general setting \eqref{dis} in arbitrary dimensions $d\geq 1$ up to arbitrary high order  we next  introduce our  decorated tree formalism.
\subsection{Main idea of decorated trees for high order resonance based schemes }\label{sec:rho}

The iteration of Duhamel's formulation \eqref{Itpi} can be expressed using decorated trees. We are interested in computing the iterated frequency interactions in \eqref{Itpi}.  This motivates us to express the latter in Fourier space. Let $ r $ be the order of the scheme and let us assume that we truncate \eqref{Itpi} at this order. Its $k-$th Fourier coefficient at order $  r $ is given by 
\begin{equs} \label{decoratedV1}
U_{k}^{r}(\tau, v) & = \sum_{T \in \CV^r_k} \frac{\Upsilon^{p}(T)(v)}{S(T)} \left( \Pi T \right)(\tau)
\end{equs}
where $ \CV^r_k $ is a set of decorated trees which incorporate the frequency $k$, $ S(T) $ is the symmetry factor associated to the tree $ T $, $ \Upsilon^{p}(T) $ is the coefficient appearing in the iteration of Duhamel's formulation and $ (\Pi T)(t) $ represents a Fourier iterated integral. The exponent $r$ in $ \CV^r_k $ means that we consider only 
trees of size $ r +1 $ which are the trees producing an iterated integral with $ r + 1$ integrals. The decorations that need to be put on the trees are illustrated in Example~\ref{ex:introExNLS trees}.

 The main difficulty then lies in developing  for every $T \in \CV^r_k$ a suitable approximation to the iterated integrals $ (\Pi T)(t) $ with the aim of minimising the local error structure (in the sense of regularity). In order to achieve this, the key idea is to embed - in the spirit of \eqref{oscNew} - the underlying resonance structure of the iterated integrals $ (\Pi T)(t) $ into the discretisation.
 
\begin{example}[cubic periodic Schr\"odinger equation with decorated trees]
\label{ex:introExNLS trees} 
When \newline  $r=2$, decorated trees for cubic NLS are given by 
\begin{equs}
T_0 = \begin{tikzpicture}[scale=0.2,baseline=-5]
\coordinate (root) at (0,1);
\coordinate (tri) at (0,-1);
\draw[kernels2] (tri) -- (root);
\node[var] (rootnode) at (root) {\tiny{$ k $}};
\node[not] (trinode) at (tri) {};
\end{tikzpicture} \qquad
T_1 = \begin{tikzpicture}[scale=0.2,baseline=-5]
\coordinate (root) at (0,2);
\coordinate (tri) at (0,0);
\coordinate (trib) at (0,-2);
\coordinate (t1) at (-2,4);
\coordinate (t2) at (2,4);
\coordinate (t3) at (0,5);
\draw[kernels2,tinydots] (t1) -- (root);
\draw[kernels2] (t2) -- (root);
\draw[kernels2] (t3) -- (root);
\draw[kernels2] (trib) -- (tri);
\draw[symbols] (root) -- (tri);
\node[not] (rootnode) at (root) {};
\node[not] (trinode) at (tri) {};
\node[var] (rootnode) at (t1) {\tiny{$ k_{\tiny{1}} $}};
\node[var] (rootnode) at (t3) {\tiny{$ k_{\tiny{2}} $}};
\node[var] (trinode) at (t2) {\tiny{$ k_{\tiny{3}} $}};
\node[not] (trinode) at (trib) {};
\end{tikzpicture} \qquad
 T_2 = \begin{tikzpicture}[scale=0.2,baseline=-5]
 \coordinate (root) at (0,2);
\coordinate (tri) at (0,0);
\coordinate (trib) at (0,-2);
\coordinate (t1) at (-2,4);
\coordinate (t2) at (2,4);
\coordinate (t3) at (0,4);
\coordinate (t4) at (0,6);
\coordinate (t41) at (-2,8);
\coordinate (t42) at (2,8);
\coordinate (t43) at (0,10);
\draw[kernels2,tinydots] (t1) -- (root);
\draw[kernels2] (t2) -- (root);
\draw[kernels2] (t3) -- (root);
\draw[symbols] (root) -- (tri);
\draw[symbols] (t3) -- (t4);
\draw[kernels2,tinydots] (t4) -- (t41);
\draw[kernels2] (t4) -- (t42);
\draw[kernels2] (t4) -- (t43);
\draw[kernels2] (trib) -- (tri);
\node[not] (trinode) at (trib) {};
\node[not] (rootnode) at (root) {};
\node[not] (rootnode) at (t4) {};
\node[not] (rootnode) at (t3) {};
\node[not] (trinode) at (tri) {};
\node[var] (rootnode) at (t1) {\tiny{$ k_{\tiny{4}} $}};
\node[var] (rootnode) at (t41) {\tiny{$ k_{\tiny{1}} $}};
\node[var] (rootnode) at (t42) {\tiny{$ k_{\tiny{3}} $}};
\node[var] (rootnode) at (t43) {\tiny{$ k_{\tiny{2}} $}};
\node[var] (trinode) at (t2) {\tiny{$ k_5 $}};
\end{tikzpicture}  \qquad
 T_3 = \begin{tikzpicture}[scale=0.2,baseline=-5]
\coordinate (root) at (0,2);
\coordinate (tri) at (0,0);
\coordinate (trib) at (0,-2);
\coordinate (t1) at (-2,4);
\coordinate (t2) at (2,4);
\coordinate (t3) at (0,4);
\coordinate (t4) at (0,6);
\coordinate (t41) at (-2,8);
\coordinate (t42) at (2,8);
\coordinate (t43) at (0,10);
\draw[kernels2] (t1) -- (root);
\draw[kernels2] (t2) -- (root);
\draw[kernels2,tinydots] (t3) -- (root);
\draw[symbols] (root) -- (tri);
\draw[symbols,tinydots] (t3) -- (t4);
\draw[kernels2] (t4) -- (t41);
\draw[kernels2,tinydots] (t4) -- (t42);
\draw[kernels2,tinydots] (t4) -- (t43);
\draw[kernels2] (trib) -- (tri);
\node[not] (trinode) at (trib) {};
\node[not] (rootnode) at (root) {};
\node[not] (rootnode) at (t4) {};
\node[not] (rootnode) at (t3) {};
\node[not] (trinode) at (tri) {};
\node[var] (rootnode) at (t1) {\tiny{$ k_{\tiny{4}} $}};
\node[var] (rootnode) at (t41) {\tiny{$ k_{\tiny{1}} $}};
\node[var] (rootnode) at (t42) {\tiny{$ k_{\tiny{3}} $}};
\node[var] (rootnode) at (t43) {\tiny{$ k_{\tiny{2}} $}};
\node[var] (trinode) at (t2) {\tiny{$ k_5 $}};
\end{tikzpicture} \label{treeK}
\end{equs}
where on the nodes we encode the frequencies such that they add up depending on the edge decorations. The root has no decoration. For example, in $T_1$ the two extremities of the blue edge have the same decoration 
given by $ -k_1 + k_2  + k_3 $ where the minus sign comes from the dashed edge. Therefore, $  \CV^r_k  $ contains infinitely many trees
(finitely many shapes, but infinitely many ways of splitting up the frequency $ k $ among the branches).  An edge of type $\<thick>$ encodes a multiplication by $ e^{-i \tau k^2} $ where $k$ is the frequency on the nodes adjacent to this edge. An edge of type $ \<thin> $ encodes an integration in time of the form
\begin{equs}
 \int_0^{\tau}  e^{i s k^2} \cdots d s.
\end{equs}
In fact $r +1$, the truncation parameter corresponds to the maximum number of integration in time that is the number of edges with type $\<thin>$.
The dashed dots on the edges correspond to a conjugate and a multiplication by $(-1)$ applied to the frequency at the top of this edge.
Then, if we apply the map $ \Pi $ (which encodes the oscillatory integrals in Fourier space, see Section~\ref{sec::recursive_pi})  to these trees, we obtain:
\begin{equs} \label{defPiex} \begin{aligned}
 (\Pi T_0)(\tau) & =  e^{-i \tau k^2},  \\ (\Pi T_1)(\tau) & =  - i e^{-i \tau k^2} \int_0^\tau  e^{i s  k^2}\left[ \left( e^{ i s   k_1^2}  \right) \left ( e^{ -i s  k_2^2}  \right) \left ( e^{ -i s  k_3^2}  \right)  \right] d s 
  \\ & = - i  e^{-i \tau k^2} \int_0^{\tau} e^{is \mathscr{F}(k)} ds
  \\ (\Pi T_2)(\tau) & = -i  e^{-i \tau k^2} \int_0^\tau  e^{i s  k^2}\left[ \left( e^{ i s   k_4^2}  \right) \Big( (\Pi T_1)(s)  \Big) \left ( e^{ - i s  k_5^2}  \right)  \right] d s 
  \\ (\Pi T_3)(\tau) & = -i e^{-i \tau k^2} \int_0^\tau  e^{i s  k^2}\left[  \Big( \overline{(\Pi T_1)(s)}  \Big)\left ( e^{ -i s  k_4^2}  \right)  \left ( e^{ -i s  k_5^2}  \right)  \right] d s
  \end{aligned}
\end{equs}
where the resonance structure $\mathscr{F}(k)$ is given in \eqref{resNLS}.  One has the constraints $k= -k_1 +k_2 +k_3$ for $T_1$,  $k= -k_1 + k_2 + k_3 -k_4 + k_5$ for $ T_2 $ and $k= k_1 - k_2 - k_3 +k_4 + k_5$ for $ T_3 $. Using the definitions in Section~\ref{sec:genScheme}, one can compute the following coefficients:
\begin{equs}
\frac{\Upsilon^p(T_0)(v)}{S(T_0)} & = \hat v_k, \quad \frac{\Upsilon^p(T_1)(v)}{S(T_1)} = \bar{\hat{v}}_{k_1} \hat{v}_{k_2} \hat{v}_{k_3} \\
\frac{\Upsilon^p(T_2)(v)}{S(T_2)} & = 2   \overline{\hat{v}}_{k_1} \hat v_{k_2} \hat v_{k_3} \overline{\hat{v}}_{k_4} \hat v_{k_5}, \quad  \frac{\Upsilon^p(T_3)(v)}{S(T_3)} =      \hat v_{k_1} \overline{\hat{v}}_{k_2}  \overline{\hat{v}}_{k_3} \hat v_{k_4}  \hat v_{k_5}
 \end{equs}
  which together with the character $ \Pi$ encode fully the identity {\eqref{decoratedV1}}.\\

\end{example}

Our general scheme is based on the approximation of $(\Pi T)(t)$ for every tree in $\CV_k^r$.  This approximation is given by a new map of decorated trees denoted by $\Pi^{n,r}$ where $r$ is the order of the scheme and $n$ corresponds to the a priori  assumed regularity of the initial value $v$.  This new  character $\Pi^{n,r}$  will embed the dominant frequency interactions and neglect the lower order terms in the spirit of \eqref{oscNew}.
Our general scheme will thus take the form
\begin{equs} \label{decoratedV2}
U_{k}^{n,r}(\tau, v) & = \sum_{T \in \CV^r_k} \frac{\Upsilon^{p}(T)(v)}{S(T)}  \left( \Pi^{n,r} T \right)(\tau)
\end{equs}
where the map $ \Pi^{n,r} T $ is  a low regularity  approximation of order $ r $ of the map $ \Pi T $ in the sense that
\begin{equs}\label{eq:loci}
\left(\Pi T - \Pi^{n,r} T \right)(\tau)  = \mathcal{O}\left( \tau^{r+2} \mathcal{L}^{r}_{\text{\tiny{low}}}(T,n) \right).
\end{equs}
Here $\mathcal{L}^{r}_{\text{\tiny{low}}}(T,n)$ involves all lower order frequency interactions that we neglect in our resonance based discretisation.  At first order this approximation is  illustrated in~\eqref{pri1}. The scheme \eqref{decoratedV2} and the local error approximations~\eqref{eq:loci} are the main results of this work (see Theorem~\ref{thm:genloc}). Let us give the main ideas on how to obtain them.

The approximation $ \Pi^{n,r} $ is constructed from a character $ \Pi^n $ defined on the vector space $ \CH $ spanned by  decorated forests taking values in  a space $ \CC $  which depends on the frequencies of the decorated trees (see, e.g., \eqref{treeK} in case of NLS). However, we will add at the root the  additional decoration $r$ which stresses that this tree will be an approximation of order $r$. 
  For this purpose we will introduce the symbol $\CD^r$  (see, e.g.,    \eqref{exDr} for $T_1$ of NLS). 
Indeed, we disregard trees which have more integrals in time than the order of the scheme.  In particular we note that $ \Pi^n \CD^r(T)  = \Pi^{n,r} T$.

The map $ \Pi^n $ is defined recursively from an operator $ \CK $ which will compute a suitable approximation (matching the regularity of the solution) of the integrals introduced by the iteration of Duhamel's formula. This map $ \CK $ corresponds to the high order counterpart of the approach described in Section~\ref{sec:res}: {It} embeds the idea of singling out the dominant parts and integrating them exactly while only approximating the lower order terms, allowing for an improved local error structure compared to classical approaches. The character $ \Pi^n $ is the main map for computing the numerical scheme in Fourier space.

\begin{example}[cubic periodic Schr\"odinger equation: computation of $\Pi^n$]
\label{ex:introExNLS Pi_n} 
We consider the  decorated trees $\CD^r(\bar T_1)$ and $\bar T_1$  given by
\begin{equs} \label{exDr}
\bar T_1 = \begin{tikzpicture}[scale=0.2,baseline=-5]
\coordinate (root) at (0,0);
\coordinate (tri) at (0,-2);
\coordinate (t1) at (-2,2);
\coordinate (t2) at (2,2);
\coordinate (t3) at (0,3);
\draw[kernels2,tinydots] (t1) -- (root);
\draw[kernels2] (t2) -- (root);
\draw[kernels2] (t3) -- (root);
\draw[symbols] (root) -- (tri);
\node[not] (rootnode) at (root) {};t
\node[not] (trinode) at (tri) {};
\node[var] (rootnode) at (t1) {\tiny{$ k_{\tiny{1}} $}};
\node[var] (rootnode) at (t3) {\tiny{$ k_{\tiny{2}} $}};
\node[var] (trinode) at (t2) {\tiny{$ k_3 $}};
\end{tikzpicture},\, \quad 
\CD^r(\bar T_1) = \begin{tikzpicture}[scale=0.2,baseline=-5]
\coordinate (root) at (0,0);
\coordinate (tri) at (0,-2);
\coordinate (t1) at (-2,2);
\coordinate (t2) at (2,2);
\coordinate (t3) at (0,3);
\draw[kernels2,tinydots] (t1) -- (root);
\draw[kernels2] (t2) -- (root);
\draw[kernels2] (t3) -- (root);
\draw[symbols] (root) -- (tri);
\node[not] (rootnode) at (root) {};t
\node[not,label= {[label distance=-0.2em]below: \scriptsize  $ r $}] (trinode) at (tri) {};
\node[var] (rootnode) at (t1) {\tiny{$ k_{\tiny{1}} $}};
\node[var] (rootnode) at (t3) {\tiny{$ k_{\tiny{2}} $}};
\node[var] (trinode) at (t2) {\tiny{$ k_3 $}};
\end{tikzpicture}.
 \end{equs}
  One can observe that 
$
  (\Pi  T_1)(t)   = e^{-i k^2 t} (\Pi \bar T_1)(t). 
$ We  will define recursively two maps $\mathscr{F}_{\text{\tiny{dom}}} $ and $ \mathscr{F}_{\text{\tiny{low}}} $ (see Definition~\ref{dom_freq} below) on decorated trees that compute the dominant and the lower part of the  nonlinear frequency interactions within the oscillatory integral  $ (\Pi \bar T_1)(t) $.  In this example, one gets back the values already computed in \eqref{domNLS0}; i.e.,
 \begin{equs} 
\mathscr{F}_{\text{\tiny{dom}}}(\bar T_1) =  \mathcal{L}_{\text{\tiny{dom}}}(k_1), \quad  \mathscr{F}_{\text{\tiny{low}}}(\bar T_1)  =\mathcal{L}_{\text{\tiny{low}}}(k_1,k_2,k_3). \end{equs}
   Moreover, the dominant part of $T_1$ is due to the observation  that  $(\Pi  T_1)(t)   = e^{-i k^2 t} (\Pi \bar T_1)(t)$ given by:
   \begin{equs}
\mathscr{F}_{\text{\tiny{dom}}}( T_1) = -k^2 +\mathscr{F}_{\text{\tiny{dom}}}(\bar T_1) ,
    \end{equs}
    because the tree $T_1$ does not start with an intregral in time. Then, one can write:
    \begin{equs}
    (\Pi \bar T_1)(t) = -i \int_0^\tau e^{i s\mathscr{F}_{\text{\tiny{dom}}}(\bar T_1) }  e^{i s\mathscr{F}_{\text{\tiny{low}}}(\bar T_1) } ds
     \end{equs}
     and Taylor-expand around $0$ the lower order term, i.e., the factor containing $  \mathscr{F}_{\text{\tiny{low}}}(\bar T_1)$. The term $\Pi^{n,1} \bar T_1 = \Pi^n \CD^1(\bar T_1)$ is then given by:
     \begin{equs} \label{schemeT1}
    ( \Pi^{n,1} \bar T_1)(t) = - i \int_0^\tau e^{i s\mathscr{F}_{\text{\tiny{dom}}}(\bar T_1) }  ds +  \mathscr{F}_{\text{\tiny{low}}}(\bar T_1)  \int_0^\tau s e^{i s  \mathscr{F}_{\text{\tiny{dom}}}(\bar T_1) }  ds .
     \end{equs}
      One observes that we obtain terms of the form $ \frac{P}{Q} e^{i R} $ where
     $P,Q, R$ are polynomials in the frequencies $ k_1, k_2, k_3 $. Linear combinations of these terms are actually the definition of the space $ \CC $. For the local error, one gets
     \begin{equs}\label{lol}
     ( \Pi^{n,1} \bar T_1)(t) - ( \Pi \bar T_1)(t) = \CO( t^3  \mathscr{F}_{\text{\tiny{low}}} (\bar T_1)^2  ).
     \end{equs}
   Here the term $ \mathscr{F}_{\text{\tiny{low}}} (\bar T_1)^2$  corresponds to the regularity that one has to impose on the solution.  One can check by hand that the expression of $ \Pi^{n,1} \bar T_1$ can be mapped back to the physical space. Such a statement  will in general hold true for the character $ \Pi^n $, see Proposition~\ref{physical_space}.  This will be important for the practical implementation of the new schemes, see also Remark \ref{rem:FFT} below. We have not used $n$ in the description of the scheme yet. In fact, it plays a role in the expression of $ \Pi^{n,1} \bar T_1 $.  One has to compare $ n $ with the regularity required by the local error \eqref{lol} introduced by the polynomial  $ \mathscr{F}_{\text{\tiny{low}}} (\bar T_1)^2 $, but  also with the term $\mathscr{F}_{\text{\tiny{dom}}}(\bar T_1)^2 $. Indeed, if the initial value is regular enough we may want to Taylor expand all the frequencies, i.e., even the dominant parts, in order to get a simpler scheme, see also Remark~\ref{rem:regi} below. 
\end{example}

In order to obtain a better understanding  of the error introduced by the character $ \Pi^n $, one needs to isolate each interaction. Therefore, we  will introduce  two characters $ \hat \Pi^n : \CH \rightarrow \CC $ and $ A^n : \CH \rightarrow \C $ such that  
\begin{equs} \label{Birkhoff1}
\Pi^n = \left( \hat \Pi^n \otimes A^n  \right) \Delta
\end{equs}
where $ \Delta : \CH \rightarrow \CH \otimes \CH_+ $ is a coaction and $ (\CH,\Delta) $ is a right comodule for a Hopf algebra  $ \CH_+ $ equipped with a coproduct $ \Deltap $ and an antipode $ \CA $. 
 In fact, on can show that:
 \begin{equs}\label{Birkhoff2}
 \hat \Pi^n = \left( \Pi^n \otimes \left( \CQ \circ \Pi^n \CA \cdot \right)(0) \right) \Delta,  \quad A^n =  (\CQ \circ \Pi^n \cdot)(0)
 \end{equs}
 where  $ \Pi^n  $ is extended to a character on $ \CH_+ $ and $ \CQ $ is a projection defined on $ \CC $ which keeps only the terms with no oscillations.
The identity \eqref{Birkhoff2} can be understood as a Birkhoff type factorisation of $ \hat \Pi^n $ using the character $ \Pi^n $. This identity is also reminiscent in the main results obtained for singular SPDEs \cite{BHZ} where two twisted antipodes play a fundamental role providing  a variant of the algebraic Birkhoff factorisation.

\begin{example}[cubic periodic Schr\"odinger equation: Birkhoff factorisation] Integrating the first term in \eqref{schemeT1} exactly yields two contributions:
\begin{equs}
\int_0^\tau e^{i s\mathscr{F}_{\text{\tiny{dom}}}(\bar T_1) }  ds = \frac{e^{i\tau\mathscr{F}_{\text{\tiny{dom}}}(\bar T_1)  }}{i\mathscr{F}_{\text{\tiny{dom}}}(\bar T_1) } - \frac{1}{i\mathscr{F}_{\text{\tiny{dom}}}(\bar T_1) }.
\end{equs}
 Plugging these two terms into $(\Pi T_2)(\tau)$ defined in \eqref{defPiex}, we see that we have to control the following two terms:
\begin{equs}
\begin{aligned} \label{firstterm}
    - e^{-i \tau k^2} & \int_0^\tau  e^{i s  k^2}\left[ \left( e^{ i s   k_4^2}  \right) \Big(  \frac{e^{i\tau\mathscr{F}_{\text{\tiny{dom}}}(\bar T_1) - is \bar k^2 } }{i\mathscr{F}_{\text{\tiny{dom}}}(\bar T_1) } \Big) \left ( e^{- i s  k_5^2}  \right)  \right] d s \\
  -  e^{-i \tau k^2} & \int_0^\tau  e^{-i s  k^2}\left[ \left( e^{ i s   k_4^2}  \right) \Big(    - \frac{e^{-i s \bar k^2 }}{i\mathscr{F}_{\text{\tiny{dom}}}(\bar T_1) } \Big) \left ( e^{ -i s  k_5^2}  \right)  \right] d s
  \end{aligned}
\end{equs}
where $ \bar k = -k_1 + k_2 + k_3$.  
The frequency analysis is needed again for approximating the time integral and defining an approximation of $(\Pi T_2)(\tau)  $. One can see that the dominant part of these two terms may differ.
This implies that one can get two different local errors for the approximation of these two terms, the final local error is the maximum between the two. At this point, we need an efficient algebraic structure for dealing with all these frequency interactions in the iterated integrals.
We first consider a character $ \hat \Pi^n $ that keeps only the main
contribution that is the second term of \eqref{firstterm}. For any decorated tree $ T $, one expects $ \hat \Pi^n $ to be of the form:
\begin{equs}
(\hat \Pi^n \CD^r(T))(t)  = B^n(\CD^r(T))(t) e^{it\mathscr{F}_{\text{\tiny{dom}}}( T) }
 \end{equs}
 where $ B^n(\CD^r(T))(t) $ is a polynomial in $ t $ depending on the decorated tree $ T $. The character   $ \hat \Pi^n  $ singles out oscillations by keeping at each iteration only the non-zero one. This  separation between the various oscillations can be encoded via the Butcher-Connes-Kreimer coaction $ \Delta : \CH \rightarrow \CH \otimes \CH_+ $. An example of computation is given below:
 \begin{equs}
\Delta & \begin{tikzpicture}[scale=0.2,baseline=-5]
\coordinate (root) at (0,0);
\coordinate (tri) at (0,-2);
\coordinate (t1) at (-2,2);
\coordinate (t2) at (2,2);
\coordinate (t3) at (0,2);
\coordinate (t4) at (0,4);
\coordinate (t41) at (-2,6);
\coordinate (t42) at (2,6);
\coordinate (t43) at (0,8);
\draw[kernels2,tinydots] (t1) -- (root);
\draw[kernels2] (t2) -- (root);
\draw[kernels2] (t3) -- (root);
\draw[symbols] (root) -- (tri);
\draw[symbols] (t3) -- (t4);
\draw[kernels2,tinydots] (t4) -- (t41);
\draw[kernels2] (t4) -- (t42);
\draw[kernels2] (t4) -- (t43);
\node[not] (rootnode) at (root) {};
\node[not] (rootnode) at (t4) {};
\node[not] (rootnode) at (t3) {};
\node[not,label= {[label distance=-0.2em]below: \scriptsize  $ r $}] (trinode) at (tri) {};
\node[var] (rootnode) at (t1) {\tiny{$ k_{\tiny{4}} $}};
\node[var] (rootnode) at (t41) {\tiny{$ k_{\tiny{1}} $}};
\node[var] (rootnode) at (t42) {\tiny{$ k_{\tiny{3}} $}};
\node[var] (rootnode) at (t43) {\tiny{$ k_{\tiny{2}} $}};
\node[var] (trinode) at (t2) {\tiny{$ k_5 $}};
\end{tikzpicture} =
\begin{tikzpicture}[scale=0.2,baseline=-5]
\coordinate (root) at (0,0);
\coordinate (tri) at (0,-2) ;
\coordinate (t1) at (-2,2);
\coordinate (t2) at (2,2);
\coordinate (t3) at (0,2);
\coordinate (t4) at (0,4);
\coordinate (t41) at (-2,6);
\coordinate (t42) at (2,6);
\coordinate (t43) at (0,8);
\draw[kernels2,tinydots] (t1) -- (root);
\draw[kernels2] (t2) -- (root);
\draw[kernels2] (t3) -- (root);
\draw[symbols] (root) -- (tri);
\draw[symbols] (t3) -- (t4);
\draw[kernels2,tinydots] (t4) -- (t41);
\draw[kernels2] (t4) -- (t42);
\draw[kernels2] (t4) -- (t43);
\node[not] (rootnode) at (root) {};
\node[not] (rootnode) at (t4) {};
\node[not] (rootnode) at (t3) {};
\node[not,label= {[label distance=-0.2em]below: \scriptsize  $ r $} ] (trinode) at (tri) {};
\node[var] (rootnode) at (t1) {\tiny{$ k_{\tiny{4}} $}};
\node[var] (rootnode) at (t41) {\tiny{$ k_{\tiny{1}} $}};
\node[var] (rootnode) at (t42) {\tiny{$ k_{\tiny{3}} $}};
\node[var] (rootnode) at (t43) {\tiny{$ k_{\tiny{2}} $}};
\node[var] (trinode) at (t2) {\tiny{$ k_5 $}};
\end{tikzpicture} \otimes \one 
+ \one \otimes \begin{tikzpicture}[scale=0.2,baseline=-5]
\coordinate (root) at (0,0);
\coordinate (tri) at (0,-2);
\coordinate (t1) at (-2,2);
\coordinate (t2) at (2,2);
\coordinate (t3) at (0,2);
\coordinate (t4) at (0,4);
\coordinate (t41) at (-2,6);
\coordinate (t42) at (2,6);
\coordinate (t43) at (0,8);
\draw[kernels2,tinydots] (t1) -- (root);
\draw[kernels2] (t2) -- (root);
\draw[kernels2] (t3) -- (root);
\draw[symbols] (root) -- (tri);
\draw[symbols] (t3) -- (t4);
\draw[kernels2,tinydots] (t4) -- (t41);
\draw[kernels2] (t4) -- (t42);
\draw[kernels2] (t4) -- (t43);
\node[not] (rootnode) at (root) {};
\node[not] (rootnode) at (t4) {};
\node[not] (rootnode) at (t3) {};
\node[not,label= {[label distance=-0.2em]below: \scriptsize  $ (r,0) $}] (trinode) at (tri) {};
\node[var] (rootnode) at (t1) {\tiny{$ k_{\tiny{4}} $}};
\node[var] (rootnode) at (t41) {\tiny{$ k_{\tiny{1}} $}};
\node[var] (rootnode) at (t42) {\tiny{$ k_{\tiny{3}} $}};
\node[var] (rootnode) at (t43) {\tiny{$ k_{\tiny{2}} $}};
\node[var] (trinode) at (t2) {\tiny{$ k_5 $}};
\end{tikzpicture} + \lambda \otimes \begin{tikzpicture}[scale=0.2,baseline=-5]
\coordinate (root) at (0,0);
\coordinate (tri) at (0,-2);
\coordinate (t1) at (-2,2);
\coordinate (t2) at (2,2);
\coordinate (t3) at (0,2);
\coordinate (t4) at (0,4);
\coordinate (t41) at (-2,6);
\coordinate (t42) at (2,6);
\coordinate (t43) at (0,8);
\draw[kernels2,tinydots] (t1) -- (root);
\draw[kernels2] (t2) -- (root);
\draw[kernels2] (t3) -- (root);
\draw[symbols] (root) -- (tri);
\draw[symbols] (t3) -- (t4);
\draw[kernels2,tinydots] (t4) -- (t41);
\draw[kernels2] (t4) -- (t42);
\draw[kernels2] (t4) -- (t43);
\node[not] (rootnode) at (root) {};
\node[not] (rootnode) at (t4) {};
\node[not] (rootnode) at (t3) {};
\node[not,label= {[label distance=-0.2em]below: \scriptsize  $ (r,1) $}] (trinode) at (tri) {};
\node[var] (rootnode) at (t1) {\tiny{$ k_{\tiny{4}} $}};
\node[var] (rootnode) at (t41) {\tiny{$ k_{\tiny{1}} $}};
\node[var] (rootnode) at (t42) {\tiny{$ k_{\tiny{3}} $}};
\node[var] (rootnode) at (t43) {\tiny{$ k_{\tiny{2}} $}};
\node[var] (trinode) at (t2) {\tiny{$ k_5 $}};
\end{tikzpicture}  + \cdots
\\ & +
\begin{tikzpicture}[scale=0.2,baseline=-5]
\coordinate (root) at (0,0);
\coordinate (tri) at (0,-2);
\coordinate (t1) at (-2,2);
\coordinate (t2) at (2,2);
\coordinate (t3) at (0,3);
\draw[kernels2,tinydots] (t1) -- (root);
\draw[kernels2] (t2) -- (root);
\draw[kernels2] (t3) -- (root);
\draw[symbols] (root) -- (tri);
\node[not] (rootnode) at (root) {};t
\node[not,label= {[label distance=-0.2em]below: \scriptsize  $ r $}] (trinode) at (tri) {};
\node[var] (rootnode) at (t1) {\tiny{$ k_{\tiny{4}} $}};
\node[var] (rootnode) at (t3) {\tiny{$ \ell $}};
\node[var] (trinode) at (t2) {\tiny{$ k_5 $}};
\end{tikzpicture} \otimes \begin{tikzpicture}[scale=0.2,baseline=-5]
\coordinate (root) at (0,0);
\coordinate (tri) at (0,-2);
\coordinate (t1) at (-2,2);
\coordinate (t2) at (2,2);
\coordinate (t3) at (0,3);
\draw[kernels2,tinydots] (t1) -- (root);
\draw[kernels2] (t2) -- (root);
\draw[kernels2] (t3) -- (root);
\draw[symbols] (root) -- (tri);
\node[not] (rootnode) at (root) {};t
\node[not,label= {[label distance=-0.2em]below: \scriptsize  $ (r-1,0) $}] (trinode) at (tri) {};
\node[var] (rootnode) at (t1) {\tiny{$ k_{\tiny{1}} $}};
\node[var] (rootnode) at (t3) {\tiny{$ k_{\tiny{2}} $}};
\node[var] (trinode) at (t2) {\tiny{$ k_3 $}};
\end{tikzpicture}    +  \begin{tikzpicture}[scale=0.2,baseline=-5]
\coordinate (root) at (0,0);
\coordinate (tri) at (0,-2);
\coordinate (t1) at (-2,2);
\coordinate (t2) at (2,2);
\coordinate (t3) at (0,3);
\draw[kernels2,tinydots] (t1) -- (root);
\draw[kernels2] (t2) -- (root);
\draw[kernels2] (t3) -- (root);
\draw[symbols] (root) -- (tri);
\node[not] (rootnode) at (root) {};t
\node[not,label= {[label distance=-0.2em]below: \scriptsize  $ r $}] (trinode) at (tri) {};
\node[var] (rootnode) at (t1) {\tiny{$ k_{\tiny{4}} $}};
\node[var1] (rootnode) at (t3) {\tiny{$ ^1_{\ell} $}};
\node[var] (trinode) at (t2) {\tiny{$ k_5 $}};
\end{tikzpicture} \otimes \begin{tikzpicture}[scale=0.2,baseline=-5]
\coordinate (root) at (0,0);
\coordinate (tri) at (0,-2);
\coordinate (t1) at (-2,2);
\coordinate (t2) at (2,2);
\coordinate (t3) at (0,3);
\draw[kernels2,tinydots] (t1) -- (root);
\draw[kernels2] (t2) -- (root);
\draw[kernels2] (t3) -- (root);
\draw[symbols] (root) -- (tri);
\node[not] (rootnode) at (root) {};t
\node[not,label= {[label distance=-0.2em]below: \scriptsize  $ (r-1,1) $}] (trinode) at (tri) {};
\node[var] (rootnode) at (t1) {\tiny{$ k_{\tiny{1}} $}};
\node[var] (rootnode) at (t3) {\tiny{$ k_{\tiny{2}} $}};
\node[var] (trinode) at (t2) {\tiny{$ k_3 $}};
\end{tikzpicture}   + \cdots,
\quad
\hat \CD^{(r,m)}  \left(\begin{tikzpicture}[scale=0.2,baseline=-5]
\coordinate (root) at (0,0);
\coordinate (tri) at (0,-2);
\coordinate (t1) at (-2,2);
\coordinate (t2) at (2,2);
\coordinate (t3) at (0,3);
\draw[kernels2,tinydots] (t1) -- (root);
\draw[kernels2] (t2) -- (root);
\draw[kernels2] (t3) -- (root);
\draw[symbols] (root) -- (tri);
\node[not] (rootnode) at (root) {};t
\node[not] (trinode) at (tri) {};
\node[var] (rootnode) at (t1) {\tiny{$ k_{\tiny{1}} $}};
\node[var] (rootnode) at (t3) {\tiny{$ k_{\tiny{2}} $}};
\node[var] (trinode) at (t2) {\tiny{$ k_3 $}};
\end{tikzpicture}  \right) = \begin{tikzpicture}[scale=0.2,baseline=-5]
\coordinate (root) at (0,0);
\coordinate (tri) at (0,-2);
\coordinate (t1) at (-2,2);
\coordinate (t2) at (2,2);
\coordinate (t3) at (0,3);
\draw[kernels2,tinydots] (t1) -- (root);
\draw[kernels2] (t2) -- (root);
\draw[kernels2] (t3) -- (root);
\draw[symbols] (root) -- (tri);
\node[not] (rootnode) at (root) {};t
\node[not,label= {[label distance=-0.2em]below: \scriptsize  $ (r,m) $}] (trinode) at (tri) {};
\node[var] (rootnode) at (t1) {\tiny{$ k_{\tiny{1}} $}};
\node[var] (rootnode) at (t3) {\tiny{$ k_{\tiny{2}} $}};
\node[var] (trinode) at (t2) {\tiny{$ k_3 $}};
\end{tikzpicture}  
\end{equs}
where $ \ell = - k_1 + k_2 + k_3$. The space $\CH_+$ corresponds to the forest of planted trees where for each planted tree the edge connecting the root to the rest of the tree must be blue (an integration in time). Only blue edges are cut (located on the right hand side of the tensor product) and the trunk is on the left hand side. The extra terms missing  in the computation correspond to the higher order terms  introduced by the Taylor approximation. Indeed, one plays with decorations introducing a $m$th derivative on the blue edges cut denoted by $ \hat \CD^{(r,m)} $ and decorations on the nodes where the edges were previously attached.  A node of the form $ \begin{tikzpicture}[scale=0.2,baseline=-5]
\coordinate (root) at (0,0);
\node[var1] (rootnode) at (root) {\tiny{$ _\ell^1 $}};
\end{tikzpicture} $ in the example above corresponds to the frequency $ \ell $ and the monomial $ \lambda $. The length of the Taylor extension is dictated by the order of the scheme $ r $. The operator $ \hat \CD^{(r,m)} $ is non-zero only if $ m \leq r+1 $. The formulae \eqref{Birkhoff1} and \eqref{Birkhoff} give the relation between $ \Pi^n $ and $ \hat \Pi^n $ which can be interpreted as a Birkhoff factorisation with the explicit formula \eqref{Birkhoff} for $ A^n $.
Such a factorisation is new and does not seem to have an equivalent in the literature. It is natural to observe this factorisation in this context: The integration in time $\int^{t}_0 ... ds$ gives two different frequency interactions that can be controlled via a projection $ \CQ $ which needs to be iterated deeper in the tree. This will be the equivalent of the Rota-Baxter map used  for such a type of factorisation.
\end{example}

The coproduct $ \Deltap $ and the coaction $ \Delta $ are extremely close in spirit to the ones defined for the recentering in  \cite{reg,BHZ}. Indeed, for designing a numerical scheme, we need to perform Taylor expansions and these two maps are performing them at the level of the algebra. The main difference with the tools used for singular SPDEs \cite{BHZ} is the length of the Taylor expansion which is now dictated by the order of the scheme. 

The structure, we propose in Section~\ref{sec:genframe} is new and reveals the universality of deformed Butcher-Connes-Kreimer coproducts which appear in \cite{BHZ}. The non-deformed version of this map is coming from the analysis of B-series in \cite{Butcher72,MR2657947,MR2803804}, which is itself an extension of the Connes-Kreimer Hopf algebra
of rooted trees \cite{CK,CKI} arising in perturbative
QFT and noncommutative geometry.

One can notice that our approximation $ \Pi^{n,r} $ depends on $ n $ which has to be understood as the regularity we assume a priori on the solution. We design our framework such that for smooth solutions the numerical schemes are simplified, recovering in the limit classical linearised approximations as in Table \ref{tab1}. 


\begin{remark}\label{rem:regi}
The term $ \mathcal{L}^{r}_{\text{\tiny{low}}}(T,n)  $  in the approximation \eqref{eq:loci} is obtained by performing several Taylor expansions. Depending on the value $ n $, we get different numerical schemes (see also the applications in Section \ref{sec:examples}). In the sequel, we focus on two specific values of $ n $ associated to two particular schemes. We consider $ n^{r}_{\text{\tiny{low}}}(T) $ and $ n^{r}_{\text{\tiny{full}}}(T) $ given by:
\begin{equs}
n^{r}_{\text{\tiny{low}}}(T) = \deg(\mathcal{L}^{r}_{\text{\tiny{low}}}(T)), \quad n^{r}_{\text{\tiny{full}}}(T) = \deg(\mathcal{L}^{r}_{\text{\tiny{full}}}(T)),
\end{equs}
where $ \mathcal{L}^{r}_{\text{\tiny{low}}}(T) $ corresponds to the error obtained when we integrate exactly the dominant part $ \mathcal{L}_{\text{\tiny{dom}}}(T)$ and Taylor expand only the lower order part $ \mathcal{L}_{\text{\tiny{low}}}(T) $, while the term $\mathcal{L}^{r}_{\text{\tiny{full}}}(T)$ corresponds to the error one obtains when we  Taylor expand the full operator $\mathcal{L}(T) =  \mathcal{L}_{\text{\tiny{dom}}}(T) +  \mathcal{L}_{\text{\tiny{low}}}(T) $. One has
\begin{equ}[e:deg]
 \deg(\mathcal{L}^{r}_{\text{\tiny{low}}}(T,n)) = \left\{ \begin{aligned}
  &  n^{r}_{\text{\tiny{low}}}(T), & \quad  \,
  \text{if } & \, n \leq n^{r}_{\text{\tiny{low}}}(T) , \\
  & n, & \quad  \,
  \text{if } & \, n^{r}_{\text{\tiny{low}}}(T) \leq  n \leq n^{r}_{\text{\tiny{full}}}(T) ,
  \\
   &  n^{r}_{\text{\tiny{full}}}(T), & \quad  \,
  \text{if } & \, n \geq  n^{r}_{\text{\tiny{full}}}(T) . \\
  \end{aligned} \right.
\end{equ}
At the level of the scheme, we get
\begin{equ}[e:scheme]
 \Pi^{n,r} T = \left\{ \begin{aligned}
  &  \Pi_{\text{\tiny{low}}}^{r} T, & \quad  \,
  \text{if } & \, n \leq n^{r}_{\text{\tiny{low}}}(T) , \\
  &  \Pi^{n,r} T , & \quad  \,
  \text{if } & \, n^{r}_{\text{\tiny{low}}}(T) \leq  n \leq n^{r}_{\text{\tiny{full}}}(T) ,
  \\
   & \Pi_{\text{\tiny{full}}}^{r} T , & \quad  \,
  \text{if } & \, n \geq  n^{r}_{\text{\tiny{full}}}(T)  \\
  \end{aligned} \right.
\end{equ}
where we call $ \Pi_{\text{\tiny{low}}}^{r} T $  the minimum regularity resonance based scheme. This scheme corresponds to the minimisation of the local-error and we can observe a plateau. Indeed, if $ n $ is too small then by convention we get this scheme. This could be the case if one does not compute the minimum regularity needed a priori.  

The other scheme $ \Pi_{\text{\tiny{full}}}^{r} T $ corresponds to a classical exponential  type discretisation, where enough regularity is assumed such that also the dominant  components of the iterated integrals can be expanded into a Taylor series as in \eqref{highT}. Then, we observe a second plateau: indeed assuming more regularity will not change the scheme as we have already Taylor-expanded all the components. 

Compared to $ \Pi_{\text{\tiny{low}}}^{r} T $  the scheme $ \Pi_{\text{\tiny{full}}}^{r} T $ is in general much simpler as no nonlinear frequency interactions are taken into account. This comes at the cost that a smaller class of equations can be solved as much higher regularity assumptions are imposed.

Between these two schemes, lies a large class of intermediate schemes  $ \Pi^{n,r} T $ which we call {\em low regularity resonance based schemes}. They take advantage of Taylor-expanding a bit more when more regularity is assumed.
Therefore, the complexity of the schemes is decreasing as $n $ increases see also Section  \ref{sec:examples}.
We can represent these different regimes through the diagram below. 
\begin{equs}
\begin{tikzpicture}[scale=1,baseline=2cm]
    \fill [blush, domain=0:2, variable=\x]
      (0, 0)
      -- plot ({\x }, 2)
      -- (2, 0)
      -- cycle;
      \fill [greenryb, domain=2:4, variable=\x]
      (2, 0)
      -- plot ({\x }, {\x})
      -- (4, 0)
      -- cycle;
 \fill [brightcerulean, domain=4:6, variable=\x]
      (4, 0)
      -- plot ({\x }, 4)
      -- (6, 0)
      -- cycle;
    \draw [thick] [->] (0,0)--(6,0) node[right, below] {$n$};
     \draw[xshift=2 cm, thick] (-1pt,0pt)--(1pt,0pt) node[below] {$ n^{r}_{\text{\tiny{low}}}(T) $};
     \draw[xshift=3 cm,yshift=1 cm]  node[below] {$\substack{ \text{\tiny{Low Regularity }} \\
   \text{\tiny{Resonance scheme }  } }$};
   \draw[xshift=1 cm,yshift=1 cm]  node[below] {$\substack{ \text{\tiny{Minimum Regularity }} \\
   \text{\tiny{Resonance scheme }  } }$};
     \draw[xshift=5 cm,yshift=1 cm]  node[below] {$\substack{ \text{\tiny{Classical Exponential  }} \\
   \text{\tiny{Integrator type scheme }  } }$};
       \draw[xshift=4 cm, thick] (-1pt,0pt)--(1pt,0pt) node[below] {$ n^{r}_{\text{\tiny{full}}}(T) $};
    \draw [thick] [->] (0,0)--(0,5) node[above, left] {$\deg(\mathcal{L}^{r}_{\text{\tiny{low}}}(T,n))$};
       \draw[yshift=2 cm, thick] (-1pt,0pt)--(1pt,0pt) node[left] {$ n^{r}_{\text{\tiny{low}}}(T) $};
       \draw[yshift=4 cm, thick] (-1pt,0pt)--(1pt,0pt) node[left] {$ n^{r}_{\text{\tiny{full}}}(T) $};
\draw [domain=0:2, variable=\x]
      plot ({\x}, {2}) node[right] at (1.5,2) {};
    \draw [domain=2:4, variable=\x]
      plot ({\x}, {\x}) node[right] at (1.5,2) {};
      \draw [domain=4:6, variable=\x]
      plot ({\x}, {4}) node[right] at (1.5,2) {};
  \end{tikzpicture}
  \end{equs}
  \end{remark}
  \begin{remark}
  Within our framework we propose a  stabilisation technique. This will allow us to improve previous higher order attempts   breaking formerly imposed order barriers of higher order resonance based schemes, such as the order reduction down to $3/2$ suggested for Schrödinger equations in dimensions $d\geq 2$  in  \cite{KOS19}. Details are given in Remark \ref{rem:stab} as well as Section \ref{sec:examples}.
  \end{remark}
  \begin{remark}\label{rem:FFT}
 The aim is to choose the central approximation $\Pi^{n,r} T$ as an interplay between optimising the local error  in the sense of regularity while allowing for a practical implementation.  We design our schemes in such a way that  products of functions can always be mapped  back to physical space.  In practical computations, this  will  allow us to benefit from the Fast Fourier Transform (FFT) with computational effort of order $\mathcal{O}\left(\vert K\vert^d \text{log}\vert K\vert ^d\right)$ in dimension $d$, where $K$ denotes the highest frequency in the discretisation. However, it comes at the cost that the approximation error  \eqref{eq:loci} involves lower order derivatives.
 If, on the other hand, we would  embed all nonlinear frequency interactions into the discretisation the resulting schemes would need to be carried out fully in Fourier space   causing large memory and computational efforts of order $\mathcal{O}\left(K^{d \cdot \text{deg}{p}}\right)$, where $\text{deg}(p)$ denotes the degree of the nonlinearity $p$.
 \end{remark}
 
 {\begin{remark}
For notational simplicity, we focus on equations with polynomial nonlinearities (cf.  \eqref{dis}). Nevertheless, our scheme \eqref{decoratedV2} allows for a generalisation to non-polynomial nonlinearities of type $$f(u) g(\overline u)$$ for smooth functions $f$ and $g$. In the latter case  the iteration of Duhamel's formula boils down to a two steps algorithm. More precisely,  imagine that we got a first expansion of the form $ e^{i s\mathcal{L}} v + A(v,s) $ where $ A(v,s) $ is a linear combination of iterated integrals. Then, when iterating Duhamel's formula we need to plug this expansion into the nonlinearity and  perform a Taylor expansion around the point $ e^{is \mathcal{L}}v $:
\begin{equs}
f(e^{i s\mathcal{L}} v + A(v,s)) = \sum_{m \leq r} \frac{A(v,s)^m}{m!} f^{(m)}(e^{i s\mathcal{L}} v ) + \mathcal{O}(A(v,s)^{r+1}).
\end{equs}
Carrying out the same manipulation for $ g(\overline{e^{i s\mathcal{L}} v + A(v,s)}) $ we  end up with terms of type
$$
\frac{A(v,s)^m}{m!} f^{(m)}(e^{i s\mathcal{L}} v ) 
\frac{\overline{A(v,s)}^n}{n!} g^{(n)}({e^{- i s\mathcal{L}}\overline v}) .
$$
At this point we can not  directly write down our  resonance based scheme due to the fact that the oscillations are still encapsulated inside $ f $ and $g$. In order to control these oscillations and their nonlinear interactions, we need to pull the oscillatory phases $ e^{\pm is\mathcal{L}} $ out of $f$ and $g$. This is achieved via expansions of the form
 \begin{equs}
f(e^{is \mathcal{L}}v) = \sum_{\ell \leq r} \frac{s^{\ell}}{\ell!} e^{is \mathcal{L}} \mathcal{C}^{\ell}[f,\mathcal{L}](v) + \mathcal{O}(s^{r+1} \mathcal{C}^{r+1}[f,\mathcal{L}](v))
\end{equs}
where $\mathcal{C}^{\ell}[f,\mathcal{L}]$ denote nested commutators which in general require (much) less regularity than powers of the full operator $\mathcal{L}^\ell$. After these two linearisation steps, we are able to use the same machinery that leads to the construction of our scheme~\eqref{decoratedV2}.

Such commutators were also recently exploited in \cite{RS} for second-order methods.
\end{remark}}

\subsection{Outline of the paper}  

Let us give a short review of the content of this paper.
In Section~\ref{sec:genframe}, we introduce the general algebraic framework by first defining a suitable vector space of decorated forests $ \hat \CH $. Next we define the dominant frequencies of a decorated forest (see Definition~\ref{dom_freq}) and show that one can map  them back into physical space   (see Corollary~\ref{physical_map}) which will be  important  for the efficiency of the numerical schemes (cf. Remark~\ref{rem:FFT}).
 Then, we introduced two spaces of decorated forests $ \CH_+ $ and $\CH$. The latter $ \CH $ is used for describing approximated iterated integrals. The main difference with the previous space is that now we project along the order $ r $ of the method. We define the maps for the coaction $ \Delta : \CH \rightarrow \CH \otimes \CH_+  $ and the coproduct $ \Deltap : \CH_+ \rightarrow \CH_+ \otimes \CH_+ $  in   \eqref{eq:co-action_plus} and \eqref{eq:coproduct_plus}. In addition we provide a recursive definition for them in  \eqref{def_deltas}.  We prove in Proposition~\ref{Hopf_algebras} that these maps give a right-comodule structure for $ \CH $ over the Hopf algebra $ \CH_+ $. Moreover, we get a simple expression for the antipode $ \CA $ in Proposition~\ref{antipode_rec}.

In Section~\ref{sec::Iterated integrals}, we construct the approximation of the iterated integrals given by the character $ \Pi : \hat \CH \rightarrow \CC  $ (see \eqref{Pi}) through the character $ \Pi^n : \CH \rightarrow \CC $ (see \eqref{recursive_pi_r}). The main operator used for the recursive construction is $ \CK $ given in Definition~\ref{Taylor_exp}. We introduce a new character $ \hat \Pi^n : \CH  \rightarrow \CC $ through a Birkhoff type factorisation obtained from the character $ \Pi^n $ (see Proposition~\ref{Birkhoff}). Thanks to $ \hat \Pi^n $, we are able to conduct the local error analysis and show one of the main results of the paper: the error estimate on the difference  between $ \Pi  $ and its approximation $ \Pi^n $ (see Theorem~\ref{approxima_tree}). 
In Section~\ref{sec:genScheme}, we introduce  decorated trees stemming from Duhamel's formula via the rules formalism (see Definition~\ref{rules}). Then, we are able to introduce the general scheme (see Definition~\ref{genscheme}) and conclude on its local error structure (see Theorem~\ref{thm:genloc}).

In Section~\ref{sec:examples} we illustrate the general framework on  various applications and conclude in Section~\ref{sec:num} with numerical experiments underlying the  favourable error behaviour of the new resonance based schemes for non-smooth, and in certain cases even for smooth, solutions.

\subsection*{Acknowledgements}

{\small
We  wish to thank the anonymous referee for her/his extremely valuable remarks. First discussions on this work
were initiated while the authors participated in the workshop "Algebraic and geometric aspects of numerical methods for differential equations"
 held at the  Institut Mittag-Leffler in July 2018.
  The authors thank the organisers of this workshop for putting together a stimulating program  bringing different communities together, and the members of the institute for providing a friendly working atmosphere. This project has received funding from the European Research Council (ERC) under the European Union's Horizon 2020 research and innovation programme  (grant agreement No. 850941).} 

\section{General framework}\label{sec:genframe}
\label{General framework}

In this section, we present the main algebraic framework   which will allow us to develop and analyse our general numerical scheme. {We start by introducing decorated trees that encode the oscillatory integrals.} Decorations on the edges represent  integrals in time and some operators  stemming from   Duhamel's formula.  In addition, we impose decorations on the nodes {for the frequencies and}  potential monomials. We will compute the   corresponding dominant and  lower order frequency interactions   associated to these trees via the recursive maps  $\mathscr{F}_{\text{\tiny{dom}}} $ and $ \mathscr{F}_{\text{\tiny{low}}} $ given in  Definition~\ref{dom_freq}. These maps are chosen such that the solution is approximated at low regularity in Fourier space with the additional property that the approximation can be mapped back to physical space. The latter will allow for an efficient practical implementation of the new scheme, see Remark \ref{rem:FFT}.

The second part of this section focuses on a different space of decorated trees that we  name \emph{approximated decorated trees}. The main difference with the trees previously introduced is the additional root decoration by some integer $ r $. The  approximated trees have to be understood as an abstract version of an approximation of order $ r $ of the corresponding oscillatory integral. In order to construct our low regularity scheme we want to carry out an abstract Taylor expansions of the time integrals at the level of these approximated trees  in the spirit of \eqref{oscNew}: We will Taylor expand only the lower parts of the frequency interactions while integrating the dominant part exactly. For these operations, we need a deformed Butcher-Connes-Kreimer coproduct in the spirit of the one which has been introduced for singular SPDEs. We consider a coproduct $\Deltap $ and a coaction $ \Delta $ with a non-recursive (see \eqref{eq:co-action_plus} and \eqref{eq:coproduct_plus}) and a recursive definition (see \eqref{def_deltas}). We show the usual coassociativity/compatibility properties in Proposition~\ref{bialgebra}. In the end we get a Hopf algebra and a comodule structures on these new spaces of approximated decorated trees. The antipode comes for free in this context since the Hopf algebra is connected, see Proposition~\ref{Hopf_algebras}.
The main novelty of this general algebraic framework is the merging of two different structures that appear in dispersive PDEs (frequency interactions) and singular SPDEs (abstract Taylor expansion).  They form the central part of the scheme - controlling the underlying oscillations and performing Taylor approximations.  Within this construction we need to introduce new objects that were not considered before in such generality.

\subsection{Decorated trees and frequency interactions}
We consider a set of decorated trees following the formalism developed in \cite{BHZ}. These trees will encode the Fourier coefficients of the  numerical scheme.

We assume   a finite set $  \Lab$ and  frequencies $ k_1,...,k_n \in \Z^{d} $. The set $ \Lab$ parametrizes a set of differential
operators with constant coefficients, whose symbols are given by the polynomials  $ (P_{\Labhom})_{\Labhom \in \Lab} $.
We define the set of decorated trees $ \hat \CT  $ 
as elements of the form  $ 
T_{\Labe}^{\Labn, \Labo} =  (T,\Labn,\Labo,\Labe) $ where 
\begin{itemize}
\item $ T $ is a non-planar rooted tree with root $ \varrho_T $, node set $N_T$ and edge set $E_T$. We denote the leaves of $ T $ by $ L_T $. $ T $ must also be a planted tree which means that there is only one edge outgoing the root.
\item the map $ \Labe : E_T \rightarrow \Lab \times \lbrace 0,1\rbrace$ are edge decorations. 
\item the map $ \Labn : N_T \setminus \lbrace\varrho_T \rbrace \rightarrow \N $ are node decorations. For every inner node $ v$, this map encodes a monomial of the form $ \xi^{\Labn(v)} $ where $ \xi $ is a time variable.
\item the map $ \Labo : N_T \setminus \lbrace\varrho_T \rbrace \rightarrow \Z^{d}$ are node decorations. These decorations are frequencies that satisfy for every inner node $ u $:
\begin{equs} \label{innerdecoration}
(-1)^{\mathfrak{p}(e_u)}\Labo(u)  = \sum_{e=(u,v) \in E_T} (-1)^{\mathfrak{p}(e)} \Labo(v)
\end{equs}
where $ \Labe(e) = (\Labhom(e),\mathfrak{p}(e)) $ and  $ e_u $ is the edge outgoing $ u $ of the form $ (v,u) $ . From this definition, one can see that the node decorations $ (\Labo(u))_{u \in L_T} $ determine the decoration of the inner nodes. We assume that the node decorations at the leaves are linear combinations of
the $ k_i $ with coefficients in $ \lbrace -1,0,1 \rbrace $.
\item we assume that the root of $ T $ has no decoration.
\end{itemize}

When the node decoration $ \Labn $ is zero, we will denote the decorated trees $ T_{\Labe}^{\Labn,\Labo} $ as
 $ T_{\Labe}^{\Labo} = (T,\Labo,\Labe)  $. The set of decorated trees satisfying such a condition is denoted by $ \hat \CT_0 $. We say that  $ \bar T_{\bar \Labe}^{\bar \Labo} $ is a decorated subtree of $ T_{\Labe}^{\Labo} \in \hat \CT_0 $ if  $ \bar T $ is a subtree of $ T $ and the restriction of the decorations $ \Labo, \Labe $ of $ T $ to $ \bar T $ are given by $ \bar \Labo $ and $ \bar \Labe $. Notice that because trees considered in this framework are always planted, we look only at subtrees that are planted. A planted subtree of $ T $ is of the form  $ T_e $ where $ e \in E_T $ and $ T_e $ corresponds to the tree above $ e $.  The nodes of $ T_e $ are given by all the nodes whose path to the root contains $ e $. 
 
\begin{example}[label=exa:cont2] \label{example1}
Below, we give an example of a decorated tree $ T_{\Labe}^{\Labn, \Labo} $ where the edges are labelled with numbers from  $ 1 $ to $ 7 $ and the set $  N_T \setminus \lbrace\varrho_T\rbrace  $ is labelled by  $\lbrace a,b,c,d,e,f,g \rbrace$:
\begin{equs}   
  \begin{tikzpicture}[scale=0.19,baseline=2cm]
        \node at (0,10)  (a) {}; 
         \node at (4,20)  (f) {}; 
          \node at (0,30)  (k) {}; 
           \node at (8,30) (l) {}; 
           \node at (20,20)  (g) {}; 
            \node at (16,30) (m) {};  
        \node at (12,15) (c) {}; 
         \node at (24,30) (p) {}; %
    \draw[kernel] (a) -- node [round1] {\tiny $\Labhom(1),\Labp(1)$}  (c) ; 
     \draw[kernel1] (c) -- node [round1] {\tiny $\Labhom(2),\Labp(2)$}  (f) ; 
     \draw[kernel1] (c) -- node [round1] {\tiny 
     $\Labhom(3),\Labp(3)$}  (g) ; 
     \draw[kernel1,black] (f) -- node [near end, round1] {\tiny $\Labhom(4),\Labp(4)$}  (k) ; 
     \draw[kernel1] (f) -- node [round1] {\tiny $\Labhom(5),\Labp(5)$}  (l) ; 
     \draw[kernel1] (g) -- node [round1] {\tiny $\Labhom(6),\Labp(6)$}  (m) ; 
     \draw[kernel1] (g) -- node [near end, round1] {\tiny $\Labhom(7),\Labp(7)$}  (p) ;
     
     \draw (p) node [rect2] {\tiny $ \Labn({g}), \mathfrak{f}({g}) $}  ;
    \draw (m) node [rect2] {\tiny $ \Labn({f}), \Labo({f}) $}  ;
    \draw (l) node [rect2] {\tiny $ \Labn({e}), \Labo({e}) $}  ;
     \draw (k) node [rect2] {\tiny $ \Labn(d), \Labo(d) $}  ;
    \draw (g) node [rect2] {\tiny $ \Labn(c),\Labo(c) $}  ;
    \draw (f) node [rect2] {\tiny $ \Labn(b),\Labo(b) $}  ;
    \draw (c) node [rect2] {\tiny $ \Labn(a),\Labo(a) $}  ;
\end{tikzpicture} 
\end{equs}
\end{example}

\begin{remark} The structure imposed on the node decorations \eqref{innerdecoration} is close to the one used in \cite{Christ,Gub11,oh1}. But in these works, the trees were designed only for one particular equation. In our framework, we cover a general class of dispersive equations  by having more decorations on the edges given by $ \Lab \times \lbrace 0,1 \rbrace $. The set $ \Lab $ keeps track of the differential operators in Duhamel's formulation.  The second edge decoration allows us to compute an abstract conjugate on the trees given in \eqref{bar}. 
\end{remark}

We  denote by $\hat H $ (resp. $ \hat H_0 $) the (unordered) forests composed of trees in $ \hat \CT $ (resp. $ \hat \CT_0 $) (including the empty forest denoted by $\one$). Their linear spans are denoted by $\hat \CH $ and  $ \hat \CH_0 $. We extend the definition of decorated  subtrees to forests by saying that $ T $ is a decorated subtree of the decorated forest $ F $ if their exists a decorated tree $  \bar T$ in $ F $ such that $ T $ is a decorated subtree of $ \bar T $.
The forest product is denoted by $ \cdot $ and the counit is $ \one^{\star} $ which is non-zero only on the empty forest.

In order to represent these decorated trees, we introduce a symbolic notation. An edge decorated by  $  o =  (\Labhom,\Labp) $ is denoted by $ \CI_{o} $. The symbol $  \CI_{o}(\lambda_{k}^{\ell} \cdot) : \hat \CH \rightarrow \hat \CH $ is viewed as  the operation that merges all the roots of the trees composing the forest into one node decorated by $(\ell,k) \in \N \times \Z^{d} $.  We obtain a decorated tree which is then grafted onto a new root with no decoration. If the condition $ \eqref{innerdecoration} $ is not satisfied on the argument then $\CI_{o}( \lambda_{k}^{\ell} \cdot)$ gives zero.
If $ \ell = 0 $, then the term $  \lambda_{k}^{\ell} $ is denoted by $  \lambda_{k} $ as a short hand notation for $  \lambda_{k}^{0} $. When $ \ell = 1 $, it will be denoted by 
$  \lambda_{k}^{1} $.
The forest product between $ \CI_{o_1}(  \lambda^{\ell_1}_{k_1}F_1) $ and $ \CI_{o_2}(  \lambda^{\ell_2}_{k_2}F_2) $ is given by:
\begin{equs}
 \CI_{o_1}(  \lambda^{\ell_1}_{k_1} F_1) \CI_{o_2}(  \lambda^{\ell_2}_{k_2} F_2) := \CI_{o_1}( \lambda^{\ell_1}_{k_1} F_1) \cdot \CI_{o_2}( \lambda^{\ell_2}_{k_2} F_2).  
\end{equs}
\begin{example}[label=exa:cont]\label{ex2}
The following symbol
\begin{equs}
  \CI_{(\Labhom(1),\Labp(1))}(  \lambda^{\Labn(a)}_{\Labo(a)}\CI_{(\Labhom(2),\Labp(2))}(  \lambda^{\Labn(b)}_{\Labo(b)})\CI_{(\Labhom(3),\Labp(3))}(  \lambda^{\Labn(c)}_{\Labo(c)}))
\end{equs}
encodes the tree
\begin{equs} \label{example2a}
 \begin{tikzpicture}[scale=0.19,baseline=2cm]
        \node at (0,10)  (a) {}; 
         \node at (4,20)  (f) {}; 
           \node at (20,20)  (g) {}; 
        \node at (12,15) (c) {}; 
    \draw[kernel] (a) -- node [round1] {\tiny $\Labhom(1),\Labp(1)$}  (c) ; 
     \draw[kernel1] (c) -- node [round1] {\tiny $\Labhom(2),\Labp(2)$}  (f) ; 
     \draw[kernel1] (c) -- node [round1] {\tiny 
     $\Labhom(3),\Labp(3)$}  (g) ; 
    \draw (g) node [rect2] {\tiny $ \Labn(c),\mathfrak{f}(c) $}  ;
    \draw (f) node [rect2] {\tiny $ \Labn(b),\Labo(b) $}  ;
    \draw (c) node [rect2] {\tiny $ \Labn(a),\Labo(a) $}  ;
\end{tikzpicture} 
\end{equs}
We will see later (in Example \ref{tex:kdv})
 that the above tree with suitabl{y} chosen decorations describes the first iterated integral of the Kortweg--de Vries equation \eqref{kdvIntro}.
 \end{example}

We are interested in the following quantity which represents the frequencies associated to this tree:
\begin{equs} \label{frenquency}
\mathscr{F}( T_{\Labe}^{\Labo} ) = \sum_{u \in N_T} 
P_{(\Labhom(e_u),\Labp(e_u))}(\Labo(u))
\end{equs}
where $ e_u $ is the edge outgoing $ u $ of the form $ (v,u) $ and 
\begin{equs}\label{Palpha}
 P_{(\Labhom(e_u),\Labp(e_u))}(\Labo(u)){ \, := \,} (-1)^{\Labp(e_u)}P_{\Labhom(e_u)}((-1)^{\Labp(e_u)}\Labo(u)) .
\end{equs} The term $ \mathscr{F}( T_{\Labe}^{\Labo})  $ has to be understood as a polynomial in multiple variables given by the $ k_i $. 

In the numerical scheme what matters  are the terms with maximal degree of frequency which are here the monomials of higher degree, cf. $\mathcal{L}_\text{\tiny{dom}}$. We compute them using the symbolic notation in the next section. 
We assume fixed $ \Lab_{+} \subset \Lab $.  This subset encodes integrals in time that are of the form $ \int_0^{\tau} e^{s P_{(\mathfrak{t},\mathfrak{p})}(\cdot)}  \cdots ds $ for $ (\mathfrak{t},\mathfrak{p}) \in \Lab_+ \times \lbrace 0,1 \rbrace$, see also its interpretation given in \eqref{Pi} below.
\subsection{Dominant parts of trees and physical space maps}
\begin{definition} \label{Dom_projection} Let $ P(k_1,...,k_n) $ a polynomial in the $ k_i $. If the higher monomials of $ P $ are of the form
\begin{equs}
a \sum_{i=1}^{n} (a_i k_i)^{m}, \quad a_i \in { \lbrace  0,1 \rbrace}, \, a \in \Z,
\end{equs}
 then  we define $ \CP_{\text{\tiny{dom}}}(P) $ as
\begin{equs} \label{physical map}
\CP_{\text{\tiny{dom}}}(P) = a \left(\sum_{i=1}^{n} a_i k_i\right)^{m}.
\end{equs}
Otherwise, it is zero.
\end{definition}
\begin{remark}
Given a polynomial $ P $, one can compute its dominant part $ \mathcal{L}_{\text{dom}} $ and its  lower part $\mathcal{L}_{\text{low}}$
\begin{equs}
\mathcal{L}_{\text{\tiny{dom}}} = \CP_{\text{\tiny{dom}}}(P), \quad \mathcal{L}_{\text{\tiny{low}}} = \left(\id  - \CP_{\text{\tiny{dom}}} \right)(P).
\end{equs}
In our discretisation we will treat the dominant parts of the frequency interactions $\mathcal{L}_{\text{\tiny{dom}}}$ exactly, while approximating the lower order parts $\mathcal{L}_{\text{\tiny{low}}} $ by Taylor series expansions (cf. also \eqref{oscNew}). This will be achieved by applying recursively the operator $ \CP_{\text{\tiny{dom}}} $ introduced in Definition~\ref{dom_freq}. 

Note that in  the special case that   $\CP_{\text{\tiny{dom}}}(P)  = 0$, we have to expand all frequency interactions into a Taylor series. The latter  for instance arises in context of quadratic Schr\"odinger equations
\begin{equs}\label{quadNLS}
i \partial_t u = -\Delta u + u^2, \quad u(0,x) = v(x)
\end{equs}
for which we face oscillations of type (cf. \eqref{oscii})
\begin{equs}
 \int_0^\tau e^{i \xi  \Delta } (e^{-i \xi \ \Delta}  v)^2 d\xi  & = 
\sum_{k_{1},k_{2}\in \Z^d}\hat{v}_{k_1}\hat{v}_{k_2}
e^{i (k_1+k_2) x}  \int_0^\tau  e^{ -i \big(k_1+k_2\big)^2 \xi} 
 e^{ i \big(k_1^2 + k_2^2  \big)\xi} d\xi
\\ &= 
\sum_{k_{1},k_{2}\in \Z^d}\hat{v}_{k_1}\hat{v}_{k_2}
e^{i (k_1+k_2) x}  \int_0^\tau  e^{ -2 i k_1 k_2 \xi} d\xi.
\end{equs}Here  we recall the notation 
 $
k \ell= k_1 \ell_1 + \ldots + k_d \ell_d
$ 
 for $k, \ell \in \Z^d$.
In contrast to  the cubic NLS  \eqref{nlsIntro2} where we have that  (cf. \eqref{domNLS0}) 
\begin{equation*}
{
\mathcal{L}_{\text{dom}}(k_1) = 2k_1^2 \quad \text{and}\quad  \mathcal{L}_{\text{low}}(k_1,k_2,k_3) = - 2 k_1 (k_2+k_3) + 2 k_2 k_3}
\end{equation*}
 we   observe for the quadratic NLS \eqref{quadNLS} with the map $ \CP_{\text{\tiny{dom}}} $ given in  \eqref{physical map} that
\begin{equs}
P(k_1,k_2) & = - (k_1 + k_2)^2 + (k_1^2 + k_2^2) = - 2k_1 k_2, \quad \CP_{\text{\tiny{dom}}}(P) = 0,
\\ \mathcal{L}_{\text{\tiny{dom}}} & = \CP_{\text{\tiny{dom}}}(P) =0, \quad \mathcal{L}_{\text{\tiny{low}}} = P - \CP_{\text{\tiny{dom}}}(P) = - 2k_1 k_2.
\end{equs}
Hence,  although $\mathcal{L}= -\Delta$ and  $\mathcal{L}_{\text{dom}}= 0  $ (which means that no oscillations are integrated exactly), we ``only''  loose one derivative in the local approximation error  (cf.  \eqref{ue1}) as 
\begin{equation*}
\mathcal{O}\left(\tau^2 \mathcal{L}_{\text{low}}v\right)= \mathcal{O}\left(\tau^2 \vert \nabla \vert v\right).
\end{equation*}
\end{remark}

%
%
%
%
%
%

\begin{remark}
Terms of type \eqref{physical map} will naturally arise when filtering out the dominant nonlinear frequency interactions in the PDE. We have to embed integrals over their exponentials into our discretisation. For their practical implementation it will  be therefore essential to map fractions of \eqref{physical map} back to physical space.

Indeed, if we apply the inverse   Fourier transform $ \mathcal{F}^{-1} $ one gets:
\begin{equs} \label{nice physical terms}
\mathcal{F}^{-1} & \left( \sum_{\substack{0 \neq k=k_1 +...+k_n\\k_\ell \neq 0} }     \frac{1}{(k_1 +...+k_n)^m}        \frac{1}{k_1^{m_1}} ... \frac{1}{k_n^{m_n}}v^1_{k_1}...v^n_{k_n} e^{i kx} \right)
\\ & =  (-\Delta)^{- m/2} \prod_{\ell = 1}^n   \left( (-\Delta)^{-m_\ell/2} v^\ell(x)\right)
\end{equs}
where by abuse of notation we define the operator $(-\Delta)^{-1}$ in Fourier space as
$
(-\Delta)^{-1} f(x) = \sum_{ k \neq 0} \frac{ \hat{f}_k }{k^2}e^{i k x}.
$
\end{remark}
 In the next proposition, we elaborate on \eqref{nice physical terms} and give a nice class of functions depending on the $k_i$ that we can map back to the physical space.
\begin{proposition} \label{Fourierproduct} Assume that we have polynomials $ Q $ in   $ k_1, \ldots, k_n $ and $ k $ is a linear combination of the $k_i$ such that:
\begin{equs}
Q & = \prod_j (\sum_{u \in V_j}  a_{u,V_j} k_u )^{m_i}, \quad V_j \subset \lbrace 1,...,n \rbrace, \quad  a_{u,V_j} \in  \lbrace-1,1\rbrace, \\
k & = \sum_{u=1}^{n} a_u k_u, \quad a_u \in \lbrace -1,1 \rbrace,
\end{equs}
where the $ V_i $ are either disjoint or if $ V_j \subset V_i $ we assume that there exist $ p_{i,j} $ such that
\begin{equs}
a_{u,V_i} = (-1)^{p_{i,j}} a_{u,V_j}, \quad u \in V_j.
\end{equs}
We also suppose that the $ V_i $ are included in $ k $ in the sense that there exist $ p_{V_i} $ such that
\begin{equs}
a_u = (-1)^{p_{V_i}} a_{u,V_i}, \quad u \in V_i.
\end{equs}
Then, one gets
\begin{equs}
\mathcal{F}^{-1} & \left( \sum_{ \substack{0 \neq k= a_1 k_1 +...+ a_n k_n\\ Q(k_1,...,k_n) \neq 0}}     \frac{1}{Q} v^{1,a_1}_{k_1}...v^{n,a_n}_{k_n} e^{i kx} \right)
\\ &  =   \left(\prod_{V_i \subset V_j} (-1)^{p_{V_i}}  (-\Delta)^{- m_i/2}_{V_i} \right)  v^{1,a_1}... v^{n,a_n}
\end{equs}
where $ v^{i,1} = v^{i} $ and $ v^{i,-1} =  \overline{v^{i}} $.
 The operator $ (-\Delta)^{- m_i/2}_{V_i} $ acts only on the functions $\prod_{u \in V_i} v^{u,a_u} $ and the product starts by the smaller elements for the inclusion order.
\end{proposition}
\begin{proof} We proceed by induction on the number of $ V_i $. Let $ V_{\max} $ an element among the $ V_i $ maximum for the inclusion order. Then, we get
\begin{equs}
\sum_{ \substack{0 \neq k= a_1 k_1 +...+ a_n k_n\\ Q(k_1,...,k_n) \neq 0}} &    \frac{1}{Q} v^{1,a_1}_{k_1}...v^{n,a_n}_{k_n}  e^{i kx} = \sum_{\substack{0 \neq k= r + \ell \\ \ell \neq 0}} \frac{(-1)^{p_{V_{\max}}}}{\ell^{m_{\max}}}\sum_{ \substack{0 \neq r=   \sum_{u \notin V_{\max}} a_u k_u \\ R \neq 0}}     \frac{1}{R} \\ & \left( \prod_{j \notin V_{\max}} v^{j,a_j}_{k_j} \right) e^{i r x}
 \times \sum_{ \substack{0 \neq \ell=  \sum_{u \in V_{\max}} a_u k_u \\ S \neq 0}}     \frac{1}{S} \left( \prod_{j \in V_{\max}} v^{j,a_j}_{k_j} \right) e^{i \ell x}
\end{equs}
where 
\begin{equs}
  S  = \prod_{V_j \varsubsetneq V_{\max}} (\sum_{u \in V_j}  a_{u,V_j} k_u )^{m_i}, \quad
R  = \prod_{V_j \cap V_{\max} = \emptyset} (\sum_{u \in V_j}  a_{u,V_j} k_u )^{m_i}, \quad Q = R S \ell^{m_{max}}.
\end{equs}
Thus, by applying the inverse Fourier transform, we get the term $ (-1)^{p_{V_{\max}}} (-\Delta)^{- m_{\max}/2}_{V_{\max}} $ from $  \frac{(-1)^{p_{V_{\max}}}}{\ell^{m_{\max}}} $. We conclude from the induction hypothesis on the two remaining sums.
\end{proof}
 The next definition will allow us to compute the dominant part of the frequency interactions of a given decorated forest in $ \hat H_0 $. The idea is to filtered out the dominant part using the operator $ \CP_{\text{\tiny{dom}}} $ which selects the  frequencies of highest order. The operator $ \CP_{\text{\tiny{dom}}} $ only appears if we face an edge in $ \Lab_+ $ which corresponds to an integral in time that we have to approximate.

\begin{definition} \label{dom_freq} We recursively define $\mathscr{F}_{\text{\tiny{dom}}}, \mathscr{F}_{\text{\tiny{low}}} : \hat H_{0} \rightarrow \mathbb{R}[\Z^d]$ as:
\begin{equs}
 \mathscr{F}_{\text{\tiny{dom}}}(\one)  = 0 \quad
 \mathscr{F}_{\text{\tiny{dom}}}(F \cdot \bar F) & =\mathscr{F}_{\text{\tiny{dom}}}(F) + \mathscr{F}_{\text{\tiny{dom}}}(\bar F) \\
 \mathscr{F}_{\text{\tiny{dom}}}\left( \CI_{(\Labhom,\Labp)}(  \lambda_{k}F) \right)  & = \left\{ \begin{aligned}
  &   \CP_{\text{\tiny{dom}}}\left( P_{(\Labhom,\Labp)}(k) +\mathscr{F}_{\text{\tiny{dom}}}(F) \right), \,
  \text{if } \Labhom \in \Lab_+  , \\
  &  P_{(\Labhom,\Labp)}(k) +\mathscr{F}_{\text{\tiny{dom}}}(F), \quad \text{otherwise}
  \\
  \end{aligned} \right.  
 \\
 \mathscr{F}_{\text{\tiny{low}}} \left( \CI_{(\Labhom,\Labp)}(  \lambda_{k}F) \right)  & =   
\left( \id - \CP_{\text{\tiny{dom}}} \right) \left( P_{(\Labhom,\Labp)}(k) +\mathscr{F}_{\text{\tiny{dom}}}(F) \right).
\end{equs}
 We extend these two maps to $ \hat H $ by ignoring the node decorations $ \Labn $.
\end{definition}

\begin{remark}\label{rem:epsi}
The definition of $ \CP_{\text{\tiny{dom}}}  $ can be adapted depending on what is considered to be the dominant part. For example, if for $ \Labhom_2  \in \Lab_+$, one has (cf. \eqref{Leps}):
\begin{equs}
P_{(\Labhom_2,p)}( \lambda) = \frac{1}{\varepsilon^{\sigma}} +  F_{(\Labhom_2,p)}( \lambda).
\end{equs}
Then we can define the dominant part only depending on $ \varepsilon $ (see Example \ref{sec:kgr})
\begin{equs}
\CP_{\text{\tiny{dom}}}\left(  P_{(\Labhom_2,p)}(k) \right)= \frac{1}{\varepsilon^{\sigma}}. 
\end{equs}
The definition considers only the case where one does not have terms of  the form  $ 1/\eps^{\sigma} $. It is sufficient for covering many interesting examples.
\end{remark}

In the following we compute the dominant part $ \mathscr{F}_\text{dom}$ for the underlying trees of the cubic Schrödinger~\eqref{nlsIntro} and KdV \eqref{kdvIntro}  equation.
\begin{example}[KdV]\label{tex:kdv} We consider the decorated tree $ T $ given in Example  \ref{example2a}, where we fix the following decorations:
\begin{equs}
\Labp(1) & = \Labp(2) = \Labp(3) = 0, \quad \Labhom(2) = \Labhom(3) = \Labhom_1, \quad \Labhom(1) = \Labhom_2,
\\
\Labo(b) & = k_1, \quad  \Labo(c) = k_2, \quad \Labo(a) = k_1 + k_2, \quad P_{\Labhom_2}( \lambda) =  \lambda^3 , \quad P_{\Labhom_1}( \lambda) = -  \lambda^3.
\end{equs}
Now, we suppose $ \Lab_+ = \lbrace \Labhom_2 \rbrace $ and $ \Lab = \lbrace \Labhom_1, \Labhom_2  \rbrace $. Then the tree \eqref{example2a} takes the form
\begin{equs} \label{example2}
 \begin{tikzpicture}[scale=0.22,baseline=2cm]
        \node at (0,10)  (a) {}; 
         \node at (4,20)  (f) {}; 
           \node at (20,20)  (g) {}; 
        \node at (12,15) (c) {}; 
    \draw[kernel] (a) -- node [round1] {\tiny $\mathfrak{t}_2, 0 $}  (c) ; 
     \draw[kernel1] (c) -- node [round1] {\tiny $ \mathfrak{t}_1, 0$}  (f) ; 
     \draw[kernel1] (c) -- node [round1] {\tiny  $ \mathfrak{t}_1, 0 $}  (g) ; 
    \draw (g) node [rect2] {\tiny $ k_2 $}  ;
    \draw (f) node [rect2] {\tiny $ k_1$}  ;
    \draw (c) node [rect2] {\tiny $k_1+k_2 $}  ;
\end{tikzpicture} 
\end{equs}
This tree corresponds to the first iterated integral for the
KdV equation \eqref{kdvIntro}.  In more formal notation, we denote this tree by: 
\begin{equs}\label{kdvTK}
  \CI_{(\Labhom_2,0)}( \lambda_{k} \CI_{(\Labhom_1,0)}( \lambda_{k_1}) \CI_{(\Labhom_1,0)}( \lambda_{k_2}))= \begin{tikzpicture}[scale=0.2,baseline=-5]
\coordinate (root) at (0,0);
\coordinate (tri) at (0,-2);
\coordinate (t1) at (-1,2);
\coordinate (t2) at (1,2);
\draw[kernels2] (t1) -- (root);
\draw[kernels2] (t2) -- (root);
\draw[symbols] (root) -- (tri);
\node[not] (rootnode) at (root) {};
\node[not] (trinode) at (tri) {};
\node[var] (rootnode) at (t1) {\tiny{$ k_{\tiny{1}} $}};
\node[var] (trinode) at (t2) {\tiny{$ k_2 $}};
\end{tikzpicture} , \quad k = k_1 + k_2
\end{equs}
where a blue edge encodes $ (\mathfrak{t}_2,0) $ and a brown edge is used for $(\mathfrak{t}_1,0)$. The frequencies are given on the leaves. The ones on the inner nodes are determined by those on the leaves. On the left hand side, we have given the symbolic notation. Together with Definition \ref{dom_freq} one gets
\begin{equs}
 \mathscr{F}_{\text{\tiny{dom}}}(T) & = \CP_{\text{\tiny{dom}}} \left( (k_1+k_2)^{3} - k_1^{3} - k_2^{3} \right) = 0
\\
 \mathscr{F}_{\text{\tiny{low}}} (T) & =   (k_1+k_2)^{3} - k_1^{3} - k_2^{3} 
 =   3k_1 k_2 (k_1+k_2).
\end{equs}
The fact that $ \mathscr{F}_{\text{\tiny{dom}}}(T)  $ is zero comes from the fundamental choice of definition of the operator $ \CP_{\text{\tiny{dom}}}  $
 in Definition~\ref{Dom_projection}.
\end{example}
\begin{example}[Cubic Schrödinger]\label{exRDoNLS}
Next we consider the Symbol
\begin{equs}
  \CI_{(\Labhom_2,0)}(  \lambda_{-k_1+k_2+k_3}\CI_{(\Labhom_2,1)}(  \lambda_{k_1})\CI_{(\Labhom_2,0)}(  \lambda_{k_2})\CI_{(\Labhom_2,0)}(  \lambda_{k_3}))
\end{equs}
with $P_{\Labhom_2}( \lambda)=   \lambda^2$, $P_{\Labhom_1}( \lambda) = -  \lambda^2$, $ \Lab_+ = \lbrace \Labhom_2 \rbrace $ and $ \Lab = \lbrace \Labhom_1, \Labhom_2 \rbrace $ which encodes the tree 
\begin{equs} \label{example2}
 \begin{tikzpicture}[scale=0.22,baseline=2cm]
        \node at (0,10)  (a) {}; 
         \node at (0,20)  (f) {}; 
           \node at (10,20)  (g) {}; 
            \node at (20,20)  (e) {}; 
        \node at (10,15) (c) {}; 
    \draw[kernel] (a) -- node [round1] {\tiny $\mathfrak{t}_2, 0$}  (c) ; 
     \draw[kernel1] (c) -- node [round1] {\tiny  $\mathfrak{t}_1,1$}  (f) ; 
     \draw[kernel1] (c) -- node [round1] {\tiny  $ \mathfrak{t}_1,0$}  (g) ; 
          \draw[kernel1] (c) -- node [round1] {\tiny  $ \mathfrak{t}_1,0$}  (e) ; 
       \draw (f) node [rect2] {\tiny $ k_1$}  ;
    \draw (g) node [rect2] {\tiny $ k_2 $}  ;
      \draw (e) node [rect2] {\tiny $ k_3$}  ;
    \draw (c) node [rect2] {\tiny $-k_1+k_2 +k_3$}  ;
\end{tikzpicture} 
\end{equs}
This tree corresponds to the frequency interaction of the first iterated integral for the cubic Schrödinger equation \eqref{nlsIntro}.   In a more formal notation, we denote this tree by: 
\begin{equs} \label{treeSE}
 \begin{tikzpicture}[scale=0.2,baseline=-5]
\coordinate (root) at (0,0);
\coordinate (tri) at (0,-2);
\coordinate (t1) at (-2,2);
\coordinate (t2) at (2,2);
\coordinate (t3) at (0,3);
\draw[kernels2,tinydots] (t1) -- (root);
\draw[kernels2] (t2) -- (root);
\draw[kernels2] (t3) -- (root);
\draw[symbols] (root) -- (tri);
\node[not] (rootnode) at (root) {};t
\node[not] (trinode) at (tri) {};
\node[var] (rootnode) at (t1) {\tiny{$ k_{\tiny{1}} $}};
\node[var] (rootnode) at (t3) {\tiny{$ k_{\tiny{2}} $}};
\node[var] (trinode) at (t2) {\tiny{$ k_3 $}};
\end{tikzpicture}
\end{equs}
where a blue edge encodes $ (\mathfrak{t}_2,0) $, a brown edge is used for $ (\mathfrak{t}_1,0) $ and a dashed brown edge is for $ (\mathfrak{t}_2,1) $. With Definition \ref{dom_freq}, we get
\begin{equs}
 \mathscr{F}_{\text{\tiny{dom}}}(T) & = \CP_{\text{\tiny{dom}}} \left( (-k_1+k_2+k_3)^{2} + (-k_1)^{2} - k_2^{2} - k_3^2 \right) \\&
=   \CP_{\text{\tiny{dom}}} \left( 2k_1^2 - 2k_1(k_2+k_3) + 2 k_2 k_3\right) = 2 k_1^2
\\
 \mathscr{F}_{\text{\tiny{low}}} (T) 
& = (-k_1+k_2+k_3)^{2} + (-k_1)^{2} - k_2^{2} - k_3^2 - 2k_1^2\\
& = - 2k_1(k_2+k_3) + 2 k_2 k_3.
\end{equs}

\end{example}

The map $ \mathscr{F}_{\text{\tiny{dom}}}$ has a nice property regarding the tree inclusions given in the next proposition. This inclusive property will be important in practical computations, see also Remark \ref{rem:FFT} and the examples in Section \ref{sec:examples}. 

 For Proposition \ref{nasty_trees} below, we need some additional assumption.
\begin{assumption} \label{assumption_physical_space} 
 We consider decorated forests whose decorations at the leaves form a partition of the $ k_1,...,k_n $ ; in the sense that for two leaves $ u $ and $ v $, $ \Labo(u) $ (resp. $ \Labo(v) $) is a linear combination of $ (k_i)_{i \in I} $ (resp. $ (k_i)_{i \in J} $) with $ I,J \subset \lbrace 1,...,n \rbrace $ and $ I \cap J = \emptyset $. This will be the case in the examples given in Section \ref{sec:examples}.
\end{assumption}
 With this assumption, for a decorated forest $ F = \prod_i T_i  $ 
such that $ \CP_{\text{\tiny{dom}}}\left(\mathscr{F}_{\text{\tiny{dom}}}(F) \right) \neq 0 $ one has the nice identity:
\begin{equs} \label{splitting_identity}
\CP_{\text{\tiny{dom}}}\left(\mathscr{F}_{\text{\tiny{dom}}}(F) \right) = \CP_{\text{\tiny{dom}}}\left(\sum_i  \CP_{\text{\tiny{dom}}}\left(\mathscr{F}_{\text{\tiny{dom}}}(T_i) \right)  \right).
\end{equs}

 We will illustrate this property and give a counterexample in an example below.
\begin{example} We give a counter-example in the setting of the Schrödinger equation to \eqref{splitting_identity} when $\CP_{\text{\tiny{dom}}}\left(\mathscr{F}_{\text{\tiny{dom}}}(F) \right) = 0  $.  We consider the following forest $ F = \prod_{i=1}^3 T_i$ where the decorated trees $ T_i $ are given by:
\begin{equs}
 T_1 = \begin{tikzpicture}[scale=0.2,baseline=-5]
\coordinate (root) at (0,1);
\coordinate (tri) at (0,-1);
\draw[kernels2] (tri) -- (root);
\node[var] (rootnode) at (root) {\tiny{$ k_4 $}};
\node[not] (trinode) at (tri) {};
\end{tikzpicture} , \quad T_2 = \begin{tikzpicture}[scale=0.2,baseline=-5]
\coordinate (root) at (0,1);
\coordinate (tri) at (0,-1);
\draw[kernels2] (tri) -- (root);
\node[var] (rootnode) at (root) {\tiny{$ k_5 $}};
\node[not] (trinode) at (tri) {};
\end{tikzpicture},  \quad T_3 =  \begin{tikzpicture}[scale=0.2,baseline=-5]
\coordinate (root) at (0,-1);
\coordinate (t3) at (0,1);
\coordinate (t4) at (0,3);
\coordinate (t41) at (-2,5);
\coordinate (t42) at (2,5);
\coordinate (t43) at (0,7);
\draw[kernels2] (t3) -- (root);
\draw[symbols] (t3) -- (t4);
\draw[kernels2,tinydots] (t4) -- (t41);
\draw[kernels2] (t4) -- (t42);
\draw[kernels2] (t4) -- (t43);
\node[not] (rootnode) at (root) {};
\node[not] (rootnode) at (t4) {};
\node[not] (rootnode) at (t3) {};
\node[var] (rootnode) at (t41) {\tiny{$ k_{\tiny{1}} $}};
\node[var] (rootnode) at (t42) {\tiny{$ k_{\tiny{3}} $}};
\node[var] (rootnode) at (t43) {\tiny{$ k_{\tiny{2}} $}};
\end{tikzpicture}   , \quad F  = \begin{tikzpicture}[scale=0.2,baseline=-5]
\coordinate (root) at (0,-1);
\coordinate (t1) at (-2,1);
\coordinate (t2) at (2,1);
\coordinate (t3) at (0,1);
\coordinate (t4) at (0,3);
\coordinate (t41) at (-2,5);
\coordinate (t42) at (2,5);
\coordinate (t43) at (0,7);
\draw[kernels2] (t1) -- (root);
\draw[kernels2] (t2) -- (root);
\draw[kernels2] (t3) -- (root);
\draw[symbols] (t3) -- (t4);
\draw[kernels2,tinydots] (t4) -- (t41);
\draw[kernels2] (t4) -- (t42);
\draw[kernels2] (t4) -- (t43);
\node[not] (rootnode) at (root) {};
\node[not] (rootnode) at (t4) {};
\node[not] (rootnode) at (t3) {};
\node[var] (rootnode) at (t1) {\tiny{$ k_{\tiny{4}} $}};
\node[var] (rootnode) at (t41) {\tiny{$ k_{\tiny{1}} $}};
\node[var] (rootnode) at (t42) {\tiny{$ k_{\tiny{3}} $}};
\node[var] (rootnode) at (t43) {\tiny{$ k_{\tiny{2}} $}};
\node[var] (trinode) at (t2) {\tiny{$ k_5 $}};
\end{tikzpicture}.
\end{equs}
Then, we can check  Assumption {\ref{assumption_physical_space}} for $ F $. One has
\begin{equs}
 \mathscr{F}_{\text{\tiny{dom}}}(T_1) & = - k_4^2, \quad\mathscr{F}_{\text{\tiny{dom}}}(T_2) = - k_5^2,
\quad\mathscr{F}_{\text{\tiny{dom}}}(T_3) = - (-k_1 + k_2 + k_3)^2 + 2 k_1^2, \\
 \mathscr{F}_{\text{\tiny{dom}}}(F) & =  - k_4^2 - k_5^2 - (-k_1 + k_2 + k_3)^2 + 2 k_1^2
 \end{equs}
 and
 \begin{equs}
 \CP_{\text{\tiny{dom}}} \left( \mathscr{F}_{\text{\tiny{dom}}}(T_1)  \right) = - k_4^2, \quad  \CP_{\text{\tiny{dom}}} \left(\mathscr{F}_{\text{\tiny{dom}}}(T_2) \right) = - k_5^2,
\quad \CP_{\text{\tiny{dom}}} \left(\mathscr{F}_{\text{\tiny{dom}}}(T_3) \right) = 0.
  \end{equs}
  Therefore, one obtains
  \begin{equs}
 \CP_{\text{\tiny{dom}}} \left( \sum_{i=1}^3 \CP_{\text{\tiny{dom}}} \left(\mathscr{F}_{\text{\tiny{dom}}}(T_i) \right)  \right) = - (k_4 + k_5)^2.
   \end{equs}
 But on the other hand,
 \begin{equs}
 \CP_{\text{\tiny{dom}}} \left(\mathscr{F}_{\text{\tiny{dom}}}(F) \right) = 0.
  \end{equs}
\end{example}

\begin{proposition} \label{nasty_trees}
 Let $   T_{\Labe}^{\Labo} $ be a decorated tree in $ \hat H_0 $  and $ e \in E_T $. We recall that $ T_e $ corresponds to the subtree of $ T $ above $ e $.  The nodes of $ T_e $ are given by all the nodes whose path to the root contains $ e $. Under Assumption~\ref{assumption_physical_space} one has 
\begin{equs} \label{subtreek}
\begin{aligned}
\CP_{\text{\tiny{dom}}} \left(\mathscr{F}_{\text{\tiny{dom}}}(T_{\Labe}^{\Labo}) \right) & = 
a \left( \sum_{u \in V } a_u k_u \right)^{m}, \quad m \in \N, a \in \Z, \, V \subset L_T, \, a_u \in \lbrace -1,1 \rbrace \\
\CP_{\text{\tiny{dom}}} \left( \mathscr{F}_{\text{\tiny{dom}}}((T_e)_{\Labe}^{\Labo}) \right) & = 
b \left( \sum_{u \in \bar V } b_u k_u \right)^{m_e}, \quad m_e \in \N, b \in \Z, \, \bar V \subset L_T, \, b_u \in \lbrace -1,1 \rbrace
\end{aligned}
\end{equs}
and $\bar V \subset V \text{ or } \bar V \cap V = \emptyset $. If $ \bar V \subset V  $, then there exists $ \bar p \in \lbrace 0,1 \rbrace $ such that $ a_u = (-1)^{\bar p} b_u $ for every $ u \in \bar V $.
\end{proposition}
\begin{proof}
 We proceed by induction  over the size of the decorated trees. We consider $T= \CI_{(\Labhom,\Labp)}(  \lambda_{k} F) $ where $ F = \prod_{i=1}^m T_i $. 

(i) If $F =\one$ then one has:
\begin{equs}
\CP_{\text{\tiny{dom}}} \left(\mathscr{F}_{\text{\tiny{dom}}}(T) \right) =
\CP_{\text{\tiny{dom}}} \left( P_{(\Labhom,\Labp)}(k) \right).
\end{equs}
We conclude from the definition of $ \CP_{\text{\tiny{dom}}}  $.

(ii) If $ m > 2 $, then one gets: 
\begin{equs}
\CP_{\text{\tiny{dom}}} \left(\mathscr{F}_{\text{\tiny{dom}}}(T) \right) =
\CP_{\text{\tiny{dom}}} \left( P_{(\Labhom,\Labp)}(k)  + \sum_{i=1}^m \mathscr{F}_{\text{\tiny{dom}}}(T_i)  \right).
\end{equs}
Using Assumption~\ref{assumption_physical_space}, one obtains:
\begin{equs}
 \CP_{\text{\tiny{dom}}} \left(\mathscr{F}_{\text{\tiny{dom}}}(T) \right) =
\CP_{\text{\tiny{dom}}} \left(  \sum_{i=1}^m    (P_{(\Labhom,\Labp)}(k^{(i)}) +\mathscr{F}_{\text{\tiny{dom}}}(T_i)) \right)
\end{equs}
where $ k^{(i)} $ corresponds to the frequency attached to the node connected to the root of $ T_i $. If $  \CP_{\text{\tiny{dom}}} \left(\mathscr{F}_{\text{\tiny{dom}}}(T) \right) \neq 0 $ then from \eqref{splitting_identity} we have
\begin{equs}
  \CP_{\text{\tiny{dom}}} \left(\mathscr{F}_{\text{\tiny{dom}}}(T) \right) =
\CP_{\text{\tiny{dom}}} \left(  \sum_{i=1}^m    \CP_{\text{\tiny{dom}}} \left(  P_{(\Labhom,\Labp)}(k^{(i)}) +\mathscr{F}_{\text{\tiny{dom}}}(T_i) \right) \right).
\end{equs}
We apply the induction hypothesis to each decorated tree $ \tilde T_i = \CI_{(\Labhom,\Labp)}(  \lambda_{k^{(i)}}  T_i) $ and we recombine the various terms in order to conclude.

(iii) If $ m=1 $, then $ F = \CI_{(\Labhom_1,\Labp_1)}( \lambda_{ \bar k} T_1) $ where $ \bar k  $ is equal to $ k $ up to a minus sign. We can assume without loss of generality that
\begin{equs} \label{case1}
 \CP_{\text{\tiny{dom}}} \left(\mathscr{F}_{\text{\tiny{dom}}}(F) \right)  = 
 \mathscr{F}_{\text{\tiny{dom}}}(F). 
\end{equs}
Indeed, otherwise
\begin{equs}
 \CP_{\text{\tiny{dom}}} \left( \mathscr{F}_{\text{\tiny{dom}}}(T)\right) = \CP_{\text{\tiny{dom}}} \left(  P_{(\Labhom,\Labp)}(k) + P_{(\Labhom_1,\Labp_1)}(\bar k) + \mathscr{F}_{\text{\tiny{dom}}}(T_1) \right) .
\end{equs}
We can see $ P_{(\Labhom,\Labp)}(k) + P_{(\Labhom_1,\Labp_1)}(\bar k)  $ as a polynomial and then apply the induction hypothesis. We are down to the case \eqref{case1} and we consider:
\begin{equs}
 \CP_{\text{\tiny{dom}}}\left(\mathscr{F}_{\text{\tiny{dom}}}(T) \right) 
  & = \CP_{\text{\tiny{dom}}}\left( P_{(\Labhom,\Labp)}(k) + \CP_{\text{\tiny{dom}}}\left(\mathscr{F}_{\text{\tiny{dom}}}(\bar T) \right) \right)
\end{equs}
where now  $  F = \bar T$ is just a decorated tree.
We apply the induction hypothesis on $ \bar T $  and we get
\begin{equs}\label{PT}
\CP_{\text{\tiny{dom}}} \left(\mathscr{F}_{\text{\tiny{dom}}}(\bar T) \right) & = 
a \left( \sum_{u \in V } a_u k_u \right)^{m}, \quad a \in \Z, \, V \subset L_T, \, a_u \in \lbrace -1,1 \rbrace \\
 k & = \sum_{u \in L_T } (c_u k_u), \quad c_{ u} = (-1)^{ p} a_{u} , \quad  u \in  V.
\end{equs}
If the degree of $P_{(\Labhom,\Labp)}(k)$ is higher than the degree of $\mathscr{F}_{\text{\tiny{dom}}}(\bar T)$ we obtain that
\begin{equs}
\CP_{\text{\tiny{dom}}}  \left( P_{(\Labhom,\Labp)}(k) +\mathscr{F}_{\text{\tiny{dom}}}(\bar T) \right) 
= \CP_{\text{\tiny{dom}}}  \left( P_{(\Labhom,\Labp)}(k) \right).
\end{equs}
 On the other hand,  if the degree of $P_{(\Labhom,\Labp)}(k)$ is lower than the degree of $\mathscr{F}_{\text{\tiny{dom}}}(\bar T)$ 
\begin{equs}
\CP_{\text{\tiny{dom}}}  \left( P_{(\Labhom,\Labp)}(k) +\mathscr{F}_{\text{\tiny{dom}}}(\bar T) \right) 
= \CP_{\text{\tiny{dom}}}  \left( \mathscr{F}_{\text{\tiny{dom}}}(\bar T)  \right).
\end{equs}
If $ P_{(\Labhom,\Labp)}(k)$ and $\mathscr{F}_{\text{\tiny{dom}}}(\bar T)$ have the same degree $ m $, we get using the definition of $ P_{(\Labhom,\Labp)}(k)$ in \eqref{Palpha} as well as the induction hypothesis on $\bar T$ given in \eqref{PT} that
\begin{equs}
   P_{(\Labhom,\Labp)}(k) + \CP_{\text{\tiny{dom}}}\left(\mathscr{F}_{\text{\tiny{dom}}}(\bar T) \right)  & =  \sum_{u \in V } \left( a (-1)^{p + m + \Labp} + (- 1)^{\Labp} \right)( (-1)^{\Labp}c_u k_u)^{m} \\ & + \sum_{u \in L_T \setminus V } (- 1)^{\Labp}  ((-1)^{\Labp}  c_u k_u)^{m} + R
\end{equs}
where $ R $ are terms of lower order and contains non-zero  

By applying  the map $ \CP_{\text{\tiny{dom}}} $ defined in \eqref{physical map}, we thus obtain an expression of the form 
\begin{equs} \label{resultb}
\CP_{\text{\tiny{dom}}} \left(\mathscr{F}_{\text{\tiny{dom}}}(T) \right) 
 = b \left( \sum_{u \in \tilde{V} } (-1)^{\Labp}  c_u k_u \right)^{m} \,\text{for some }  b \in \Z
\end{equs} 
where $ \tilde V $ could be either $ L_T $ or $ L_T \setminus V $. 

 Let $ e \in E_T  $, $ T_e \neq T $, then $ T_e $  is a subtree of $ \bar T $. By the induction hypothesis, one obtains \eqref{subtreek}, meaning that  if we denote by $ V $ (resp. $ \bar V $) the set associated to $ \bar T $ (resp. $ T_e $), we get: $ \bar V \subset V $ or $ \bar V \cap V = \emptyset $.

In the first case $\bar V \subset V  $, the assertion follows as $ V \subset \tilde{V} $ or $ V \cap \tilde{V} = \emptyset $ such that necessarily $\bar V \subset \tilde V$ or $\bar V \cap \tilde V = \emptyset$.
 
In the second case $ \bar V \cap V = \emptyset $, we apply the induction hypothesis on $ T_e $. Then, for $ v $ of $ T_e $, there exists $ p $ such that the decoration $ \Labo(v) $ is given by:
 \begin{equs}
 \Labo(v) = \sum_{u \in L_{T_v} } (d_u k_u), \quad d_{ u} = (-1)^{ p} b_{u} , \quad  u \in  \bar V.
 \end{equs}
 As $  \Labo(v) $ appears as a subfactor in $ k $, one has $ \bar V \subset L_T $. Then, $ \bar {V} \cap V = \emptyset $ gives also that $ \bar V \subset L_T \setminus V $. Therefore, we have $ \bar V \subset \tilde{V} $ which concludes the proof.
\end{proof}

\begin{corollary} \label{physical_map} Let $ T_{\Labe}^{\Labo} $ a decorated tree in $ \hat \CT_0 $. We assume  that Assumption~\ref{assumption_physical_space} holds true for a set $ A $ of decorated subtrees of $   T_{\Labe}^{\Labo} $ such that $\CP_{\text{\tiny{dom}}} \left(\mathscr{F}_{\text{\tiny{dom}}}(\bar T) \right) =\mathscr{F}_{\text{\tiny{dom}}}(\bar T)  \neq 0$ for $ \bar T \in A $. Moreover, we assume that the $ \bar T $ are of the form $ (T_e)_{\Labe}^{\Labo} $ where $ e \in E_T $. Then, the following product 
\begin{equs}
\prod_{\bar T \in A} \frac{1}{\left(\mathscr{F}_{\text{\tiny{dom}}}(\bar T) \right)^{m_T}}
\end{equs}
can be mapped back to Physical space using operators of the form $(- \Delta)^{-m/2}_V $ as defined in Proposition~\ref{Fourierproduct}.
\end{corollary}
\begin{proof}
Proposition~\ref{nasty_trees} gives us the structure needed for applying Proposition~\ref{Fourierproduct} which allows us to conclude.
\end{proof}

 \begin{example}
 We illustrate Corollary~\ref{physical_map} via an example extracted from the cubic Schrödinger equation~\eqref{nlsIntro}.We consider the following decorated trees:
\begin{equs}\label{nlsTK}
 T_1  = \begin{tikzpicture}[scale=0.2,baseline=-5]
\coordinate (root) at (0,0);
\coordinate (tri) at (0,-2);
\coordinate (t1) at (-2,2);
\coordinate (t2) at (2,2);
\coordinate (t3) at (0,3);
\draw[kernels2,tinydots] (t1) -- (root);
\draw[kernels2] (t2) -- (root);
\draw[kernels2] (t3) -- (root);
\draw[symbols] (root) -- (tri);
\node[not] (rootnode) at (root) {};t
\node[not] (trinode) at (tri) {};
\node[var] (rootnode) at (t1) {\tiny{$ k_{\tiny{1}} $}};
\node[var] (rootnode) at (t3) {\tiny{$ k_{\tiny{2}} $}};
\node[var] (trinode) at (t2) {\tiny{$ k_3 $}};
\end{tikzpicture}, \quad T_2  = \begin{tikzpicture}[scale=0.2,baseline=-5]
\coordinate (root) at (0,0);
\coordinate (tri) at (0,-2);
\coordinate (t1) at (-2,2);
\coordinate (t2) at (2,2);
\coordinate (t3) at (0,2);
\coordinate (t4) at (0,4);
\coordinate (t41) at (-2,6);
\coordinate (t42) at (2,6);
\coordinate (t43) at (0,8);
\draw[kernels2,tinydots] (t1) -- (root);
\draw[kernels2] (t2) -- (root);
\draw[kernels2] (t3) -- (root);
\draw[symbols] (root) -- (tri);
\draw[symbols] (t3) -- (t4);
\draw[kernels2,tinydots] (t4) -- (t41);
\draw[kernels2] (t4) -- (t42);
\draw[kernels2] (t4) -- (t43);
\node[not] (rootnode) at (root) {};
\node[not] (rootnode) at (t4) {};
\node[not] (rootnode) at (t3) {};
\node[not] (trinode) at (tri) {};
\node[var] (rootnode) at (t1) {\tiny{$ k_{\tiny{4}} $}};
\node[var] (rootnode) at (t41) {\tiny{$ k_{\tiny{1}} $}};
\node[var] (rootnode) at (t42) {\tiny{$ k_{\tiny{3}} $}};
\node[var] (rootnode) at (t43) {\tiny{$ k_{\tiny{2}} $}};
\node[var] (trinode) at (t2) {\tiny{$ k_5 $}};
\end{tikzpicture}.
\end{equs}
We observe that these trees satisfy Assumption~\ref{assumption_physical_space} and that $ T_1 $ is a subtree of $ T_2 $. One has that:
\begin{equs}
\mathscr{F}_{\text{\tiny{dom}}}( T_1) =  2 k_1^2 , \quad\mathscr{F}_{\text{\tiny{dom}}}( T_2) = 2 (k_1 + k_4)^2
\end{equs}
and the following quantity can be mapped back into physical space:
\begin{equs}
\sum_{\substack{k= -k_1-k_4+k_2+k_3+k_5\\ k_1\neq 0, k_1+k_4\neq 0 \\k_1,k_2,k_3,k_4,k_5\in \Z^d}}\frac{1}{ \mathscr{F}_{\text{\tiny{dom}}}( T_1)} \frac{1}{ \mathscr{F}_{\text{\tiny{dom}}}( T_2)} \overline{\hat{v}_{k_1}} \overline{\hat{v}_{k_4}}\hat{v}_{k_2}\hat{v}_{k_3}\hat{v}_{k_5} e^{i k x}\\= \sum_{\substack{k= -k_1-k_4+k_2+k_3+k_5\\ k_1\neq 0, k_1+k_4\neq 0 \\k_1,k_2,k_3,k_4,k_5\in \Z^d}} \frac{1}{4 (k_1^2) (k_1 + k_4)^2}\overline{\hat{v}_{k_1}} \overline{\hat{v}_{k_4}}\hat{v}_{k_2}\hat{v}_{k_3}\hat{v}_{k_5} e^{i k x}
\\=
\frac14 v(x)^3(-\Delta)^{-1}\left(\overline v(x)  (-\Delta)^{-1} \overline v(x)\right)
.
 \end{equs}
\end{example}

\subsection{Approximated decorated trees}

 We denote by $ \CT $ the set of decorated trees $ T_{\Labe,r}^{\Labn,\Labo} = (T,\Labn,\Labo,\Labe,r)  $ where
\begin{itemize}
\item $ T_{\Labe}^{\Labn,\Labo} \in \hat \CT $.
\item The decoration of the root is given by $ r \in \Z $, $ r \geq -1 $ such that
\begin{equs} \label{condition_trees}
 r +1 \geq  \deg(T_{\Labe}^{\Labn,\Labo})
\end{equs}
 where $ \deg $ is defined recursively by 
 \begin{equs}
\deg(\one) & = 0, \quad \deg(T_1 \cdot T_2 )  = \max(\deg(T_1),\deg(T_2)),  \\
\deg(\CI_{(\Labhom,\Labp)}(  \lambda^{\ell}_{k}T_1)   ) & = \ell +   \one_{\lbrace\Labhom \in \CL_+\rbrace}   +  \deg(T_1)
\end{equs}
where $ \one  $ is the empty forest and  $ T_1, T_2  $ are forests composed of trees in $ \CT $. The quantity $\deg(T_{\Labe}^{\Labn,\Labo})$ is the maximum number of edges with type in $ \Lab_+ $ and node decorations $ \Labn $ lying on the same path from one leaf to the root.
\end{itemize}
We call decorated trees in $ \CT $ approximated decorated trees. The main difference with decorated trees introduced before is the adjunction of the decoration $ r $ at the root. The idea is that these trees correspond to different analytical objects. We summarise this below:
\begin{equs}
T_{\Labe}^{\Labn,\Labo} & \in  \hat{\CT} \equiv \text{Iterated integral} \\
T_{\Labe,r}^{\Labn,\Labo} & \in  \hat{\CT} \equiv \text{Approximation of an iterated integral of order $r$}.
\end{equs}
The interpretation \eqref{recursive_pi_r} of approximated decorated trees gives the numerical scheme see Definition~\ref{scheme}.

\begin{remark}
The condition \eqref{condition_trees} encodes the fact that the order of the scheme must be higher than the maximum number of iterated integrals and monomials lying on the same path from one leaf to the root. Moreover, we can only have monomials of degree less than $ r+1 $ at  order $ r+2 $.
\end{remark}

\begin{example}
We continue Example \ref{ex2} with the decorated tree $  T_{\Labe}^{\Labn,\mathfrak{f}} $ given in \eqref{example2a}. We  suppose that $ \Labhom(1) $ is in $ \mathfrak{L}_+ $ but not $ \Labhom(2) $ and  $\Labhom(3) $. We are now in the context of Example~\ref{tex:kdv}. Then, one has
\begin{equs}
\deg(T_{\Labe}^{\Labn,\Labo}) =  \Labn(a) + 1+  \max(\Labn(b),\Labn(c)).
\end{equs}
\end{example}

We denote by $ \CH $ the vector space spanned by  forests composed of trees in $ \CT$ and $  \lambda^n $, $ n \in \N $ where $  \lambda^n $ is the tree with one node decorated by $ n $. When the decoration $ n $  is equal to zero we identify this tree with the empty forest: $  \lambda^{0}  =\one $. Using the symbolic notation, one has: 
\begin{equs}
\CH = \langle  \lbrace  \prod_j  \lambda^{m_j} \prod_i \CI^{r_i}_{o_i}( \lambda_{k_i}^{\ell_i} F_i), \, 
\CI_{o_i}( \lambda_{k_i}^{\ell_i} F_i) \in \hat \CT \rbrace \rangle
\end{equs}
where the product used is the forest product. We call decorated forests in $ \CH $ approximated decorated forests. The map $    \CI^{r}_{o}( \lambda_{k}^{\ell} \cdot) :  \hat \CH \rightarrow    \CH   $ is defined as the same as for $ \CI_{o}( \lambda_{k}^{\ell} \cdot) $ except now the root is decorated by $ r $ and it could be zero if the inequality \eqref{condition_trees} is not satisfied.
We extend this map to $ \CH $ by:
\begin{equs}
\CI^{r}_{o}( \lambda_{k}^{\ell} (\prod_j  \lambda^{m_j}  \prod_i \CI^{r_i}_{o_i}( \lambda_{k_i}^{\ell_i} F_i))) { \, := }\CI^{r}_{o}( \lambda_{k}^{\ell+\sum_j m_j} (\prod_i \CI_{o_i}( \lambda_{k_i}^{\ell_i} F_i))).
\end{equs}
In the extension, we remove the decorations $ r_i $ and we add up the decorations $ m_j $ with $ \ell $. In the sequel, we will use a recursive formulation and move from $ \hat \CH $ to $ \CH $. Therefore, we define the map $ \CD^{r} : \hat \CH \rightarrow \CH $ which  replaces the root decoration of a decorated tree by $ r $ and performs the projection along the identity \eqref{condition_trees}. It is given by
\begin{equs}\label{DR}
\CD^{r}(\one)= \one_{\lbrace 0 \leq  r+1\rbrace} , \quad \CD^r\left( \CI_{o}( \lambda_{k}^{\ell} F) \right) =  \CI^{r}_{o}( \lambda_{k}^{\ell} F)
\end{equs}
and we extend it multiplicatively to any forest in $ \hat \CH $. The map $ \CD^r $ projects according to the order of the scheme $ r $. We disregard decorated trees having a degree bigger than $ r $.

 We denote by $ \CT_{+} $ the set of decorated forests composed of trees of the form $ (T,\Labn,\Labo,\Labe,(r,m))  $ where
\begin{itemize}
\item $ T_{\Labe,r}^{\Labn,\Labo} \in  \CT $.
\item The edge connecting the root has a decoration of the form $ (\Labhom,\Labp) $ where $ \Labhom \in \Lab_+ $. 
\item The decoration $ (r,m) $ is at the root of $ T $ and $ m \in \N $ is such that $ m \leq r +1 $.
\end{itemize}
The linear span of $ \CT_+ $ is denoted by $ \CH_+ $. One can observe that the main difference with $\CH$ is that $ \lambda^m \notin \CT_+ $. We can define the same grafting operator as before $ \CI^{(r,m)}_{o}(  \lambda_{k}^{\ell} \cdot)  : \CH \rightarrow \CH_+ $ as the same as $ \CI^{r}_{o}(  \lambda_{k}^{\ell} \cdot)  : \CH \rightarrow \CH $  but now we add the decoration $ (r,m) $ at the root where $ m \leq r+1 $. 
We also define $  \hat \CD^{(r,m)}:  \hat\CH \rightarrow \CH_+ $   as the same as $  \CD^{r}  $.  It is given by 
\begin{equs}\label{DRhat}
\hat \CD^{(r,m)}(\one)= \one, \quad \hat \CD^{(r,m)}\left( \CI_{o}( \lambda_{k}^{\ell} F) \right) =  \CI^{(r,m)}_{o}( \lambda_{k}^{\ell} F).
\end{equs}

\begin{example}
For the tree  \eqref{example2a} we obtain  when applying $ \CD^r $ and $ \hat \CD^{(r,m)} $
\begin{equs} \label{example3}
 \begin{tikzpicture}[scale=0.19,baseline=2cm]
        \node at (0,10)  (a) {}; 
         \node at (4,20)  (f) {}; 
           \node at (20,20)  (g) {}; 
        \node at (12,15) (c) {}; 
    \draw[kernel1] (a) -- node [round1] {\tiny $\Labhom(1),\Labp(1)$}  (c) ; 
     \draw[kernel1] (c) -- node [round1] {\tiny $\Labhom(2),\Labp(2)$}  (f) ; 
     \draw[kernel1] (c) -- node [round1] {\tiny 
     $\Labhom(3),\Labp(3)$}  (g) ; 
    \draw (g) node [rect2] {\tiny $ \Labn(c),\mathfrak{f}(c)  $}  ;
    \draw (f) node [rect2] {\tiny $ \Labn(b),\Labo(b) $}  ;
    \draw (c) node [rect2] {\tiny $ \Labn(a),\Labo(a) $}  ;
     \draw (a) node [rect2] {\tiny $ r $}  ;
\end{tikzpicture}, \quad  \begin{tikzpicture}[scale=0.19,baseline=2cm]
        \node at (0,10)  (a) {}; 
         \node at (4,20)  (f) {}; 
           \node at (20,20)  (g) {}; 
        \node at (12,15) (c) {}; 
    \draw[kernel1] (a) -- node [round1] {\tiny $\Labhom(1),\Labp(1)$}  (c) ; 
     \draw[kernel1] (c) -- node [round1] {\tiny $\Labhom(2),\Labp(2)$}  (f) ; 
     \draw[kernel1] (c) -- node [round1] {\tiny 
     $\Labhom(3),\Labp(3)$}  (g) ; 
    \draw (g) node [rect2] {\tiny $ \Labn(c),\mathfrak{f}(c) $}  ;
    \draw (f) node [rect2] {\tiny $ \Labn(b),\Labo(b) $}  ;
    \draw (c) node [rect2] {\tiny $ \Labn(a),\Labo(a) $}  ;
     \draw (a) node [rect2] {\tiny $ (r,m) $}  ;
\end{tikzpicture}
\end{equs}

 For the decorated tree in \eqref{treeSE}, one obtains
\begin{equs} \label{DRnot}
   \CD^{r}  \left(\begin{tikzpicture}[scale=0.2,baseline=-5]
\coordinate (root) at (0,0);
\coordinate (tri) at (0,-2);
\coordinate (t1) at (-2,2);
\coordinate (t2) at (2,2);
\coordinate (t3) at (0,3);
\draw[kernels2,tinydots] (t1) -- (root);
\draw[kernels2] (t2) -- (root);
\draw[kernels2] (t3) -- (root);
\draw[symbols] (root) -- (tri);
\node[not] (rootnode) at (root) {};t
\node[not] (trinode) at (tri) {};
\node[var] (rootnode) at (t1) {\tiny{$ k_{\tiny{1}} $}};
\node[var] (rootnode) at (t3) {\tiny{$ k_{\tiny{2}} $}};
\node[var] (trinode) at (t2) {\tiny{$ k_3 $}};
\end{tikzpicture}  \right) = \begin{tikzpicture}[scale=0.2,baseline=-5]
\coordinate (root) at (0,0);
\coordinate (tri) at (0,-2);
\coordinate (t1) at (-2,2);
\coordinate (t2) at (2,2);
\coordinate (t3) at (0,3);
\draw[kernels2,tinydots] (t1) -- (root);
\draw[kernels2] (t2) -- (root);
\draw[kernels2] (t3) -- (root);
\draw[symbols] (root) -- (tri);
\node[not] (rootnode) at (root) {};t
\node[not,label= {[label distance=-0.2em]below: \scriptsize  $ r $}] (trinode) at (tri) {};
\node[var] (rootnode) at (t1) {\tiny{$ k_{\tiny{1}} $}};
\node[var] (rootnode) at (t3) {\tiny{$ k_{\tiny{2}} $}};
\node[var] (trinode) at (t2) {\tiny{$ k_3 $}};
\end{tikzpicture}, \quad \hat \CD^{(r,m)}  \left(\begin{tikzpicture}[scale=0.2,baseline=-5]
\coordinate (root) at (0,0);
\coordinate (tri) at (0,-2);
\coordinate (t1) at (-2,2);
\coordinate (t2) at (2,2);
\coordinate (t3) at (0,3);
\draw[kernels2,tinydots] (t1) -- (root);
\draw[kernels2] (t2) -- (root);
\draw[kernels2] (t3) -- (root);
\draw[symbols] (root) -- (tri);
\node[not] (rootnode) at (root) {};t
\node[not] (trinode) at (tri) {};
\node[var] (rootnode) at (t1) {\tiny{$ k_{\tiny{1}} $}};
\node[var] (rootnode) at (t3) {\tiny{$ k_{\tiny{2}} $}};
\node[var] (trinode) at (t2) {\tiny{$ k_3 $}};
\end{tikzpicture}  \right) = \begin{tikzpicture}[scale=0.2,baseline=-5]
\coordinate (root) at (0,0);
\coordinate (tri) at (0,-2);
\coordinate (t1) at (-2,2);
\coordinate (t2) at (2,2);
\coordinate (t3) at (0,3);
\draw[kernels2,tinydots] (t1) -- (root);
\draw[kernels2] (t2) -- (root);
\draw[kernels2] (t3) -- (root);
\draw[symbols] (root) -- (tri);
\node[not] (rootnode) at (root) {};t
\node[not,label= {[label distance=-0.2em]below: \scriptsize  $ (r,m) $}] (trinode) at (tri) {};
\node[var] (rootnode) at (t1) {\tiny{$ k_{\tiny{1}} $}};
\node[var] (rootnode) at (t3) {\tiny{$ k_{\tiny{2}} $}};
\node[var] (trinode) at (t2) {\tiny{$ k_3 $}};
\end{tikzpicture}.  
\end{equs}
\end{example}

\subsection{Operators on approximated decorated forests}
In the subsection, we introduce two maps $ \Delta $ and $ \Delta^{\!+}  $ that will act on approximated decorated forests splitting them into two parts by the use of the tensor product. They act at two levels. First, on the shapes of the trees,  they extract a subtree at the root. Then, they induce  subtle changes in the decorations that could be interpreted as abstract Taylor expansions.
 
 We define a map $\Delta : \CH 
\rightarrow \CH \otimes \CH_+$   
for a given $T^{\Labn,\Labo}_{\Labe,r} \in \cT$ by 
\begin{equs}
\label{eq:co-action_plus}
 \Delta  T^{\Labn,\Labo}_{\Labe,r} & = 
 \sum_{A \in \Adm(T) } \sum_{\Labe_A}  \frac1{\Labe_A!}
 (A,\Labn + \pi\Labe_A, \Labo, \Labe,r) 
   \\ & \otimes \prod_{e \in \partial(A,T)}(  T_{e}, \Labn , \Labo, \Labe, (r-\deg(e),\Labe_A(e)))\;, 
   \\ & =  \sum_{A \in \Adm(T) } \sum_{\Labe_A}  \frac1{\Labe_A!}
 A^{\Labn +\pi\Labe_A, \Labo }_{\Labe,r} 
   \otimes \prod_{e \in \partial(A,T)}  (T_{e})^{\Labn , \Labo}_{\Labe, (r-\deg(e),\Labe_A(e))}
\end{equs}
where we use the following notations
\begin{itemize}
\item  	We write $T_e $ as the planted tree above the edge $ e $ in $ T $. For $g : E_T  
\rightarrow \N$, we define for every $x \in N_T$, $(\pi g)(x) = 
\sum_{e=(x,y) \in E_T} g(e)$.
\item 	In $ A^{\Labn +\pi\Labe_A, \Labo }_{\Labe,r}  $, the maps $ \Labn, \Labo $ and $ \Labe $ are restricted to $ N_A $ and $ E_A $. The same is valid for $ (T_{e})^{\Labn , \Labo}_{\Labe, (r-\deg(e),\Labe_A(e))} $ where the restriction is on $ N_{T_e} \setminus \lbrace \varrho_{T_e}\rbrace $ and $ E_{T_e} $, $ \varrho_{T_e} $ is the root of $ T_e $. When $ A $ is reduced to a single node, we set $ A^{\Labn +\pi\Labe_A, \Labo }_{\Labe,r} = \lambda^{\Labn +\pi\Labe_A} $.
\item 	The first sum runs over $\Adm(T)$, the set of all subtrees $A$ of $T$ containing the root $ \varrho $ of $ T $. The second sum runs $\Labe_{A} : \partial(A,T) \rightarrow \N$ where 
$\partial(A,T)$ denotes the edges  in $E_T \setminus E_A$ of type in $ \Lab_+ $ that are adjacent to 
$N_A$.
\item 	Factorial coefficients are understood in multiindex 
notation.
\item We define $ \deg(e) $ for $ e \in E_T $ as the number of edges having $ \Labhom(e) \in \Lab_+ $ lying on the path from $ e  $ to the root in the decorated tree $ T_{\Labe}^{\Labn, \Labo} $. We also add up the decoration $ \Labn $ on this path.
\end{itemize}

Let us comment briefly that the play on decorations can be interpreted as asbract Taylor expansions. We can use the following dictionary:
\begin{equs}
(T_e)_{\Labe}^{\Labn,\Labo} \equiv \int_{0}^{\tau} e^{i\xi P(k)} f(k_1,...,k_n,\xi) d\xi
\end{equs}
where $ k_1,...,k_n $ are the frequencies appearing on the leaves of $  (T_e)_{\Labe}^{\Labn,\Labo}$, $ P $ is the polynomial associated to the decoration of the edge $ e $ and $ k $ is the frequency on the node connected to the root.
This iterated integral appears inside the iterated integral associated to $ T_{\Labe}^{\Labn,\Labo} $. For our numerical approximation, we need to approximate this integral by giving a scheme of the form:
\begin{equs}
(T_e)_{\Labe, \tilde{r}}^{\Labn,\Labo} \equiv \sum_{\ell \leq \tilde{r}} \frac{\tau^{\ell}}{\ell!} f_{\tilde{r},\ell}(k_1,...,k_n), \quad (T_{e})^{\Labn , \Labo}_{\Labe, (\tilde{r},\ell)} \equiv f_{\tilde{r},\ell}(k_1,...,k_n)
\end{equs}
where the order of the expansion is given by $ \tilde{r} = r  - \deg(e) $. 
Then, the $ \tau^{\ell} $ are part of the original iterated integrals and cannot be detached as the $ f_{\tilde{r},\ell} $. That's why we increase the polynomial decorations where the tree $ T_e $ was originally attached via the term $ \Labn +\pi\Labe_A $. The choice for $ \tilde{r} $ is motivated by the fact that we do not need to go too far for the approximation of $ (T_e)_{\Labe}^{\Labn,\Labo} $. We take into account all the time integrals and polynomials which lie on the path connecting the root of $ T_e $ to the root of $ T $.

The map $ \Delta $ is compatible with the projection induced by the decoration $ r $. Indeed, one has
\begin{equs}
\deg(T_{\Labe}^{\Labn, \Labo}) = \max_{e \in E_T} \left( \deg( (T_e)_{\Labe}^{\Labn, \Labo} ) + \deg(e) \right).
\end{equs}
Therefore, if $ \deg(T_{\Labe}^{\Labn, \Labo}) < r $ then for every $ e \in E_T $
\begin{equs}
\deg( (T_e)_{\Labe}^{\Labn, \Labo} ) + \deg(e) < r.
\end{equs}
We deduce that if $ T^{\Labn,\Labo}_{\Labe,r} $ is zero then the $  (T_e)_{\Labe,(r-\deg(e),\Labe_A(e))}^{\Labn, \Labo} $ are zero too.

We define a map $\Deltap : \CH_+ 
\rightarrow \CH_+ \otimes \CH_+$  given 
for $T^{\Labn, \Labo}_{\Labe,(r,m)} \in \cT_+$ by 
\begin{equs}
\label{eq:coproduct_plus}
 \Deltap  T^{\Labn,\Labo}_{\Labe,(r,m)} & =  \sum_{A \in \Adm(T) } \sum_{\Labe_A}  \frac1{\Labe_A!}
 A^{\Labn +\pi\Labe_A, \Labo }_{\Labe,(r,m)} 
   \otimes \prod_{e \in \partial(A,T)}  (T_{e})^{\Labn , \Labo}_{\Labe, (r-\deg(e),\Labe_A(e))}.
\end{equs}
We require that $ A^{\Labn +\pi\Labe_A, \Labo }_{\Labe,(r,m)} \in \CH_+ $, so we implicitly have a projection on zero when $ A $ happens to be a single node with $ \Labn +\pi\Labe_A \neq 0 $.
We illustrate this coproduct on a well-chosen example.
\begin{example}
\label{example3}
We continue with the tree in Example \ref{example1}. We suppose that $ \Lab_+ = \lbrace \Labhom(2),  \Labhom(3), \Labhom(4), \Labhom(5) \rbrace $. Below, the subtree $ A \in \Adm(T)$ is colored in blue. We have $ N_A = \lbrace  \varrho, a,b\rbrace $, $ E_A = \lbrace 1,2 \rbrace $ and $ \partial(A,T) = \lbrace 3, 4,5 \rbrace $ .
\begin{equs}   
  \begin{tikzpicture}[scale=0.19,baseline=2cm]
        \node at (0,10)  (a) {}; 
         \node at (4,20)  (f) {}; 
          \node at (0,30)  (k) {}; 
           \node at (8,30) (l) {}; 
           \node at (20,20)  (g) {}; 
            \node at (16,30) (m) {};  
        \node at (12,15) (c) {}; 
         \node at (24,30) (p) {}; %
    \draw[kernel1,blue] (a) -- node [round3] {\tiny $\Labhom(1),\Labp(1)$}  (c) ; 
     \draw[kernel1,blue] (c) -- node [round3] {\tiny $\Labhom(2),\Labp(2)$}  (f) ; 
     \draw[kernel1] (c) -- node [round1] {\tiny 
     $\Labhom(3),\Labp(3)$}  (g) ; 
     \draw[kernel1,black] (f) -- node [near end, round1] {\tiny $\Labhom(4),\Labp(4)$}  (k) ; 
     \draw[kernel1] (f) -- node [round1] {\tiny $\Labhom(5),\Labp(5)$}  (l) ; 
     \draw[kernel1] (g) -- node [round1] {\tiny $\Labhom(6),\Labp(6)$}  (m) ; 
     \draw[kernel1] (g) -- node [near end, round1] {\tiny $\Labhom(7),\Labp(7)$}  (p) ;
     
     \draw (p) node [rect2] {\tiny $ \Labn({g}), \mathfrak{f}({g}) $}  ;
    \draw (m) node [rect2] {\tiny $ \Labn({f}), \Labo({f}) $}  ;
    \draw (l) node [rect2] {\tiny $ \Labn({e}), \Labo({e}) $}  ;
     \draw (k) node [rect2] {\tiny $ \Labn(d), \Labo(d) $}  ;
    \draw (g) node [rect2] {\tiny $ \Labn(c),\Labo(c) $}  ;
    \draw[blue] (f) node [round3] {\tiny $ \Labn(b),\mathfrak{f}(b) $}  ;
    \draw[blue] (c) node [round3] {\tiny $ \Labn(a),\Labo(a) $}  ;
    \draw[blue] (a) node [round3] {\tiny $ r $}  ;
\end{tikzpicture} 
\end{equs}
 We also get:
\begin{equs}
\deg(4) = \deg(5) = 1 + \Labn(a) + \Labn(b), \quad \deg(3) = \Labn(a). 
\end{equs}
We have for a fix $ \Labe_A :  \partial(A,T) \rightarrow \N $:
\begin{equs}   
  \begin{tikzpicture}[scale=0.19,baseline=2cm]
        \node at (0,10)  (a) {}; 
         \node at (4,20)  (f) {}; 
        \node at (12,15) (c) {}; 
    \draw[kernel1,blue] (a) -- node [round3] {\tiny $\Labhom(1),\Labp(1)$}  (c) ; 
     \draw[kernel1,blue] (c) -- node [round3] {\tiny $\Labhom(2),\Labp(2)$}  (f) ; 
    \draw[blue] (f) node [round3] {\tiny $ \Labn(b) + \Labe_A(4) + \Labe_A(5),\mathfrak{f}(b) $}  ;
    \draw[blue] (c) node [round3] {\tiny $ \Labn(a) + \Labe_A(3),\Labo(a) $}  ;
    \draw[blue] (a) node [round3] {\tiny $ r $}  ;
\end{tikzpicture} \otimes \quad
\begin{tikzpicture}[scale=0.19,baseline=2cm]
        \node at (0,10)  (a) {}; 
         \node at (4,20)  (f) {}; 
          \node at (0,30)  (k) {}; 
           \node at (8,30) (l) {}; 
           \node at (0,20)  (g) {}; 
            \node at (-6,30) (m) {};  
        \node at (0,10) (c) {}; 
         \node at (6,30) (p) {}; %
     \draw[kernel1] (c) -- node [round1] {\tiny 
     $\Labhom(3),\Labp(3)$}  (g) ; 
     \draw[kernel1] (g) -- node [round1] {\tiny $\Labhom(6),\Labp(6)$}  (m) ; 
     \draw[kernel1] (g) -- node [near end, round1] {\tiny $\Labhom(7),\Labp(7)$}  (p) ;
     
     \draw (p) node [rect2] {\tiny $ \Labn({g}), \mathfrak{f}({g}) $}  ;
    \draw (m) node [rect2] {\tiny $ \Labn({f}), \Labo({f}) $}  ;
    \draw (g) node [rect2] {\tiny $ \Labn(c),\mathfrak{f}(c) $}  ;
    \draw (c) node [rect1] {\tiny $ r - \Labn(a),\Labe_A(3)$}  ;
\end{tikzpicture} 
 \begin{tikzpicture}[scale=0.19,baseline=2cm]
           \node at (0,20)  (g) {}; 
        \node at (0,10) (c) {}; 
     \draw[kernel1] (c) -- node [round1] {\tiny 
     $\Labhom(4),\Labp(4)$}  (g) ;    
    \draw (g) node [rect2] {\tiny $ \Labn(d),\mathfrak{f}(d) $};
    \draw (c) node [rect1] {\tiny $ \bar r,\Labe_A(4)$}  ;
\end{tikzpicture} 
\qquad 
\begin{tikzpicture}[scale=0.19,baseline=2cm]
           \node at (0,20)  (g) {}; 
        \node at (0,10) (c) {}; 
     \draw[kernel1] (c) -- node [round1] {\tiny 
     $\Labhom(5),\Labp(5)$}  (g) ;    
    \draw (g) node [rect2] {\tiny $ \Labn(e),\mathfrak{f}(e) $};
    \draw (c) node [rect1] {\tiny $ \bar r,\Labe_A(5)$}  ;
\end{tikzpicture} 
\end{equs}
where $ \bar r = r - \Labn(a) - \Labn(b) -1  $.
Now if $ A $ is just equal to the root of $ T $, then one gets $ N_A = \lbrace  \varrho \rbrace $, $ E_A = \emptyset$ and $ \partial(A,T) = \lbrace 1\rbrace $ as illustrated below
\begin{equs} \label{example_root}   
 \begin{tikzpicture}[scale=0.19,baseline=2cm]
        \node at (0,10) (c) {}; 
    \draw (c) node [rect1] {\tiny $ \Labe_A(1)$}  ;
\end{tikzpicture} \qquad \otimes \qquad \begin{tikzpicture}[scale=0.19,baseline=2cm]
        \node at (0,10)  (a) {}; 
         \node at (4,20)  (f) {}; 
          \node at (0,30)  (k) {}; 
           \node at (8,30) (l) {}; 
           \node at (20,20)  (g) {}; 
            \node at (16,30) (m) {};  
        \node at (12,15) (c) {}; 
         \node at (24,30) (p) {}; %
    \draw[kernel1] (a) -- node [round1] {\tiny $\Labhom(1),\Labp(1)$}  (c) ; 
     \draw[kernel1] (c) -- node [round1] {\tiny $\Labhom(2),\Labp(2)$}  (f) ; 
     \draw[kernel1] (c) -- node [round1] {\tiny 
     $\Labhom(3),\Labp(3)$}  (g) ; 
     \draw[kernel1,black] (f) -- node [near end, round1] {\tiny $\Labhom(4),\Labp(4)$}  (k) ; 
     \draw[kernel1] (f) -- node [round1] {\tiny $\Labhom(5),\Labp(5)$}  (l) ; 
     \draw[kernel1] (g) -- node [round1] {\tiny $\Labhom(6),\Labp(6)$}  (m) ; 
     \draw[kernel1] (g) -- node [near end, round1] {\tiny $\Labhom(7),\Labp(7)$}  (p) ;
     
     \draw (p) node [rect2] {\tiny $ \Labn({g}), \mathfrak{f}({g}) $}  ;
    \draw (m) node [rect2] {\tiny $ \Labn({f}), \Labo({f}) $}  ;
    \draw (l) node [rect2] {\tiny $ \Labn({e}), \Labo({e}) $}  ;
     \draw (k) node [rect2] {\tiny $ \Labn(d), \Labo(d) $}  ;
    \draw (g) node [rect2] {\tiny $ \Labn(c),\Labo(c) $}  ;
    \draw (f) node [rect2] {\tiny $ \Labn(b),\Labo(b) $}  ;
    \draw (c) node [rect2] {\tiny $ \Labn(a),\Labo(a) $}  ;
    \draw (a) node [rect1] {\tiny $ r,\Labe_A(1) $}  ;
\end{tikzpicture}
\end{equs}
The map $ \Deltap $ behaves the same way except that we start with a tree decorated by $ (r,m) $ at the root and that we exclude the case described in \eqref{example_root}. 
\end{example}

\begin{example}
Next we provide a more explicit example of the computations for the maps $ \Deltap $ and $ \Delta $ on the tree
\begin{equs}\label{kdvTK2}
 \begin{tikzpicture}[scale=0.2,baseline=-5]
\coordinate (root) at (0,0);
\coordinate (tri) at (0,-2);
\coordinate (t1) at (-1,2);
\coordinate (t11) at (-2,4);
\coordinate (t12) at (-3,6);
\coordinate (t13) at (-1,6);
\coordinate (t2) at (1,2);
\draw[kernels2] (t11) -- (t13);
\draw[kernels2] (t11) -- (t12);
\draw[kernels2] (t1) -- (root);
\draw[symbols] (t1) -- (t11);
\draw[kernels2] (t2) -- (root);
\draw[symbols] (root) -- (tri);
\node[not] (rootnode) at (root) {};
\node[not,label= {[label distance=-0.2em]below: \scriptsize  $ r $}] (trinode) at (tri) {};
\node[not] (trinode) at (t1) {};
\node[var] (rootnode) at (t12) {\tiny{$ k_{\tiny{1}} $}};
\node[var] (rootnode) at (t13) {\tiny{$ k_{\tiny{2}} $}};
\node[var] (trinode) at (t2) {\tiny{$ k_3 $}};
\end{tikzpicture} 
\end{equs}
 that appears for the KdV equation \eqref{kdvIntro}. We have that
\begin{equs}
 \Delta \begin{tikzpicture}[scale=0.2,baseline=-5]
\coordinate (root) at (0,0);
\coordinate (tri) at (0,-2);
\coordinate (t1) at (-1,2);
\coordinate (t11) at (-2,4);
\coordinate (t12) at (-3,6);
\coordinate (t13) at (-1,6);
\coordinate (t2) at (1,2);
\draw[kernels2] (t11) -- (t13);
\draw[kernels2] (t11) -- (t12);
\draw[kernels2] (t1) -- (root);
\draw[symbols] (t1) -- (t11);
\draw[kernels2] (t2) -- (root);
\draw[symbols] (root) -- (tri);
\node[not] (rootnode) at (root) {};
\node[not,label= {[label distance=-0.2em]below: \scriptsize  $ r $}] (trinode) at (tri) {};
\node[not] (trinode) at (t1) {};
\node[var] (rootnode) at (t12) {\tiny{$ k_{\tiny{1}} $}};
\node[var] (rootnode) at (t13) {\tiny{$ k_{\tiny{2}} $}};
\node[var] (trinode) at (t2) {\tiny{$ k_3 $}};
\end{tikzpicture}  = \begin{tikzpicture}[scale=0.2,baseline=-5]
\coordinate (root) at (0,0);
\coordinate (tri) at (0,-2);
\coordinate (t1) at (-1,2);
\coordinate (t11) at (-2,4);
\coordinate (t12) at (-3,6);
\coordinate (t13) at (-1,6);
\coordinate (t2) at (1,2);
\draw[kernels2] (t11) -- (t13);
\draw[kernels2] (t11) -- (t12);
\draw[kernels2] (t1) -- (root);
\draw[symbols] (t1) -- (t11);
\draw[kernels2] (t2) -- (root);
\draw[symbols] (root) -- (tri);
\node[not] (rootnode) at (root) {};
\node[not,label= {[label distance=-0.2em]below: \scriptsize  $ r $}] (trinode) at (tri) {};
\node[not] (trinode) at (t1) {};
\node[var] (rootnode) at (t12) {\tiny{$ k_{\tiny{1}} $}};
\node[var] (rootnode) at (t13) {\tiny{$ k_{\tiny{2}} $}};
\node[var] (trinode) at (t2) {\tiny{$ k_3 $}};
\end{tikzpicture} \otimes \one + \sum_{m  \leq r+1} \frac{ \lambda^m}{m!} \otimes \begin{tikzpicture}[scale=0.2,baseline=-5]
\coordinate (root) at (0,0);
\coordinate (tri) at (0,-2);
\coordinate (t1) at (-1,2);
\coordinate (t11) at (-2,4);
\coordinate (t12) at (-3,6);
\coordinate (t13) at (-1,6);
\coordinate (t2) at (1,2);
\draw[kernels2] (t11) -- (t13);
\draw[kernels2] (t11) -- (t12);
\draw[kernels2] (t1) -- (root);
\draw[symbols] (t1) -- (t11);
\draw[kernels2] (t2) -- (root);
\draw[symbols] (root) -- (tri);
\node[not] (rootnode) at (root) {};
\node[not,label= {[label distance=-0.2em]below: \scriptsize  $ (r,m) $}] (trinode) at (tri) {};
\node[not] (trinode) at (t1) {};
\node[var] (rootnode) at (t12) {\tiny{$ k_{\tiny{1}} $}};
\node[var] (rootnode) at (t13) {\tiny{$ k_{\tiny{2}} $}};
\node[var] (trinode) at (t2) {\tiny{$ k_3 $}};
\end{tikzpicture} 
 + \sum_{m \leq r} \frac{1}{m!}  \begin{tikzpicture}[scale=0.2,baseline=-5]
\coordinate (root) at (0,0);
\coordinate (tri) at (0,-2);
\coordinate (t1) at (-1,2);
\coordinate (t2) at (1,2);
\draw[kernels2] (t1) -- (root);
\draw[kernels2] (t2) -- (root);
\draw[symbols] (root) -- (tri);
\node[not] (rootnode) at (root) {};
\node[not,label= {[label distance=-0.2em]below: \scriptsize  $ r $}] (trinode) at (tri) {};
\node[var] (rootnode) at (t1) {\tiny{$ _\ell^m $}};
\node[var] (trinode) at (t2) {\tiny{$ k_2 $}};
\end{tikzpicture}   \otimes \begin{tikzpicture}[scale=0.2,baseline=-5]
\coordinate (root) at (0,0);
\coordinate (tri) at (0,-2);
\coordinate (t1) at (-1,2);
\coordinate (t2) at (1,2);
\draw[kernels2] (t1) -- (root);
\draw[kernels2] (t2) -- (root);
\draw[symbols] (root) -- (tri);
\node[not] (rootnode) at (root) {};
\node[not,label= {[label distance=-0.2em]below: \scriptsize  $ (r-1,m) $}] (trinode) at (tri) {};
\node[var] (rootnode) at (t1) {\tiny{$ k_{\tiny{1}} $}};
\node[var] (trinode) at (t2) {\tiny{$ k_2 $}};
\end{tikzpicture}  
\end{equs}
where $ \ell = k_1 + k_2 $ and we have used similar notations as in Example~\ref{tex:kdv} and in \eqref{DRnot}. We have introduced a new graphical notation for a node decorated
by the decoration $(m,\ell)$ when $ m \neq 0 $: 
$ \begin{tikzpicture}[scale=0.2,baseline=-5]
\coordinate (root) at (0,0);
\node[var] (rootnode) at (root) {\tiny{$ _\ell^m $}};
\end{tikzpicture} $. One can notice that the second abstract Taylor expansion is shorter. This is due to the fact that there was one blue edge (in $ \Lab_+ $) on the path connecting the cutting tree to the root. In addition we have
\begin{equs}
 \Deltap \begin{tikzpicture}[scale=0.2,baseline=-5]
\coordinate (root) at (0,0);
\coordinate (tri) at (0,-2);
\coordinate (t1) at (-1,2);
\coordinate (t11) at (-2,4);
\coordinate (t12) at (-3,6);
\coordinate (t13) at (-1,6);
\coordinate (t2) at (1,2);
\draw[kernels2] (t11) -- (t13);
\draw[kernels2] (t11) -- (t12);
\draw[kernels2] (t1) -- (root);
\draw[symbols] (t1) -- (t11);
\draw[kernels2] (t2) -- (root);
\draw[symbols] (root) -- (tri);
\node[not] (rootnode) at (root) {};
\node[not,label= {[label distance=-0.2em]below: \scriptsize  $ (r,m) $}] (trinode) at (tri) {};
\node[not] (trinode) at (t1) {};
\node[var] (rootnode) at (t12) {\tiny{$ k_{\tiny{1}} $}};
\node[var] (rootnode) at (t13) {\tiny{$ k_{\tiny{2}} $}};
\node[var] (trinode) at (t2) {\tiny{$ k_3 $}};
\end{tikzpicture}  = \begin{tikzpicture}[scale=0.2,baseline=-5]
\coordinate (root) at (0,0);
\coordinate (tri) at (0,-2);
\coordinate (t1) at (-1,2);
\coordinate (t11) at (-2,4);
\coordinate (t12) at (-3,6);
\coordinate (t13) at (-1,6);
\coordinate (t2) at (1,2);
\draw[kernels2] (t11) -- (t13);
\draw[kernels2] (t11) -- (t12);
\draw[kernels2] (t1) -- (root);
\draw[symbols] (t1) -- (t11);
\draw[kernels2] (t2) -- (root);
\draw[symbols] (root) -- (tri);
\node[not] (rootnode) at (root) {};
\node[not,label= {[label distance=-0.2em]below: \scriptsize  $ (r,m) $}] (trinode) at (tri) {};
\node[not] (trinode) at (t1) {};
\node[var] (rootnode) at (t12) {\tiny{$ k_{\tiny{1}} $}};
\node[var] (rootnode) at (t13) {\tiny{$ k_{\tiny{2}} $}};
\node[var] (trinode) at (t2) {\tiny{$ k_3 $}};
\end{tikzpicture} \otimes \one + \one \otimes \begin{tikzpicture}[scale=0.2,baseline=-5]
\coordinate (root) at (0,0);
\coordinate (tri) at (0,-2);
\coordinate (t1) at (-1,2);
\coordinate (t11) at (-2,4);
\coordinate (t12) at (-3,6);
\coordinate (t13) at (-1,6);
\coordinate (t2) at (1,2);
\draw[kernels2] (t11) -- (t13);
\draw[kernels2] (t11) -- (t12);
\draw[kernels2] (t1) -- (root);
\draw[symbols] (t1) -- (t11);
\draw[kernels2] (t2) -- (root);
\draw[symbols] (root) -- (tri);
\node[not] (rootnode) at (root) {};
\node[not,label= {[label distance=-0.2em]below: \scriptsize  $ (r,m) $}] (trinode) at (tri) {};
\node[not] (trinode) at (t1) {};
\node[var] (rootnode) at (t12) {\tiny{$ k_{\tiny{1}} $}};
\node[var] (rootnode) at (t13) {\tiny{$ k_{\tiny{2}} $}};
\node[var] (trinode) at (t2) {\tiny{$ k_3 $}};
\end{tikzpicture} 
 + \sum_{n \leq r} \frac{1}{n!}  \begin{tikzpicture}[scale=0.2,baseline=-5]
\coordinate (root) at (0,0);
\coordinate (tri) at (0,-2);
\coordinate (t1) at (-1,2);
\coordinate (t2) at (1,2);
\draw[kernels2] (t1) -- (root);
\draw[kernels2] (t2) -- (root);
\draw[symbols] (root) -- (tri);
\node[not] (rootnode) at (root) {};
\node[not,label= {[label distance=-0.2em]below: \scriptsize  $ (r,m) $}] (trinode) at (tri) {};
\node[var] (rootnode) at (t1) {\tiny{$ _\ell^n $}};
\node[var] (trinode) at (t2) {\tiny{$ k_2 $}};
\end{tikzpicture}   \otimes \begin{tikzpicture}[scale=0.2,baseline=-5]
\coordinate (root) at (0,0);
\coordinate (tri) at (0,-2);
\coordinate (t1) at (-1,2);
\coordinate (t2) at (1,2);
\draw[kernels2] (t1) -- (root);
\draw[kernels2] (t2) -- (root);
\draw[symbols] (root) -- (tri);
\node[not] (rootnode) at (root) {};
\node[not,label= {[label distance=-0.2em]below: \scriptsize  $ (r-1,n) $}] (trinode) at (tri) {};
\node[var] (rootnode) at (t1) {\tiny{$ k_{\tiny{1}} $}};
\node[var] (trinode) at (t2) {\tiny{$ k_2 $}};
\end{tikzpicture}  
\end{equs}
One of the main difference between $ \Delta $ and $ \Deltap $ is that we do not have a higher Taylor expansion on the edge connecting the root for $ \Deltap $ because the elements $  \lambda^k $ are not in $ \CH_+ $. With this property, $ \CH_+ $ will be a connected Hopf algebra.
\end{example}

We use the symbolic notation to provide an alternative,  recursive definition of the two maps $  \Delta : \CH \rightarrow \CH \otimes \CH_+ $ and $ \Deltap : \CH_+ \rightarrow \CH_+ \otimes \CH_+ $

\begin{equation} \label{def_deltas}
\begin{aligned}
\Delta \one & = \one \otimes \one, \quad \Delta  \lambda^{\ell} =  \lambda^{\ell} \otimes \one\\
\Delta \CI^{r}_{o_1}( \lambda_{k}^{\ell} F)  & = \left(  \CI^{r}_{o_1}( \lambda_{k}^{\ell}\cdot)  \otimes \id  \right) \Delta \CD^{r-\ell}(F) \\
\Delta \CI^{r}_{o_2}( \lambda_{k}^{\ell} F)  & = \left( \CI^{r}_{o_2}(  \lambda_{k}^{\ell} \cdot)  \otimes \id  \right) \Delta \CD^{r-\ell-1}(F) + \sum_{m \leq r +1} \frac{ \lambda^{m}}{m!} \otimes  \CI^{(r,m)}_{o_2}( \lambda_{k}^{\ell} F) \\
\Deltap \CI^{(r,m)}_{o_2}( \lambda_{k}^{\ell} F)  & = \left( \CI^{(r,m)}_{o_2}(  \lambda_{k}^{\ell} \cdot)  \otimes \id  \right) \Delta \CD^{r-\ell-1}(F) +  \one \otimes  \CI^{(r,m)}_{o_2}( \lambda_{k}^{\ell} F)
\end{aligned}
\end{equation}
where $ o_1 = (\Labhom_1,p_1) $, $ \Labhom_1 \notin \Lab_+ $ and  $ o_2 = (\Labhom_2,p_2) $, $ \Labhom_2 \in \Lab_+ $.
\begin{remark} \label{comparisonalgebra}
The maps $ \Deltap $ and $ \Delta $ are a variant of the maps used in \cite{reg,BHZ} for the recentering of iterated integrals in the context of singular SPDEs. They could be understood as a deformed Butcher-Connes-Kreimer coproduct.
We present the ideas behind their construction:
\begin{itemize}
\item the decoration $ \Labo $ are inert and behave nicely toward the extraction/cutting operation. 
\item The recursive definition \eqref{def_deltas} is close to the definition of the Connes-Kreimer coproduct with an operator which grafts a forest onto a new root (see \cite{CK}).
\item The deformation is given by the sum over $ m $ where root decorations are increased. The number of terms in the sum is finite bounded by $ r+1 $. This deformation corresponds to the one used for SPDEs but with a different projection. In Numerical Analysis, the length of the Taylor expansion is governed by the path connecting the root to the edge we are considering, whereas for SPDEs, it depends on the tree above the edge. Therefore, the structure proposed here is new in comparison to the literature and shows the universality of the deformation observed for singular SPDEs.
\item There are some interesting simplifications in the definition of $ \Delta $ and $ \Deltap $ in comparison to \cite{reg,BHZ}. Indeed, one has 
\begin{equs}
\Delta  \lambda^{\ell} =  \lambda^{\ell} \otimes \one
\end{equs}
instead of the expected definition for the polynomial coproduct
\begin{equs}
\Delta  \lambda =  \lambda \otimes \one + \one \otimes  \lambda, \quad \Delta  \lambda^{\ell} = \sum_{m \leq \ell} 
\binom{\ell}{m}  \lambda^{m} \otimes  \lambda^{\ell - m}
\end{equs}
and $ \Deltap  $ is not defined on polynomials. This comes from our numerical scheme: We are only interested in recentering around $ 0 $. Therefore, all the right part of the tensor product will be evaluated at zero and all these terms can be omitted at the level of the algebra.
Such simplifications can also be used in the context of SPDEs where one considers random objects of the form $ \Pi_x T $ which are recentered iterated integrals around the point $ x $. When one wants to construct these stochastic objects, the interest lies in their law and it turns out  that their law is invariant by translation. Then, one can only consider  the term $ \Pi_0 T $ which corresponds to the Numerical Analysis framework.
With this simplification, we obtain an easier formulation for the antipode given in \eqref{antipode_rec}. 
\end{itemize}
\end{remark}

\begin{remark}
For the sequel, we will use mainly the symbolic notation \eqref{def_deltas} which is very useful for carrying out recursive proofs. We will also develop a recursive formulation of the general numerical scheme. This approach is also crucial in \cite{reg} and has been pushed forward in \cite{BR18} for singular SPDEs.
\end{remark}

In the next proposition, we prove the equivalence between the recursive and non-recursive definitions.
\begin{proposition}
The definitions \eqref{eq:co-action_plus} and \eqref{eq:coproduct_plus} coincide with \eqref{def_deltas}.
\end{proposition}
\begin{proof} 
The operator $\Delta$ is multiplicative on $\hat \CH$.
It remains to verify that the recursive identities hold as well. 
We consider $\Delta \sigma$ with $ \sigma = \CI^{r}_{(\Labhom_2,p)}( \lambda^{\ell}_{k} \tau)$ and $ \Labhom_2 \in \Lab_+ $. We write $  \tau =  F_{\Labe}^{\Labn,\Labo}$, $ \bar \tau = F_{\Labe, r-\ell-1}^{\Labn,\Labo} $ where $ F = \prod_i T_i $ is a forest formed of the trees $ T_i $. One has $ \bar \tau = \prod_i (T_i)_{\Labe_i, r-\ell-1}^{\Labn_i,\Labo_i} $ and the maps $ \Labn, \Labo, \Labe $ are obtained as disjoint sums of the $ \Labn_i, \Labo_i, \Labe_i $. We write  $\sigma = T^{\bar \Labn,\bar \Labo}_{\bar \Labe,r}$ where
\begin{equs}
\bar \Labe = \Labe + \one_e (\Labhom_2,p), \quad \bar \Labo(u) = \Labo + \one_u k, \quad \bar \Labn(u) = \Labn + \one_u \ell
\end{equs} 
and $e$ denotes a
trunk of type $\Labhom$ created by $\CI_{(\Labhom_2,p)}$, $\rho$ is the root of $T$ and $ u  $ is such that $ e=(\rho,u) $.
It follows from these definitions that
\begin{equs}
\Adm(T) = \{\{\rho\}\} \cup \{ A \cup \{\rho,e\}\,:\, A \in \Adm(F)\}\;
\end{equs}
where $ \Adm(F) = \sqcup_i \Adm(T_i)$ and the $ A $ are forests. One can actually rewrite \eqref{eq:co-action_plus}   exactly as the same for forests.
Then, we have the identity
\begin{equs}
 \Delta \sigma &= (\CI^{r}_{(\Labhom_2,p)}( \lambda^{\ell}_k \cdot) \otimes \id) \Delta \bar \tau
+ \sum_{\Labe_{\bullet}} {1\over \Labe_{\bullet}!}
(\bullet, \pi \Labe_{\bullet},0,0 ,0)  \otimes (T, \bar{\Labn},\bar{\Labo},\bar{\Labe},(r,\Labe_{\bullet}))     
 \\ 
& = \left( \CI^{r}_{(\Labhom_2,p)}(  \lambda_{k}^{\ell} \cdot)  \otimes \id  \right) \Delta \CD^{r-\ell-1}(\tau) + \sum_{m \leq r +1} \frac{ \lambda^{m}}{m!} \otimes  \CI^{(r,m)}_{(\Labhom_2,p)}( \lambda_{k}^{\ell} \tau)\;
\end{equs}
where the recursive $\left( \CI^{r}_{(\Labhom_2,p)}( \lambda^{\ell}_k \cdot) \otimes \id \right) \Delta$ encodes the extraction of $ A \cup \{\rho,e\} $, $ A \in \Adm(T) $.
We can perform a similar proof for $ \Labhom_1 \notin \Lab_+ $ and for $ \Deltap $. The main difference is that  the sum on the polynomial decoration is removed for an edge not in $ \Lab_+ $ and for $ \Deltap $ such that we just keep the first term.  
\end{proof}

\begin{example}\label{rem:RecDnlsT1} We illustrate the recursive definition of $ \Delta $ by performing some computations on some relevant decorated trees that one can face in practice for instance in case of cubic NLS \eqref{nlsIntro}.
For the decorated tree
\begin{equs}
T_1  =  \CI_{(\Labhom_2,0)} \left( \lambda_k  F_1 \right) \quad
 F_1  = \CI_{(\Labhom_1,1)}( \lambda_{k_1}) \CI_{(\Labhom_1,0)}( \lambda_{k_2})  \CI_{(\Labhom_1,0)}( \lambda_{k_3})
\end{equs}
which appears in context of the cubic Schrödinger equation, see Section \ref{sec:nls},
we have that
\begin{equs}\label{comps}
\Delta \CD^r(T_1) =  \CD^r(T_1) \otimes \one + \sum_{m \leq r+1} \frac{ \lambda^m}{m!} \otimes  \hat \CD^{(r,m)}( T_1).
\end{equs}
Relation \eqref{comps} is proven as follows: Using the definition of $ \CD^r$ in \eqref{DR} as well as~\eqref{def_deltas}  yields that
\begin{equs}\label{comp1}
\Delta \CD^r(T_1)  & = \Delta \CI^{r}_{(\Labhom_2,0)}( \lambda_{k} F_1)  \\& = \left( \CI^{r}_{(\Labhom_2,0)}(  \lambda_k \cdot)  \otimes \id  \right) \Delta \CD^{r-1}(F_1) + \sum_{m \leq r +1} \frac{ \lambda^{m}}{m!} \otimes  \CI^{(r,m)}_{(\Labhom_2,0)}( \lambda_k F_1).\end{equs}
Thanks to the definition of $\hat \CD^r$ in \eqref{DRhat} we can conclude that
\[
 \CI^{(r,m)}_{(\Labhom_2,0)}( \lambda_k F_1) =  \hat \CD^{(r,m)}\CI_{(\Labhom_2,0)}( \lambda_{k} F_1)  =\hat \CD^{(r,m)} (T_1) 
\]
which yields  together with \eqref{comp1} that
\begin{equ}\label{compi}
\Delta \CD^r(T_1)  = \left( \CI^{r}_{(\Labhom_2,0)}(  \lambda_k \cdot)  \otimes \id  \right) \Delta \CD^{r-1}(F_1) + \sum_{\ell \leq r +1} \frac{ \lambda^{\ell}}{\ell!} \otimes  \hat \CD^{(r,\ell)} (T_1). \end{equ}
Next we need to analyse the term $ \left( \CI^{r}_{(\Labhom_2,0)}(  \lambda_k \cdot)  \otimes \id  \right) \Delta \CD^{r-1}(F_1) $. First we use  the multiplicativity of $\CD^r${,} (cf.  \eqref{DR}) and  coproduct which yields that
\begin{equs}\label{comp2}
\Delta \CD^{r-1}(F_1) = \left( \Delta \CI^{r-1}_{(\Labhom_1,1)}( \lambda_{k_1})\right)\left( \Delta \CI^{r-1}_{(\Labhom_1,0)}( \lambda_{k_2})  \right)\left(\Delta \CI^{r-1}_{(\Labhom_1,0)}( \lambda_{k_3})\right) .
\end{equs}
Thanks to \eqref{def_deltas} we furthermore have that 
\begin{equs}\label{comp3}
\Delta \CI^{r-1}_{(\Labhom_1,p)}( \lambda_{k_j})   &=  \left(  \CI^{r-1}_{(\Labhom_1,p)}( \lambda_{k_j}\cdot)  \otimes \id  \right) \Delta \CD^{r-1}( \one )\\&  =  \left(  \CI^{r-1}_{(\Labhom_1,p)}( \lambda_{k_j}\cdot)  \otimes \id  \right)\left( \one \otimes \one\right)
 =  \CI^{r-1}_{(\Labhom_1,p)}( \lambda_{k_j})  \otimes \one
\end{equs}
where we have used that $ \CD^{r-1}( \one ) = \one$ and $\Delta \one  = \one \otimes \one$, see also  \eqref{def_deltas}. Plugging~\eqref{comp3} into \eqref{comp2} yields that
\begin{equs}\label{CoF1}
\Delta \CD^{r-1}(F_1)  \\& = \left(  \CI^{r-1}_{(\Labhom_1,1)}( \lambda_{k_1})  \otimes \one  \right)
 \left(  \CI^{r-1}_{(\Labhom_1,0)}( \lambda_{k_2})  \otimes \one  \right) \left(  \CI^{r-1}_{(\Labhom_1,0)}( \lambda_{k_3})  \otimes \one  \right)\\
  &= \CI^{r-1}_{(\Labhom_1,1)}( \lambda_{k_1})  \CI^{r-1}_{(\Labhom_1,0)}( \lambda_{k_2})   \CI^{r-1}_{(\Labhom_1,0)}( \lambda_{k_3}) \otimes \one 
  =  \CD^{r-1}(F_1) \otimes \one .
\end{equs}
Hence,
\begin{equs}
\left( \CI^{r}_{(\Labhom_2,0)}(  \lambda_k \cdot)  \otimes \id  \right) \Delta \CD^{r-1}(F_1)&  =
\left( \CI^{r}_{(\Labhom_2,0)}(  \lambda_k \cdot)  \otimes \id  \right) \left( \CD^{r-1}(F_1) \otimes \one  \right)\\
& = \CI^{r}_{(\Labhom_2,0)}(  \lambda_k F_1)  \otimes  \one =  \CD^{r}(T_1) \otimes  \one.
\end{equs}
Plugging this into \eqref{compi}   yields \eqref{comps}.
\end{example}

\subsection{Hopf {algebra} and comodule structures}
Using the two maps $ \Delta $ and $ \Deltap $, we want to identify a comodule structure over a Hopf algebra.
Here, we provide a brief reminder of this structure for a reader not familiar with it. For simplicity, we will use the notation of the spaces introduced above as well as the maps $ \Delta $ and $ \Deltap $. The proof that we are indeed in this framework is then given in Proposition~\ref{Hopf_algebras} below.

A {\it bialgebra} $(\CH_+,\CM,\one,\Deltap,\one^\star)$ is given by:
\begin{itemize}
\item A vector space $ \CH_+ $ over $\C$
\item A linear map $\CM:\CH_+\otimes \CH_+ \to \CH_+$ (product) and an element $\eta:r\mapsto r\one$, $\one\in \CH_+$ 
(identity) such that $(\CH_+,\CM,\eta)$ is a unital associative algebra.
\item Linear maps $\Deltap :\CH_+ \to \CH_+\otimes \CH_+$ (coproduct) and $\one^\star:\CH_+\to\C$ (counit), 
such that $(\CH_+,\Deltap,\one^\star)$ is a counital coassociative coalgebra, namely
\begin{equation}\label{e:coasso}
(\Deltap\otimes\id)\Deltap=(\id\otimes\Deltap)\Deltap, \qquad (\one^\star\otimes\id)\Deltap=
(\id\otimes\one^\star)\Deltap=\id
\end{equation}
\item $ \Deltap $ and $ \one^{*} $ (resp. $ \CM $ and $ \one $) are homomorphisms of algebras (coalgebras).
\end{itemize}

A {\it Hopf algebra} is a bialgebra $(\CH_+,\CM,\one,\Deltap,\one^\star)$ endowed with a linear map
$\CA: \CH_+ \to \CH_+$ such that
\begin{equation}\label{anti}
\CM(\id\otimes \CA)\Deltap = \CM(\CA\otimes\id)\Deltap= \one^\star\one.
\end{equation}

A {\it right comodule} over a bialgebra $(\CH_+,\CM,\one,\Deltap,\one^\star)$ is a pair $(\CH,\Delta)$ where $\CH$ is a vector space and $\Delta: \CH \to \CH \otimes \CH_+$ is a linear map such that
\begin{equs}[e:coaction]
(\Delta\otimes\id)\Delta=(\id\otimes\Deltap)\Delta, \qquad (\id \otimes \one^{\star})\Delta=\id.
\end{equs}
In our framework, the product $ \CM $ is given by the forest product:
\begin{equs}
 \CM (F_1 \otimes F_2) = F_1 \cdot F_2.
 \end{equs}
 Most of the properties listed above are quite straightforward to check. In the next proposition we focus on the coassociativity of the maps $ \Delta $ and $ \Deltap $ in \eqref{e:coasso} and~\eqref{e:coaction}.
Before, let us explain why this structure is useful. If one considers characters that are multiplicative maps $ g : \CH_+ \rightarrow \C$, then the coproduct $ \Delta^{\!+} $ and the antipode $\CA  $ allow us  to put a group structure on them. We denote the group of such characters by $ \mathcal{G} $. The product for this group is the convolution product $ \star $ given for $ f,g \in \mathcal{G} $ by:
\begin{equs}
f \star g = \left( f \otimes g \right) \Delta^{\!+}.
\end{equs}
We do not need a multiplication because we use the identification  $ \C \otimes \C \cong \C $. The inverse is given by the antipode:
\begin{equs}
f^{-1} = f(\CA \cdot).
\end{equs}
The comodule structure allows us to have an action of $ \mathcal{G} $ onto $ \CH $ defined by
\begin{equs}
\Gamma_{f} = \left( \id \otimes f \right) \Delta^{\!+}, \quad \Gamma_{f} \Gamma_{g} = \Gamma_{f \star g}, \quad \Gamma_f^{-1} = \Gamma_{f^{-1}}.
\end{equs}
In Section~\ref{sec::Brikhoff}, we will use these structures to decompose our scheme for iterated integrals, that is a character $ \Pi^n : \CH \rightarrow \CC$, into
\begin{equs}
\Pi^n = \left( \hat \Pi^n \otimes A^n \right) \Delta
\end{equs}
where $ \CC $ is a space introduced in Section~\ref{sec::recursive_pi}, $ A^n \in \mathcal{G} $ (defined from $ \Pi^n $) and $ \hat \Pi^n :  \CH \rightarrow \CC$. In the identity above we have made the following identification $ \CC \otimes \C \cong \CC $.
The main point of this decomposition is to let the character $ \hat \Pi^n $ appear  which is simpler than $ \Pi_n $. This will help us in carrying out  the local error analysis in Section~\ref{local error analysis}.

\begin{proposition} \label{bialgebra}
One has:
\begin{equs} 
\left( \Delta \otimes  \id \right) \Delta = \left( \id \otimes  \Deltap \right) \Delta, \quad \left( \Deltap \otimes  \id \right) \Deltap = \left( \id \otimes  \Deltap \right) \Deltap.
\end{equs}
\end{proposition}
\begin{proof}
We proceed by induction and we perform the proof only for  $ \CI^{r}_{o_2}( \lambda_{k}^{\ell} F)  $. The other case follows similar steps. Note that
\begin{equs}
& \left(  \Delta \otimes  \id \right) \Delta \CI^{r}_{o_2}( \lambda_{k}^{\ell} F)   \\
& = \left(\Delta \CI^{r}_{o_2}(  \lambda_{k}^{\ell} \cdot)  \otimes \id  \right) \Delta \CD^{r-\ell-1}(F) + \sum_{m \leq r + 1} \Delta \frac{ \lambda^{m}}{m!} \otimes  \CI^{(r,m)}_{o_2}( \lambda_{k}^{\ell} F)
\\ & = \sum_{m \leq r + 1} \frac{ \lambda^{m}}{m!} \otimes \left( \CI^{(r,m)}_{o_2}(  \lambda_{k}^{\ell} \cdot)   \otimes \id  \right) \Delta \CD^{r-\ell-1}(F) \\ & + \sum_{m \leq r +1} \frac{ \lambda^{m}}{m!} \otimes \one \otimes  \CI^{(r,m)}_{o_2}( \lambda_{k}^{\ell} F) 
 +\left( \left( \CI^{r}_{o_2}(  \lambda_{k}^{\ell} \cdot) \otimes \id \right) \Delta  \otimes \id  \right) \Delta \CD^{r-\ell-1}(F).
\end{equs}
On the other hand, we get
\begin{equs}
& \left(  \id \otimes  \Deltap \right) \Delta \CI^{r}_{o_2}( \lambda_{k}^{\ell} F)   \\
& = \left(\CI^{r}_{o_2}(  \lambda_{k}^{\ell} \cdot)  \otimes \Deltap  \right) \Delta \CD^{r-\ell-1}(F) + \sum_{m \leq r+1}  \frac{ \lambda^{m}}{m!} \otimes  \Deltap \CI^{(r,m)}_{o_2}( \lambda_{k}^{\ell} F)
\\ & =  \left(\CI^{r}_{o_2}(  \lambda_{k}^{\ell} \cdot)  \otimes \Deltap  \right) \Delta \CD^{r-\ell-1}(F) +  \sum_{m \leq r +1} \frac{ \lambda^{m}}{m!} \otimes \one \otimes  \CI^{(r,m)}_{o_2}( \lambda_{k}^{\ell} F)  \\ & +\sum_{m \leq r +1} \frac{ \lambda^{m}}{m!} \otimes \left( \CI^{(r,m)}_{o_2}(  \lambda_{k}^{\ell} \cdot)   \otimes \id  \right) \Delta \CD^{r-\ell-1}(F).
\end{equs}
Next we observe that
\begin{equs}
\left(\CI^{r}_{o_2}(  \lambda_{k}^{\ell} \cdot)  \otimes \Deltap  \right) \Delta \CD^{r-\ell-1}(F) & =
\left( \CI^{r}_{o_2}(  \lambda_{k}^{\ell} \cdot)   \otimes \id  \otimes \id \right) \left( \id \otimes \Deltap \right) \Delta \CD^{r-\ell-1}(F) \\
& = \left( \CI^{r}_{o_2}(  \lambda_{k}^{\ell} \cdot)   \otimes \id  \otimes \id \right) \left( \Delta \otimes \id \right) \Delta \CD^{r-\ell-1}(F)
\\ & = \left( \left( \CI^{r}_{o_2}(  \lambda_{k}^{\ell} \cdot)   \otimes \id \right) \Delta \otimes \id \right) \Delta \CD^{r-\ell-1}(F)
\end{equs}
where we used an inductive argument for 
\begin{equs}
 \left( \Delta \otimes  \id \right) \Delta  \CD^{r-\ell-1}(F) = \left( \id \otimes  \Deltap \right)  \Delta \CD^{r-\ell-1}(F) .
\end{equs} 
This yields the assertion.
\end{proof}

 \begin{proposition} \label{Hopf_algebras} There exists an algebra morphism $ \CA : \CH_+ \rightarrow \CH_+ $ so that $  (\CH_+, \cdot, \Deltap, \one, \one^{\star}, \CA  ) $ is a Hopf algebra. The map $ \Delta : \CH \rightarrow \CH \otimes \CH_+ $ turns $ \CH $ into a right comodule for  $ \CH_+ $ with counit $ \one^{\star} $.
 \end{proposition}
 \begin{proof}
 From Proposition~\ref{bialgebra}, $ \CH_+ $ is a bialgebra and $ \Delta $ is a coaction. In fact,
 $ \CH_+ $ is a connected graded bialgebra with the grading given by the number of edges. Therefore from \cite[Corollary II.3.2]{Man}, it is a Hopf algebra and we get the existence of a unique map called the antipode such that:
  \begin{equs} \label{antipode_identity}
 \CM \left( \CA \otimes \id \right) \Deltap = \CM \left( \id \otimes \CA \right) \Deltap = \one \one^{\star}.
 \end{equs}
 This concludes the proof.
 \end{proof}
 We use the identity \eqref{antipode_identity} to write a recursive formulation for the antipode:
 \begin{proposition}
 For every $ T \in \hat \CH $, one has
\begin{equs} \label{antipode_rec}
\CA \CI^{(r,m)}_{o_2}(   \lambda_k^{\ell} F ) & = \CM \left(  \CI^{(r,m)}_{o_2}(   \lambda_k^{\ell} \cdot )\otimes \CA \right) \Delta  \CD^{r-\ell-1}(F).
\end{equs}
\end{proposition}
\begin{proof} We use the identity~\eqref{antipode_identity} which implies that
\begin{equs}
\CM \left( \id \otimes \CA \right) \Deltap \CI^{(r,m)}_{o_2}(   \lambda_k^{\ell} F )=  \one \, \one^{\star}\left( \CI^{(r,m)}_{o_2}(   \lambda_k^{\ell} F ) \right).
\end{equs}
As $\one^{*}$ is non-zero only on the empty forest we can thus conclude that
\begin{equs}
\CM \left( \id \otimes \CA \right) \Deltap \CI^{(r,m)}_{o_2}(   \lambda_k^{\ell} F )=  0.\end{equs}
Then, we have by the definition of $\Deltap \CI^{(r,m)}_{o_2}(   \lambda_k^{\ell} F )$ given in \eqref{def_deltas} that
\begin{equs}
\CM \left( \id \otimes \CA \right) \left( \left( \CI^{(r,m)}_{o_2}(  \lambda_{k}^{\ell} \cdot)  \otimes \id  \right) \Delta \CD^{r-\ell-1}(F) +  \one \otimes  \CI^{(r,m)}_{o_2}( \lambda_{k}^{\ell} F) \right)=  0
\end{equs}
which yields \eqref{antipode_rec}.
\end{proof}
\begin{remark} The formula \eqref{antipode_rec} can be rewritten in a non-recursive form. Indeed,  let us introduce the reduced coproduct:
 \begin{equs}
 \tilde{\Delta} F = \Deltap F - F \otimes \one - \one \otimes F.
 \end{equs}
 Then, we can rewrite \eqref{antipode_identity} as follows
\begin{equs} \label{recursive formula}
\CA F = - F - \sum_{(T)} F' \cdot (\CA F'') \quad \tilde \Delta F = \sum_{(T)} F' \otimes F''
\end{equs} 
where we have used Sweedler notations.
\end{remark}

\begin{example}
We compute the antipode on some KdV trees (cf. \eqref{kdvTK} and \eqref{kdvTK2}) using the formula \eqref{recursive formula}:
 \begin{equs}\CA \begin{tikzpicture}[scale=0.2,baseline=-5]
\coordinate (root) at (0,0);
\coordinate (tri) at (0,-2);
\coordinate (t1) at (-1,2);
\coordinate (t2) at (1,2);
\draw[kernels2] (t1) -- (root);
\draw[kernels2] (t2) -- (root);
\draw[symbols] (root) -- (tri);
\node[not] (rootnode) at (root) {};
\node[not,label= {[label distance=-0.2em]below: \scriptsize  $ (r,n) $}] (trinode) at (tri) {};
\node[var] (rootnode) at (t1) {\tiny{$ k_{\tiny{1}} $}};
\node[var] (trinode) at (t2) {\tiny{$ k_2 $}};
\end{tikzpicture}  & = - \begin{tikzpicture}[scale=0.2,baseline=-5]
\coordinate (root) at (0,0);
\coordinate (tri) at (0,-2);
\coordinate (t1) at (-1,2);
\coordinate (t2) at (1,2);
\draw[kernels2] (t1) -- (root);
\draw[kernels2] (t2) -- (root);
\draw[symbols] (root) -- (tri);
\node[not] (rootnode) at (root) {};
\node[not,label= {[label distance=-0.2em]below: \scriptsize  $ (r,n) $}] (trinode) at (tri) {};
\node[var] (rootnode) at (t1) {\tiny{$ k_{\tiny{1}} $}};
\node[var] (trinode) at (t2) {\tiny{$ k_2 $}};
\end{tikzpicture},
 \\
 \CA \begin{tikzpicture}[scale=0.2,baseline=-5]
\coordinate (root) at (0,0);
\coordinate (tri) at (0,-2);
\coordinate (t1) at (-1,2);
\coordinate (t11) at (-2,4);
\coordinate (t12) at (-3,6);
\coordinate (t13) at (-1,6);
\coordinate (t2) at (1,2);
\draw[kernels2] (t11) -- (t13);
\draw[kernels2] (t11) -- (t12);
\draw[kernels2] (t1) -- (root);
\draw[symbols] (t1) -- (t11);
\draw[kernels2] (t2) -- (root);
\draw[symbols] (root) -- (tri);
\node[not] (rootnode) at (root) {};
\node[not,label= {[label distance=-0.2em]below: \scriptsize  $ (r,m) $}] (trinode) at (tri) {};
\node[not] (trinode) at (t1) {};
\node[var] (rootnode) at (t12) {\tiny{$ k_{\tiny{1}} $}};
\node[var] (rootnode) at (t13) {\tiny{$ k_{\tiny{2}} $}};
\node[var] (trinode) at (t2) {\tiny{$ k_3 $}};
\end{tikzpicture}  & = - \begin{tikzpicture}[scale=0.2,baseline=-5]
\coordinate (root) at (0,0);
\coordinate (tri) at (0,-2);
\coordinate (t1) at (-1,2);
\coordinate (t11) at (-2,4);
\coordinate (t12) at (-3,6);
\coordinate (t13) at (-1,6);
\coordinate (t2) at (1,2);
\draw[kernels2] (t11) -- (t13);
\draw[kernels2] (t11) -- (t12);
\draw[kernels2] (t1) -- (root);
\draw[symbols] (t1) -- (t11);
\draw[kernels2] (t2) -- (root);
\draw[symbols] (root) -- (tri);
\node[not] (rootnode) at (root) {};
\node[not,label= {[label distance=-0.2em]below: \scriptsize  $ (r,m) $}] (trinode) at (tri) {};
\node[not] (trinode) at (t1) {};
\node[var] (rootnode) at (t12) {\tiny{$ k_{\tiny{1}} $}};
\node[var] (rootnode) at (t13) {\tiny{$ k_{\tiny{2}} $}};
\node[var] (trinode) at (t2) {\tiny{$ k_3 $}};
\end{tikzpicture}  
 - \sum_{n \leq r} \frac{1}{n!}  \begin{tikzpicture}[scale=0.2,baseline=-5]
\coordinate (root) at (0,0);
\coordinate (tri) at (0,-2);
\coordinate (t1) at (-1,2);
\coordinate (t2) at (1,2);
\draw[kernels2] (t1) -- (root);
\draw[kernels2] (t2) -- (root);
\draw[symbols] (root) -- (tri);
\node[not] (rootnode) at (root) {};
\node[not,label= {[label distance=-0.2em]below: \scriptsize  $ (r,m) $}] (trinode) at (tri) {};
\node[var] (rootnode) at (t1) {\tiny{$ _\ell^n $}};
\node[var] (trinode) at (t2) {\tiny{$ k_3 $}};
\end{tikzpicture}   \CA \begin{tikzpicture}[scale=0.2,baseline=-5]
\coordinate (root) at (0,0);
\coordinate (tri) at (0,-2);
\coordinate (t1) at (-1,2);
\coordinate (t2) at (1,2);
\draw[kernels2] (t1) -- (root);
\draw[kernels2] (t2) -- (root);
\draw[symbols] (root) -- (tri);
\node[not] (rootnode) at (root) {};
\node[not,label= {[label distance=-0.2em]below: \scriptsize  $ (r-1,n) $}] (trinode) at (tri) {};
\node[var] (rootnode) at (t1) {\tiny{$ k_{\tiny{1}} $}};
\node[var] (trinode) at (t2) {\tiny{$ k_2 $}};
\end{tikzpicture} 
 = - \begin{tikzpicture}[scale=0.2,baseline=-5]
\coordinate (root) at (0,0);
\coordinate (tri) at (0,-2);
\coordinate (t1) at (-1,2);
\coordinate (t11) at (-2,4);
\coordinate (t12) at (-3,6);
\coordinate (t13) at (-1,6);
\coordinate (t2) at (1,2);
\draw[kernels2] (t11) -- (t13);
\draw[kernels2] (t11) -- (t12);
\draw[kernels2] (t1) -- (root);
\draw[symbols] (t1) -- (t11);
\draw[kernels2] (t2) -- (root);
\draw[symbols] (root) -- (tri);
\node[not] (rootnode) at (root) {};
\node[not,label= {[label distance=-0.2em]below: \scriptsize  $ (r,m) $}] (trinode) at (tri) {};
\node[not] (trinode) at (t1) {};
\node[var] (rootnode) at (t12) {\tiny{$ k_{\tiny{1}} $}};
\node[var] (rootnode) at (t13) {\tiny{$ k_{\tiny{2}} $}};
\node[var] (trinode) at (t2) {\tiny{$ k_3 $}};
\end{tikzpicture}  
 + \sum_{n \leq r} \frac{1}{n!}  \begin{tikzpicture}[scale=0.2,baseline=-5]
\coordinate (root) at (0,0);
\coordinate (tri) at (0,-2);
\coordinate (t1) at (-1,2);
\coordinate (t2) at (1,2);
\draw[kernels2] (t1) -- (root);
\draw[kernels2] (t2) -- (root);
\draw[symbols] (root) -- (tri);
\node[not] (rootnode) at (root) {};
\node[not,label= {[label distance=-0.2em]below: \scriptsize  $ (r,m) $}] (trinode) at (tri) {};
\node[var] (rootnode) at (t1) {\tiny{$ _\ell^n $}};
\node[var] (trinode) at (t2) {\tiny{$ k_3 $}};
\end{tikzpicture}    \begin{tikzpicture}[scale=0.2,baseline=-5]
\coordinate (root) at (0,0);
\coordinate (tri) at (0,-2);
\coordinate (t1) at (-1,2);
\coordinate (t2) at (1,2);
\draw[kernels2] (t1) -- (root);
\draw[kernels2] (t2) -- (root);
\draw[symbols] (root) -- (tri);
\node[not] (rootnode) at (root) {};
\node[not,label= {[label distance=-0.2em]below: \scriptsize  $ (r-1,n) $}] (trinode) at (tri) {};
\node[var] (rootnode) at (t1) {\tiny{$ k_{\tiny{1}} $}};
\node[var] (trinode) at (t2) {\tiny{$ k_2 $}};
\end{tikzpicture}. 
\end{equs}
\end{example}

\section{Approximating iterated integrals} \label{sec::Iterated integrals}

 In this section, we introduce the main characters that map decorated trees to oscillatory integrals  denoted by the space $ \CC $.
The first character $ \Pi : \hat \CH \rightarrow \CC $ corresponds to the integral in Fourier space stemming form Duhamel's formula.
Then, the second character $ \Pi^n : \CH \rightarrow \CC $ gives an approximation of the first character, in the sense that if $ F \in \hat \CH $, then $ \Pi^n \CD^r(F) $ is a   low regularity approximation of order $ r $ of $ \Pi F $. This is the main result of the section: Theorem~\ref{approxima_tree}.
In order to prove this result, one needs to introduce an intermediate character $ \hat \Pi^n : \CH \rightarrow \CC $ that singles out the dominant oscillations (see Proposition~\ref{factorisation_dom}). The connection between $ \Pi^n $ and $ \hat \Pi^n $ is performed via a Birkhoff factorisation where one uses the coaction $ \Delta $. Such a factorisation seems natural in our context, see Remark~\ref{Birkoffnatural}, and it shows a different application from the existing literature on Birkhoff factorisations. Indeed, the approximation in $ \Pi^n $ is centered around well-chosen Taylor expansions that depends on the frequency interactions. It is central for the local error analysis to understand these interactions. The Birkhoff factorisation allows us to control the contributions of the dominant and lower order parts. We conclude this section by checking that the approximation given by $ \Pi^n $ can be mapped back to the physical space see Proposition~\ref{physical_space} which is an important property for  the practical implementation of the numerical scheme (cf. Remark \ref{rem:FFT}).

\subsection{A recursive formulation}
\label{sec::recursive_pi}

 For the rest of this section, an element of $ \Lab_+ $ (resp, $ \Lab_+ \times \lbrace 0,1\rbrace $) is denoted by $ \Labhom_2 $ (resp. $ o_2 $) and an element of $ \Lab \setminus \Lab_+ $ (resp.  $ \Lab \setminus \Lab_+ \times \lbrace  0,1 \rbrace $) is denoted by $ \Labhom_1 $ (resp. $ o_1 $). 
We denote by $ \CC $  the space of functions of the form $ z \mapsto \sum_j Q_j(z)e^{i z P_j(k_1,...,k_n) } $ where the $ Q_j(z) $ are polynomials in $ z $ and the $ P_j $ are polynomials in $ k_1,...,k_n \in \Z^{d} $. The $ Q_j $ may also depend on  $ k_1,...,k_n $. We use the pointwise product on $ \CC $ for $ G_1(z) =  Q_1(z)e^{i z P_1(k_1,...,k_n) }  $ and  $ G_2(z) =  Q_2(z)e^{i z P_2( k_1,..., k_n) }  $ given by:
 \begin{equs}
 ( G_1  G_2)(z) = Q_1(z)  Q_2(z)  e^{i z  P(k_1,...,k_n)}, \quad P = P_1 + P_2.
 \end{equs}
 We want to define characters on decorated trees using their recursive construction.
A character is a map defined from  $ \hat \CH $ into $ \CC $  which respects the forest product. In the sense, that $ g : \hat \CH \rightarrow \CC  $ is a character if one has:
\begin{equs}
g(F \cdot \bar F) = g(F) g(\bar F), \quad F, \bar{F} \in \hat \CH.
\end{equs}
We define the following character $ \Pi : \hat \CH \rightarrow \mathcal{C} $  by
\begin{equation}\label{Pi}
\begin{aligned}
\Pi \left( F \cdot \bar F \right)(\tau) & =  ( \Pi F)(\tau) ( \Pi \bar F )(\tau),
\\
\Pi \left( \CI_{o_1}( \lambda_k^{\ell} F)\right)(\tau) & = 
 e^{i \tau P_{o_1}(k)} \tau^{\ell} (\Pi F)(\tau),
\\
 \Pi \left(  \CI_{o_2}( \lambda_k^{\ell} F)\right)(\tau) & = -i  \vert \nabla\vert^{\alpha} (k)
\int_{0}^{\tau} e^{i \xi P_{o_2}(k)} \xi^{\ell}(\Pi F)(\xi) d \xi, 
\end{aligned}
\end{equation}
where $ F, \bar F \in \hat \CH $. 

\begin{example}
With the aid of \eqref{Pi}, one can compute recursively the following oscillatory integrals arising in the cubic NLS equation~\eqref{nlsIntro}
\begin{equs}
 (\Pi  \begin{tikzpicture}[scale=0.2,baseline=-5]
\coordinate (root) at (0,1);
\coordinate (tri) at (0,-1);
\draw[kernels2] (tri) -- (root);
\node[var] (rootnode) at (root) {\tiny{$ k_2 $}};
\node[not] (trinode) at (tri) {};
\end{tikzpicture}) (\tau) & = e^{-i \tau k_2^2}, \quad (\Pi  \begin{tikzpicture}[scale=0.2,baseline=-5]
\coordinate (root) at (0,1);
\coordinate (tri) at (0,-1);
\draw[kernels2,tinydots] (tri) -- (root);
\node[var] (rootnode) at (root) {\tiny{$ k_1 $}};
\node[not] (trinode) at (tri) {};
\end{tikzpicture}) (\tau) = e^{i \tau k_1^2}, \quad 
(\Pi \begin{tikzpicture}[scale=0.2,baseline=-5]
\coordinate (root) at (0,-1);
\coordinate (t1) at (-2,1);
\coordinate (t2) at (2,1);
\coordinate (t3) at (0,2);
\draw[kernels2,tinydots] (t1) -- (root);
\draw[kernels2] (t2) -- (root);
\draw[kernels2] (t3) -- (root);
\node[not] (rootnode) at (root) {};t
\node[var] (rootnode) at (t1) {\tiny{$ k_{\tiny{1}} $}};
\node[var] (rootnode) at (t3) {\tiny{$ k_{\tiny{2}} $}};
\node[var] (trinode) at (t2) {\tiny{$ k_3 $}};
\end{tikzpicture}  )(\tau) = e^{i \tau (k_1^2 - k_2^2 - k_3^2)},
\\ ( \Pi \begin{tikzpicture}[scale=0.2,baseline=-5]
\coordinate (root) at (0,0);
\coordinate (tri) at (0,-2);
\coordinate (t1) at (-2,2);
\coordinate (t2) at (2,2);
\coordinate (t3) at (0,3);
\draw[kernels2,tinydots] (t1) -- (root);
\draw[kernels2] (t2) -- (root);
\draw[kernels2] (t3) -- (root);
\draw[symbols] (root) -- (tri);
\node[not] (rootnode) at (root) {};t
\node[not] (trinode) at (tri) {};
\node[var] (rootnode) at (t1) {\tiny{$ k_{\tiny{1}} $}};
\node[var] (rootnode) at (t3) {\tiny{$ k_{\tiny{2}} $}};
\node[var] (trinode) at (t2) {\tiny{$ k_3 $}};
\end{tikzpicture}) (\tau) & = -i \int^{\tau}_0 e^{is (-k_1 + k_2 + k_3)^2} e^{i s (k_1^2 - k_2^2 - k_3^2)} ds.
  \end{equs}
\end{example}

We need a well chosen approximation of order $ r $ for the character $ \Pi $ defined in~\eqref{Pi}, which is suitable in the sense that it embeds those dominant frequencies matching   the regularity $n$ of the solution, see Remark~\ref{rem:regi}. Therefore, we consider a new family of characters defined now on $ \CH $ and parametrised by $ n \in \N $:
\begin{equation} \label{recursive_pi_r}
\begin{aligned}
\Pi^n \left( F \cdot \bar F \right)(\tau) & = 
\left( \Pi^n F  \right)(\tau)  \left( \Pi^n  \bar F \right)(\tau), \quad (\Pi^n  \lambda^{\ell})(\tau) = \tau^{\ell}, \\
(\Pi^n  \CI^{r}_{o_1}( \lambda_{k}^{\ell}  F ))(\tau)   & =\tau^{\ell} e^{i \tau P_{o_1}(k)} (\Pi^n \CD^{r-\ell}(F))(\tau),  \\
\left( \Pi^n \CI^{r}_{o_2}( \lambda^{\ell}_k F) \right) (\tau) & = \CK^{k,r}_{o_2} \left(  \Pi^n \left( \lambda^{\ell} \CD^{r-\ell-1}(F) \right),n \right)(\tau).
\end{aligned}
\end{equation}
All approximations are thereby carried out in the map $  \CK^{k,r}_{o_2} \left( \cdot,n \right) $ which is given in Definition~\ref{Taylor_exp} below.
The main idea behind the map $  \CK^{k,r}_{o_2} \left( \cdot,n \right) $  is that all  integrals are approximated through well-chosen Taylor expansions  depending on the regularity $ n $ of the solution assumed a priori, and the interaction of the frequencies in the decorated trees. For a polynomial $ P(k_1,...,k_n)$, we define the degree of $ P $ denoted by $ \deg(P) $ as the maximum $ m $  such that $ k_i^{m} $ appears as a factor of one monomial in $ P $ for some~$i$.
For example:
\begin{equs}
& P(k_1,k_2,k_3) = - 2 k_1 (k_2+k_3) + 2 k_2 k_3, \quad \deg(P) = 1,\\
& P(k_1,k_2,k_3) = k_1^2- 2 k_1 (k_2+k_3) + 2 k_2 k_3 , \quad \deg(P) = 2.
\end{equs}

\begin{definition} \label{Taylor_exp}
Assume that $ G:  \xi \mapsto \xi^{q} e^{i \xi P(k_1,...,k_n)} $ where $ P $ is a polynomial in the frequencies $ k_1,...,k_n $ and let $ o_2 =  (\Labhom_2,p) \in \Lab_+ \times \lbrace 0,1 \rbrace$ and $ r \in \N $. Let $ k $ be a linear map in $ k_1,...,k_n $ using coefficients in $ \lbrace -1,0,1 \rbrace $ and 
 \[
\begin{aligned}
\CL_{\text{\tiny{dom}}} & = \CP_{\text{\tiny{dom}}} \left(  P_{o_2}(k) + P \right), \quad \CL_{\text{\tiny{low}}}  = \CP_{\text{\tiny{low}}} \left(  P_{o_2}(k) +  P \right)
\\
f(\xi) & =  e^{i \xi \CL_{\text{\tiny{dom}}}}, \quad
g(\xi)  =  e^{i \xi \CL_{\text{\tiny{low}}}}, \quad { \tilde g}(\xi) = e^{i \xi \left(  P_{o_2}(k) +  P \right)} .
\end{aligned}
\]
Then, we define for $ n \in \N $ and $ r \geq q $
\begin{equ}[e:Pim1]
 \CK^{k,r}_{o_2} (  { G},n)(\tau) = \left\{ \begin{aligned}
  &  -i \vert \nabla\vert^\alpha(k) \sum_{\ell \leq r - q} \frac{ { \tilde g}^{(\ell)}(0)}{\ell!} \int_0^{\tau}   \xi^{\ell+q} d \xi, \,
  \text{if }  n \geq \text{deg}\left(\mathcal{L}_{\text{\tiny{dom}}}^{r+1}\right) + \alpha , \\
  & -i\vert \nabla\vert^\alpha(k)  \sum_{\ell \leq r -q} \frac{g^{(\ell)}(0)}{\ell !} \, \Psi^{r}_{n,q}\left( \CL_{\text{\tiny{dom}}} ,\ell\right)(\tau), \quad \text{otherwise}.
  \\
  \end{aligned} \right.
\end{equ}
Thereby we set for  $ \left(r - q - \ell+1 \right) \deg(\CL_{\text{\tiny{dom}}})  + \ell \deg(\CL_{\text{\tiny{low}}}) + \alpha > n  $
\begin{equs}[e:Pim]
\Psi^{r}_{n,q}\left( \CL_{\text{\tiny{dom}}},\ell \right)(\tau) =   \int_0^{\tau}   \xi^{\ell+q} f(\xi)   d \xi.
\end{equs}
Otherwise,
\begin{equ}[e:Pim2]
 \Psi^{r}_{n,q}\left( \CL_{\text{\tiny{dom}}},\ell \right)(\tau) = \sum_{m \leq r - q - \ell} \frac{f^{(m)}(0)}{m!} \int_0^{\tau}   \xi^{\ell+m+q} d \xi.
\end{equ}
Here $ \deg(\CL_{\text{\tiny{dom}}})$ and $\deg(\CL_{\text{\tiny{low}}}) $  denote the degree of the polynomial $ \CL_{\text{\tiny{dom}}} $ and $ \CL_{\text{\tiny{low}}} $, respectively and $\vert \nabla\vert^\alpha(k) = \prod_{\alpha = \sum \gamma_j <  \deg(\CL)} k_j^{\gamma_j}$ (cf. \eqref{Lldef}). If $ r <q $, the map $ \CK^{k,r}_{o_2} (  { G},n)(\tau)$ is equal to zero.
\end{definition}
\begin{remark}[Practical implementation]\label{rem:stab}
In practical computations we need to stabilise the above approach, as the Taylor series expansion of $g$ may introduce derivatives on the numerical solution causing instability of the discretisation. We propose two ways to obtain stabilised high-order resonance based schemes without changing the underlying structure of the local error:
\begin{itemize}
\item  Instead of straightforwardly applying a  Taylor series expansion of $g$  we introduce  a stabilisation in the Taylor series expansion itself  based on  finite difference approximations of type 
$
g'(0) = \frac{g(t)-g(0)}{t}  + \mathcal{O}(t g'').
$ For instance, at second- and third order we will use that
\begin{equation}
\begin{aligned}\label{gExp1}
g(\xi) & = g(0) + \xi \frac{g(t)-g(0)}{t} + \mathcal{O}(t \xi g'')\\
g(\xi) &= g(0) + \xi \frac{g(t)+ g(-t)}{t} +  \frac{\xi^2}{2} \frac{g(t)- 2 g(0) + g(-t)}{t^2} +  \mathcal{O}(t \xi^2 g''').
\end{aligned}
\end{equation}

We refer to \cite{Fornberg} for a simple recursive algorithm calculating the weights in compact finite difference formulas for any order of derivative and to any order of accuracy.
\item We carry out a straightforward Taylor series expansion of $g$, but include suitable filter functions $\Psi$ in the discretisation. At second-order they may for instance take the form\begin{equation}\label{psi}
\Psi = \Psi \left(i \tau \mathcal{L}_\text{low}\right) \quad \text{ with }\quad \Psi(0) = 1 \quad \text{and} \quad \left\Vert \tau  \Psi\left(i \tau \mathcal{L}_\text{low}\right) g'(0)  \right\Vert \leq 1.
\end{equation}
For details on filter functions we refer to  \cite{H2Tri} and the references therein.
\end{itemize}
Practical computations and choices of this stabilisation for concrete examples are detailed in Section \ref{sec:examples}.
\end{remark}

\begin{example}  We consider $P_{\Labhom_2}( \lambda) = -  \lambda^2$, $p = 0$, $ \alpha =0 $,  
$ k = -k_1+k_2+k_3 $ and 
\begin{equs}
{ G}(\xi) = \xi e^{i \xi (  k_1^2 - k_2^2 - k_3^2 )}.
\end{equs}
With the notation of Definition \ref{Taylor_exp} we observe that $q = 1$ and $P(k_1,k_2,k_3) = k_1^2 - k_2^2 - k_3^2$ such that 
\begin{equs}
 P_{o_2}(k) &+ P  = (-1)^{p} P_{\Labhom_2}( (-1)^{p} k) + P
\\
& =  (-k_1+k_2+k_3)^2 + k_1^2 - k_2^2 - k_3^2 = 2k_1^2 - 2 k_1 (k_2+k_3) + 2 k_2 k_3.
\end{equs}
Hence,
\begin{equs}
\CL_{\text{\tiny{dom}}}  
& = 2 k_1^2, \quad \CL_{\text{\tiny{low }}}  =  - 2 k_1 (k_2+k_3) + 2 k_2 k_3,
\end{equs}
cf. also the Schrödinger Example  \ref{exRDoNLS}.
Furthermore, we observe as $\deg(\CL_{\text{\tiny{dom}}})   = 2$, $\deg(\CL_{\text{\tiny{low}}})   = 1$ and $q = 1$ that
\begin{equs}\label{condE}
\left(r - q - \ell+1 \right) \deg(\CL_{\text{\tiny{dom}}})  + \ell \deg(\CL_{\text{\tiny{low}}}) > n \quad \text{ if }\quad 2r - n > \ell.
\end{equs}
In the following we will also exploit that $f(0) = g(0) = 1$.
\begin{itemize}
\item {\bf Case $r = 0$:} As $ r< q = 1$ we have for all $n$  that
\[
\CK^{k,0}_{o_2} (  { G},n)(\tau)  = 0.
\]
\item {\bf Case $r = 1$:}  For $n = 1$ we obtain 
\[
\CK^{k,1}_{o_2} ( { G},n)(\tau)  = -i  \Psi^{1}_{n,1}\left( \CL_{\text{\tiny{dom}}},0 \right) 
= \int_0^\tau \xi f(\xi) d\xi = \frac{1}{2i k_1^2}\left( \tau  e^ {2i \tau k_1^2}-  \frac{ e^{2 i \tau k_1^2}-1}{2 i k_1^2}\right)
\]
as condition \eqref{condE} takes  for $\ell = 0$ the form $2-n> 0$.

 On the other hand, for $n \geq  2$ we have that $n \geq \deg(\CL_{\text{\tiny{dom}}})$ such that 
\[
\CK^{k,1}_{o_2} (  { G},n)(\tau)  = -i  {  \tilde g}(0) \int_0^\tau \xi d\xi =- i  \frac{\tau^2}{2}.
\]

\item \noindent {\bf Case $r = 2$:} If $ n \geq 4$ we obtain
\[
\CK^{k,2}_{\alpha} (  { G},n)(\tau) = - i \left( \frac{\tau^2}{2} + \frac{ {  \tilde g}(\tau) - 1}{\tau} \frac{\tau^3}{2}\right).
\]

Let $n \leq 3$. We  have that
 \begin{equs}
\CK^{k,2}_{\alpha} (  { G},n)(\tau)  & =  -i \left( \Psi^{2}_{n,1}\left( \CL_{\text{\tiny{dom}}},0 \right)  +
\frac{g(\tau)-1}{\tau}  \Psi^{2}_{n,1}\left( \CL_{\text{\tiny{dom}}},1 \right) \right) \\
& =   -i \left(\Psi^{2}_{n,1}\left( \CL_{\text{\tiny{dom}}},0 \right)  +
\frac{g(\tau)-1}{\tau}  \Psi^{2}_{n,1}\left( \CL_{\text{\tiny{dom}}},1 \right)\right)
\end{equs}
and condition \eqref{condE} takes  the form $4-n> \ell$.

If $\ell = 1$ we thus obtain for $n =1,2$ 
\begin{equs}
 \Psi^{2}_{n \leq 2 ,1}\left( \CL_{\text{\tiny{dom}}},1 \right)  = 
 \int_0^\tau \xi^2 f(\xi) d\xi 
 = \frac{ \tau^2 }{2 i k_2^2} \left(e^{2i \tau k_1^2}
 - 2  \Psi^{1}_{1,1}\left( \CL_{\text{\tiny{dom}}},0 \right) 
 \right)
\end{equs}
and
for $n = 3$
\begin{equs}
\Psi^{2}_{n>2,1}\left( \CL_{\text{\tiny{dom}}},1 \right) = f(0)\int_0^\tau \xi^{2}d\xi = \frac{\tau^3}{3}.
\end{equs}

If $\ell = 0$, on the other hand,  condition \eqref{condE} holds for $n = 1,2,3$. 
Henceforth,  we have that
\[
 \Psi^{2}_{n \leq 3,1}\left( \CL_{\text{\tiny{dom}}},0 \right) 
 = \int_0^\tau \xi f(\xi)d\xi =  \Psi^{1}_{1,1}\left( \CL_{\text{\tiny{dom}}},0 \right) .
\]
\end{itemize}
\end{example}

\begin{lemma} \label{Taylor_bound}  We keep the notations of Definition~\ref{Taylor_exp}. We suppose that $   q \leq r$ then one has
\begin{equation}
- i \vert \nabla\vert^{\alpha} (k) \int_{0}^{\tau} \xi^{q} { e^{i \xi \left (\CL_{\text{\tiny{dom}}} + \CL_{\text{\tiny{low}}}\right)}}  d\xi -\CK^{k,r}_{o_2} (  { G},n)(\tau) = \CO(\tau^{r+2} k^{\bar n})
\end{equation} 
where $ \bar n = \max(n, \deg(\CL_{\text{\tiny{low}}}^{r-q +1}) + \alpha) $.
\end{lemma}
\begin{proof}

Recall the notation of Definition~\ref{Taylor_exp} which implies that
$$
 \int_{0}^{\tau} \xi^{q} { e^{i \xi \left (\CL_{\text{\tiny{dom}}} + \CL_{\text{\tiny{low}}}\right)}}  d\xi 
 =  \int_{0}^{\tau}  \xi^{q} f(\xi) g(\xi) d\xi 
$$
with $f(\xi)  =  e^{i \xi \CL_{\text{\tiny{dom}}}}$ and $
g(\xi)  =  e^{i \xi \CL_{\text{\tiny{low}}}}$.
 It is just a consequence of Taylor expanding the functions $ g,  {  \tilde g}$ and $ f $. If $  n \geq \text{deg}\left(\mathcal{L}_{\text{\tiny{dom}}}^{r}\right) + \alpha $ we have 
\begin{equs}
-i \vert \nabla\vert^{\alpha} (k) & \int_{0}^{\tau}  \xi^{q} f(\xi) g(\xi) d\xi +i \vert \nabla\vert^{\alpha} (k) \sum_{\ell \leq r - q} \frac{ {  \tilde g}^{(\ell)}(0)}{\ell!} \int_0^{\tau}   \xi^{\ell+q} d \xi \\ &  =
\CO( \tau^{r+2} \vert \nabla\vert^{\alpha} (k) \mathcal{L}_{\text{\tiny{dom}}}^{r+1}) 
\\ & = \CO(\tau^{r+2} k^{\bar n}).
\end{equs} 
Else, we get 
\begin{equs}
- i \vert \nabla\vert^{\alpha} (k) & \int_{0}^{\tau}  \xi^{q} f(\xi) g(\xi) d\xi + i \vert \nabla\vert^{\alpha} (k) \sum_{\ell \leq r - q} \frac{g^{(\ell)}(0)}{\ell!} \int_0^{\tau}   \xi^{\ell+q} f(\xi)  d \xi   \\ & = \CO(\tau^{r+2} \vert \nabla\vert^{\alpha} (k) g^{(r-q+1)})
 \\& = \CO(\tau^{r+2} \vert \nabla\vert^{\alpha} (k) \CL_{\text{\tiny{low}}}^{r-q+1})
\end{equs}
where the latter follows from the observation that  $ g^{(\ell)}(\xi) = (i \CL_{\text{\tiny{low}}})^{\ell} e^{i \xi \CL_{\text{\tiny{low}}}} $.

If  on the other hand $  \left(r- q - \ell +1 \right) \deg(\CL_{\text{\tiny{dom}}})  + \ell \deg(\CL_{\text{\tiny{low}}}) + \alpha \leq n $ then
\begin{equs}
& - i \vert \nabla\vert^{\alpha} (k) \frac{g^{(\ell)}(0)}{\ell!} \int_0^{\tau}   \xi^{\ell+q} f(\xi)  d \xi +i \vert \nabla\vert^{\alpha} (k)\frac{g^{(\ell)}(0)}{\ell!} \sum_{m \leq r - q - \ell} \frac{f^{(m)}(0)}{m!}  \int_0^{\tau}    \xi^{\ell+m+q}  d \xi 
\\ &   = \frac{g^{(\ell)}(0)}{\ell!} \int_0^{\tau}   \xi^{\ell+q} \CO \left(\xi^{r-q-\ell+1} \vert \nabla\vert^{\alpha} (k)\CL_{\text{\tiny{dom}}}^{r-q-\ell+1} \right)  d \xi
\\ & = \CO(\tau^{r+2} \vert \nabla\vert^{\alpha} (k) \CL_{\text{\tiny{low}}}^{\ell} \CL_{\text{\tiny{dom}}}^{r-q-\ell+1} )
\\ & = \CO(\tau^{r+2} k^n )
\end{equs}
which allows us to conclude.
\end{proof}

\begin{remark} In the proof of Lemma~\ref{Taylor_bound}, one has 
\begin{equs}
\deg \left( \CL_{\text{\tiny{low}}}^{\ell} \CL_{\text{\tiny{dom}}}^{r-q-\ell+1} \right) \geq\deg \left( \CL_{\text{\tiny{low}}}^{r-q+1} \right). 
 \end{equs}
 If $ n = \deg \left( \CL_{\text{\tiny{low}}}^{r-q+1} \right) + \alpha $ we cannot carry out  a Taylor series expansion of $ f $ and we have to perform the integration exactly. This will give a more complicate numerical scheme. If on the other hand,  $ n $ is larger, part of the Taylor expansions of $ f $ will be possible. In fact, $ n $ corresponds to the regularity of the solution we assume a priori, see Remark \ref{rem:regi}.
\end{remark}

\begin{remark} \label{better error}
In Lemma~\ref{Taylor_bound} we express the  approximation error in terms of  powers of $ k $. This will be enough for conducting the local error analysis for the general scheme. One can be more precise and keep the full structure by replacing $ k $ by monomials in $ \CL_{\text{\tiny{dom}}} $ and $ \CL_{\text{\tiny{low}}} $. This could be certainly useful when one wants to perform the global error analysis and needs to keep track of the full error structure.
\end{remark}

\subsection{A Birkhoff type factorisation}
\label{sec::Brikhoff}
The character $ \Pi^n : \CH \rightarrow \CC  $ is quite complex since one needs to  compute  several nonlinear interactions (oscillations) at the same time.
Indeed, most of the time the operator $  \CK^{k,r}_{o_2} \left( \cdot,n \right) $ is applied to a linear combination of monomials of the form $ e^{i \xi P_j(k)} $. We want to single out every oscillation through a factorisation of this character. 
We start by introducing a splitting with a projection $ \CQ $:
\begin{equs}
\CC = \CC_- \oplus \CC_+, \quad \CQ : \CC \rightarrow \CC_-
\end{equs}
where $ \CC_- $ is the space of polynomials $ Q(\xi) $ and $ \CC_+ $ is the subspace of  functions of the form $ z \mapsto \sum_j Q_j(z)e^{i z P_j(k_1,...,k_n) } $ with $ P_j \neq 0 $. 

\begin{remark} \label{Birkoffnatural}  In the classical Birkhoff factorisation for  Laurent series, we consider $A = \C[[t,t^{-1}]$, with finite pole-part. In this context, the splitting reads $ A = A_- \oplus A_+  $ where $A_-  =  t^{-1}\C[t^{-1}]$ and $A_+ =  \C[[t]]$, such that $ \CQ $ keeps only the pole part of a series: 
\begin{equs}
	\CQ\Big( \sum_{n} a_n t^{n} \Big)  = \sum_{ n< 0} a_n t^{n} \in A_-.
\end{equs}
The idea is to remove the divergent part of the series, i.e.,  its pole part. In our context the structure of the factorisation is quite different, as we are interested in singling out the oscillations. Let suppose that we start with
a term of the form $ z \mapsto e^{i z P(k_1,...,k_n) }$, where $P$ corresponds to the dominant part of some  differential operator. Then, the integral over this term  yields two contributions:
\begin{align*}
\int_{0}^{t} e^{i \xi P(k_1,...,k_n)} d\xi = 
\frac{ e^{i t  P(k_1,...,k_n)} - 1}{i \small{P(k_1,...,k_n)}}.
\end{align*}
One is the oscillation $e^{i z P(k_1,...,k_n) }$ we start with evaluated at time $z = t$ and the other one is the constant term $ -1 $. These terms are obtained by applying recursively the projection $ \CQ $. This approach seems new from the literature and it is quite different in spirit from what has been observed for singular SPDEs.
\end{remark}

We set 
\begin{equs}\label{K}
\CK^{k,r}_{o_2,-}:= \CQ \circ \CK^{k,r}_{o_2} , \quad \CK^{k,r}_{o_2,+}:=\left( \id - \CQ \right) \circ \CK^{k,r}_{o_2} .
\end{equs}
One has 
\[
\CK^{k,r}_{o_2} = \CK^{k,r}_{o_2,-} + \CK^{k,r}_{o_2,+}.
\]
We define a character $ A^n : \CH_+ \rightarrow \C  $ by
\begin{equation}
\begin{aligned}\label{An}
 A^n( \CI^{(r,m)}_{o_2}(  \lambda_{k}^{\ell} F)) & = \left( \CQ \circ  \partial^{m} \Pi^{n} \CI^{r}_{o_2}(  \lambda_{k}^{\ell} F) \right)(0).
\end{aligned}
\end{equation}
 where  $\partial^{m} \Pi^{n} \CI^{r}_{o_2}(  \lambda_{k}^{\ell} F)$ is the $ m $th derivative of the function $ t \mapsto  (\Pi^{n} \CI^{r}_{o_2}(  \lambda_{k}^{\ell} F))(t) $.
The character $ A^n $ applied to $ \CI^{(r,m)}_{o_2}(  \lambda_{k}^{\ell} F) $ is extracting the coefficient of $ \tau^{m} $ multiplied by $ m! $  in $ \Pi^n \CI^{r}_{o_2}(  \lambda_{k}^{\ell} F) $. If we extend $ \Pi^n $ to $ \CH_+ $ by setting
\begin{equs}
 \Pi^n( \CI^{(r,m)}_{o_2}(  \lambda_{k}^{\ell} F))  =
  \partial^{m} \Pi^{n} \CI^{r}_{o_2}(  \lambda_{k}^{\ell} F)
\end{equs} 
then we have a new expression for $ A^n $
\begin{equs} \label{nice expressiion An}
A^n = \left( \CQ \circ \Pi^n \cdot \right)(0).
\end{equs}

We define a character $ \hat \Pi^n : \CH \rightarrow \CC $ which computes only one interaction by applying repeatedly the projection $ \id - \CQ $
\begin{equation} \label{def_A_B}
\begin{aligned}
\hat{\Pi}^n \left( F \cdot \bar F \right)(\tau) & = 
\left( \hat{\Pi}^n F  \right)(\tau)  \left( \hat{\Pi}^n  \bar F \right)(\tau), \quad (\hat{\Pi}^n  \lambda^{\ell})(\tau) = \tau^{\ell}, \\
\hat{\Pi}^n(\CI^{r}_{o_1}( \lambda_{k}^{\ell}  F ))(\tau)  & = \tau^{\ell }e^{i \tau P_{o_1}(k)} \hat{\Pi}^n (\CD^{r-\ell}(F))(\tau),  \\
 \hat \Pi^n(\CI^{r}_{o_2}( \lambda_{k}^{\ell}  F ))   & = \CK_{o_2,+}^{k,r}( \hat \Pi^n(  \lambda^{\ell} \CD^{r-\ell-1}(F)),n   )
  \end{aligned}
\end{equation} 
{for $ F, \bar{F} \in \CH $. One can notice that $ \hat \Pi^n $ takes value in $ \CC_+ $ on elements of the form $\CI^{r}_{o_2}( \lambda_{k}^{\ell}  F ) $.}
The character $ \hat{\Pi}^n $ is central for deriving the local error analysis for the numerical scheme. It has nice properties outlined in Proposition~\ref{factorisation_dom} which are crucial for proving Theorem~\ref{approxima_tree}. We provide an identity for the approximation $ \Pi^n $:
\begin{equation} \label{antipode_def}
\begin{aligned}
\Pi^n = \left( \hat  \Pi^n \otimes A^n \right) \Delta.
\end{aligned} 
\end{equation}
 In the identity above, we do not need a multiplication because $ A^n(F) \in \C $ for every $ F \in \CH_+ $ and $ \CC$ is a $ \C $-vector space. Therefore, we use the identification $ \mathcal{C} \otimes \C \cong \mathcal{C} $.

\begin{proposition} \label{Prop 3.7}
The two definitions \eqref{recursive_pi_r} and \eqref{antipode_def} coincide.
\end{proposition}
\begin{proof} We prove this identity by induction on the number of edges of a forest.
We first consider a tree of the form $ \CI^{r}_{o_1}(  \lambda_k^{\ell} F)  $ then we get:
\[
\begin{aligned}
 \left( \hat \Pi^n(\cdot)(\tau) \otimes A^n  \right)
\Delta \CI^{r}_{o_1}(   \lambda_k^{\ell} F )
 & = \left(\hat \Pi^n \left( \CI^{r}_{o_1}(  \lambda_k^{\ell} \cdot  )  \right)(\tau) \otimes A^n \right) \Delta \CD^{r - \ell}(F) 
\\ & = \tau^{\ell}  e^{i \tau P_{o_1}(k)} \left( \hat \Pi^n \left( \cdot \right)(\tau) \otimes A^n  \right) \Delta \CD^{r - \ell}(F)
\\ & = \tau^{\ell} e^{i \tau P_{o_1}(k)} (\Pi^n \CD^{r - \ell}(F))(\tau)
\\ & = \left( \Pi^n  \CI^{r}_{o_1}(   \lambda_k^{\ell} F ) \right)(\tau),
\end{aligned}
\]
where we have used our inductive hypothesis.
We look now at a tree of the form  $   \CI^{r}_{o_2}( \lambda_{k}^{\ell} F)  $:
\[
\begin{aligned}
& \left( \hat  \Pi^n \otimes A^n  \right)
\Delta \CI^{r}_{o_2}( \lambda_{k}^{\ell} F)   = 
\left( \hat \Pi^n(\CI^{r}_{o_2}(   \lambda_k^{\ell} \cdot)) \otimes A^n \right)
\Delta \CD^{r-\ell-1}(F) \\&  + \sum_{m \leq r + 1} \frac{1}{m!}\hat{\Pi}^n( \lambda^{m}) A^n( \CI^{(r,m)}_{o_2}( \lambda_{k}^{\ell} F) ) 
\\ & = \CK_{o_2,+}^{r,k}
\left( \hat \Pi^n( \lambda^{\ell} \cdot) \otimes A^n, n  \right)
\Delta \CD^{r-\ell-1}(F)  + \CK_{o_2,-}^{r,k}(\Pi^n  \lambda^{\ell} \CD^{r -\ell-1}(F),n)
\\ 
& = \CK_{o_2,+}^{r,k}(\Pi^n  \lambda^{\ell} \CD^{r-\ell-1}(F),n) +  \CK_{o_2,-}^{r,k}(\Pi^n  \lambda^{\ell} \CD^{r-\ell-1}(F),n)
\\ & =  \Pi^n \left(\CI^{r}_{o_2}( \lambda^{\ell}_{k} F) \right)
\end{aligned}
\]
where we used the following identification
\begin{equs}
(\hat \Pi^n \otimes A^n) (F_1 \otimes F_2) = ( \hat \Pi^n F_1 \otimes A^n F_2 ) = \hat \Pi^n(F_1) A^n(F_2)
\end{equs}
and that
\begin{equs}
& \sum_{m \leq r+1}  \frac{1}{m!}\hat{\Pi}^n( \lambda^{m}) \, A^n( \CI^{(r,m)}_{o_2}( \lambda_{k}^{\ell} F) ) \\ & = \sum_{m \leq r + 1} \frac{1}{m!}\hat{\Pi}^n( \lambda^{m}) \,\left( \CQ \circ  \partial^{m} \Pi^n \CI^{r}_{o_2}(  \lambda_{k}^{\ell} F) \right)(0) 
 \\ & =  \CK_{o_2,-}^{r,k}(\Pi^n  \lambda^{\ell} \CD^{r - 1-\ell}(F),n).
\end{equs}
\end{proof}
 The interest of the decomposition given
by Proposition~\ref{Prop 3.7} comes from Proposition~\ref{factorisation_dom}: $\hat{\Pi}^n$ only involves one oscillation.

\begin{example} \label{AnBn}
We compute $ A^n $ and $ \hat \Pi^n $ for the following decorated trees that appear  for the  cubic NLS equation, see also  \eqref{nlsTK}:
\begin{equs}
  T_1  = \begin{tikzpicture}[scale=0.2,baseline=-5]
\coordinate (root) at (0,0);
\coordinate (tri) at (0,-2);
\coordinate (t1) at (-2,2);
\coordinate (t2) at (2,2);
\coordinate (t3) at (0,3);
\draw[kernels2,tinydots] (t1) -- (root);
\draw[kernels2] (t2) -- (root);
\draw[kernels2] (t3) -- (root);
\draw[symbols] (root) -- (tri);
\node[not] (rootnode) at (root) {};t
\node[not] (trinode) at (tri) {};
\node[var] (rootnode) at (t1) {\tiny{$ k_{\tiny{1}} $}};
\node[var] (rootnode) at (t3) {\tiny{$ k_{\tiny{2}} $}};
\node[var] (trinode) at (t2) {\tiny{$ k_3 $}};
\end{tikzpicture}, \quad  T_2  = \begin{tikzpicture}[scale=0.2,baseline=-5]
\coordinate (root) at (0,0);
\coordinate (tri) at (0,-2);
\coordinate (t1) at (-2,2);
\coordinate (t2) at (2,2);
\coordinate (t3) at (0,3);
\draw[kernels2,tinydots] (t1) -- (root);
\draw[kernels2] (t2) -- (root);
\draw[kernels2] (t3) -- (root);
\draw[symbols] (root) -- (tri);
\node[not] (rootnode) at (root) {};t
\node[not] (trinode) at (tri) {};
\node[var] (rootnode) at (t1) {\tiny{$ k_{\tiny{4}} $}};
\node[var] (rootnode) at (t3) {\tiny{$ \ell $}};
\node[var] (trinode) at (t2) {\tiny{$ k_5 $}};
\end{tikzpicture}, \quad T_3  = \begin{tikzpicture}[scale=0.2,baseline=-5]
\coordinate (root) at (0,0);
\coordinate (tri) at (0,-2);
\coordinate (t1) at (-2,2);
\coordinate (t2) at (2,2);
\coordinate (t3) at (0,2);
\coordinate (t4) at (0,4);
\coordinate (t41) at (-2,6);
\coordinate (t42) at (2,6);
\coordinate (t43) at (0,8);
\draw[kernels2,tinydots] (t1) -- (root);
\draw[kernels2] (t2) -- (root);
\draw[kernels2] (t3) -- (root);
\draw[symbols] (root) -- (tri);
\draw[symbols] (t3) -- (t4);
\draw[kernels2,tinydots] (t4) -- (t41);
\draw[kernels2] (t4) -- (t42);
\draw[kernels2] (t4) -- (t43);
\node[not] (rootnode) at (root) {};
\node[not] (rootnode) at (t4) {};
\node[not] (rootnode) at (t3) {};
\node[not] (trinode) at (tri) {};
\node[var] (rootnode) at (t1) {\tiny{$ k_{\tiny{4}} $}};
\node[var] (rootnode) at (t41) {\tiny{$ k_{\tiny{1}} $}};
\node[var] (rootnode) at (t42) {\tiny{$ k_{\tiny{3}} $}};
\node[var] (rootnode) at (t43) {\tiny{$ k_{\tiny{2}} $}};
\node[var] (trinode) at (t2) {\tiny{$ k_5 $}};
\end{tikzpicture}
\end{equs}
where $ \ell = -k_1 + k_2 + k_3 $. Then, when $ r = 0  $ and $ n < 2 $, one has from \eqref{Psi0}
\begin{equs}
 (\Pi^n \CD^r(T_1))(\tau) = -i \frac{e^{2 i \tau k_1^2}-1}{2 i k_1^2}.
\end{equs}
On the other hand, 
\begin{equs}
 \Delta \CD^{0}(T_1) = \CD^{0}(T_1) \otimes \one + \one \otimes \hat \CD^{(0,0)}(T_1)
 \end{equs}
and
 \begin{equs} \hat \Pi^{n}(\CD^{0}(T_1)) = -\frac{e^{2 i \tau k_1^2}}{2  k_1^2}, \quad A^n(\hat \CD^{(0,0)}(T_1)) =  \frac{1}{2  k_1^2}.
\end{equs}
When $ n=2 $, one gets
\begin{equs}
(\Pi^n \CD^0(T_1))(\tau) = \tau,  \quad \hat \Pi^{n}(\CD^{0}(T_1)) = 0, \quad A^n(\hat \CD^{(0,0)}(T_1)) =  1.
\end{equs}
Now, we consider the tree $ T_3 $ and we assume that $ r=1 $ and $ n =2 $. We calculate that (for details see  \eqref{computescheme2} in Section \ref{sec:examples}) 
\begin{equs}
 (\Pi^n \CD^1(T_3))(\tau) =  - \frac{\tau^2}{2}.
\end{equs}
Then, 
\begin{equs}
\Delta & \begin{tikzpicture}[scale=0.2,baseline=-5]
\coordinate (root) at (0,0);
\coordinate (tri) at (0,-2);
\coordinate (t1) at (-2,2);
\coordinate (t2) at (2,2);
\coordinate (t3) at (0,2);
\coordinate (t4) at (0,4);
\coordinate (t41) at (-2,6);
\coordinate (t42) at (2,6);
\coordinate (t43) at (0,8);
\draw[kernels2,tinydots] (t1) -- (root);
\draw[kernels2] (t2) -- (root);
\draw[kernels2] (t3) -- (root);
\draw[symbols] (root) -- (tri);
\draw[symbols] (t3) -- (t4);
\draw[kernels2,tinydots] (t4) -- (t41);
\draw[kernels2] (t4) -- (t42);
\draw[kernels2] (t4) -- (t43);
\node[not] (rootnode) at (root) {};
\node[not] (rootnode) at (t4) {};
\node[not] (rootnode) at (t3) {};
\node[not,label= {[label distance=-0.2em]below: \scriptsize  $ 1 $}] (trinode) at (tri) {};
\node[var] (rootnode) at (t1) {\tiny{$ k_{\tiny{4}} $}};
\node[var] (rootnode) at (t41) {\tiny{$ k_{\tiny{1}} $}};
\node[var] (rootnode) at (t42) {\tiny{$ k_{\tiny{3}} $}};
\node[var] (rootnode) at (t43) {\tiny{$ k_{\tiny{2}} $}};
\node[var] (trinode) at (t2) {\tiny{$ k_5 $}};
\end{tikzpicture} =
\begin{tikzpicture}[scale=0.2,baseline=-5]
\coordinate (root) at (0,0);
\coordinate (tri) at (0,-2) ;
\coordinate (t1) at (-2,2);
\coordinate (t2) at (2,2);
\coordinate (t3) at (0,2);
\coordinate (t4) at (0,4);
\coordinate (t41) at (-2,6);
\coordinate (t42) at (2,6);
\coordinate (t43) at (0,8);
\draw[kernels2,tinydots] (t1) -- (root);
\draw[kernels2] (t2) -- (root);
\draw[kernels2] (t3) -- (root);
\draw[symbols] (root) -- (tri);
\draw[symbols] (t3) -- (t4);
\draw[kernels2,tinydots] (t4) -- (t41);
\draw[kernels2] (t4) -- (t42);
\draw[kernels2] (t4) -- (t43);
\node[not] (rootnode) at (root) {};
\node[not] (rootnode) at (t4) {};
\node[not] (rootnode) at (t3) {};
\node[not,label= {[label distance=-0.2em]below: \scriptsize  $ 1 $} ] (trinode) at (tri) {};
\node[var] (rootnode) at (t1) {\tiny{$ k_{\tiny{4}} $}};
\node[var] (rootnode) at (t41) {\tiny{$ k_{\tiny{1}} $}};
\node[var] (rootnode) at (t42) {\tiny{$ k_{\tiny{3}} $}};
\node[var] (rootnode) at (t43) {\tiny{$ k_{\tiny{2}} $}};
\node[var] (trinode) at (t2) {\tiny{$ k_5 $}};
\end{tikzpicture} \otimes \one 
+ \one \otimes \begin{tikzpicture}[scale=0.2,baseline=-5]
\coordinate (root) at (0,0);
\coordinate (tri) at (0,-2);
\coordinate (t1) at (-2,2);
\coordinate (t2) at (2,2);
\coordinate (t3) at (0,2);
\coordinate (t4) at (0,4);
\coordinate (t41) at (-2,6);
\coordinate (t42) at (2,6);
\coordinate (t43) at (0,8);
\draw[kernels2,tinydots] (t1) -- (root);
\draw[kernels2] (t2) -- (root);
\draw[kernels2] (t3) -- (root);
\draw[symbols] (root) -- (tri);
\draw[symbols] (t3) -- (t4);
\draw[kernels2,tinydots] (t4) -- (t41);
\draw[kernels2] (t4) -- (t42);
\draw[kernels2] (t4) -- (t43);
\node[not] (rootnode) at (root) {};
\node[not] (rootnode) at (t4) {};
\node[not] (rootnode) at (t3) {};
\node[not,label= {[label distance=-0.2em]below: \scriptsize  $ (1,0) $}] (trinode) at (tri) {};
\node[var] (rootnode) at (t1) {\tiny{$ k_{\tiny{4}} $}};
\node[var] (rootnode) at (t41) {\tiny{$ k_{\tiny{1}} $}};
\node[var] (rootnode) at (t42) {\tiny{$ k_{\tiny{3}} $}};
\node[var] (rootnode) at (t43) {\tiny{$ k_{\tiny{2}} $}};
\node[var] (trinode) at (t2) {\tiny{$ k_5 $}};
\end{tikzpicture} +  \lambda \otimes \begin{tikzpicture}[scale=0.2,baseline=-5]
\coordinate (root) at (0,0);
\coordinate (tri) at (0,-2);
\coordinate (t1) at (-2,2);
\coordinate (t2) at (2,2);
\coordinate (t3) at (0,2);
\coordinate (t4) at (0,4);
\coordinate (t41) at (-2,6);
\coordinate (t42) at (2,6);
\coordinate (t43) at (0,8);
\draw[kernels2,tinydots] (t1) -- (root);
\draw[kernels2] (t2) -- (root);
\draw[kernels2] (t3) -- (root);
\draw[symbols] (root) -- (tri);
\draw[symbols] (t3) -- (t4);
\draw[kernels2,tinydots] (t4) -- (t41);
\draw[kernels2] (t4) -- (t42);
\draw[kernels2] (t4) -- (t43);
\node[not] (rootnode) at (root) {};
\node[not] (rootnode) at (t4) {};
\node[not] (rootnode) at (t3) {};
\node[not,label= {[label distance=-0.2em]below: \scriptsize  $ (1,1) $}] (trinode) at (tri) {};
\node[var] (rootnode) at (t1) {\tiny{$ k_{\tiny{4}} $}};
\node[var] (rootnode) at (t41) {\tiny{$ k_{\tiny{1}} $}};
\node[var] (rootnode) at (t42) {\tiny{$ k_{\tiny{3}} $}};
\node[var] (rootnode) at (t43) {\tiny{$ k_{\tiny{2}} $}};
\node[var] (trinode) at (t2) {\tiny{$ k_5 $}};
\end{tikzpicture}  + \frac{ \lambda^2}{2} \otimes \begin{tikzpicture}[scale=0.2,baseline=-5]
\coordinate (root) at (0,0);
\coordinate (tri) at (0,-2);
\coordinate (t1) at (-2,2);
\coordinate (t2) at (2,2);
\coordinate (t3) at (0,2);
\coordinate (t4) at (0,4);
\coordinate (t41) at (-2,6);
\coordinate (t42) at (2,6);
\coordinate (t43) at (0,8);
\draw[kernels2,tinydots] (t1) -- (root);
\draw[kernels2] (t2) -- (root);
\draw[kernels2] (t3) -- (root);
\draw[symbols] (root) -- (tri);
\draw[symbols] (t3) -- (t4);
\draw[kernels2,tinydots] (t4) -- (t41);
\draw[kernels2] (t4) -- (t42);
\draw[kernels2] (t4) -- (t43);
\node[not] (rootnode) at (root) {};
\node[not] (rootnode) at (t4) {};
\node[not] (rootnode) at (t3) {};
\node[not,label= {[label distance=-0.2em]below: \scriptsize  $ (1,2) $}] (trinode) at (tri) {};
\node[var] (rootnode) at (t1) {\tiny{$ k_{\tiny{4}} $}};
\node[var] (rootnode) at (t41) {\tiny{$ k_{\tiny{1}} $}};
\node[var] (rootnode) at (t42) {\tiny{$ k_{\tiny{3}} $}};
\node[var] (rootnode) at (t43) {\tiny{$ k_{\tiny{2}} $}};
\node[var] (trinode) at (t2) {\tiny{$ k_5 $}};
\end{tikzpicture}
\\ & +
\begin{tikzpicture}[scale=0.2,baseline=-5]
\coordinate (root) at (0,0);
\coordinate (tri) at (0,-2);
\coordinate (t1) at (-2,2);
\coordinate (t2) at (2,2);
\coordinate (t3) at (0,3);
\draw[kernels2,tinydots] (t1) -- (root);
\draw[kernels2] (t2) -- (root);
\draw[kernels2] (t3) -- (root);
\draw[symbols] (root) -- (tri);
\node[not] (rootnode) at (root) {};t
\node[not,label= {[label distance=-0.2em]below: \scriptsize  $ 1 $}] (trinode) at (tri) {};
\node[var] (rootnode) at (t1) {\tiny{$ k_{\tiny{4}} $}};
\node[var] (rootnode) at (t3) {\tiny{$ \ell $}};
\node[var] (trinode) at (t2) {\tiny{$ k_5 $}};
\end{tikzpicture} \otimes \begin{tikzpicture}[scale=0.2,baseline=-5]
\coordinate (root) at (0,0);
\coordinate (tri) at (0,-2);
\coordinate (t1) at (-2,2);
\coordinate (t2) at (2,2);
\coordinate (t3) at (0,3);
\draw[kernels2,tinydots] (t1) -- (root);
\draw[kernels2] (t2) -- (root);
\draw[kernels2] (t3) -- (root);
\draw[symbols] (root) -- (tri);
\node[not] (rootnode) at (root) {};t
\node[not,label= {[label distance=-0.2em]below: \scriptsize  $ (0,0) $}] (trinode) at (tri) {};
\node[var] (rootnode) at (t1) {\tiny{$ k_{\tiny{1}} $}};
\node[var] (rootnode) at (t3) {\tiny{$ k_{\tiny{2}} $}};
\node[var] (trinode) at (t2) {\tiny{$ k_3 $}};
\end{tikzpicture}  +  \begin{tikzpicture}[scale=0.2,baseline=-5]
\coordinate (root) at (0,0);
\coordinate (tri) at (0,-2);
\coordinate (t1) at (-2,2);
\coordinate (t2) at (2,2);
\coordinate (t3) at (0,3);
\draw[kernels2,tinydots] (t1) -- (root);
\draw[kernels2] (t2) -- (root);
\draw[kernels2] (t3) -- (root);
\draw[symbols] (root) -- (tri);
\node[not] (rootnode) at (root) {};t
\node[not,label= {[label distance=-0.2em]below: \scriptsize  $ 1 $}] (trinode) at (tri) {};
\node[var] (rootnode) at (t1) {\tiny{$ k_{\tiny{4}} $}};
\node[var] (rootnode) at (t3) {\tiny{$ ^1_\ell $}};
\node[var] (trinode) at (t2) {\tiny{$ k_5 $}};
\end{tikzpicture} \otimes \begin{tikzpicture}[scale=0.2,baseline=-5]
\coordinate (root) at (0,0);
\coordinate (tri) at (0,-2);
\coordinate (t1) at (-2,2);
\coordinate (t2) at (2,2);
\coordinate (t3) at (0,3);
\draw[kernels2,tinydots] (t1) -- (root);
\draw[kernels2] (t2) -- (root);
\draw[kernels2] (t3) -- (root);
\draw[symbols] (root) -- (tri);
\node[not] (rootnode) at (root) {};t
\node[not,label= {[label distance=-0.2em]below: \scriptsize  $ (0,1) $}] (trinode) at (tri) {};
\node[var] (rootnode) at (t1) {\tiny{$ k_{\tiny{1}} $}};
\node[var] (rootnode) at (t3) {\tiny{$ k_{\tiny{2}} $}};
\node[var] (trinode) at (t2) {\tiny{$ k_3 $}};
\end{tikzpicture}.
\end{equs}
 When, one applies $ (\hat \Pi^n \otimes A^n) $, the only non-zero contribution is given by the third term from the previous computation:
\begin{equs}
 (\hat \Pi^n  \lambda^2)(\tau) = \tau^2, \quad A^n ( \hat \CD^{(1,2)}(T_3)) = -1.
\end{equs}
In order to see a non-trivial interaction in the previous term, one has to consider a higher order approximation such as for $ r = 2  $ and $ n=3 $. In this case, one has
\begin{equs}\label{KD}
(\Pi^{n} \CD^{1}(T_1) )(\tau) & = - i \int_{0}^{\tau} e^{2i s k_1^2}    ds + \frac{\tau^2}{2} \mathscr{F}_{\text{\tiny{low}}} \left( T_1 \right) 
\\ & = -  \frac{e^{2i \tau k_1^2} }{2  k_1^2} + \frac{1}{2  k_1^2}  + \frac{\tau^2}{2} \mathscr{F}_{\text{\tiny{low}}} \left( T_1 \right) .
\end{equs}
Then, 
\begin{equs}
 A^n(\hat \CD^{(1,0)}(T_1)) =  \frac{1}{2k_1^2}, \quad A^n(\hat \CD^{(1,1)}(T_1)) = 0, \quad A^n(\hat \CD^{(1,2)}(T_1)) =  \mathscr{F}_{\text{\tiny{low}}} \left( T_1 \right) .
\end{equs}
Next one has to approximate 
\begin{equs} 
 - i \int_0^\tau  \mathrm{e}^{i s k^2} \Big( \mathrm{e}^{i s(   k_4^2 - k_5^2 - k_{123}^2)} (\Pi^n \CD^1(T_1))(s)  \Big)  d s
\end{equs}
where $ k = -k_1 + k_2 + k_3 - k_4 + k_5 $ and $ k_{123} = -k_1 + k_2 + k_3 $.   Thanks to the structure of $(\Pi^n \CD^1(T_1))(s)$, see \eqref{KD}, it remains to control the following two  oscillations 
\begin{equs}
k^2  + k_4^2 - k_5^2 - k_{123}^2 & =\mathscr{F}_{\text{\tiny{dom}}}\left( T_2 \right) +  \mathscr{F}_{\text{\tiny{low}}} \left( T_2 \right) \\
k^2  + k_4^2 - k_5^2 - k_{123}^2 + 2k_1^2 & = \mathscr{F}_{\text{\tiny{dom}}}\left( T_3 \right) +  \mathscr{F}_{\text{\tiny{low}}} \left( T_3 \right).
 \end{equs}
Then
\begin{equs}
 (\Pi^{n} \CD^2(T_3))(\tau) & = - \int_{0}^{\tau} \frac{e^{i s \mathscr{F}_{\text{\tiny{dom}}}\left( T_3 \right) }}{2ik_1^2} \left( 1+ i \mathscr{F}_{\text{\tiny{low}}} \left( T_3 \right) s - \mathscr{F}_{\text{\tiny{low}}} \left( T_3 \right)^2 \frac{s^2}{2}\right)     ds   \\ & +
 \int_{0}^{\tau} \frac{e^{i s\mathscr{F}_{\text{\tiny{dom}}}\left( T_2 \right)}}{2ik_1^2} \left( 1+ i \mathscr{F}_{\text{\tiny{low}}} \left( T_2 \right) s - \mathscr{F}_{\text{\tiny{low}}} \left( T_2 \right)^2 \frac{s^2}{2}\right)     ds - i \frac{\tau^3}{3 !} \mathscr{F}_{\text{\tiny{low}}} \left( T_1 \right) 
\end{equs}
One has the following identities:
\begin{equs}
 (\hat \Pi^n \CD^2(T_3) ) (\tau) & = -  \left( \id - \CQ \right) \int_{0}^{\tau} \frac{e^{i s \mathscr{F}_{\text{\tiny{dom}}}\left( T_3 \right) }}{2ik_1^2} \left( 1+ i \mathscr{F}_{\text{\tiny{low}}} \left( T_3 \right) s - \mathscr{F}_{\text{\tiny{low}}} \left( T_3 \right)^2 \frac{s^2}{2}\right)     ds, \\
 A^n ( \hat \CD^{(2,0)}(T_3) ) & =  \CQ (\Pi^{n} \CD^2(T_3))(\tau), \quad
  A^n ( \hat \CD^{(2,1)}(T_3) )  = A^n ( \hat \CD^{(2,2)}(T_3) ) = 0, \\
  A^n ( \hat \CD^{(2,3)}(T_3) )  & = -i \mathscr{F}_{\text{\tiny{low}}} \left( T_1 \right), \quad  \Big(\hat \Pi_n  \begin{tikzpicture}[scale=0.2,baseline=-5]
\coordinate (root) at (0,0);
\coordinate (tri) at (0,-2);
\coordinate (t1) at (-2,2);
\coordinate (t2) at (2,2);
\coordinate (t3) at (0,3);
\draw[kernels2,tinydots] (t1) -- (root);
\draw[kernels2] (t2) -- (root);
\draw[kernels2] (t3) -- (root);
\draw[symbols] (root) -- (tri);
\node[not] (rootnode) at (root) {};t
\node[not,label= {[label distance=-0.2em]below: \scriptsize  $ 2 $}] (trinode) at (tri) {};
\node[var] (rootnode) at (t1) {\tiny{$ k_{\tiny{4}} $}};
\node[var] (rootnode) at (t3) {\tiny{$ ^1_\ell $}};
\node[var] (trinode) at (t2) {\tiny{$ k_5 $}};
\end{tikzpicture} \Big)(\tau) = \Big(\hat \Pi^n  \begin{tikzpicture}[scale=0.2,baseline=-5]
\coordinate (root) at (0,0);
\coordinate (tri) at (0,-2);
\coordinate (t1) at (-2,2);
\coordinate (t2) at (2,2);
\coordinate (t3) at (0,3);
\draw[kernels2,tinydots] (t1) -- (root);
\draw[kernels2] (t2) -- (root);
\draw[kernels2] (t3) -- (root);
\draw[symbols] (root) -- (tri);
\node[not] (rootnode) at (root) {};t
\node[not,label= {[label distance=-0.2em]below: \scriptsize  $ 2 $}] (trinode) at (tri) {};
\node[var] (rootnode) at (t1) {\tiny{$ k_{\tiny{4}} $}};
\node[var] (rootnode) at (t3) {\tiny{$ ^2_\ell $}};
\node[var] (trinode) at (t2) {\tiny{$ k_5 $}};
\end{tikzpicture} \Big)(\tau) =0, 
  \\ \hat \Pi^n(  \CD^{2}(T_2)  ) & = (\id - \CQ)  \int_{0}^{\tau} e^{i s\mathscr{F}_{\text{\tiny{dom}}}\left( T_2 \right)} \left( 1+ i \mathscr{F}_{\text{\tiny{low}}} \left( T_2 \right) s - \mathscr{F}_{\text{\tiny{low}}} \left( T_2 \right)^2 \frac{s^2}{2}\right)     ds.
\end{equs}
\end{example}

In the next proposition, we write a Birkhoff type factorisation for the character $ \hat \Pi^n $ defined from $ \Pi^n $ and the antipode. Such an identity was also obtained in the context of SPDEs (see \cite{BHZ}) but with a twisted antipode. Our formulation is slightly simpler due to the simplifications observed at the level of the algebra (see Remark~\ref{comparisonalgebra}). The Proposition~\ref{Birkhoff} is not used  in the sequel but it gives an inductive way to compute 
$ \hat \Pi^n  $ in terms of $ \Pi^n $. It can be seen as a rewriting of identity~\eqref{antipode_def}.
\begin{proposition} \label{Birkhoff}
One has 
\begin{equs} \label{factorisation}
\hat \Pi^n = \left( \Pi^n \otimes (\CQ \circ \Pi^n \CA \cdot)(0) \right) \Delta.
\end{equs}
\end{proposition}
\begin{proof} From \eqref{antipode_def}, one gets:
\begin{equs} \label{identityA}
\Pi^n = \left( \hat \Pi^n \otimes (\CQ \circ \Pi^n  \cdot)(0) \right) \Delta =  \hat \Pi^n * (\CQ \circ \Pi^n  \cdot)(0) 
\end{equs}
where  the product $ * $ is defined from the coaction $ \Delta $. Then, if we multiply the identity \eqref{identityA} by the inverse $ (\CQ \circ \Pi^n \CA  \cdot)(0) $, we get
\begin{equs}
\Pi^n * (\CQ \circ \Pi^n \CA  \cdot)(0) & =\left( \hat \Pi^n * (\CQ \circ \Pi^n  \cdot)(0) \right) * (\CQ \circ \Pi^n \CA  \cdot)(0)  \\
& = \left( \left( \hat \Pi^n \otimes (\CQ \circ \Pi^n  \cdot)(0) \right) \Delta \otimes (\CQ \circ \Pi^n \CA  \cdot)(0)  \right) \Delta
\\ & = \left( \hat \Pi^n \otimes \left( (\CQ \circ \Pi^n \cdot)(0) \otimes (\CQ \circ \Pi^n \CA  \cdot)(0) \right) \Deltap  \right) \Delta
\\ & = \hat \Pi^n  
\end{equs}
where we have used
\begin{equs}
\left( \Delta \otimes \id \right) \Delta = \left( \id \otimes \Deltap \right) \Delta, \quad
\CM \left( \id \otimes \CA \right) \Deltap = \one \one^{\star}.
\end{equs}
This concludes the proof.
\end{proof}

\subsection{Local error analysis}
\label{local error analysis}
In this section, we explore the properties of the character $ \hat \Pi^n $ which allow us  to conduct the local error analysis of the approximation given by $ \Pi^n $. Proposition \ref{factorisation_dom}  below shows that only one oscillation is treated through $ \hat \Pi^n $.

\begin{proposition} \label{factorisation_dom}
For every forest  $  F \in \hat \CH $, there exists a polynomial 
$ B^n\left( \CD^r(F) \right){ (\xi)} $  such that
\begin{equs} \label{simple_formula}
\hat \Pi^n\left( \CD^r(F) \right)(\xi) =  B^n\left( \CD^r(F) \right)(\xi) e^{i \xi\mathscr{F}_{\text{\tiny{dom}}}( F)}
\end{equs}
where $ \mathscr{F}_{\text{\tiny{dom}}}(F)  $ is given in Definition~\ref{dom_freq}.  Moreover, $ B^n(\CD^r(F))(\xi) $ is given by:
\begin{equs} \label{nicefactorisation}
B^n(\CD^r(F))(\xi) =  \frac{P(\xi)}{Q}, \quad  Q = \prod_{\bar T \in A} \left(\mathscr{F}_{\text{\tiny{dom}}}(\bar T) \right)^{m_{\bar T}}
\end{equs}
where $ P(\xi) $ is a polynomial in $ \xi $ and the $k_i$, $ A $ is a set of decorated subtrees of $ F $ satisfying the same property as in Corollary~\ref{physical_map}.
\end{proposition}
\begin{proof} We proceed by induction. We get 
\begin{equs}
\hat{\Pi}^n\left( \CD^r(F \cdot \bar F) \right)(\xi) & = \hat{\Pi}^n\left( \CD^r(F ) \right)(\xi)   \hat{\Pi}^n\left( \CD^r( \bar F) \right)(\xi) \\
& =  B^n\left( \CD^r(F) \right)(\xi) e^{i \xi\mathscr{F}_{\text{\tiny{dom}}}( F)}  B^n\left( \CD^r(\bar F) \right)(\xi) e^{i \xi\mathscr{F}_{\text{\tiny{dom}}}( \bar F)} \\
& =   B^n\left( \CD^r(F) \right)(\xi)  B^n\left( \CD^r(\bar F) \right)(\xi) e^{i \xi(\mathscr{F}_{\text{\tiny{dom}}}(  F) +\mathscr{F}_{\text{\tiny{dom}}}( \bar F) )}
\\ & =   B^n\left( \CD^r(F \cdot \bar F) \right)(\xi) e^{i \xi\mathscr{F}_{\text{\tiny{dom}}}( F \cdot \bar F)}.
\end{equs}
The pointwise product preserves the structure given by \eqref{nicefactorisation}.

Then for  $ T = \CI_{o_1}( \lambda_{k}^{\ell}F) $, one gets by the definition of $\hat{\Pi}^n$ given in \eqref{def_A_B} that
\begin{equs}
\hat{\Pi}^n\left( \CD^r(T) \right)(\xi) & =
 \xi^{\ell} e^{i \xi  P_{o_1}(k) } \hat{\Pi}^n\left( \CD^{r-\ell}(F) \right)(\xi)
 \\ & = \xi^{\ell} e^{i \xi  P_{o_1}(k) +i \xi\mathscr{F}_{\text{\tiny{dom}}}( F) } B^n\left( \CD^r( F) \right)(\xi) 
 \\ & =  B^n\left( \CD^r(T) \right)(\xi) e^{i \xi\mathscr{F}_{\text{\tiny{dom}}}( T)}.
\end{equs}
One gets $ B^n\left( \CD^r(T) \right)(\xi) = \xi^{\ell} B^n\left( \CD^r(F) \right)(\xi) $ and can conclude on the preservation of the factorisation \eqref{nicefactorisation}.
We end with the decorated tree $ T = \CI_{o_2}( \lambda_{k}^{\ell} F) $. By  the definition of $\hat{\Pi}^n$ given in~\eqref{def_A_B}  we obtain that
\begin{equs}
\hat \Pi^n\left( T \right)(\xi) = \CK_{o_2,+}^{k,r}( \hat \Pi^n(  \lambda^{\ell} \CD^{r-\ell-1}(F)),n   )(\xi).
\end{equs}
Now, we apply the induction hypothesis on $ F $ which yields
\begin{equs}
\hat \Pi^n( \CD^{r-\ell-1}(F))(\xi)  =  B^n( \CD^{r-\ell-1}( F))(\xi) e^{i \xi\mathscr{F}_{\text{\tiny{dom}}}(F)}
\end{equs}
and we conclude by applying Definition~\ref{Taylor_exp}. Indeed, by applying $  \CK_{o_2,+}^{k,r} $  to $ \xi^{\ell} e^{i \xi\mathscr{F}_{\text{\tiny{dom}}}( F)} $, we can get in \eqref{e:Pim} extras terms $ \frac{1}{Q} $ coming from expressions of the form
\begin{equs}
 \int_0^{\tau}   \xi^{\ell+q}    e^{i \xi\mathscr{F}_{\text{\tiny{dom}}}(T)} d \xi.
\end{equs}
Then, by computing this integral, we obtain coefficients of the form 
\begin{equs}
 \frac{1}{(i\mathscr{F}_{\text{\tiny{dom}}}(T))^{m_T}}
\end{equs}
if  $\mathscr{F}_{\text{\tiny{dom}}}(T)^{m_T} \neq 0 $. This term will be multiplied by $ B^n( \CD^{r-\ell-1}( F))(\xi) e^{i \xi\mathscr{F}_{\text{\tiny{dom}}}(F)} $ and preserves the structure given in \eqref{nicefactorisation}.
 When performing the Taylor series expansions by applying  $  \CK_{o_2,+}^{k,r} $ (cf. \eqref{e:Pim1} and \eqref{e:Pim2}) we can also get  some extra polynomials in the $ k_i $. This leads to  the factor $P(\xi)$ in \eqref{nicefactorisation} and concludes the proof.
\end{proof}
 \begin{remark}
The identity \eqref{simple_formula} shows that the character $ \hat \Pi^n $ has selected the oscillation $ e^{i \xi\mathscr{F}_{\text{\tiny{dom}}}( F)} $ for the decorated forest $F$. The complexity of this character is hidden behind the polynomial $ B^n\left( \CD^r(F) \right)(\xi)  $. It depends on the parameters $ n $ and $ r $. The explicit formula \eqref{simple_formula}  is strong enough for conducting the local error analysis. 
\end{remark}


The next recursive definition introduces a systematic way to compute the local error from the structure of the decorated tree and the coaction $ \Delta $.

\begin{definition}\label{def:Llow}
Let $ n \in \N $, $ r \in \Z $. We recursively define $ \mathcal{L}^{r}_{\text{\tiny{low}}}(\cdot,n)$ as
\begin{equs}
\mathcal{L}^{r}_{\text{\tiny{low}}}(F,n) = 1, \quad r < 0.
\end{equs}
Else
\begin{equs}
\mathcal{L}^{r}_{\text{\tiny{low}}}(\one,n) = 1, \quad
\mathcal{L}^{r}_{\text{\tiny{low}}}(F \cdot \bar F,n)  = 
\mathcal{L}^{r}_{\text{\tiny{low}}}(F,n ) + \mathcal{L}^{r}_{\text{\tiny{low}}}( \bar F,n) \\
\mathcal{L}^{r}_{\text{\tiny{low}}}(\CI_{o_1}( \lambda_{k}^{\ell}  F ),n)  = \mathcal{L}^{r-\ell}_{\text{\tiny{low}}}(  F,n ) \\
\mathcal{L}^{r}_{\text{\tiny{low}}}(\CI_{o_2}( \lambda^{\ell}_{k}  F ),n)  =   k^{\alpha} \mathcal{L}^{r-\ell-1}_{\text{\tiny{low}}}(  F,n ) + \one_{\lbrace r-\ell \geq 0 \rbrace} \sum_j k^{\bar n_j}
\end{equs}
where 
\begin{equs}
 \bar n_j   =  \max_{ m}\left(n,\deg\left( P_{(F^{(1)}_j, F^{(2)}_j,m)}  \mathscr{F}_{\text{\tiny{low}}} (\CI_{(\Labhom_2,p)}( \lambda^{\ell}_{k}  F^{(1)}_j ))^{r-\ell +1- m} + \alpha \right) \right)
\end{equs}
with
\begin{equs}
\Delta \CD^{r-\ell-1}(F) & = \sum_{j} F^{(1)}_j \otimes F^{(2)}_j, \\ \quad A^n(F^{(2)}_j)  B^n\left( F^{(1)}_j \right)(\xi)  & = \sum_{m \leq r-\ell -1} \frac{P_{(F^{(1)}_j, F^{(2)}_j,m)}}{Q_{(F^{(1)}_j, F^{(2)}_j,m)}}\xi^m
\end{equs}
and $ \mathscr{F}_{\text{\tiny{low}}} $ is defined in Definition \ref{dom_freq}. 
\end{definition}

 \begin{remark}
For a tree $ T $ with $ n $ leaves, the quantity $ \mathcal{L}^{r}_{\text{\tiny{low}}}(T,n) $ is a polynomial in the frequencies $ k_1,...,k_n $  attached to its leaves. The recursive definition of $ \mathcal{L}^{r}_{\text{\tiny{low}}}(\cdot,n) $ follows exactly the mechanism involved in the proof of Theorem~\ref{approxima_tree}.
\end{remark}
  
\begin{remark}
One can observe that the local error strongly depends  on the Birkhoff factorisation  which provides a systematic way to get all the potential contributions. Indeed, $ \bar n $ depends on $ A^n, B^n $ applied to decorated forests coming from the coaction $ \Delta $.
\end{remark}

\begin{remark}\label{rem:simpL} In Section~\ref{sec:examples}  we derive the low regularity resonance based schemes on concrete examples up to order two. In the discussed examples, one does not get any contribution from $ B^n $ and $ A^n $  (see also Example~\ref{AnBn}). Thus, one can work with a simplify definition of $\bar n_j $ given by:
\begin{equs}
\bar n_j & = \max(n, \deg( \mathscr{F}_{\text{\tiny{low}}} (\CI_{(\Labhom_2,p)}( \lambda^{\ell}_{k}  F_j ))^{r-\ell +1} + \alpha ), \end{equs}
where
\begin{equs}
  \sum_j F_j & =  \CM_{(1)} \Delta \CD^{r-\ell-1}(F), \quad F_j \in H , \quad \CM_{(1)} \left( F_1 \otimes F_2 \right) = F_1.
  \end{equs}
\end{remark}

\begin{remark} For $ n=0 $, one obtains an optimal scheme in terms of regularity of the solution which is equal to $n(T,0) = \mathcal{L}^{r}_{\text{\tiny{low}}}(T,0)$. Then, one may wish to use this information to simplify the scheme, i.e., carry out additional Taylor series expansions of the dominant parts if the regularity allows for it. This can be achieved by introducing the new decoration $  n_0  =\max_{T} n(T,0) $.
In the examples in  Section~\ref{sec:examples}, one can observe that for $  n \geq  n_0 $ one has: $ n = \mathcal{L}^{r}_{\text{\tiny{low}}}(T, n)$. In order to guarantee that this holds true in general, one can extend naturally the algebraic structure by introducing the regularity $ n $ as a decoration at the root of a decorated tree. We list below what will be the potential changes:

\begin{itemize}
    \item First in Definition~\ref{Taylor_exp}, we take into account only monomials $ \xi^\ell $ that will give the length of the Taylor approximation. We could potentially insert some polynomials in the frequency that will refine the analysis with  respect to $ n $ in case we already used some regularity coming from previous approximations.
     One can introduce an extended decoration by having another component on the monomials $  \lambda^{\ell}, \ell \in \N^{2} $ where the second decorations will stand for this regularity already used.
    \item The trees will carry at the root a decoration in $ \Z^2 $ of the form $ (r,n) \in \Z^2 $. The decoration will have the same behaviour as $ r $, it will decrease for each edge in $ \mathfrak{L}_+ $ according to the derivative $ |\nabla|^{\alpha} $ that appears in Duhamel's formula. The recursive formula \eqref{def_deltas} will have two Taylor expansions one determined by $ r $, the other one determined by $ n $ for $  n \geq  n_0 $. The Birkhoff factorisation will remain the same based on these two Taylor expansions.
\end{itemize}
This potential extension for optimising the scheme shows that the algebraic structure chosen in this work is robust and can encode various behaviours such as the order of the scheme as well as its regularity. 
\end{remark}

\begin{remark} As in Remark~\ref{better error}, we use only powers of $ k $ in Definition~\ref{def:Llow} but more structures can be preserved if one wants to conduct a  global error analysis. 
\end{remark}
Now we are in the position to state the  approximation error of $\Pi^{n,r}$ to $\Pi$ (cf. \eqref{eq:loci}).
\begin{theorem} \label{approxima_tree}
For every $ T \in \CT $ one has,
\begin{equs}
\left(\Pi T - \Pi^{n,r} T \right)(\tau)  = \mathcal{O}\left( \tau^{r+2} \mathcal{L}^{r}_{\text{\tiny{low}}}(T,n) \right)
\end{equs}
where $\Pi $ is defined in \eqref{Pi},  $\Pi^n$ is given in \eqref{recursive_pi_r} and $\Pi^{n,r} = \Pi^n  \CD^r$.
\end{theorem}
\begin{proof}  We proceed by induction on the size of a forest by using the recursive definition \eqref{recursive_pi_r} of $ \Pi^n $ and we prove a more general version of the theorem for forests. In fact, only the version for trees is needed for the local error analysis. 
First, one gets:
\begin{equs}
\left(\Pi - \Pi^{n,r} \right)(\one)(\tau)  = 0 = \mathcal{O}\left( \tau^{r+2} \mathcal{L}^{r}_{\text{\tiny{low}}}(\one,n) \right).
\end{equs}
One also has 
\begin{equs}
\left(\Pi -  \Pi^{n,r} \right)(\CI_{o_1}( \lambda_{k}^{\ell}  F ))(\tau) & = \tau^{\ell} e^{i \tau P_{o_1}(k)} (\Pi - \Pi^{n,r-\ell})(F)(\tau) \\
 & = \mathcal{O}\left( \tau^{r+2} \mathcal{L}^{r-\ell}_{\text{\tiny{low}}}(F,n) \right).
\end{equs}
Then, one gets again by  \eqref{Pi} and   \eqref{recursive_pi_r} that
\begin{equs}
\left(\Pi - \Pi^{n,r} \right)(F \cdot \bar F)(\tau) & = \left(\Pi - \Pi^{n,r} \right)(F)(\tau)  ( \Pi^{n,r} \bar F)(\tau) + (\Pi F)(\tau)  \left(\Pi -  \Pi^{n,r} \right)(\bar F)(\tau) \\
& =  \mathcal{O}\left( \tau^{r+2} \mathcal{L}^{r}_{\text{\tiny{low}}}(F,n) \right) + \mathcal{O}\left( \tau^{r+2} \mathcal{L}^{r}_{\text{\tiny{low}}}(\bar F,n) \right) \\
& = \CO \left( \tau^{r+2} \mathcal{L}^{r}_{\text{\tiny{low}}}(F \cdot \bar F,n) \right) 
\end{equs}
where we use Definition \ref{def:Llow}.
At the end, by \eqref{Pi} and   \eqref{recursive_pi_r}  and inserting zero in terms of
\[
\pm \Pi^{n,r-\ell - 1}( F)(\xi) 
\]
as well as using that $ \Pi^{n,r-\ell - 1}( F) = \Pi^n  \CD^{r-\ell-1}(F) $   we obtain 
\begin{equs}
&\left( \Pi- \Pi^{n,r} \right) \left( \CI_{o_2}( \lambda_k^{\ell} F) \right) (\tau)  = - i \vert \nabla\vert^{\alpha} (k) \int_{0}^{\tau} \xi^{\ell} e^{i \xi P_{o_2}(k)} (\Pi - \Pi^{n,r-\ell - 1})( F)(\xi) d \xi 
\\ &  - i \vert \nabla\vert^{\alpha} (k)\int_{0}^{\tau} e^{i \xi P_{o_2}(k)}  (\Pi^{n} {  \lambda^{\ell} }  \CD^{r-\ell-1}(F) )(\xi) d \xi -\CK^{k,r}_{o_2} (  \Pi^{n}( \lambda^\ell\CD^{r-\ell-1} (F)),n )(\tau)
 \\ & = \int_{0}^{\tau} \CO \left( \xi^{r+1} k^{\alpha} \mathcal{L}^{r-\ell-1}_{\text{\tiny{low}}}(F,n) \right)  d\xi + \one_{\lbrace r-\ell \geq 1 \rbrace} \sum_{j} \CO(\tau^{r+2} k^{\bar n_j}) \\
& =   \CO \left( \tau^{r+2} \mathcal{L}^{r}_{\text{\tiny{low}}}(F,n) \right) .
\end{equs}
Note that in the above calculation we have used the following decomposition
\begin{equs}
\Pi^n \left( \CD^{r-\ell-1} (F)\right) =
\left( \hat \Pi^n \otimes A^n \right) \Delta \CD^{r-\ell-1} (F)
\end{equs}
 which by Proposition~\ref{factorisation_dom} implies that one has
\begin{equs} \label{Bnu}
 \Pi^n \left( \CD^{r-\ell-1} (F)\right)(\xi) =
 \sum_{j} A^n(F^{(2)}_j)  B^n\left( F^{(1)}_j \right)(\xi) e^{i \xi\mathscr{F}_{\text{\tiny{dom}}}( F^{(1)}_j)}
\end{equs}
where we have used Sweedler notations for the coaction   $ \Delta $. Then, one has
\begin{equs}
A^n(F^{(2)}_j)  B^n\left( F^{(1)}_j \right)(\xi)  = \sum_{m \leq r-\ell -1} \frac{P_{(F^{(1)}_j, F^{(2)}_j,m)}}{Q_{(F^{(1)}_j, F^{(2)}_j,m)}}\xi^m
 \end{equs}
 where $ P_{(F^{(1)}_j, F^{(2)}_j,m)} $ and $  Q_{(F^{(1)}_j, F^{(2)}_j,m)}$ are polynomials in the frequencies $ k_1,..,k_n $.

Thus by applying Lemma~\ref{Taylor_bound}, we get an error for every term in the sum \eqref{Bnu} which is at most of the form:
\begin{equs}
\sum_j \CO(\tau^{r+2}  k^{\bar n_j} )
\end{equs}
where
\begin{equs}
 \bar n_j   =  \max_{ m}\left(n,\deg\left( P_{(F^{(1)}_j, F^{(2)}_j,m)}  \mathscr{F}_{\text{\tiny{low}}} (\CI_{(\Labhom_2,p)}( \lambda^{\ell}_{k}  F^{(1)}_j ))^{r-\ell +1- m} + \alpha \right) \right).
\end{equs}
This concludes the proof.
\end{proof}

\begin{proposition} \label{physical_space} For every decorated tree $ T = \CI^{r}_{(\Labhom,p)}( \lambda^{\ell}_{k}  F) $ in $ \CT  $  with disjoint leaves decorations  being a subset of the $k_i$ as in Assumption~\ref{assumption_physical_space}, one can map $ \Pi^n T $ back  into physical  space which means that for functions indexed by the leaves of $ T $, $ (v_u)_{u \in L_T} $, the term
\begin{equs} \label{inverse Fourier}
\mathcal{F}^{-1} \left( \sum_{k = \sum_{u \in L_T} a_u k_u}(\Pi^n T)(\xi) \right) (v_{u,a_u}, u \in L_T)
\end{equs}
can be expressed by applying classical differential operators $ \nabla^{\ell} , e^{i\xi \nabla^{m}}$, $ m,\ell \in \Z  $ to  $ v_{u,a_u} $ which are defined by $ v_{u,1} = v_u $ and $ v_{u,-1} = \overline{v_u} $. 
\end{proposition}
\begin{proof} We proceed by induction using the identity \eqref{antipode_def}. The latter implies that
\begin{equs}
\Pi^n T & = \left( \hat  \Pi^n \otimes A^n \right) \Delta T 
 = \sum_j \hat \Pi^n(T_j^{(1)}) A^n(T_j^{(2)})
 \quad \Delta T  = \sum_{j} T_j^{(1)} \otimes T_{j}^{(2)}.
\end{equs} 
Then, for every $ \hat \Pi^n T_j^{(1)} $, we apply Proposition~\ref{factorisation_dom} and we get:
\begin{equs}
\hat \Pi^n\left( T_j^{(1)} \right)(\xi) =  B^n\left( T_j^{(1)} \right)(\xi) e^{i \xi\mathscr{F}_{\text{\tiny{dom}}}( T_j^{(1)})}.
\end{equs}
 Moreover, $ B^n(T_j^{(1)})(\xi) $ is given by:
\begin{equs} \label{nicefactorisationb}
B^n(T_j^{(1)})(\xi) =  \frac{P(\xi)}{Q}, \quad  Q = \prod_{\bar T \in A({T_j^{(1)}})} \left(\mathscr{F}_{\text{\tiny{dom}}}(\bar T) \right)^{m_{\bar T}}
\end{equs} 
with the notations defined in Proposition~\ref{factorisation_dom} and $ A(T_j^{(1)}) $ are some decorated subtrees of $ T_j^{(1)} $. The term $ \hat \Pi^n\left( T_j^{(1)} \right)(\xi) $ can be mapped back to physical space using Proposition~\ref{Fourierproduct}. The polynomial $ P(\xi) $ will produce derivatives  of type $ \nabla^{\ell} $ and the term $  e^{i \xi\mathscr{F}_{\text{\tiny{dom}}}( T_j^{(1)})}$ is of the form $ e^{i\xi \nabla^{m}} $.
For the terms $ A^n(T_j^{(2)}) $, we use the non-recursive definition of the map $ \Delta $. Indeed, $ T_j^{(2)} $ is a product of trees of the form $ T_e $ where $ e $ is an edge in $ T $ which was cut. The map $ A^n$ is defined from $ \Pi^n $ in \eqref{An}. Then we can apply the induction hypothesis on $ A^n(T_e) $. For each $ T_e = \CI_{(\Labhom,p)}( \lambda_{k_e}^{\ell_e} \bar T_e) $, $ k_e $ appears as a decoration at a leaf of $ T_j^{(1)} $. Then, it is either included in  $ Q $ or disjoint. This allows us to apply the inverse Fourier Transform concluding the proof.
 \end{proof}

\section{A general numerical scheme}\label{sec:genScheme}

Recall the mild solution of  \eqref{dis} given by Duhamel's formula
\begin{equation}\label{duhLin_it}
u(t) = e^{ it  \mathcal{L}\left(\nabla, \frac{1}{\varepsilon}\right)} v  - i\vert \nabla \vert^\alpha e^{ it  \mathcal{L}\left(\nabla, \frac{1}{\varepsilon}\right)}  \int_0^t e^{ -i\xi  \mathcal{L}\left(\nabla, \frac{1}{\varepsilon}\right)} p(u(t),\bar u(t)) d\xi .
\end{equation}
 For simplicity we restrict our attention to  nonlinearities of type
\begin{equs}\label{poly}
p(u, \bar u) = u^N \bar u^M
\end{equs}
which includes all examples in Section~\ref{sec:examples}.
The analysis which follows can straightforwardly be generalised to polynomials and coupled systems.

In order to describe our general numerical scheme, we first describe the  iterated integrals produced by the (high order) iteration of Duhamel's formula \eqref{duhLin_it} through a class of suitable decorated trees.

\subsection{Decorated trees generated by Duhamel's formula}

With the aid of the Fourier series expansion $u(x) = \sum_{k\in \Z^d} \hat u_k(t) e^{i k x}$ we first rewrite Duhamel's formula \eqref{duhLin_it} at the level of the Fourier coefficients:
\begin{equation}\label{duhLin_it_Fourier}
\hat u_k(t) = e^{ it  P(k)} \hat v_k  - i  \vert \nabla\vert^{\alpha} (k)e^{ it  P(k)} \int_0^t e^{ -i\xi  P(k) } p_k(u(\xi),\bar u(\xi)) d\xi 
\end{equation}
where $P(k)$ denotes the differential operator $\mathcal{L}$ in Fourier space, i.e., 
\begin{equs}
P(k) = \mathcal{L}\left(\nabla,\frac{1}{\varepsilon}\right)(k)
\end{equs} (cf. \eqref{Lldef} and \eqref{Leps}, respectively) and
\begin{equs}
 p_k(u(t),\bar u(t)) {\, :=} \sum_{k =\sum_{i} k_i - \sum_j \bar k_j} \prod_{i=1}^N \hat u_{k_i}(t) \prod_{j=1}^M \bar{\hat{u}}_{\bar k_j}(t).
\end{equs}
This equation is given in an abstract way by
\begin{equs} \label{recursion_tree_f}
U_k = \CI_{(\Labhom_1,0)}( \lambda_k) + \CI_{(\Labhom_1,0)}( \lambda_k \CI_{(\Labhom_2,0)}( \lambda_k  p_k(U, \bar U ) )  )
\end{equs}
where
\begin{equs}
 p_k(U,\bar U) {\, :=} \sum_{k =\sum_i k_i - \sum_j \bar k_j} \prod_{i=1}^N  U_{k_i} \prod_{j=1}^M  \bar{U}_{\bar k_j}
\end{equs}
and $\Lab=\{\Labhom_1,\Labhom_2\}$,  $P_{\Labhom_1}( \lambda) = P( \lambda)$ and $P_{\Labhom_2}( \lambda) = - P( \lambda)$.
\begin{remark} The two systems \eqref{duhLin_it_Fourier} and \eqref{recursion_tree_f} are equivalent. Indeed,
we can define a map $ \psi $ such that:
\begin{equs}
\psi(U_k)(u,v,t) & = \hat u_k(t) , \quad \psi(\bar U_k)(u,v,t)  = \bar{\hat{u}}_k(t)
\\ \psi\left( \CI_{o_1}( \lambda_k)  \right)(v,u,t)  & =  e^{ it   P_{o_1}(k)} \hat v_k(t) \\
\psi\left( \CI_{o_1}( \lambda_k T)  \right)(v,u,t) & =  e^{ it  P_{o_1}(k)} \psi\left( T  \right)(v,u,t) \\
\psi\left( \CI_{o_2}( \lambda_k T)  \right)(v,u,t) & =  - i \vert \nabla \vert^\alpha(k) \int_{0}^{t} e^{ i \xi  P_{o_2}(k)} \psi\left( T  \right)(v,u,\xi) d\xi.
\end{equs}
In this notation \eqref{duhLin_it_Fourier} takes the form 
\begin{equs}
\psi(U_k )= \psi\left(\CI_{(\Labhom_1,0)}( \lambda_k) \right)(v,u,0)+ \psi\left(\CI_{(\Labhom_1,0)}( \lambda_k \CI_{(\Labhom_2,1)}( \lambda_k  p_k(U, \bar U ) )  )\right).
\end{equs}
\end{remark}

 We define the notion of a rule in the same spirit as in \cite{BHZ}.  A rule is then a map $R$ assigning to each element of $\Lab \times \lbrace 0,1 \rbrace$ a non-empty collection
of tuples in $\Lab \times \lbrace 0,1 \rbrace$. 
The relevant rule describing the class of equations \eqref{recursion_tree_f}
is given by 
\begin{equs}
R(  (\Labhom_1,p) ) & = \{(), (   (\Labhom_2,p) )  \} \\
R(  (\Labhom_2,p) ) & = \{ ( (\Labhom_1,p)^{N}, (\Labhom_1,p+1)^{M})   \}
\end{equs}
where $N$ and $M$ depend on the polynomial nonlinearity \eqref{poly} and  the notation $ (\Labhom_1,p+1)^{M} $ means that $ (\Labhom_1,p+1) $ is repeated $ M $ times and the sum $ p+1 $ is performed modulo $2$. Using graphical notation, one gets:
\begin{equation}\label{e:ru}
\begin{aligned}
R\bigg( \begin{tikzpicture}[scale=0.2,baseline=0.32cm]
          \node at (0,0) [dot] (k) {}; 
           \node at (0,5)  (l) {}; 
    \draw[kernel1] (l) -- node [rect1] {\tiny$\Labhom_1, p $}   (k)  ;
\end{tikzpicture}  
\bigg)
& =\bigg\{ (), \ 
\bigg( \begin{tikzpicture}[scale=0.2,baseline=0.32cm]
          \node at (0,0) [dot] (k) {}; 
           \node at (0,5)  (l) {}; 
    \draw[kernel1] (l) -- node [rect1] {\tiny$\Labhom_2, p $}   (k)  ;
\end{tikzpicture}  
\bigg)
\bigg\} \\
R\bigg( \begin{tikzpicture}[scale=0.2,baseline=0.32cm]
          \node at (0,0) [dot] (k) {}; 
           \node at (0,5)  (l) {}; 
    \draw[kernel1] (l) -- node [rect1] {\tiny$\Labhom_2, p $}   (k)  ;
\end{tikzpicture}  
\bigg)
& =\bigg\{ 
\bigg( \bigg( \begin{tikzpicture}[scale=0.2,baseline=0.32cm]
          \node at (0,0) [dot] (k) {}; 
           \node at (0,5)  (l) {}; 
    \draw[kernel1] (l) -- node [rect1] {\tiny$\Labhom_1, p $}   (k)  ;
\end{tikzpicture}  
\bigg)^{N}, \ \bigg( \begin{tikzpicture}[scale=0.2,baseline=0.32cm]
          \node at (0,0) [dot] (k) {}; 
           \node at (0,5)  (l) {}; 
    \draw[kernel1] (l) -- node [rect1] {\tiny$\Labhom_1, p +1 $}   (k)  ;
\end{tikzpicture}  
\bigg)^{M} \bigg)
\bigg\}
\end{aligned}
\end{equation}

\begin{definition} \label{rules}
A decorated tree $ T_{\Labe}^{\Labo} $ in $ \hat \CT_0 $ is generated by $ R $ if for every node $ u $ in $ N_T $, one has 
\begin{equs}
\cup_{e \in E_u}(\Labhom(e),\Labp(e)) \in  R(e_u) 
\end{equs}
where $ E_u \subset E_T $ are the edges of the form $ (u,v) $ and $ e_u $ is the edge of the form
$ (w,u) $. The set of decorated trees generated by $ R $ is denoted by $ \hat \CT_0(R)  $ and for $ r \in \Z $, $ r \geq -1 $, we set:
\begin{equs}
 \hat \CT_0^{r}(R) = \lbrace  T_{\Labe}^{\Labo} \in \hat \CT_0{ (R)}   \,  , \deg(T_{\Labe}^{\Labo}) \leq r +1  \rbrace. 
\end{equs}
\end{definition}

Given a decorated tree $ T_{\Labe} = (T,\Labe)$ where we just have the edge decoration, the symmetry 
factor $S(T_{\Labe})$ is defined inductively by setting $S(\one)\,  { =} 1$, while if 
$T$ is of the form
\begin{equs}  
\prod_{i,j}  \mathcal{I}_{(\Labhom_{t_i},p_i)}\left( T_{i,j}\right)^{\beta_{i,j}}  
\end{equs}
with $T_{i,j} \neq T_{i,\ell}$ for $j \neq \ell$, then 
\begin{align}\label{S}
S(T)
\,  { :=} 
\Big(
\prod_{i,j}
S(T_{i,j})^{\beta_{i,j}}
\beta_{i,j}!
\Big)\;.
\end{align}
We extend this definition to any tree $ T_{\Labe}^{\Labn,\Labo} $ in $ \CT $ by setting:
\begin{equs}
S(T_{\Labe}^{\Labn,\Labo} )\,  { :=}  S(T_{\Labe} ).
\end{equs}

Then, we define the map $ \Upsilon^{p}(T)(v) $ for 
\begin{equs}
 T  = 
\CI_{(\Labhom_2,0)}( \lambda_k   \prod_{i=1}^N \CI_{(\Labhom_1,0)}( \lambda_{k_i} T_i) \prod_{j=1}^M \CI_{(\Labhom_1,1)}( \lambda_{\tilde k_j} \tilde T_j)  )  
\end{equs}
by
\begin{equs}\label{upsi}
\Upsilon^{p}(T)(v)& \,  { :=}  \partial_v^{N} \partial_{\bar v}^{M} p(v,\bar v) \prod_{i=1}^N  \Upsilon^p( \lambda_{k_i}  T_i)(v) \prod_{j=1}^M \overline{\Upsilon^p( \lambda_{\tilde k_j}\tilde T_j)(v)}\\
& = N! M! \prod_{i=1}^N  \Upsilon^p( \lambda_{k_i} T_i)(v) \prod_{j=1}^M \overline{\Upsilon^p( \lambda_{\tilde k_j}\tilde T_j)(v)}
\end{equs}
and 
\begin{equs}
\Upsilon^{p}(\CI_{(\Labhom_1,0)}( \lambda_{k})  )(v)  &\,  { :=}  \hat v_k, \quad \Upsilon^{p}(\CI_{(\Labhom_1,0)}( \lambda_{k} \tilde T ))(v)   \,  { :=}  \Upsilon^{p}( \lambda_k \tilde T)  (v), \\
\Upsilon^{p}(\CI_{(\Labhom_1,1)}( \lambda_{k})  )(v)  &\,  { :=} \bar{\hat{v}}_k,
 \quad\Upsilon^{p}(\CI_{(\Labhom_1,1)}( \lambda_{k} \tilde T)  )(v) \,  { :=} \overline{\Upsilon^{p}( \lambda_k \tilde T)  (v)}, \quad  \tilde T \neq \one .
\end{equs}

\begin{example}\label{ex:SUpsNLS1}
Assume that we have the tree 
 \begin{equs}
 T  = 
\CI_{(\Labhom_2,0)}( \lambda_k   \CI_{(\Labhom_1,1)}( \lambda_{k_1} ) \CI_{(\Labhom_1,0)}( \lambda_{k_2}) \CI_{(\Labhom_1,0)}( \lambda_{k_3})  )  
\end{equs}
then
\begin{equs}
\Upsilon^{p}\left(T\right) (v) & = 2\Upsilon^{p}\left( \CI_{(\Labhom_1,1)}( \lambda_{k_1} ) \right) (v)
 \Upsilon^{p}\left( \CI_{(\Labhom_1,0)}( \lambda_{k_2} ) \right) (v)  \Upsilon^{p}\left( \CI_{(\Labhom_1,0)}( \lambda_{k_3} ) \right) (v) \\& =2  \overline{\hat{v}}_{k_1} \hat{v}_{k_2} \hat{v}_{k_3}
\end{equs}
and
$
S(T) = 2.
$
\end{example}

If we want to find a solution $ U $ to \eqref{recursion_tree_f} as a linear combination of decorated trees, then we need to give a meaning to the conjugate of $ U $ that is $ \bar U $. We define this operation on $ \hat \CT $ recursively as:
\begin{equs}\label{bar}
\overline{ \CI_{(\Labhom,p)}( \lambda_k T)} =  \CI_{(\Labhom,p+1)}( \lambda_{k} \overline{T}), \quad \overline{T_1 \cdot T_2 } =   \overline{T}_1 \cdot \overline{T}_2.
\end{equs}
This map is well-defined from $ \hat \CH $ into itself and preserves the identity \eqref{innerdecoration}. We want to find maps $ V^{r}_k $, $ r \in \Z $, $ r \geq -1 $ such that
 \begin{equs} \label{equa_recurs}
 V^{r+1}_k = \CI_{(\Labhom_1,0)}( \lambda_k ) +  \CI_{(\Labhom_1,0)}( \lambda_k \CI_{(\Labhom_2,0)}( \lambda_k  p_k(V^{r}, \overline{V^{r}} ) )  ) .
 \end{equs}
 An explicit expression is given in the next proposition
 
 \begin{proposition}
The solution in $ \hat \CH$ of \eqref{equa_recurs} is given by the trees generated by the rule $ R $
\begin{equs}
V^{r}_k (v) = \sum_{T \in \hat \CT_0^{r}(R)} \frac{\Upsilon^{p}( \lambda_k T)(v)}{S(T)} \CI_{(\Labhom_1,0)}( \lambda_k T).
\end{equs}
\end{proposition}
\begin{proof}We will prove this by induction. We need to expand 
\begin{equs} \label{recursion_tree_f_t}
Z_k = \CI_{(\Labhom_1,0)}( \lambda_k ) +  \CI_{(\Labhom_1,0)}( \lambda_k \CI_{(\Labhom_2,0)}( \lambda_k  p_k(V^{r}, \overline{V^{r}} ) )  ) .
\end{equs}
One has
\begin{equs}
Z_k & =  \CI_{(\Labhom_1,0)}( \lambda_k ) + \sum_{T_i, \tilde T_j \in \hat \CT_0^{r}(R)}  \sum_{k = \sum_i k_i - \sum_j \tilde k_j } \prod_{i,j} \frac{\Upsilon^{p}(T_i)}{S(T_i)} \frac{\overline{\Upsilon^{p}(\tilde T_j)}}{S(\tilde T_j)} 
\\ & \cdot \CI_{(\Labhom_1,0)}( \lambda_k \CI_{(\Labhom_2,0)}( \lambda_k\left( \prod_i   \CI_{(\Labhom_1,0)}( \lambda_{k_i} T_i) \prod_j \CI_{(\Labhom_1,1)}( \lambda_{\tilde k_j} \tilde T_j) \right))).  
\end{equs}
Then, we fix the products $ \prod_i   \CI_{(\Labhom_1,0)}( \lambda_{k_i} T_i)  $ and $ \prod_j   \CI_{(\Labhom_1,0)}( \lambda_{\tilde k_j} \tilde T_j)  $ which can be rewritten as follows
\begin{equs}
\prod_i   \CI_{(\Labhom_1,0)}( \lambda_{k_i} T_i) = \prod_{\ell}   \CI_{(\Labhom_1,0)}( \lambda_{m_\ell} F_{\ell})^{\beta_{\ell}} 
\\ 
\prod_j   \CI_{(\Labhom_1,1)}( \lambda_{\tilde k_j} \tilde T_j) = \prod_{\ell}   \CI_{(\Labhom_1,1)}( \lambda_{\tilde{m}_\ell} \tilde F_{\ell})^{\alpha_{\ell}} 
\end{equs}
where the $ F_{\ell} $ (resp. $ \tilde F_{\ell} $) are disjoints.
The number of times the same term appears when we sum over the $ F_{\ell} $ (resp. $ \tilde F_{\ell} $) and the $ k_i $ (resp. $ \tilde k_j $) is equal to $ \frac{N!}{\prod_{\ell} \beta_{\ell}  !} $ (resp. $ \frac{M!}{\prod_{\ell} \alpha_{\ell}  !} $). 
With the  identity
\begin{equs}
\prod_{i,j}   \frac{\Upsilon^{p}(T_i)}{S(T_i)} \frac{\overline{\Upsilon^{p}(\tilde T_j)}}{S(\tilde T_j)}  \frac{N!}{\prod_{\ell} \beta_{\ell}  !} \frac{M!}{\prod_{\ell} \alpha_{\ell}  !}  = \frac{\Upsilon^{p}( \lambda_k T)}{S(T)}
\end{equs}
 we thus get
\begin{equs}
Z_k = \sum_{T \in \hat \CT_0^{r+1}(R)} \frac{\Upsilon^{p}( \lambda_k T)}{S(T)} \CI_{(\Labhom_1,0)}( \lambda_k T) = V^{r+1}_{k}
\end{equs}
which concludes the proof.
\end{proof}
Before describing our numerical scheme, we need to remove the trees which are already of size $\CO(\tau^{r+2})$. Indeed, one has 
\begin{equs}
(\Pi T)( \tau) = \CO(\tau^{n_+(T)})
\end{equs} 
where  $n_+(T) $ is the number of edges of type in 
$ \Lab_+ $  corresponding to the number of integration in the definition of $ \Pi T $. Therefore, we define the space of trees $  \CT^{r}_{0}(R) $ as 
\begin{equs} \label{space_scheme}
\CT^{r}_{0}(R) = \lbrace  T \in \hat \CT^{r}_{0}(R), \, n_+(T)  \leq r +1 \rbrace. 
\end{equs}

\subsection{Numerical scheme and local error analysis}

Now, we are able to describe the general numerical scheme:
\begin{definition}[The general numerical scheme] \label{scheme} For fixed $ n, r \in \N $, we define the general numerical scheme in Fourier space as:
\begin{equs}\label{genscheme}
U_{k}^{n,r}(\tau, v) = \sum_{T \in \CT^{r+2}_{0}(R)} \frac{\Upsilon^{p}( \lambda_kT)(v)}{S(T)} \Pi^n \left( \CD^r(\CI_{(\Labhom_1,0)}( \lambda_k T)) \right)(\tau).
\end{equs}
\end{definition}

 \begin{remark}
The sum appearing in the expression~\eqref{genscheme} runs over infinitely many trees
(finitely many shapes, but infinitely many ways of splitting up the frequency $ k $ among the branches). One can notice that each leaf decorated by $ k_i $ is associated with $ v_{k_i} $ coming from the term $\Upsilon^{p}( \lambda_kT)(v)  $.  Therefore with an appropriate analytical assumption on the initial value  $ v $, i.e., if $v$ belongs to a sufficiently smooth Sobolev space, this sum converges in a suitable norm, see also Remark~\ref{sobolev_remark} below and the regularity assumptions detailed in Section \ref{sec:examples}. 
\end{remark}

\begin{remark} We can always map the term $ U_{k}^{n,r}(\tau, v) $ back to physical using classical operators. Indeed, from Proposition~\ref{physical_space} this holds true for each term  $ \Pi^n \left( \CD^r(\CI_{(\Labhom_1,0)}( \lambda_k T)) \right)(\tau) $. In practical applications this will allow us to carry out the multiplication of functions in physical space, using the Fast Fourier Transform (cf. Remark \ref{rem:FFT}). Details for concrete applications are given in Setion \ref{sec:examples}.
\end{remark}

\begin{remark} The spaces $ \CV_{k}^{r} $ given in \eqref{decoratedV1} and \eqref{decoratedV2} are  defined by:
\begin{equs}
 \CV_{k}^{r} = \lbrace \CI_{(\Labhom_1,0)}( \lambda_k T), \, T \in \CT^{r+2}_{0}(R) \rbrace.
\end{equs}
\end{remark}
The numerical scheme \eqref{genscheme}  approximates  the exact solution locally up to order $r+2$. More precisely, the following Theorem holds,
\begin{theorem}[Local error]\label{thm:genloc} 
The numerical scheme \eqref{genscheme}  with initial value $v = u(0)$ approximates the exact solution $U_{k}(\tau,v) $ up to a  local error of type
\begin{equs}
U_{k}^{n,r}(\tau,v) - U_{k}(\tau,v) = \sum_{T \in \CT^{r+2}_{0}(R)} \CO\left(\tau^{r+2} \CL^{r}_{\text{\tiny{low}}}(T,n) \Upsilon^{p}( \lambda_kT)(v) \right)
\end{equs}
where the operator $\CL^{r}_{\text{\tiny{low}}}(T,n)$, given in Definition \ref{def:Llow}, embeds the necessary regularity of the solution.
\end{theorem}
\begin{proof}
First we define the exact solution $u^r$ up to order $ r $ in Fourier space by
\begin{equs}
U_{k}^{r}(\tau,v) = \sum_{T \in \CT^{r+2}_{0}(R)} \frac{\Upsilon^{p}( \lambda_kT)(v)}{S(T)} \Pi \left( \CI_{(\Labhom_1,0)}( \lambda_k T) \right)(\tau)
\end{equs}
which satisfies
\begin{equs}\label{app1}
 u(\tau) - u^r(\tau)  = 
 \CO\left( \tau^{r+2} \vert \nabla \vert^{\alpha(r+2)} \tilde p (u(t))\right)
\end{equs}
 for some polynomial $\tilde p$ and $0 \leq t \leq \tau$. Thanks to  Proposition~\ref{approxima_tree} we furthermore obtain that
\begin{equs}\label{app2}
U_{k}^{n,r} & (\tau,v) - U_{k}^{r}(\tau,v) 
\\ & = \sum_{T \in \CT^{r+2}_{0}(R)} \frac{\Upsilon^{p}( \lambda_kT)}{S(T)}(v) \left( \Pi -\Pi^{n,r} \right) \left( \CI_{(\Labhom_1,0)}( \lambda_k T)\right)(\tau) \\
& = \sum_{T \in \CT^{r+2}_{0}(R)} \CO\left(\tau^{r+2} \CL^{r}_{\text{\tiny{low}}}(T,n)  \Upsilon^{p}( \lambda_kT)(v)\right).
\end{equs}
Next we write
\begin{equs}
U_{k}^{n,r}(\tau,v) - U_{k}(\tau,v)  & = U_{k}^{n,r}(\tau,v) - U_k^r(\tau,v)  + U_k^r(\tau,v) - U_{k}(\tau,v)
\end{equs}
where by the definition of $\CL^{r}_{\text{\tiny{low}}}(T,n)$ we easily see that the approximation error~\eqref{app2} is dominant compared to \eqref{app1}.
\end{proof}
\begin{remark}\label{rem:glob}
Theorem \ref{thm:genloc}  allows us to  state the order of consistency of the general scheme \eqref{genscheme} as well as the necessary regularity requirements on the exact solution to meet the error bound.
In  Section~\ref{sec:examples} we detail the particular form of the general scheme \eqref{genscheme} on concrete examples and explicitly determine the required regularity of the solution in the local error imposed by the operator $\CL^{r}_{\text{\tiny{low}}}(T,n)$.
\end{remark}
 \begin{remark} \label{sobolev_remark}
Theorem \ref{thm:genloc} provides a local error estimate (order of consistency) for the new resonance based schemes \eqref{genscheme}. With the aid of stability one can easily obtain a global error estimate  with the aid of Lady Windamere's fan argument \cite{H2Tri}. However,  the necessary stability estimates in general rely on the algebraic structure of the underlying space. In the stability analysis of dispersive PDEs set in Sobolev spaces $H^r$ one classically exploits bilinear estimates of type 
 \[
 \Vert v w \Vert_r \leq c_{r,d} \Vert v \Vert_r \Vert w \Vert_r.
 \]
The latter only hold for $r>d/2$ and thus restricts the analysis to sufficiently smooth Sobolev spaces $H^r$ with $r>d/2$. To obtain (sharp) $L^2$ global error estimates one needs to exploit discrete Strichartz estimates and discrete  Bourgain spaces in the periodic setting, see, e.g., \cite{IZ09,ORS19,ORS20}. This is out of the scope of this paper.
 \end{remark}

 \begin{proposition}\label{Tay} For $ n $ sufficiently large the scheme \eqref{genscheme}  coincides with a classical numerical discretisation based on Taylor series expansions of the full operator $\mathcal{L}$.
 \end{proposition}
 \begin{proof}
  From Theorem~\ref{thm:genloc}, one has
 \begin{equs}
U_{k}^{n,r}(\tau,v) - U_{k}^{r}(\tau,v) = \sum_{T \in \CT^{r+2}_{0}(R)} \CO\left(\tau^{r+2} \CL^{r}_{\text{\tiny{low}}}(T,n) \Upsilon^{p}( \lambda_kT)(v)\right).
\end{equs}
We need to show that for $ n $ bigger than $ \deg(P_{\Labhom_2}^r) $, $ U_{k}^{n,r}(\tau,v) $ is a polynomial in $ \tau $ and that $ \CL^{r}_{\text{\tiny{low}}}(T,n) = \CO(P_{\Labhom_2}^{r}(k)) $.
Then by mapping it back to the physical space, we get:
\begin{equs}
U^{n,r}(\tau,v) - U^{r}(\tau,v) =  \CO(\tau^{r+2} \partial_t^{r} v)
\end{equs}
where $ U^{n,r}(\tau,v) $ is a polynomial in $ \tau $. The two statements can be proven by induction. One needs to see how these properties are preserved by applying the map $ \CK^{k,r}_{o_2} ( \cdot,n) $. Indeed, one has 
\begin{equs}
\left( \Pi^n \CI^{r}_{o_2}( \lambda^{\ell}_k F) \right) (\tau) & = \CK^{k,r}_{o_2} \left(  \Pi^n \left( \lambda^{\ell} \CD^{r-\ell-1}(F) \right),n \right)(\tau).
\end{equs}
Then by the induction hypothesis, we can assume that $ \Pi^n \left( \lambda^{\ell} \CD^{r-\ell-1}(F) \right)(\xi) $ is a polynomial in $ \xi $ where the coefficients are polynomials in $ k $. Then by applying Definition~\ref{Taylor_exp}, one has:
\begin{equs}
\tilde{g}(\xi) = e^{i \xi P_{o_2}(k)}.
\end{equs}
One can see that if $ n \geq \deg(P_{\Labhom_2}^{r}) $ then one Taylor expands $ \tilde{g} $ which yields a polynomial in $ \tau $ with polynomial coefficients in $ k $. Lemma~\ref{Taylor_bound} implies an error of order 
$ \CO(k^{n}) $. This concludes the proof.
 \end{proof}
 
 \begin{remark}\label{rem:Tay} 
 Proposition \ref{Tay} implies that we indeed recover classical numerical discretisations with our general framework   for smooth solutions. More precisely, one could check that depending on the particular choice of filter functions $\Psi$ (cf. Remark~\ref{rem:stab}) we recover exponential Runge--Kutta  methods and exponential integrators, respectively. For details on the latter we refer to \cite{Berland,ButcherExp,HochOst10} and the references therein. 
  \end{remark}
  
 In the examples in Section \ref{sec:examples} we will also state the local error in physical space. For this purpose we introduce the notation $\varphi^\tau$ and $\Phi^\tau$ which will denote the exact and numerical solution at time $t = \tau$,  i.e., $\varphi^\tau(v) = u(\tau)$ and $\Phi^\tau(v) = u^1\approx u(\tau)$. We write
\begin{equs}\label{lo}
\varphi^\tau (v) - \Phi^\tau(v) = \mathcal{O}_{\Vert \cdot \Vert} \left(\tau^\gamma \tilde{\CL} v\right)
\end{equs}
 if in a suitable norm $\Vert \cdot \Vert$ (e.g., Sobolev norm) it holds that
\begin{equs}
\Vert \varphi^\tau (v) - \Phi^\tau(v)\Vert \leq C(T,d,r) \tau^\gamma \sup_{0 \leq t \leq \tau}  \Vert  q\left(\tilde{\CL} \varphi^t (v)\right)\Vert
\end{equs}
for some polynomial $q$, differential operator $\tilde{\CL}$ and constant $C$ independent of $\tau$. If~\eqref{lo} holds we say that the numerical solution $u^1$ approximates the exact solution $u(t)$ at time $t = \tau$ with a local error of order $\mathcal{O}\left(\tau^\gamma \tilde{\CL} v\right)
$.


\section{Applications}\label{sec:examples}

We illustrate the general framework presented in Section \ref{sec:genScheme} on three concrete examples.  First we consider the nonlinear 
Schrödinger and the Korteweg-de Vries equation, for which we find a new class of second-order resonance based schemes.   For an extensive overview on classical, non-resonance based discretisations we thereby refer to \cite{BeDe02,CanG15,CCO08,CoGa12,Duj09,Faou12,GauLu,HochOst10,HLRS10,HLR12,HKRT12,HKR99,IZ09,Klein06,Lubich08,YMQ88,T74,Ta12} and the references therein. In addition, we illustrate the general framework on a  highly oscillatory system: The Klein--Gordon equation in the so-called non-relativistic limit regime, where the speed of light   formally tends to infinity, see, e.g.,  \cite{BaoKGZUA,BaoZ,BD,BS19,BFS17,ChC}. 
 
\subsection{Nonlinear Schrödinger} \label{sec:nls}

We consider the cubic nonlinear Schrödinger equation
\begin{equation}\label{nls}
i \partial_t u + \Delta u  = \vert u\vert^2 u
\end{equation}
with mild solution given by Duhamel's formula
\begin{equation}\label{DuhNLS}
u(\tau) = e^{i \tau \Delta} u(0) - i e^{i \tau \Delta} \int_0^\tau e^{-i \xi \Delta}
\left(\vert u(\xi)\vert^2 u(\xi)\right)d\xi.
\end{equation}
The Schrödinger equation \eqref{nls} fits into the general framework \eqref{dis} with
\begin{equation*}
\begin{aligned}
\mathcal{L}\left(\nabla, \frac{1}{\varepsilon}\right)  =\Delta, \quad \alpha = 0   \quad \text{and}\quad  p(u,\overline u) = u^2 \overline u.
 \end{aligned}
\end{equation*} 

Here $ \CL = \lbrace \Labhom_1, \Labhom_2 \rbrace $, $ P_{\Labhom_1} = - \lambda^2 $ and $ P_{\Labhom_2} =  \lambda^2 $ (cf \eqref{Palpha}). Then, we denote by~$ \<thick> $ an edge decorated by $ (\Labhom_1,0) $, $ \<thick2> $ an edge denoted by $ (\Labhom_1,1) $ by $\<thin>$ an edge decorated by $ (\Labhom_2,0) $ and by $\<thin2>$ an edge decorated by $ (\Labhom_2,1) $.
The rules that generate the trees obtained by iterating  Duhamel's formulation are given by:
\begin{equ}
R(\<thin>) = \{(\<thick>,\<thick>,\<thick2>) \}\;, \quad  R(\<thin2>) =  \{(\<thick2>,\<thick2>,\<thick>) \}\;,\quad
R(\<thick>) = \{(\<thin>) , ()\}\;,
\quad  R(\<thick2>) = \{(\<thin2>) , ()\}\;.
\end{equ}

\subsubsection{First-order schemes}
The general framework~\eqref{genscheme} derived in Section \ref{sec:genScheme}  builds the foundation of the  first-order resonance based schemes presented below for the nonlinear Schrödinger equation~\eqref{nls}. The structure of the schemes depends  on the regularity of the solution.
\begin{corollary}\label{corNLS1} For the nonlinear Schrödinger equation \eqref{nls} the general scheme~ \eqref{genscheme} takes at first order the form
\begin{align} 
\label{nls1low}
u^{\ell+1}& = e^{i \tau \Delta} u^{\ell}- i \tau e^{i \tau \Delta} \left( (u^\ell)^2 \varphi_1(-2 i \tau \Delta) \overline{u^\ell}\right) 
\end{align}
with a local error of order $\mathcal{O}(\tau^2 \vert \nabla \vert u)$ and  filter function $\varphi_1(\sigma) = \frac{e^\sigma-1}{\sigma}$. 

In case of regular solutions the general scheme~ \eqref{genscheme} takes the simplified form
\begin{align}
\label{nls1class}
u^{\ell+1} & = e^{i \tau \Delta} u^{\ell} - i \tau e^{i \tau \Delta} \left( \vert u^\ell\vert^2 u^\ell \right)  
\end{align}
with a local error of order $\mathcal{O}(\tau^2 \Delta u)$ .

\end{corollary}
\begin{remark}
With the general framework introduced in Section \ref{sec:genScheme} we exactly recover the first-order resonance based  low regularity scheme \eqref{nls1low} proposed in \cite{OS18}. In addition, for smooth solution we recover a classical first-order approximation~\eqref{nls1class}, i.e., the exponential Euler method,  with a classical local error $\mathcal{O}\left(\tau^2 \Delta u\right)$. The low regularity scheme \eqref{nls1low} allows us to solve a larger class of solution due to its favorable error behavior at low regularity.\end{remark}

\begin{proof}

{\bf Construction of the schemes.}
For the first-order scheme we need a local error of order $\mathcal{O}(\tau^2)$. Therefore, we need to choose $r = 0$ in Definition \ref{scheme}  and the corresponding trees accordingly. From Definition \eqref{scheme}, the scheme is given by
\begin{equs}\label{s1}
U_{k}^{n,0}(\tau,v) = \sum_{T \in \CT_0^2(R)} \frac{\Upsilon^{p}( \lambda_k T)(v)}{S(T)} \Pi^n \left( \CD^0(\CI_{(\Labhom_1,0)}( \lambda_k T)) \right)(\tau)
\end{equs}
where one has
\begin{equs}
\CT_0^2(R) = \lbrace  T_0, T_1, \, k_i \in \Z^{d} \rbrace, \quad 
T_0 = \one \quad \text{and} \quad T_1 = \begin{tikzpicture}[scale=0.2,baseline=-5]
\coordinate (root) at (0,0);
\coordinate (tri) at (0,-2);
\coordinate (t1) at (-2,2);
\coordinate (t2) at (2,2);
\coordinate (t3) at (0,3);
\draw[kernels2,tinydots] (t1) -- (root);
\draw[kernels2] (t2) -- (root);
\draw[kernels2] (t3) -- (root);
\draw[symbols] (root) -- (tri);
\node[not] (rootnode) at (root) {};t
\node[not] (trinode) at (tri) {};
\node[var] (rootnode) at (t1) {\tiny{$ k_{\tiny{1}} $}};
\node[var] (rootnode) at (t3) {\tiny{$ k_{\tiny{2}} $}};
\node[var] (trinode) at (t2) {\tiny{$ k_3 $}};
\end{tikzpicture}
\end{equs}

with $T_1$  associated to the first order iterated integral
\begin{equation}\label{I1}
\begin{aligned}
 \CI_1(v^2,\overline v, \xi ) & = \int_0^\xi  e^{-i \xi_1 \Delta}
\left[\left(e^{i \xi_1 \Delta} v\right)^2  \left(e^{-i \xi_1 \Delta} \overline v\right) \right]d\xi_1.
 \end{aligned}
\end{equation}
Hence, our first-order scheme \eqref{s1} takes the form
\begin{equs}\label{nlsorder1}
U_{k}^{n,0}(\tau, v) & =  \frac{\Upsilon^{p}( \lambda_k)(v)}{S(\one)} \Pi^n \left( \CD^0(\CI_{(\Labhom_1,0)}( \lambda_k)) \right)(\tau) \\& +\sum_{\substack{k_1,k_2,k_3\in \Z^d\\-k_1+k_2+k_3 = k}}  \frac{\Upsilon^{p}( \lambda_k T_1)(v)}{S(T_1)} \Pi^n \left( \CD^0(\CI_{(\Labhom_1,0)}( \lambda_k T_1)) \right)(\tau).
\end{equs}
In order to write down the scheme \eqref{nlsorder1} explicitly, we need to compute 
\begin{equs}
&  \Upsilon^{p}( \lambda_k)(v), \quad S(\one) \quad\text{and} \quad \Pi^n \left( \CD^0(\CI_{(\Labhom_1,0)}( \lambda_k ))\right)(\tau)
\\
& \Upsilon^{p}( \lambda_k T_1)(v), \quad S(T_1) \quad\text{and} \quad \Pi^n \left( \CD^0(\CI_{(\Labhom_1,0)}( \lambda_k T_1))\right)(\tau).
\end{equs}
  Note that if we use the symbolic notation, one gets:
\begin{equs}\label{tnls}
T_1  =  \CI_{(\Labhom_2,0)} \left( \lambda_k  F_1 \right) \quad
 F_1  = \CI_{(\Labhom_1,1)}( \lambda_{k_1}) \CI_{(\Labhom_1,0)}( \lambda_{k_2})  \CI_{(\Labhom_1,0)}( \lambda_{k_3}),\\ \quad k = -k_1+k_2+k_3.
\end{equs}
Thanks to the definition of $ \Upsilon^{p}, S$ and Example \ref{ex:SUpsNLS1} we already know that
\begin{equs}
 & \Upsilon^{p}( \lambda_k)(v) = \hat v_k , \quad S(\one) =1,\quad  \Upsilon^{p}( \lambda_k T_1)(v) = 2 \overline{\hat{v}}_{k_1}\hat{v}_{k_2}\hat{v}_{k_3}, \quad S(T) = 2.
\end{equs}
Hence, the scheme \eqref{nlsorder1} takes the form
\begin{equs}\label{nlsorder1a}
U_{k}^{n,0}(\tau, v) & = \hat{v}_k \Pi^n \left( \CD^0(\CI_{(\Labhom_1,0)}( \lambda_k)) \right)(\tau) \\& +  \sum_{\substack{k_1,k_2,k_3\in \Z^d\\-k_1+k_2+k_3 = k}}\overline{\hat{v}}_{k_1}\hat{v}_{k_2}\hat{v}_{k_3} \Pi^n \left( \CD^0(\CI_{(\Labhom_1,0)}( \lambda_k T_1)) \right)(\tau).
\end{equs}
It remains to compute $\Pi^n \left( \CD^0(\CI_{(\Labhom_1,0)}( \lambda_kT_j)) \right)(\tau)$ for $j=0,1$  with $T_0 = \one$ and $T_1$ given  in  \eqref{tnls}.\\
1. {Computation of $\Pi^n \left( \CD^0(\CI_{(\Labhom_1,0)}( \lambda_k)) \right)(\tau)$.}
The second line in \eqref{recursive_pi_r} together with $P_{\Labhom_1}(k) = - k^2$ implies that
\[
\Pi^n \left( \CD^0(\CI_{(\Labhom_1,0)}( \lambda_k)) \right)(\tau)  = \Pi^n( \CI^0_{(\Labhom_1,0)}( \lambda_k))(\tau)
= e^{i \tau P_{\Labhom_1}(k)} = e^{-i \tau k^2}.
\]
2. {Computation of $\Pi^n \left( \CD^0(\CI_{(\Labhom_1,0)}( \lambda_kT_1)) \right)(\tau)$.}
The definition of the tree $T_1$ in \eqref{tnls} furthermore  implies that
\begin{equation}\label{a1}
\begin{aligned}
\Pi^n \left( \CD^0(\CI_{(\Labhom_1,0)}( \lambda_k T_1))\right)(\tau) & = (\Pi^n  \CI^{0}_{(\Labhom_1,0)}( \lambda_{k} T_1  ))(\tau)    =e^{i \tau P_{\Labhom_1}(k)} (\Pi^n \CD^{0}(T_1))(\tau) 
\\&   = e^{- i \tau k^2} (\Pi^n
\CI^0_{(\Labhom_2,0)} \left( \lambda_k  F_1 \right) 
)(\tau).
\end{aligned}
\end{equation}
Furthermore, by the third line in \eqref{recursive_pi_r} we have
\begin{equs}
\left( \Pi^n \CI^{0}_{(\Labhom_2,0)}( \lambda_k F_1) \right) (\tau) & = \CK^{k,0}_{(\Labhom_2,0)} \left(  \Pi^n  \CD^{-1}(F_1) ,n \right)(\tau) .
\end{equs}
By the multiplicativity of $\Pi^n$ (see \eqref{recursive_pi_r}) we furthermore obtain
\begin{equs}
 \Pi^n   \CD^{-1}(F_1) (\xi) & =  \left(\Pi^n \CI^{-1}_{(\Labhom_1,1)}( \lambda_{k_1}) \right)(\xi) \left(\Pi^n \CI^{-1}_{(\Labhom_1,0)}( \lambda_{k_2}) \right)(\xi) \left(\Pi^n \CI^{-1}_{(\Labhom_1,0)}( \lambda_{k_3})\right)(\xi) \\
 & = e^{(-1)i \xi P_{\Labhom_1}(-k_1)} \left(\Pi^n \one \right)(\xi)  e^{i \xi P_{\Labhom_1}(k_2)} \left(\Pi^n \one \right)(\xi)  e^{i \xi P_{\Labhom_1}(k_3)} \left(\Pi^n \one \right)(\xi)  \\
 & = e^{ i \xi (k_1^2- k_2^2-k_3^2)}
\end{equs}
where we used again that $P_{\Labhom_1}(k_j) = - k_j^2$. Collecting the results and plugging them into \eqref{a1} yields that
\begin{equs}\label{a2}
(\Pi^n  \CI^{0}_{(\Labhom_1,0)}( \lambda_{k} T_1  ))(\tau)  = e^{-i \tau k^2}
\CK^{k,0}_{(\Labhom_2,0)} \left( e^{ i \xi (k_1^2- k_2^2-k_3^2)},n\right)(\tau).
\end{equs}
Next we use Definition \ref{Taylor_exp}. Observe that (see Example  \ref{exRDoNLS}),
\begin{equs}\label{lowdom1}
& \mathcal{L}_{\text{\tiny{dom}}} = 2 k_1^2, \quad \mathcal{L}_{\text{\tiny{low}}}  = - 2 k_1 (k_2+k_3) + 2 k_2 k_3\\
& f(\xi) = e^{i \xi\mathcal{L}_{\text{\tiny{dom}}}}, \quad g(\xi) = e^{i \xi\mathcal{L}_{\text{\tiny{low}} }}
\end{equs}
and thus as $g(0) = 1$ we have 
\begin{equation*}\label{K0}
\CK^{k,0}_{(\Labhom_2,0)} \left( e^{ i \xi (k_1^2- k_2^2-k_3^2)},n\right)(\tau)
= \left\{ \begin{aligned}
& - i \tau,\,\quad \text{if } n\geq 2\\
& -i\Psi_{n,0}^0 \left(\mathcal{L}_\text{dom}, 0\right)(\tau),\,  \quad \text{if } n < 2 
  \\
  \end{aligned} \right.
\end{equation*}
with
\begin{equ}\label{Psi0}
 \Psi^{0}_{n,0}\left( \CL_{\text{\tiny{dom}}},0 \right)(\tau) =
     \int_0^{\tau}   f(\xi)   d \xi = \frac{e^{2 i \tau k_1^2}-1}{2 i k_1^2} = \tau \varphi_1(2i \tau k_1^2), 
 \quad  \text{if } n < 2. 
\end{equ}
Plugging this into \eqref{a2}  yields together with \eqref{a1} and \eqref{nlsorder1a} that
\begin{equs}\label{Fscheme1}
\begin{aligned}
U_{k}^{n= 1,0}(\tau, v) & = e^{-i \tau k^2}  \hat{v}_k -i \tau \sum_{\substack{k_1,k_2,k_3\in \Z^d\\-k_1+k_2+k_3 = k}}\overline{\hat{v}}_{k_1}\hat{v}_{k_2}\hat{v}_{k_3} e^{-i \tau k^2} \varphi_1(2i \tau k_1^2) \\
U_{k}^{n > 1,0}(\tau, v) & =e^{-i \tau k^2}  \hat{v}_k - i \tau \sum_{\substack{k_1,k_2,k_3\in \Z^d\\-k_1+k_2+k_3 = k}}   \overline{\hat{v}}_{k_1}\hat{v}_{k_2}\hat{v}_{k_3} e^{-i \tau k^2}.
\end{aligned}
\end{equs}
Note that the above Fourier based resonance schemes can easily be transformed back to physical space yielding 
the    two first-order iterative schemes  \eqref{nls1low} and \eqref{nls1class}  which depend on the smoothness $n$ of the exact solution.

\noindent {\bf Local error analysis.} 
It remains to show that   the general local error bound
stated in Theorem \ref{thm:genloc} implies the  local error 
\begin{equs}\label{loc1}
\mathcal{O}\left(\tau^2 \nabla^n\right), \quad n = 1,2
\end{equs}
for the schemes \eqref{nls1low} (with $n=1$) and \eqref{nls1class} (with $n=2$), respectively. Theorem~\ref{thm:genloc} implies that
\begin{equs}
U_{k}^{n,0}(\tau,v) - U_{k}^{0}(\tau,v) = \sum_{ T \in \CT_0^2(R)} \CO\left(\tau^{2} \CL^{0}_{\text{\tiny{low}}}(T,n)  \Upsilon^{p}( \lambda_k T)(v) \right).
\end{equs}
 By Definition \ref{def:Llow}  and Remark \ref{rem:simpL} we have that  $ \mathcal{L}^{0}_{\text{\tiny{low}}}(T_0) = 1$ and 
\begin{equs} 
& \mathcal{L}^{0}_{\text{\tiny{low}}}(T_1) =
\mathcal{L}^{0}_{\text{\tiny{low}}}(\CI_{(\Labhom_2,0)}( \lambda_{k}  F_1 ),n) 
= \mathcal{L}^{-1}_{\text{\tiny{low}}}(F_1,n) + \sum_j k^{\overline n_j}\\
& \overline{n}_j  =  \max(n, \deg( \mathscr{F}_{\text{\tiny{low}}} (\CI_{(\Labhom_2,0)}( \lambda^{\ell}_{k}  F_j ))))
\end{equs} 
where $\sum_j F_j  =  \CM_{(1)} \Delta F_1$ with $\CM_{(1)} \left( F_1 \otimes F_2 \right) = F_1$. Note that by \eqref{CoF1} we have that
$
\Delta F_1 = F_1 \otimes \one
$ such that  $\sum_j F_j = F_1$. Hence,
\begin{equs}
\overline{n}_j = \overline{n}_1  =  \max(n, \deg( \mathscr{F}_{\text{\tiny{low}}} (\CI_{(\Labhom_2,0)}( \lambda_{k}  F_1)))).
\end{equs}
By Definition \ref{dom_freq} and Example \ref{exRDoNLS} we obtain that 
\begin{equs}
 \mathscr{F}_{\text{\tiny{low}}} (\CI_{(\Labhom_2,0)}( \lambda_{k}  F_1)
=  - 2k_1(k_2+k_3) + 2 k_2 k_3.
\end{equs}
Hence, $\overline n_1 = \max(n, 1)$ and 
\begin{equs}
 \mathcal{L}^{0}_{\text{\tiny{low}}}(T_1)
& = \mathcal{L}^{-1}_{\text{\tiny{low}}}(F_1,n) + k^{\max(n, 1)} \\
& = \mathcal{L}^{-1}_{\text{\tiny{low}}}\left(\CI_{(\Labhom_1,1)}( \lambda_{k_1}) ,n\right) +
\mathcal{L}^{-1}_{\text{\tiny{low}}}\left(\CI_{(\Labhom_1,0)}( \lambda_{k_2}) ,n\right)\\ & +
\mathcal{L}^{-1}_{\text{\tiny{low}}}\left(\CI_{(\Labhom_1,0)}( \lambda_{k_3}) , n\right) + k^{\max(n, 1)}\\
& = \mathcal{L}^{-1}_{\text{\tiny{low}}}\left(\one,n\right) +
\mathcal{L}^{-1}_{\text{\tiny{low}}}\left(\one,n \right)
+
\mathcal{L}^{-1}_{\text{\tiny{low}}}\left(\one,n \right) + k^{\max(n, 1)}.
\end{equs}
Using that $\mathcal{L}^{-1}_{\text{\tiny{low}}}\left(\one,n\right) =1 $ we finally obtain that
$
 \mathcal{L}^{0}_{\text{\tiny{low}}}(T_1) = \CO( k^{\max(n, 1)}).
$
Therefore,
we recover \eqref{loc1}.
\end{proof}

\subsubsection{Second-order approximation}\label{sec:2NLS}
The general framework~\eqref{genscheme} derived in Section \ref{sec:genScheme}  builds the foundation of the  second-order resonance based schemes presented below for the nonlinear Schrödinger equation~\eqref{nls}. The structure of the schemes depends  on the regularity of the solution.
\begin{corollary}\label{corNLS2} For the nonlinear Schrödinger equation \eqref{nls} the general scheme~ \eqref{genscheme} takes at second order the form
\begin{equs}\label{nls2low}
 u^{\ell+1} & = e^{i \tau \Delta} u^\ell   - i \tau e^{i \tau \Delta}
\Big(
(u^\ell)^2 \left(\varphi_1(-2i \tau \Delta) - \varphi_2(-2i \tau \Delta)\right) \overline{u^\ell}{ \Big)} \\&  - i \tau  
\left(e^{i \tau \Delta} u^\ell\right)^2 \varphi_2 (-2i \tau \Delta)e^{i \tau \Delta} \overline{u^\ell}
 -\frac{\tau^2}{2} e^{i \tau \Delta} \left( \vert u^\ell\vert^4 u^\ell\right)
\end{equs}
with a local error of order $\mathcal{O}(\tau^3 \Delta u)$ and filter function $\varphi_2(\sigma) = \frac{e^\sigma-\varphi_1(\sigma)}{\sigma}$.  

In case of more regular solutions the general scheme \eqref{genscheme} takes the simplified form
\begin{equation}\label{nls2mid}
\begin{aligned}
& u^{\ell+1} = e^{i \tau \Delta} u^\ell \\& - i \tau e^{i \tau \Delta}
\left(
(u^\ell)^2 \left(\varphi_1 (-2i \tau \Delta)-\frac{1}{2}\right) \overline{u^\ell} + 
\frac{1}{2} \left(e^{i \tau \Delta} u^\ell\right)^2e^{i \tau \Delta} \overline{u^\ell}
\right)\\
& -\frac{\tau^2}{2} e^{i \tau \Delta} \left( \vert u^\ell\vert^4 u^\ell\right)
\end{aligned}
\end{equation}
with a local error of order $\mathcal{O}(\tau^3 \nabla^3 u)$, and for smooth solutions
\begin{equs}\label{nls2smooth}
 u^{\ell+1} 
  & = e^{i \tau \Delta} u^\ell  - i \tau e^{i \tau \Delta} \left(\Psi_1 - i \Psi_2 \frac12 \tau \Delta \right) 
 \vert u^\ell\vert^2 u^\ell 
  \\ &  +  \frac{\tau^2}{2} e^{i \tau \Delta}\Psi_3 \left(
 - (u^\ell)^2 \Delta  \overline{u^\ell} + 2 \vert u^\ell\vert^2 \Delta u^\ell-  \vert u^\ell\vert^4 u^\ell  \right)
\end{equs}
with a local error of order $\mathcal{O}(\tau^3 \Delta^2 u)$ and   suitable filter functions $\Psi_{1,2,3}$ satisfying
\begin{equs}
\Psi_{1,2,3}= \Psi_{1,2,3}\left(i \tau \Delta\right), \quad \Psi_{1,2,3}(0) = 1, \quad \Vert \tau   \Psi_{1,2,3}\left(i \tau \Delta\right) \Delta \Vert \leq 1.
\end{equs}

\end{corollary}
\begin{remark}\label{rem:nls2}
With the general framework we had recovered at first-order exactly the resonance based first-order scheme \eqref{nls1low} derived in \cite{OS18}. The second-order schemes \eqref{nls2low} and  \eqref{nls2mid} are new and allow us to improve the classical local error structure $\mathcal{O}\left(
\tau^3 \Delta^4 u
\right)$. We refer to \cite{Lubich08,CoGa12} for the error analysis of classical splitting and exponential integrators for the Schrödinger equation.

 The new second-order low regularity scheme \eqref{nls2low} moreover allows us to improve  the recently introduced scheme \cite{KOS19}, which  only allows for a local error of order $\mathcal{O}\left(\tau^{2+1/2} \Delta u \right)$. Thus we break the order barrier of $3/2$ previously assumed for resonance based approximations for Schrödinger equations.
  
With the new framework we in addition recover for smooth solutions  classical second-order Schrödinger approximations obeying the  classical local error structure $\mathcal{O}\left(\tau^2 \Delta^2 u\right)$: Depending on the choice of filter functions $\Psi_{1,2,3}$ the second-order scheme \eqref{nls2smooth} coincides with second-order exponential Runge--Kutta or exponential integrator  methods (\cite{Berland,ButcherExp,HochOst10}), see  also Remark \ref{rem:Tay}.  The favorable error behavior of the new schemes for non-smooth solutions is underlined by numerical experiments, see Figure~\ref{fig:NLS}.

\end{remark}

\begin{proof}
{\bf Construction of the schemes.} 
For the second-order scheme we need a local error of order $\mathcal{O}(\tau^3)$. Therefore, we need to choose $r = 1$ in Definition \ref{scheme} and the corresponding trees accordingly.  This yields that
\begin{equs}\label{nlsorder2}
U_{k}^{n,1}(\tau, v) & =  \frac{\Upsilon^{p}( \lambda_k)(v)}{S( \lambda_k)} \Pi^n \left( \CD^1(\CI_{(\Labhom_1,0)}( \lambda_k)) \right)(\tau) \\& +\sum_{\substack{k_1,k_2,k_3\in \Z^d\\-k_1+k_2+k_3 = k}} \frac{\Upsilon^{p}( \lambda_k T_1)(v)}{S(T_1)} \Pi^n \left( \CD^1(\CI_{(\Labhom_1,0)}( \lambda_k T_1)) \right)(\tau)
\\& +\sum_{\substack{k_1,k_2,k_3,k_4,k_5\in \Z^d\\-k_1+k_2+k_3 -k_4+k_5= k}} \frac{\Upsilon^{p}( \lambda_k T_2)(v)}{S(T_2)} \Pi^n \left( \CD^1(\CI_{(\Labhom_1,0)}( \lambda_k T_2)) \right)(\tau)
\\& +\sum_{\substack{k_1,k_2,k_3,k_4,k_5\in \Z^d\\k_1-k_2-k_3 +k_4+k_5= k}} \frac{\Upsilon^{p}( \lambda_k T_3)(v)}{S(T_3)} \Pi^n \left( \CD^1(\CI_{(\Labhom_1,0)}( \lambda_k T_3)) \right)(\tau)
\end{equs}
with $T_1$ defined in \eqref{tnls} and
\begin{equs}
\CT_0^3(R) = \lbrace  T_0, T_1, T_2, T_3, \, k_i \in \Z^{d} \rbrace, \quad
T_2 = \begin{tikzpicture}[scale=0.2,baseline=-5]
\coordinate (root) at (0,0);
\coordinate (tri) at (0,-2);
\coordinate (t1) at (-2,2);
\coordinate (t2) at (2,2);
\coordinate (t3) at (0,2);
\coordinate (t4) at (0,4);
\coordinate (t41) at (-2,6);
\coordinate (t42) at (2,6);
\coordinate (t43) at (0,8);
\draw[kernels2,tinydots] (t1) -- (root);
\draw[kernels2] (t2) -- (root);
\draw[kernels2] (t3) -- (root);
\draw[symbols] (root) -- (tri);
\draw[symbols] (t3) -- (t4);
\draw[kernels2,tinydots] (t4) -- (t41);
\draw[kernels2] (t4) -- (t42);
\draw[kernels2] (t4) -- (t43);
\node[not] (rootnode) at (root) {};
\node[not] (rootnode) at (t4) {};
\node[not] (rootnode) at (t3) {};
\node[not] (trinode) at (tri) {};
\node[var] (rootnode) at (t1) {\tiny{$ k_{\tiny{4}} $}};
\node[var] (rootnode) at (t41) {\tiny{$ k_{\tiny{1}} $}};
\node[var] (rootnode) at (t42) {\tiny{$ k_{\tiny{3}} $}};
\node[var] (rootnode) at (t43) {\tiny{$ k_{\tiny{2}} $}};
\node[var] (trinode) at (t2) {\tiny{$ k_5 $}};
\end{tikzpicture}, \quad
T_3 = \begin{tikzpicture}[scale=0.2,baseline=-5]
\coordinate (root) at (0,0);
\coordinate (tri) at (0,-2);
\coordinate (t1) at (-2,2);
\coordinate (t2) at (2,2);
\coordinate (t3) at (0,2);
\coordinate (t4) at (0,4);
\coordinate (t41) at (-2,6);
\coordinate (t42) at (2,6);
\coordinate (t43) at (0,8);
\draw[kernels2] (t1) -- (root);
\draw[kernels2] (t2) -- (root);
\draw[kernels2,tinydots] (t3) -- (root);
\draw[symbols] (root) -- (tri);
\draw[symbols,tinydots] (t3) -- (t4);
\draw[kernels2] (t4) -- (t41);
\draw[kernels2,tinydots] (t4) -- (t42);
\draw[kernels2,tinydots] (t4) -- (t43);
\node[not] (rootnode) at (root) {};
\node[not] (rootnode) at (t4) {};
\node[not] (rootnode) at (t3) {};
\node[not] (trinode) at (tri) {};
\node[var] (rootnode) at (t1) {\tiny{$ k_{\tiny{4}} $}};
\node[var] (rootnode) at (t41) {\tiny{$ k_{\tiny{1}} $}};
\node[var] (rootnode) at (t42) {\tiny{$ k_{\tiny{3}} $}};
\node[var] (rootnode) at (t43) {\tiny{$ k_{\tiny{2}} $}};
\node[var] (trinode) at (t2) {\tiny{$ k_5 $}};
\end{tikzpicture}.
\end{equs}
In symbolic notation one gets
\begin{equs}\label{nlst23}
\begin{aligned}
& T_2  =  \CI_{(\Labhom_2,0)} \left( \lambda_k  F_2 \right), \quad k = -k_1+k_2+k_3-k_4+k_5
\\ &   F_2  = \CI_{(\Labhom_1,1)}( \lambda_{k_4}) \CI_{(\Labhom_1,0)} \left( \lambda_{-k_1+k_2+k_3} T_1 \right)\CI_{(\Labhom_1,0)}( \lambda_{k_5}), 
\\
 &   T_3 =   \CI_{(\Labhom_2,0)} \left( \lambda_k  F_3 \right), \quad k = k_1-k_2-k_3+k_4+k_5
\\ &   F_3  = \CI_{(\Labhom_1,0)}( \lambda_{k_4}) \overline{\CI_{(\Labhom_1,0)}\left(  \lambda_{-k_1+k_2+k_3} T_1 \right)}\CI_{(\Labhom_1,0)}( \lambda_{k_5}), 
\end{aligned}
\end{equs}
where thanks to \eqref{bar} 
\begin{equs}
\overline{\CI_{(\Labhom_1,0)}\left(  \lambda_{-k_1+k_2+k_3} T_1 \right)}\ = \CI_{(\Labhom_1,1)}\left(  \lambda_{-k_1+k_2+k_3} \overline{T_1} \right).
\end{equs}

Note that the trees $T_2,T_3$ correspond to the next iterated integrals
\begin{equation}\label{I2}
\begin{aligned}
 \CI_2( v^3,\overline v^2, \xi ) & =  \int_0^\xi  e^{-i \xi_1 \Delta} \left[\left(e^{i \xi_1 \Delta} v\right) \left(e^{-i \xi_1 \Delta} \overline v\right)  \left( e^{i \xi_1 \Delta}  \CI_1( v^2,\overline v, \xi_1 ) \right) \right]  d\xi_1 \\
 \CI_3( v^3,\overline v^2, \xi ) & =  \int_0^\xi  e^{-i \xi_1 \Delta} \left[\left(e^{i \xi_1 \Delta} v\right)^{2}   \left( e^{-i \xi_1 \Delta}  \overline{\CI_1( v^2,\overline v, \xi_1 )} \right) \right]  d\xi_1.
 \end{aligned}
\end{equation}

The definitions \eqref{upsi} imply  that
\begin{equs}
 \Upsilon^{p}( \lambda_k T_2)(v)  & = 2\Upsilon^{p}\left( \CI_{(\Labhom_1,1)}( \lambda_{k_4} ) \right) (v)
 \Upsilon^{p}\left( \CI_{(\Labhom_1,0)}( \lambda_{k} T_1) \right) (v)  \Upsilon^{p}\left( \CI_{(\Labhom_1,0)}( \lambda_{k_5} ) \right) (v)\\
 & = 2\overline{\hat{v}}_{k_4} \left( 2 \overline{\hat{v}}_{k_1}\hat{v}_{k_2}\hat{v}_{k_3}\right) 
{\hat{v}}_{k_5} \\
 \Upsilon^{p}( \lambda_k T_3)(v) & = 2 \Upsilon^{p}\left( \CI_{(\Labhom_1,0)}( \lambda_{k_4} ) \right) (v) \Upsilon^{p}\left(  \lambda_{k} \overline{T_1} \right)(v)\Upsilon^{p}\left( \CI_{(\Labhom_1,0)}( \lambda_{k_5} ) \right) (v)\\
 & = 2 \hat v_{k_4}\overline{ \left( 2 \overline{\hat{v}}_{k_1}\hat{v}_{k_2}\hat{v}_{k_3}\right) }  \hat v_{k_5} = 2  \hat v_{k_4} \left(2 \hat v_{k_1} \overline{\hat{v}}_{k_2}  \overline{\hat{v}}_{k_3}\right) \hat v_{k_5} 
\end{equs}
and by  \eqref{S} we obtain
\begin{equs}
S(T_2) = 1 \cdot 2 = 2, \quad S(T_3) = 2 \cdot 2 = 4.
\end{equs}
Next we have to compute $ \Pi^n \left( \CD^1(\CI_{(\Labhom_1,0)}( \lambda_k T_j)) \right)$ for $j=1, 2,3$.\\
\noindent 1. {Computation of $ \Pi^n \left( \CD^1(\CI_{(\Labhom_1,0)}( \lambda_k T_1)) \right)$:} Here $k= - k_1+k_2+k_3$. Similarly to \eqref{a2} we obtain that
\begin{equ}\label{a3}
\Pi^n \left( \CD^1(\CI_{(\Labhom_1,0)}( \lambda_k T_1)) \right)(\tau)  = e^{-i \tau k^2}
\CK^{k,1}_{(\Labhom_2,1)} \left( e^{ i \xi (k_1^2- k_2^2-k_3^2)},n\right)(\tau),
\end{equ}
where by \eqref{lowdom1} we have that
if $ n\geq 4$
\begin{equs}
 \CK^{k,1}_{(\Labhom_2,1)} \left( e^{ i \xi (k_1^2- k_2^2-k_3^2)},n\right)(\tau)  & = 
 -i \tau  + \left(k^2 + k_1^2 - k_2^2-k_3^2\right) \frac{\tau^2}{2} .
  \end{equs}
If $ n < 4 $ 
  \begin{equs}
 \CK^{k,1}_{(\Labhom_2,1)} \left( e^{ i \xi (k_1^2- k_2^2-k_3^2)},n\right)(\tau)  = 
&-i  \Psi_{n,0}^1(\mathcal{L}_{\text{\tiny{dom}}},0)(\tau) -i  \frac{g(\tau)-1}{\tau} \Psi_{n,0}^1(\mathcal{L}_{\text{\tiny{dom}}},1)(\tau) ,
  \end{equs} 
with
\begin{equ} \Psi^{1}_{n,0}\left( \CL_{\text{\tiny{dom}}},\ell \right)(\tau) = \left\{ \begin{aligned}
  &   \int_0^{\tau}   \xi^{\ell} f(\xi)   d \xi, \,
  \quad \text{if } 4 -  \ell > n \text{ and } n < 4 , \\
  & \sum_{m \leq 1-\ell} \frac{f^{(m)}(0)}{m!} \int_0^{\tau}   \xi^{\ell+m} d \xi, \quad \text{if } 4 -  \ell \leq n \text{ and } n < 4.
  \\
  \end{aligned} \right.
\end{equ}
Hence,\begin{equation}\label{simpNLS}
\begin{aligned}
\Pi^{n=2} \left( \CD^1(\CI_{(\Labhom_1,0)}( \lambda_k T_1)) \right)(\tau)  &=-i \tau  e^{-i \tau k^2}
\Big(
\varphi_1(2i \tau k_1^2) + \left(g(\tau)-1\right) \varphi_2(2i \tau k_1^2) 
\Big)\\
  \Pi^{n= 3} \left( \CD^1(\CI_{(\Labhom_1,0)}( \lambda_k T_1)) \right)(\tau)  & = -i \tau e^{-i \tau k^2}
\Big( \varphi_1(2i \tau k_1^2) + \frac{1}{2}\left(g(\tau)-1\right)\Big)\\
  \Pi^{n \geq 4} \left( \CD^1(\CI_{(\Labhom_1,0)}( \lambda_k T_1)) \right)(\tau) 
& = e^{-i \tau k^2}
\Big(-i \tau  + \left(k^2 + k_1^2 - k_2^2-k_3^2\right) \frac{\tau^2}{2} \Big)
\end{aligned}
\end{equation}
where $g(\tau) = e^{i \tau (k^2 - k_1^2-k_2^2-k_3^2)}$ and $\mathcal{L}(k) = \mathcal{L}_\text{dom}(k) + \mathcal{L}_\text{low}(k)$.

\noindent 2. {Computation of $ \Pi^n \left( \CD^1(\CI_{(\Labhom_1,0)}( \lambda_k T_2)) \right)$:} Here $k = -k_1+k_2+k_3-k_4+k_5$. By \eqref{recursive_pi_r} and $ P_{\Labhom_1}(k) = -k^2$ we have
\begin{equs}
 \Pi^n \left( \CD^1(\CI_{(\Labhom_1,0)}( \lambda_k T_2)) \right)(\tau) & =
 e^{-i \tau k^2 } \left( \Pi^n \CD^{1}(T_2)\right)(\tau)
 =  e^{-i \tau k^2 } \left( \Pi^n \CI^1_{(\Labhom_2,0)} \left( \lambda_k  F_2 \right) \right)(\tau)\\
 & = e^{-i \tau k^2 } \CK_{(\Labhom_2,0)}^{k,1}\left(\Pi^n\left(\CD^{0}(F_2)\right),n\right)(\tau).
\end{equs}
Furthermore, by the mulitiplicativity of $\Pi^n$ we obtain with the aid of \eqref{a2} and~\eqref{K0}
\begin{align*}
\Pi^n\left(\CD^{0}(F_2)\right) (\tau)& =  \left(  \Pi^n  \CI^0_{(\Labhom_1,1)}( \lambda_{k_4}) \right)(\tau) \left(\Pi^n \CI^0_{(\Labhom_1,0)} \left( \lambda_{k}T_1 \right)\right)(\tau) \left(\Pi^n \CI^0_{(\Labhom_1,0)}( \lambda_{k_5})\right) (\tau)\\
&= e^{-i \tau k_4^2} e^{-i \tau (-k_1+k_2+k_3)^2}
\CK^{-k_1+k_2+k_3,0}_{(\Labhom_2,0)} \left( e^{ i \xi (k_1^2- k_2^2-k_3^2)},n\right)(\tau) e^{-i \tau k_5^2} \\
&=  - i e^{-i \tau (k_4^2+k_5^2)} e^{-i \tau (-k_1+k_2+k_3)^2}  \Psi_{n,0}^0 \left(\mathcal{L}_{\text{\tiny{dom}}}, 0\right)(\tau)
\end{align*}
where by \eqref{Psi0} and the fact that $n \geq 2$ we have that
\begin{equs}
 \Psi_{n,0}^0 \left(\mathcal{L}_{\text{\tiny{dom}}}, 0\right)(\tau) = \tau.
\end{equs}
Hence,
\begin{align} \label{computescheme2}
\begin{aligned}
\Pi^n \left( \CI^{1}_{(\Labhom_1,0)}( \lambda_k T_2) \right)(\tau)  & = -i 
e^{-i \tau k^2 }  \CK_{(\Labhom_2,0)}^{k,1}\left(\xi e^{-i \xi (k_4^2+k_5^2)} e^{-i \xi (-k_1+k_2+k_3)^2}
,n\right)(\tau)\\
& =-  e^{-i \tau k^2 }  \frac{\tau^2}{2}
\end{aligned}
\end{align}
where we used again that $n \geq 2$.

\noindent 3. {Computation of $ \Pi^n \left( \CD^1(\CI_{(\Labhom_1,0)}( \lambda_k T_3)) \right)$:} Here $k = k_1-k_2-k_3+k_4+k_5$. Similarly we can show that 
\begin{align*}
\Pi^n \left( \CD^1(\CI_{(\Labhom_1,0)}( \lambda_k T_3)) \right)(\tau) & = +   e^{-i \tau k^2 }  \frac{\tau^2}{2}.
\end{align*}

Plugging the results from Computations 1--3 into \eqref{nlsorder2} yields that
\begin{equs}\label{Fscheme2}
U_k^{n,1}(\tau,v) & = e^{-i \tau k^2} \hat v_k
\\& -i \tau \sum_{-k_1+k_2+k_3 = k}e^{-i \tau k^2} \frac{1}{\tau} \Pi^{n} \left( \CI^{1}_{(\Labhom_1,0)}( \lambda_k T_1) \right) (\tau) \, \overline{\hat{v}}_{k_1} \hat v_{k_2} \hat v_{k_2} \\
& - \frac{\tau^2}{2} \sum_{-k_1+k_2+k_3-k_4+k_5 = k} e^{-i \tau k^2}\overline{\hat{v}}_{k_1}  \hat v_{k_2} \hat v_{k_2}
\overline{\hat{v}}_{k_4}  \hat v_{k_5}
\end{equs}
with $\Pi^{n} \left( \CD^1(\CI_{(\Labhom_1,0)}( \lambda_k T_1)) \right)$ given in \eqref{simpNLS} and we have used that the last two sums in \eqref{nlsorder2} can be merged into one.
The  Fourier based resonance schemes \eqref{Fscheme2} can easily be transformed back to physical space yielding the   three  low-to-high regularity second-order iterative schemes  \eqref{nls2low} -- \eqref{nls2smooth}  which depend on the smoothness $n$ of the exact solution.

{\bf Local error analysis.} It remains to show that  the general local error bound
given in Theorem \ref{thm:genloc} implies the local error
\begin{equs}\label{loc2}
\mathcal{O}\left( \tau^3 \nabla^n\right), \quad n = 2,3,4
\end{equs}
of the schemes \eqref{nls2low} -- \eqref{nls2smooth}.  Theorem \ref{thm:genloc} implies that
\begin{equs}\label{locoErr2}
U_{k}^{n,1}(\tau,v) - U_{k}^{1}(\tau,v) = \sum_{T \in \CT_0^3(R)  } \CO\left(\tau^{3} \CL^{1}_{\text{\tiny{low}}}(T,n)  \Upsilon^{p}( \lambda_k T)(v) \right)
\end{equs}
where $\CT_0^3(R)   = \{ T_0, T_1, T_2, T_3\}$ with $T_0 = \one$, $T_1$ given  in  \eqref{tnls} and $T_2, T_3$ defined in \eqref{nlst23}.

 1) Computation of $\CL^{1}_{\text{\tiny{low}}}(T_0, n)$ and $\CL^{1}_{\text{\tiny{low}}}(T_1, n)$:\\ By Definition \ref{def:Llow}  and Remark  \ref{rem:simpL} we obtain similarly to the first-order scheme (here with  $r=1$) that 
\begin{equs}
\mathcal{L}^1_{\text{\tiny{low}}}\left(T_0,n\right) = 1, \quad \mathcal{L}^1_{\text{\tiny{low}}}\left(T_1,n\right) = k^{\text{\tiny{max}}(n,2)}= k^{\text{\tiny{max}}(n,2)}.
\end{equs}

 2) Computation of $\CL^{1}_{\text{\tiny{low}}}(T_2, n)$:\\
Next we calculate ( using again Definition \ref{def:Llow}  and Remark  \ref{rem:simpL}) that
\begin{equs}\label{L1low}
\mathcal{L}^1_{\text{\tiny{low}}}\left(T_2,n\right) & = 
\mathcal{L}^1_{\text{\tiny{low}}}\left(\CI_{(\Labhom_2,0)}( \lambda_k F_2),n\right) 
= \mathcal{L}^0_{\text{\tiny{low}}}\left(F_2,n\right) + \sum_{j} k^{\bar n_j}
\end{equs}
with
\[
 \bar n_j  = \max(n, \deg( \mathscr{F}_{\text{\tiny{low}}} (\CI_{(\Labhom_2,0)}( \lambda_{k}  F^j_2 ))^{2}))
\]
where $ \sum_j F_2^j  =  \CM_{(1)} \Delta \CD^{r-1} (F_2)$ with $r=1$.  Hence, we have to calculate $\Delta \CD^{0} (F_2)$. By 
the multiplicativity of the coproduct, its recursive definition~\eqref{def_deltas} and the calculation of $\CD^r(T_1)$ given in \eqref{comps} we obtain that
\begin{equs}
\Delta \CD^r (F_2) & =\Delta \CI^r_{(\Labhom_1,1)}\left( \lambda_{k_4}\right) 
\Delta \CI^r_{(\Labhom_1,0)}\left( \lambda_{-k_1+k_2+k_3}T_1\right)\Delta \CI^r_{(\Labhom_1,0)}\left( \lambda_{k_5}\right)\\
&= \left(\CI^r_{(\Labhom_1,1)}\left( \lambda_{k_4}\right) \otimes \one \right)\left(  \CI^{r}_{(\Labhom_1,0)}( \lambda_{-k_1+k_2+k_3}\cdot)  \otimes \id  \right) \\& \qquad\Delta \CD^{r}(T_1))
 \left(\CI^r_{(\Labhom_1,0)}\left( \lambda_{k_5}\right) \otimes \one \right)\\
& =   \CD^r(F_2) \otimes \one 
\\&+  \CI^r_{(\Labhom_1,1)}\left( \lambda_{k_4}\right) \CI^r_{(\Labhom_1,0)}\left( \lambda_{k_5}\right) \CI^{r}_{(\Labhom_1,0)}\left(\sum_{m \leq r +1 } \frac{ \lambda_{-k_1+k_2+k_3}^m}{m!} \right) \otimes  \CD^{(r,m)}( T_1)  .
\end{equs}
Hence, as $r = 1$ we obtain
\begin{equs}
 \sum_{j=1}^4 F^j_2  & = \CD^0(F_2) + 
\sum_{m \leq r +1 }  \frac{1}{m!}  \CI^0_{(\Labhom_1,1)}\left( \lambda_{k_4}\right) \CI^0_{(\Labhom_1,0)}\left( \lambda_{k_5}\right)   \CI^{0}_{(\Labhom_1,0)}\left(  \lambda^{m}_{-k_1+k_2+k_3} \right) .
\end{equs}
Now we are in the position to compute $\bar n_1$: By Definition \ref{dom_freq} we have
\begin{equs}\label{RD1}
 \mathscr{F}_{\text{\tiny{dom}}}\left(\CI_{(\Labhom_2,0)}\left(  \lambda_k F_2\right)\right)& 
= \CP_{\text{\tiny{dom}}}\left( P_{(\Labhom_2,0)}(k) +\mathscr{F}_{\text{\tiny{dom}}}(F_2) \right)\\& 
=  \CP_{\text{\tiny{dom}}}\left(k^2 +\mathscr{F}_{\text{\tiny{dom}}}( F_2)\right)
 \end{equs}
 where we used that $P_{\Labhom_2}(k) = + k^2$ as well as \eqref{Palpha}. Furthermore, we have that
\begin{equs}
 \mathscr{F}_{\text{\tiny{dom}}}\left(F_2\right)&  = \mathscr{F}_{\text{\tiny{dom}}}\left( \CI_{(\Labhom_1,1)}( \lambda_{k_4}) \right) +\mathscr{F}_{\text{\tiny{dom}}}\left(\CI_{(\Labhom_1,0)} \left( \lambda_{-k_1+k_2+k_3} T_1 \right)\right) \\ & +\mathscr{F}_{\text{\tiny{dom}}}\left(\CI_{(\Labhom_1,0)}( \lambda_{k_5})\right)
\end{equs}
with
\begin{equs}
 \mathscr{F}_{\text{\tiny{dom}}}\left( \CI_{(\Labhom_1,1)}( \lambda_{k_4}) \right)& = P_{(\Labhom_1,1)}(k_4) +\mathscr{F}_{\text{\tiny{dom}}}(\one)  = - P_{\Labhom_1}(-k_4) 
= +k_4^2 
\end{equs}
where we used that $P_{\Labhom_1}(k) = - k^2$ as well as \eqref{Palpha}. Similarly we obtain that
\begin{equs}
 \mathscr{F}_{\text{\tiny{dom}}}\left( \CI_{(\Labhom_1,0)}( \lambda_{k_5}) \right) = - k_5^2  .
\end{equs}
Furthermore, we have that
\begin{equs}
\mathscr{F}_{\text{\tiny{dom}}}\left(\CI_{(\Labhom_1,0)} \left( \lambda_{-k_1+k_2+k_3} T_1 \right)\right)
 & = - (-k_1+k_2+k_3)^2 + \mathscr{F}_{\text{\tiny{dom}}}\left(T_1 \right)\\
 & =  - (-k_1+k_2+k_3)^2 + 2 k_1^2
\end{equs}
where we used Example \ref{exRDoNLS}. Hence,
\begin{equs}
 \mathscr{F}_{\text{\tiny{dom}}}\left(F_2\right)& = - (-k_1+k_2+k_3)^2 + 2 k_1^2 + k_4^2  - k_5^2.
\end{equs}
Plugging this into \eqref{RD1} yields
\begin{equs}
 \mathscr{F}_{\text{\tiny{dom}}}\left(\CI_{(\Labhom_2,0)}\left (  \lambda_k F_2 \right)\right) 
& = \CP_{\text{\tiny{dom}}}\left(k^2   - (-k_1+k_2+k_3)^2 + 2 k_1^2 + k_4^2  - k_5^2\right).
\end{equs}
Therefore, $ \mathscr{F}_{\text{low}}\left(\CI_{(\Labhom_2,0)}\left (  \lambda_kF_2\right)\right) = k$ such that
\begin{equs}
\overline n_1 = \text{max}(n,\text{deg}(k^{2})) = \text{max}(n,2).
\end{equs} 
Next we compute $\overline n_2$:
\begin{equs}
 \mathscr{F}_{\text{\tiny{dom}}}\left(\CI_{(\Labhom_2,0)}\left ( \lambda_k F_2^2 \right)\right)
= \CP_{\text{\tiny{dom}}}\left(k^2  +\mathscr{F}_{\text{\tiny{dom}}}\left(F_2^2\right)\right)
\end{equs}
with 
\begin{equs}
 \mathscr{F}_{\text{\tiny{dom}}}\left(F_2^2\right)
&  = \mathscr{F}_{\text{\tiny{dom}}}\left(  \CI_{(\Labhom_1,1)}\left( \lambda_{k_4}\right)\right)+\mathscr{F}_{\text{\tiny{dom}}}\left( \CI_{(\Labhom_1,0)}\left( \lambda_{k_5}\right)  \right)\\&+ \mathscr{F}_{\text{\tiny{dom}}}\left( \CI_{(\Labhom_1,0)}\left(  \lambda_{-k_1+k_2+k_3} \right) \right).
\end{equs}

Note that
\begin{equs}
 \mathscr{F}_{\text{\tiny{dom}}}\left( \CI_{(\Labhom_1,0)}\left({ \lambda_{-k_1+k_2+k_3}^m}  \right)\right)
 & = \mathscr{F}_{\text{\tiny{dom}}}\left( \CI_{(\Labhom_1,0)}\left( \lambda_{-k_1+k_2+k_3} \right)\right)\\
 & = - (-k_1+k_2+k_3)^2.
\end{equs}
Together with the previous computations we can thus conclude that
\begin{equs}
 \mathscr{F}_{\text{\tiny{dom}}}\left(\CI_{(\Labhom_2,0)}\left ( \lambda_k F_2^2 \right)\right)
& =  \CP_{\text{\tiny{dom}}}\left(k^2
 - (-k_1+k_2+k_3)^2 + k_4^2 - k_5^2
\right).
\end{equs}
Therefore, $ \mathscr{F}_{\text{low}}\left(\CI_{(\Labhom_2,0)}\left ( \lambda_k F_2^2 \right)\right) = k$ and 
\begin{equs}
\bar n_2 = \text{max}(n,2).
\end{equs}

Hence
\begin{equs}
\bar n_1 = \bar n_2 = \bar n_3 = \bar n_4 = \text{max}(n,2).
\end{equs}
Furthermore, by Definition  \ref{def:Llow} we have 
\begin{equs}
\mathcal{L}^1_{\text{\tiny{low}}}\left( F_2,n\right)  &  =
\mathcal{L}^1_{\text{\tiny{low}}} \left( \CI_{(\Labhom_1,1)}\left( \lambda_{k_4}\right) ,n\right)
+ \mathcal{L}^1_{\text{\tiny{low}}} \left( \CI_{(\Labhom_1,0)}\left( \lambda_{-k_1+k_2+k_3} T_1\right) ,n\right) \\& + \mathcal{L}^1_{\text{\tiny{low}}} \left( \CI_{(\Labhom_1,0)}\left( \lambda_{k_5}\right) ,n\right)\\
& = \mathcal{L}^1_{\text{\tiny{low}}}(\one) + 
 \mathcal{L}^1_{\text{\tiny{low}}}\left(T_1 ,n\right)
+ \mathcal{L}^1_{\text{\tiny{low}}}(\one) \\
& = 2 +  k^{\text{\tiny{max}}(n,2)}.
\end{equs}
Plugging this into \eqref{L1low} yields that
\begin{equs}
\mathcal{L}^1_{\text{\tiny{low}}}\left(T_2,n\right) & = \mathcal{O}\left( k^{\text{\tiny{max}}(n,2)}\right).
\end{equs}

 \noindent 3) Computation of $\CL^{1}_{\text{\tiny{low}}}(T_3, n)$:\\
 Similarly we obtain that
 \begin{equs}
\mathcal{L}^1_{\text{\tiny{low}}}\left(T_3,n\right) & = \mathcal{O}\left( k^{\text{\tiny{max}}(n,2)}\right).
\end{equs}
 
Plugging the computations 1-3 into \eqref{locoErr2} we 
recover the local error structure \eqref{loc2}.

\end{proof}

 \subsection{Korteweg--de Vries}\label{sec:kdv}
 We consider the Korteweg--de Vries (KdV) equation
\begin{equs}\label{kdv}
\partial_t u + \partial_x^{3} u = \frac12 \partial_x u^2
\end{equs}
with mild solution given by Duhamel's formula
\begin{equs}
u(\tau) = e^{-\tau \partial_x^{3}} v + \frac12 e^{-\tau \partial_x^{3}} \int_{0}^{\tau} e^{ \xi \partial_x^{3}} \partial_x u^2(\xi) d\xi.
\end{equs}
The KdV equation \eqref{kdv} fits into the general framework \eqref{dis} with
\begin{equation*}\label{kgrDo}
\begin{aligned}
 \mathcal{L}\left(\nabla, \frac{1}{\varepsilon}\right)  = i \partial_x^3, \quad \alpha = 1   \quad \text{and}\quad   p(u,\overline u) =  p(u) = i  \frac12 u^2.
 \end{aligned}
\end{equation*} 

Here $ \CL = \lbrace \Labhom_1, \Labhom_2 \rbrace $, $ P_{\Labhom_1} = - \lambda^3 $ and $ P_{\Labhom_2} =  \lambda^3 $ . Then, we denoted by $ \<thick> $ an edge decorated by $ (\Labhom_1,0) $ and by $\<thin>$ an edge decorated by $ (\Labhom_2,0) $.
Following the formalism given in \cite{BHZ}, one can provide the rules that generate the trees obtained by iterating the Duhamel formulation:
\begin{equ}
R(\<thin>) = \{(\<thick>,\<thick>) \}\;, \quad
R(\<thick>) = \{(\<thin>), ()\}\;.
\end{equ}

The general framework~\eqref{genscheme} derived in Section \ref{sec:genScheme} builds the foundation of the  first- and second-order resonances based schemes presented below for the KdV equation~\eqref{kdv}. The structure of the schemes depends  on the regularity of the solution.

\begin{corollary}\label{corKdV} For the KdV equation \eqref{kdv} the general scheme~ \eqref{genscheme} takes at first order the form
\begin{equation}
\begin{aligned}\label{schemeKdV1}
u^{\ell+1} &=  e^{-\tau \partial_x^3} u^\ell + \frac16 \left(e^{-\tau\partial_x^3 }\partial_x^{-1} u^\ell\right)^2 - \frac16 e^{-\tau\partial_x^3} \left(\partial_x^{-1} u^\ell\right)^2\end{aligned}
\end{equation}
with a local error  of order $\mathcal{O}\Big(
\tau^2 \partial_x^2 u
\Big)$
and at  second-order 
\begin{equation}
\begin{aligned}\label{schemeKdV}
u^{\ell+1} &=  e^{-\tau \partial_x^3} u^\ell + \frac16 \left(e^{-\tau\partial_x^3 }\partial_x^{-1} u^\ell\right)^2 - \frac16 e^{-\tau\partial_x^3} \left(\partial_x^{-1} u^\ell\right)^2\\& +\frac{\tau^2}{4} e^{- \tau \partial_x^3}\Psi\big(i \tau \partial_x^2\big)  \Big(\partial_x \Big(u^\ell \partial_x (u^\ell u^\ell)\Big)\Big)
\end{aligned}
\end{equation}
with a local error  of order $\mathcal{O}\Big(
\tau^3 \partial_x^4 u
\Big)$ and  a suitable filter function $\Psi$ satisfying
\begin{equs}
\Psi= \Psi\left(i \tau \partial_x^2 \right), \quad \Psi(0) = 1, \quad \Vert \tau   \Psi \left(i \tau \partial_x^2\right) \partial_x^2 \Vert_r \leq 1.
\end{equs}
\end{corollary}
\begin{remark}
Note that the first-order scheme \eqref{schemeKdV1} which was originally derived in  \cite{HS16}  is optimised as the resonance structure factorises in such a way that all frequencies can be integrated exactly (details are given in the proof). This is in general true, for equations in one dimension with quadratic nonlinearities up to first order.  However, this trick can not be applied to derive second-order methods. The second-order scheme is new and allows us to improve the local error structure $\mathcal{O}\left(
\tau^3 \partial_x^5 u
\right)$ introduced by  the classical Strang splitting scheme \cite{HLR12}. Due to the stability constrain induced by  Burger's nonlinearity it is  preferable   to embed the  resonance structure into the numerical discretisation even for smooth solutions. In Figure~\ref{fig:KdV} we numerically observe the favourable error behaviour of the new resonance based scheme \eqref{schemeKdV} for $\mathcal{C}^\infty$ solutions.
\end{remark}
\begin{proof}
The proof follows the line of argumentation to the analysis for the Schrödinger equation. The construction of the schemes is again based on the general framework~\eqref{genscheme}. Hence, we have to consider for $r = 0,1$
\begin{equs}\label{kdvU}
U_{k}^{n,r}(\tau, v) = \sum_{T \in \CT^{r+2}_{0}(R)} \frac{\Upsilon^{p}( \lambda_kT)(v)}{S(T)} \Pi^n \left( \CD^r(\CI_{(\Labhom_1,0)}( \lambda_k T)) \right)(\tau).
\end{equs} 
Thereby, for the first-order scheme the trees of interests are 
\begin{equs}
\CT_0^2(R) = \lbrace  T_0, T_1, \, k_i \in \Z^{d} \rbrace, \quad T_0 = \one\quad\text{and}\quad T_1 =   \begin{tikzpicture}[scale=0.2,baseline=-5]
\coordinate (root) at (0,0);
\coordinate (tri) at (0,-2);
\coordinate (t1) at (-1,2);
\coordinate (t2) at (1,2);
\draw[kernels2] (t1) -- (root);
\draw[kernels2] (t2) -- (root);
\draw[symbols] (root) -- (tri);
\node[not] (rootnode) at (root) {};
\node[not] (trinode) at (tri) {};
\node[var] (rootnode) at (t1) {\tiny{$ k_{\tiny{1}} $}};
\node[var] (trinode) at (t2) {\tiny{$ k_2 $}};
\end{tikzpicture}  
\end{equs}
where $T_1$ is associated to the first-order iterated integral
\begin{equs}
\CI_1(v^2,s) =   \int_{0}^{s} e^{ s_1 \partial_x^{3}} \partial_x (e^{-s_1 \partial_x^{3}}v)^2 ds_1
\end{equs}
and in symbolic notation  takes the form
\begin{equs}
T_1 = \CI_{(\Labhom_2,0)}\left ( \lambda_{k} F_1\right) \quad F_1 = \CI_{(\Labhom_1,0)}(  \lambda_{k_1})  \CI_{(\Labhom_1,0) }( \lambda_{k_2}) \quad \text{with } k = k_1+k_2.
\end{equs}
For the first-order scheme we set $r=0$  in \eqref{kdvU} such that
\begin{equs}\label{kdvU}
U_{k}^{n,0}(\tau, v) & = \frac{\Upsilon^{p}( \lambda_kT_0 )(v)}{S(T_0)} \Pi^n \left( \CD^0(\CI_{(\Labhom_1,0)}( \lambda_k T_0)) \right)(\tau)\\
& +  \sum_{k = k_1 + k_2} \frac{\Upsilon^{p}( \lambda_kT_1 )(v)}{S(T_1)} \Pi^n \left( \CD^0(\CI_{(\Labhom_1,0)}( \lambda_k T_1)) \right)(\tau).
\end{equs}

For the first term we readily obtain that
\begin{equs}
 \frac{\Upsilon^{p}( \lambda_kT_0)(v)}{S(T_0)} \Pi^n \left( \CD^0(\CI_{(\Labhom_1,0)}( \lambda_k T_0)) \right)(\tau)
 = e^{-i \tau k^3} \hat{v}_k.
\end{equs}
It remains to compute the second term.  Note that thanks to \eqref{recursive_pi_r} we have that
\begin{equs}\label{kdvST}
 \Pi^n \left( \CD^0(\CI_{(\Labhom_1,0)}( \lambda_k T_1)) \right)(\tau) &= e^{i \tau P_{\Labhom_1}(k)}\Pi^n (\CD^0(T_1))(\tau)  \\& = e^{-i \tau k^3} \Pi^n (\CI_{(\Labhom_2,0)}^0( \lambda_k F_1) )(\tau)
 \\ & =  e^{-i \tau k^3} \mathcal{K}_{(\Labhom_2,0)}^{k,0} \left( \Pi^n (\CD^{-1}(F_1))\right)(\tau).
\end{equs}
By the product formula we furthermore obtain that
\begin{equs}
\Pi^n (\CD^{-1}(F_1))(\tau)&  = \Pi^n \left( \CI^{-1}_{(\Labhom_1,0)} ( \lambda_{k_1})\right) (\tau) \Pi^n \left( \CI^{-1}_{(\Labhom_1,0)}(  \lambda_{k_2})\right) (\tau)
 \\ & = e^{-i \tau k_1^3} e^{-i \tau k_2^3}.
\end{equs}
Plugging the above relation into \eqref{kdvST} yields that
\begin{equs}
 \Pi^n \left( \CD^0(\CI_{(\Labhom_1,0)}( \lambda_k T_1)) \right)(\tau)
 =  e^{-i \tau k^3} \mathcal{K}_{(\Labhom_2,0)}^{k,0} \left( e^{i \xi (-k_1^3 - k_2^3)} \right)(\tau).
\end{equs}
Next  we observe that
\begin{equs}
 P_{(\Labhom_2,0)}(k) - k_1^3-k_2^3  =  k^3 - k_1^3 - k_2^3 = 3 k_1 k_2 (k_1+k_2)
\end{equs}
such that
\begin{equs}
\frac{1}{ P_{(\Labhom_2,0)}(k) - k_1^3-k_2^3  } 
\end{equs}
can be mapped back to physical space. Therefore, we set
\begin{equs}
\mathcal{L}_{\tiny\text{dom}} =  P_{(\Labhom_2,0)}(k) - k_1^3-k_2^3  = 3 k_1 k_2 (k_1+k_2)
\end{equs}
and integrate all frequencies exactly. This implies
\begin{equs}
 \Pi^n \left( \CD^0(\CI_{(\Labhom_1,0)}( \lambda_k T_1)) \right)(\tau)
 & = e^{-i \tau k^3} \frac{i (k_1+k_2) }{3i  k_1 k_2 (k_1+k_2)} \left(e^{i \tau (k^3 - k_1^3 -k_2^3)}-1\right)\\
 & =\frac{1 }{3  k_1 k_2 } \left(e^{- i \tau ( k_1^3 + k_2^3)}- e^{-i \tau k^3} \right).
 \end{equs}
Together with \eqref{kdvU} this yields the scheme \eqref{schemeKdV1}.

For the second-order scheme we need to take into account the following trees
\begin{equs}\CT_0^3(R) = \lbrace  T_0, T_1, T_2, \, k_i \in \Z^{d} \rbrace, \quad 
T_2 = \begin{tikzpicture}[scale=0.2,baseline=-5]
\coordinate (root) at (0,0);
\coordinate (tri) at (0,-2);
\coordinate (t1) at (-1,2);
\coordinate (t11) at (-2,4);
\coordinate (t12) at (-3,6);
\coordinate (t13) at (-1,6);
\coordinate (t2) at (1,2);
\draw[kernels2] (t11) -- (t13);
\draw[kernels2] (t11) -- (t12);
\draw[kernels2] (t1) -- (root);
\draw[symbols] (t1) -- (t11);
\draw[kernels2] (t2) -- (root);
\draw[symbols] (root) -- (tri);
\node[not] (rootnode) at (root) {};
\node[not] (trinode) at (tri) {};
\node[not] (trinode) at (t1) {};
\node[var] (rootnode) at (t12) {\tiny{$ k_{\tiny{1}} $}};
\node[var] (rootnode) at (t13) {\tiny{$ k_{\tiny{2}} $}};
\node[var] (trinode) at (t2) {\tiny{$ k_3 $}};
\end{tikzpicture} 
\end{equs}
which is associated to the second-order iterated integral
 \begin{equs}
\CI_2(v^3,s) =  \int_{0}^{s} e^{ s_1 \partial_x^{3}} \partial_x \left( (e^{-s_1 \partial_x^{3}}v)  e^{-s_1 \partial_x^{3}} \int_0^{s_1} e^{s_2 \partial_x^{3}}\partial_x (e^{-s_2 \partial_x^{3}}v)^2   ds_2 \right) ds_1.
\end{equs}
Then one can proceed as in the second-order schemes for the Schrödinger equation. We omit the details here.
The local error analysis is then given by Theorem \ref{thm:genloc} noting that $\alpha = 1$ and 
\begin{equs}
\mathcal{L}_{\tiny\text{low}}^0(T_1,\cdot) = k^{1+\alpha}, \quad 
\mathcal{L}_{\tiny\text{low}}^1(T_1,\cdot) = \mathcal{L}_{\tiny\text{low}}^1(T_2,\cdot) = k^{2(1+\alpha)}.
\end{equs}
\end{proof}
\subsection{Klein--Gordon}\label{sec:kgr}
In this section we apply the general  framework \eqref{genscheme} to the  Klein--Gordon equation
\begin{equation}\label{kgo}
\partial_t^2 z - \Delta z + \frac{1}{\varepsilon^2} z = \vert z\vert^2 z, \quad z(0,x) = \gamma(x),\quad \partial_t z(0,x) = \frac{1}{\varepsilon^2} \delta(x).
\end{equation}
Here we are in particular interested in resolving the highly oscillatory, so-called non-relativistic, structure of the PDE when the speed of light $c = \frac{1}{\varepsilon}$ formally tends to infinity.

Via  the transformation  $u = z - i \varepsilon \nab^{-1} \partial_t z$ we express \eqref{kgo}   in its first-order form
\begin{align}\label{kgr}
& i \partial_t u = -\frac{1}{\varepsilon}\nab u + \frac{1}{\varepsilon}\nab^{-1} \textstyle \frac18 (u + \overline u)^3, \qquad \frac{1}{\varepsilon}\nab = \frac{1}{\varepsilon}\sqrt{\frac{1}{\varepsilon^2}-\Delta}.  
\end{align}
  The first-order form  \eqref{kgr} casts into the general form \eqref{dis} with
\begin{equation*}\label{kgrDo}
\begin{aligned}
\mathcal{L}\left(\nabla, \frac{1}{\varepsilon}\right)  =  \frac{1}{\varepsilon}\nab, \quad \alpha = 0  \quad \text{and}\quad  p(u,\overline u) =  \frac{1}{\varepsilon}\nab^{-1} \textstyle \frac18 (u + \overline u)^3.
 \end{aligned}
\end{equation*} 
The leading operator $\mathcal{L}\left(\nabla, \frac{1}{\varepsilon}\right) $ thereby triggers oscillations of type
\begin{equation*}\label{kgosc}
\sum_{\ell \in \Z } e^{ i t \ell \frac{1}{\varepsilon^2}}
\end{equation*} which can be formally seen by the Taylor series expansion 
$$
\mathcal{L}\left(\nabla, \frac{1}{\varepsilon}\right)  =   \frac{1}{\varepsilon} \nab = \frac{1}{\varepsilon^2} - \frac12 \Delta + \mathcal{O}\left(\frac{\Delta^2}{\varepsilon^2}\right).
$$
 In oder to determine these dominant oscillations we define the non-oscillatory operators
 \begin{equs}\label{Bdelta}
&  \mathcal{B}_\Delta = \frac{1}{\varepsilon} \nab - \frac{1}{\varepsilon^2}, \qquad
  \mathcal{B}_\Delta (k) = \frac{1}{\varepsilon^2} \sqrt{1 + \frac{k^2}{\varepsilon^2}} - \frac{1}{\varepsilon^2} \\
&    \mathcal{C}_\Delta =  \frac{1}{\varepsilon}\nab^{-1}, \qquad 
  \mathcal{C}_\Delta(k)  = \frac{1}{\sqrt{1+ \frac{k^2}{\varepsilon^2}}}
 \end{equs}
 which both can be uniformly bounded  in $\varepsilon$ thanks to the estimates $\Vert \mathcal{B}_\Delta w \Vert \leq \frac12 \Vert \Delta w\Vert$, $\frac{1}{1+x^2} \leq 1$. The latter motivates us to rewrite the oscillatory equation \eqref{kgr} in the following form 
\begin{align}\label{kgrB}
& i \partial_t u = - \left( \frac{1}{\varepsilon^2} +\mathcal{B}_\Delta\right) u + \mathcal{C}_\Delta \textstyle \frac18 (u + \overline u)^3.
 \end{align}

 Here $ \CL = \lbrace \Labhom_1, \Labhom_2 \rbrace $, $ P_{\Labhom_1} = -\left(  \frac{1}{\varepsilon^2} + \mathcal{B}_\Delta( \lambda) \right) $ and $ P_{\Labhom_2} =  \frac{1}{\varepsilon^2} + \mathcal{B}_\Delta( \lambda) $ . Then, we denoted by $ \<thick> $ an edge decorated by $ (\Labhom_1,0) $, $ \<thick2> $ an edge denoted by $ (\Labhom_1,1) $ by $\<thin>$ an edge decorated by $ (\Labhom_2,0) $ and by $\<thin2>$ an edge decorated by $ (\Labhom_2,1) $.
The rules that generate the trees obtained by iterating the Duhamel formulation are given by:
\begin{equs}
R(\<thin>)  = R(\<thin2>) = \{ (\<thick>,\<thick>,\<thick>), (\<thick>,\<thick>,\<thick2>), (\<thick>,\<thick2>,\<thick2>), 
(\<thick2>,\<thick2>,\<thick2>) \}\;, 
\quad
R(\<thick>)   = \{(\<thin>) , ()\}\;,
\quad  R(\<thick2>) = \{(\<thin2>) , ()\}\;.
\end{equs}

The general framework~\eqref{genscheme} derived in Section \ref{sec:genScheme}  builds the foundation of the  first-order resonance based schemes presented below for the Klein--Gordon equation equation~\eqref{kgr}. 

\begin{corollary}\label{corKGR} For the Klein--Gordon equation \eqref{kgr} the general scheme~ \eqref{genscheme} takes at first order the form
\begin{equs}\label{kgr1}
 u^{\ell+1} &  = e^{i \tau \left(\frac{1}{\varepsilon^2} + \mathcal{B}_\Delta\right)}  u^\ell - \tau \frac{3i}{8}  e^{i \tau \left(\frac{1}{\varepsilon^2} + \mathcal{B}_\Delta\right)}\mathcal{C}_\Delta  \vert u^\ell\vert^2 u^\ell\\
&  - \tau \frac{i}{8}   e^{i \tau \left(\frac{1}{\varepsilon^2} + \mathcal{B}_\Delta\right)} \mathcal{C}_\Delta \Big(
\varphi_1\left(2i \frac{1}{\varepsilon^2} \tau\right) \left(u^\ell \right)^3+ 3 \varphi_1\left(-2i \frac{1}{\varepsilon^2} \tau\right) \left \vert u^\ell\right\vert^2 u^\ell \\& + \varphi_1\left(-4 i \frac{1}{\varepsilon^2} \tau\right) \left(\overline{u^\ell}\right)^3
\Big)
\end{equs}
with a local error $\mathcal{O}\left( \tau^2 \Delta u\right)$ and the filter function $\varphi_1(\sigma) = \frac{e^\sigma-1}{\sigma}$.

 If we allow  step size restrictions $\tau = \tau\left(\frac{1}{\varepsilon^2}\right)$ the general scheme~ \eqref{genscheme} takes the simplified form 

\begin{equs}\label{kgrSimp}
 u^{\ell+1}   = e^{i \tau \left(\frac{1}{\varepsilon^2} + \mathcal{B}_\Delta\right)}  u^\ell   - \tau \frac{i}{8}   e^{i \tau \left(\frac{1}{\varepsilon^2} + \mathcal{B}_\Delta\right)} \mathcal{C}_\Delta \left(u^\ell + \overline{u^\ell}\right)^3
\end{equs}
with a local error of order $ \mathcal{O}\left( \frac{\tau^2}{\varepsilon^{2}} u \right) +  \mathcal{O}(\tau^2 \Delta u )$.
\end{corollary}
\begin{remark}
With the general framework we recover at first-order exactly the resonance based first-order scheme \eqref{kgr1} derived in \cite{BFS17}. If we allow for step size restrictions, we recover 
a classical approximation  the classical local error structure $\mathcal{O}\left( \frac{\tau^2}{\varepsilon^2}\right)$ introduced by   classical Strang splitting or Gautschi-type schemes (\cite{BD}).
\end{remark}

\begin{proof}
From the general framework~\eqref{genscheme} (with $r=0$) we obtain that
\begin{equs}\label{kdvU}
U_{k}^{n,0}(\tau, v) = \sum_{T \in \CT^{2}_{0}(R)} \frac{\Upsilon^{p}( \lambda_kT)(v)}{S(T)} \Pi^n \left( \CD^0(\CI_{(\Labhom_1,0)}( \lambda_k T)) \right)(\tau)
\end{equs}
with
\begin{equs}
\CT_0^2(R) & = \lbrace  T_0, T_1, T_2, T_3, T_4,  \, k_i \in \Z^{d} \rbrace, \quad T_0 = \one \\ 
 T_1 & = \begin{tikzpicture}[scale=0.2,baseline=-5]
\coordinate (root) at (0,0);
\coordinate (tri) at (0,-2);
\coordinate (t1) at (-2,2);
\coordinate (t2) at (2,2);
\coordinate (t3) at (0,3);
\draw[kernels2] (t1) -- (root);
\draw[kernels2] (t2) -- (root);
\draw[kernels2] (t3) -- (root);
\draw[symbols] (root) -- (tri);
\node[not] (rootnode) at (root) {};t
\node[not] (trinode) at (tri) {};
\node[var] (rootnode) at (t1) {\tiny{$ k_{\tiny{1}} $}};
\node[var] (rootnode) at (t3) {\tiny{$ k_{\tiny{2}} $}};
\node[var] (trinode) at (t2) {\tiny{$ k_3 $}};
\end{tikzpicture}, \quad T_2 = \begin{tikzpicture}[scale=0.2,baseline=-5]
\coordinate (root) at (0,0);
\coordinate (tri) at (0,-2);
\coordinate (t1) at (-2,2);
\coordinate (t2) at (2,2);
\coordinate (t3) at (0,3);
\draw[kernels2,tinydots] (t1) -- (root);
\draw[kernels2] (t2) -- (root);
\draw[kernels2] (t3) -- (root);
\draw[symbols] (root) -- (tri);
\node[not] (rootnode) at (root) {};t
\node[not] (trinode) at (tri) {};
\node[var] (rootnode) at (t1) {\tiny{$ k_{\tiny{1}} $}};
\node[var] (rootnode) at (t3) {\tiny{$ k_{\tiny{2}} $}};
\node[var] (trinode) at (t2) {\tiny{$ k_3 $}};
\end{tikzpicture},
\quad T_3 = \begin{tikzpicture}[scale=0.2,baseline=-5]
\coordinate (root) at (0,0);
\coordinate (tri) at (0,-2);
\coordinate (t1) at (-2,2);
\coordinate (t2) at (2,2);
\coordinate (t3) at (0,3);
\draw[kernels2,tinydots] (t1) -- (root);
\draw[kernels2] (t2) -- (root);
\draw[kernels2,tinydots] (t3) -- (root);
\draw[symbols] (root) -- (tri);
\node[not] (rootnode) at (root) {};t
\node[not] (trinode) at (tri) {};
\node[var] (rootnode) at (t1) {\tiny{$ k_{\tiny{1}} $}};
\node[var] (rootnode) at (t3) {\tiny{$ k_{\tiny{2}} $}};
\node[var] (trinode) at (t2) {\tiny{$ k_3 $}};
\end{tikzpicture}, \quad T_4 = \begin{tikzpicture}[scale=0.2,baseline=-5]
\coordinate (root) at (0,0);
\coordinate (tri) at (0,-2);
\coordinate (t1) at (-2,2);
\coordinate (t2) at (2,2);
\coordinate (t3) at (0,3);
\draw[kernels2,tinydots] (t1) -- (root);
\draw[kernels2,tinydots] (t2) -- (root);
\draw[kernels2,tinydots] (t3) -- (root);
\draw[symbols] (root) -- (tri);
\node[not] (rootnode) at (root) {};t
\node[not] (trinode) at (tri) {};
\node[var] (rootnode) at (t1) {\tiny{$ k_{\tiny{1}} $}};
\node[var] (rootnode) at (t3) {\tiny{$ k_{\tiny{2}} $}};
\node[var] (trinode) at (t2) {\tiny{$ k_3 $}};
\end{tikzpicture}. 
\end{equs}
Let us carry out the computation for the tree $T_1$ which in symbolic notation takes the form
\begin{equs}
T_1 & = \CI_{(\Labhom_2,0)} ( \lambda_k F_1),  \quad  k = k_1+k_2+k_3
\\
F_1 & = \CI_{(\Labhom_1,0)} ( \lambda_{k_1}) \CI_{(\Labhom_1,0)}(  \lambda_{k_2}) \CI_{(\Labhom_1,0)} ( \lambda_{k_3}).
\end{equs}
Thanks to  \eqref{recursive_pi_r} and the fact that
$
P_{(\Labhom_1,0)}\left(k,\frac{1}{\varepsilon^2}\right)
=  \frac{1}{\varepsilon^2} + \mathcal{B}_\Delta(k) 
$ 
we obtain
\begin{equs}
 \Pi^n \left( \CD^r(\CI_{(\Labhom_1,0)}( \lambda_k T_1)) \right)(\tau)
 & = e^{i \tau P_{(\Labhom_1,0)}(k,\frac{1}{\varepsilon})} (\Pi^n \CD^0 (T_1))(\tau) 
 \\& = e^{i \tau \left( \frac{1}{\varepsilon^2} + \mathcal{B}_\Delta(k)\right ) }  (\Pi^n  \CI^0_{(\Labhom_2,0)} ( \lambda_k F_1))(\tau)
\\&  = e^{i \tau \left( \frac{1}{\varepsilon^2} + \mathcal{B}_\Delta(k)\right ) } 
\mathcal{K}_{(\Labhom_2,0)}^{k,0}\left( \Pi^n \CD^{-1}(F_1) \right)(\tau).
\end{equs}
By the product rule we furthermore have that 
\begin{equs}
\Pi^n \CD^{-1}(F_1)(\xi) &  = \left( \Pi^n \CI^{-1}_{(t_1,0)}  \lambda_{k_1}\right)(\xi)  \left( \Pi^n \CI^{-1}_{(t_1,0)}  \lambda_{k_2}\right)(\xi)  \left(  \Pi^n \CI^{-1}_{(t_1,0)}  \lambda_{k_3} \right)(\xi)   \\& = e^{i \xi  \left(3 \frac{1}{\varepsilon^2} + \mathcal{B}_\Delta(k_1)+ \mathcal{B}_\Delta(k_2)+ \mathcal{B}_\Delta(k_3)\right ) }
\end{equs}
such that
\begin{equs}\label{KGC1}
 \Pi^n & \left( \CD^0(\CI_{(\Labhom_1,0)}( \lambda_k T_1)) \right)(\tau)
 \\ & = e^{i \tau \left( \frac{1}{\varepsilon^2} + \mathcal{B}_\Delta(k)\right ) } 
\mathcal{K}_{(\Labhom_2,0)}^{k,0}\left(  e^{i \xi  \left(3 \frac{1}{\varepsilon^2} + \mathcal{B}_\Delta(k_1)+ \mathcal{B}_\Delta(k_2)+ \mathcal{B}_\Delta(k_3)\right ) } \right)(\tau).
 \end{equs}
 Definition \ref{Taylor_exp} together with Remark \ref{rem:epsi} and the observation
$
P_{(\Labhom_2,0)}\left(k,\frac{1}{\varepsilon^2}\right)
= - \left(  \frac{1}{\varepsilon^2} + \mathcal{B}_\Delta(k) \right)
$ 
implies that
 \begin{equs}\label{domKg}
 \mathcal{L}_{\tiny \text{dom}} & = \mathcal{P} _{\tiny \text{dom}}
 \left( 
 P_{(\Labhom_2,0)}\left(k,\frac{1}{\varepsilon^2}\right) + P
 \right)
 \\& =  \mathcal{P} _{\tiny \text{dom}}
 \left( 
2 \frac{1}{\varepsilon^2} - \mathcal{B}_\Delta(k) + \mathcal{B}_\Delta(k_1)+ \mathcal{B}_\Delta(k_2)+ \mathcal{B}_\Delta(k_3 )
\right)
 = 2 \frac{1}{\varepsilon^2} .
 \end{equs}
 Hence,
 \begin{equs}
 \mathcal{K}_{(\Labhom_2,0)}^{k,0}\left(  e^{i \xi  \left(3 \frac{1}{\varepsilon^2} + \mathcal{B}_\Delta(k_1)+ \mathcal{B}_\Delta(k_2)+ \mathcal{B}_\Delta(k_3)\right ) } \right)(\tau)
 & = -i  \frac{1}{8} \mathcal{C}_\Delta(k)    \frac{e^{2 i \tau   \frac{1}{\varepsilon^2}} -1}{2 i  \frac{1}{\varepsilon^2}}.
 \end{equs}
Plugging this into \eqref{KGC1} yields that 
\begin{equs}
 \Pi^n  \left( \CD^0(\CI_{(\Labhom_1,0)}( \lambda_k T_1)) \right)(\tau)
  =  - i \tau   \frac{1}{8} \mathcal{C}_\Delta(k)    \varphi_1\left(2 i \tau   \frac{1}{\varepsilon^2}\right)  e^{i \tau \left( \frac{1}{\varepsilon^2} + \mathcal{B}_\Delta(k)\right ) }.
 \end{equs}
 Together with the observation that
 \begin{equs}
 \frac{\Upsilon^{p}( \lambda_kT_1)(v)}{S(T_1)} =  \hat{v}_{k_1} \hat v_{k_2} \hat v_{k_3}  \end{equs}
 we obtain for the first tree $T_1$ in Fourier space that
 \begin{equs}
 \frac{\Upsilon^{p}( \lambda_kT)(v)}{S(T)} &  \Pi^n  \left( \CD^0(\CI_{(\Labhom_1,0)}( \lambda_k T_1)) \right)(\tau) \\& =  - i\tau  \frac{1}{8} \mathcal{C}_\Delta(k)     \hat{v}_{k_1} \hat v_{k_2} \hat v_{k_3} \varphi_1\left(2 i \tau   \frac{1}{\varepsilon^2}\right)  e^{i \tau \left( \frac{1}{\varepsilon^2} + \mathcal{B}_\Delta(k)\right ) }.
 \end{equs}
In physical space  the latter takes the form
 \begin{equs}\label{termi}
\mathcal{F}^{-1} \left(\frac{\Upsilon^{p}( \lambda_kT)(v)}{S(T)}   \Pi^n  \left( \CD^0(\CI_{(\Labhom_1,0)}( \lambda_k T_1)) \right)(\tau) \right)\\ = -i  \frac{1}{8}    \tau  e^{i \tau \left( \frac{1}{\varepsilon^2} + \mathcal{B}_\Delta \right ) }  \mathcal{C}_\Delta\varphi_1\left(2 i \tau   \frac{1}{\varepsilon^2}\right) v^3
 \end{equs}
 with $ \mathcal{B}_\Delta $ and $ \mathcal{C}_\Delta  $ defined in \eqref{Bdelta}. The term \eqref{termi} is exactly the third term (corresponding to $(u^\ell)^3$) in the first-order scheme \eqref{kgr1}. The other terms in the scheme can be computed in a similar way with the aid of the trees $T_j$, $j = 1,2,3,4$.  The local error is then given by Theorem \ref{thm:genloc}. For instance, the local error   introduced by the approximation of $T_1$ reads
 \begin{equs}
 \mathcal{O}\left(\tau^2 \left(- \mathcal{B}_\Delta(k) + \mathcal{B}_\Delta(k_1)+ \mathcal{B}_\Delta(k_2)+ \mathcal{B}_\Delta(k_3 )\right) \hat v_{k_1} \hat v_{k_2} \hat v_{k_3}\right) = \mathcal{O}(\tau^2 \Delta v^3).
 \end{equs}
 The other approximations obey a similar error structure.  
 
The simplified scheme \eqref{kgrSimp} on the other hand is constructed by carrying out a Taylor series expansion of the full operator $\mathcal{L}_{\tiny\text{dom}} + \mathcal{L}_{\tiny\text{low}} $. For instance, when calculating $\mathcal{K}$, instead of integrating the dominant part \eqref{domKg}  for the first tree $T_1$, exactly, we Taylor expand the full operator 
\begin{equs}
P_{(\Labhom_2,0)}\left(k,\frac{1}{\varepsilon^2}\right) + P
= 
2 \frac{1}{\varepsilon^2} - \mathcal{B}_\Delta(k) + \mathcal{B}_\Delta(k_1)+ \mathcal{B}_\Delta(k_2)+ \mathcal{B}_\Delta(k_3 ).
\end{equs}
This implies a local error structure of type
\begin{equs}
 \mathcal{O}\left( \frac{\tau^2}{\varepsilon^{2}} u \right) +  \mathcal{O}(\tau^2 \Delta u ).
 \end{equs}
We omit further details here.

\end{proof}
\begin{remark}
Our framework \eqref{genscheme} allows us to also derive  second- and higher-order schemes for the Klein--Gordon equation \eqref{kgr}. The trees have the same shape as for the cubic Schrödinger equation \eqref{nls}, but many more trees $T\in \CT^{3}_{0}(R)$  are needed for the description. With our framework we can in particular recover the first- and second-order uniformly accurate method proposed in \cite{BFS17}.

\end{remark}

\subsection{Numerical experiments}\label{sec:num}
We underline the favorable error behavior  of the new resonance based schemes compared to  classical approximation techniques in case of non-smooth solutions. We choose $M=2^{8}$ spatial grid points and carry out the simulations up to $T = 0.1$. 
\begin{example}[Schrödinger]
In Figure \ref{fig:NLS} we compare the convergence of the new resonance based approach with classical splitting and exponential integration schemes in case  of the Schrödinger equation \eqref{nls} with smooth and non-smooth solutions. The numerical experiments underline the favorable error behavior of the resonance based schemes presented in Corollary \ref{corNLS2} in case of non-smooth solutions.  While the second-order Strang splitting faces high oscillations in the error causing severe order reduction, the second-order resonance based scheme maintains its second-order convergence for less regular solutions.
\end{example}
\begin{example}[Kortweg--de Vries]
Figure \ref{fig:KdV} underlines the  preferable  choice of embedding the resonance structure into the numerical discretisation even for smooth solutions of the KdV equation \eqref{kdv}. While the second-order classical exponential integrator suffers from spikes in the error when hitting certain (resonant) time steps, the second-order resonance based scheme presented in Corollary \ref{corKdV} allows for full-order convergence without any oscillations.
\end{example}
\begin{figure}[h!]
\begin{subfigure}[c]{0.48\textwidth}
\includegraphics[width=1\textwidth]{NewH2.eps}
\subcaption{ $H^2$ data}
\end{subfigure}
\begin{subfigure}[c]{0.48\textwidth}
\includegraphics[width=1\textwidth]{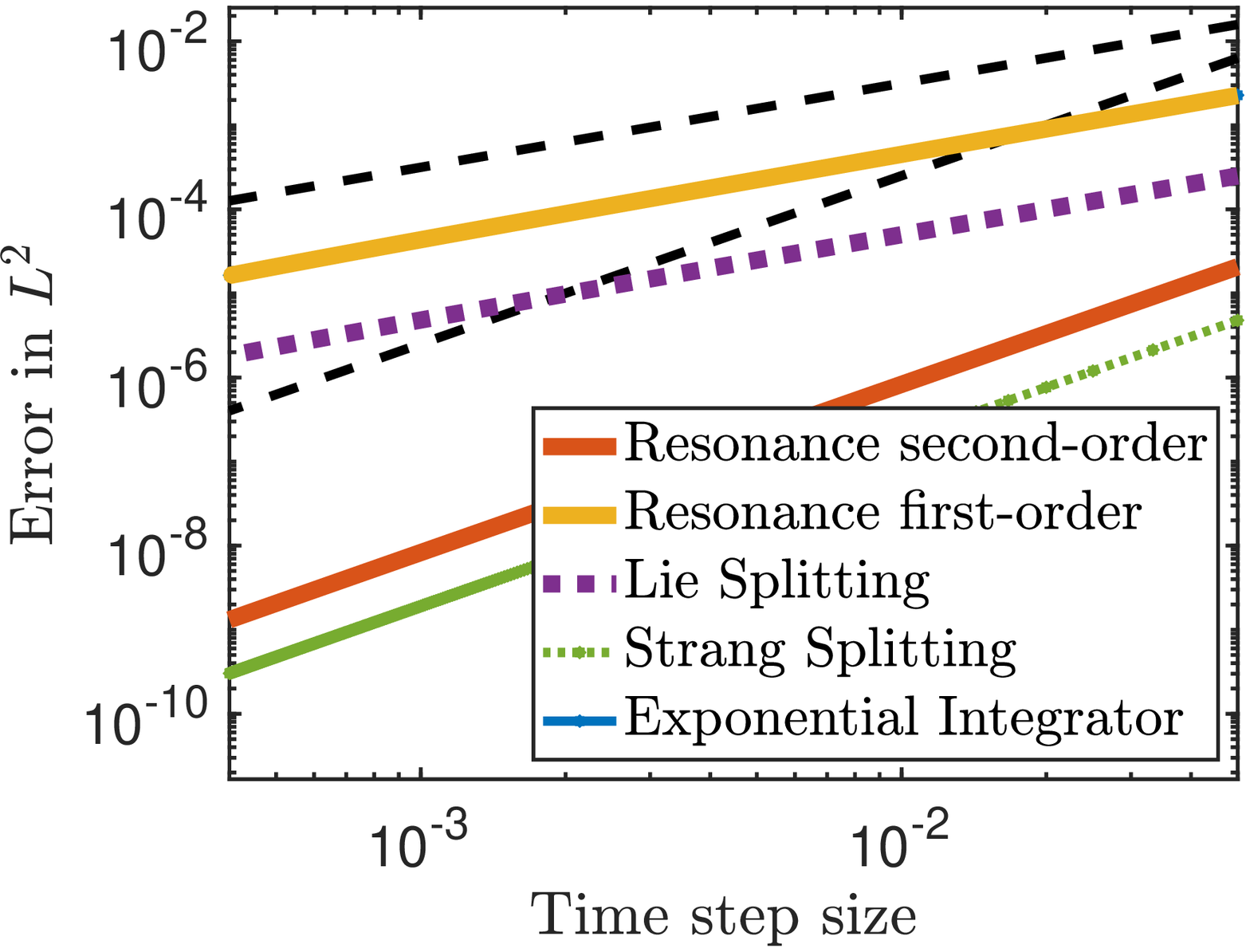}
\subcaption{$\mathcal{C}^\infty$ data}
\end{subfigure}
\caption{Error versus step size (double logarithmic plot). Comparison of classical and resonance based schemes for the Schrödinger equation for smooth (right picture) and non-smooth (left picture) solutions.}\label{fig:NLS}
\end{figure}

\begin{figure}[h!]

\centering \includegraphics[width=0.48\textwidth]{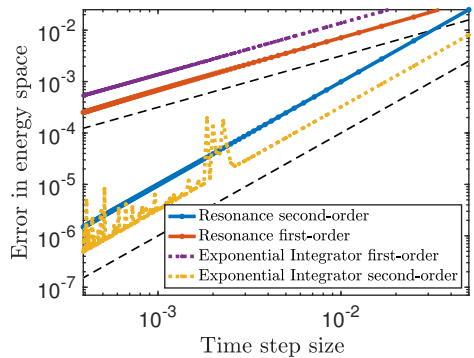}

\caption{Error versus step size (double logarithmic plot). Comparison of classical and resonance based schemes for the KdV equation with smooth data in $\mathcal{C}^\infty$.}\label{fig:KdV}
\end{figure}

\newpage

\end{document}